\documentclass[reqno,11pt]{amsart}
\usepackage{amsmath,amssymb,mathrsfs,amsthm,amsfonts}
\usepackage[inline]{enumitem}
\usepackage[usenames,dvipsnames]{xcolor}
\usepackage{hyperref}
\usepackage[percent]{overpic}
\usepackage{comment}
\usepackage{stmaryrd}
\usepackage{algorithm}
\usepackage{algorithmic}
\usepackage{pgfplots}
\usepackage{subcaption}
\usepackage{tikz}
\usepackage{ulem}
\usetikzlibrary{calc}
\usetikzlibrary{arrows.meta,backgrounds}
\usepackage{multirow,array,longtable,booktabs}
\usetikzlibrary{arrows}
\usetikzlibrary{shapes.multipart}

\hypersetup{
	colorlinks=true, linkcolor=blue,
	citecolor=ForestGreen
}
\usepackage[paper=letterpaper,margin=1in]{geometry}
\DeclareMathAlphabet{\mathpzc}{OT1}{pzc}{m}{it}

\newtheorem{theorem}{Theorem}[section]
\newtheorem{lemma}[theorem]{Lemma}
\newtheorem{proposition}[theorem]{Proposition}
\newtheorem{corollary}[theorem]{Corollary}
\theoremstyle{definition}
\newtheorem{definition}[theorem]{Definition}

\newtheorem{remark}[theorem]{Remark}
\numberwithin{equation}{section}
\allowdisplaybreaks
\usepackage{acronym}

\acrodef{KPZ}{Kardar--Parisi--Zhang}
\acrodef{SHE}{Stochastic Heat Equation}
\acrodef{LDP}{Large Deviation Principle}

\renewcommand{\Pr}{\mathbb{P}}
\renewcommand{\P}{\mathbf{P}}
\newcommand{\2}{\textcolor{black}{2}}
\newcommand{\p}{\mathfrak{p}}
\newcommand{\Ex}{\mathbb{E}}
\newcommand{\E}{\mathbf{E}}	
\newcommand{\ind}{\mathbf{1}}	
\renewcommand{\H}{\mathcal{H}}
\newcommand{\hh}{H}
\newcommand{\mc}{\mathbf{w}}

\renewcommand{\wp}{\mathrm{whp}}
\makeatletter
\renewcommand{\ll}{\llbracket}
\newcommand{\rr}{\rrbracket}
\newcommand{\pp}{\mathfrak{p}}
\newcommand{\Br}{\mathfrak{B}}

\newcommand{\norm}[1]{\Vert#1\Vert}
\newcommand{\rev}{\textcolor{red!50!black}}
\newcommand{\0}{0}

\newcommand{\R}{\mathbb{R}} 
\newcommand{\Z}{\mathbb{Z}} 
\newcommand{\e}{\varepsilon}
\newcommand{\calZ}{\mathcal{Z}}
\renewcommand{\L}{\mathcal{L}}
\newcommand{\m}{\mathsf}
\renewcommand{\hat}{\widehat}
\newcommand{\til}{\widetilde}
\renewcommand{\bar}{\overline}
\newcommand{\sd}[1]{{\color{red}\ttfamily\upshape\small[#1]\color{black}}}

\usepackage{graphicx}
\pgfplotsset{compat=1.18}

\title[Scaling limit of the half-space KPZ line ensemble]{The half-space KPZ line ensemble and its scaling limit}

\author[S.\ Das]{Sayan Das}
\address{S.\ Das,
	Department of Mathematics, University of Chicago,
	\newline\hphantom{\quad \ \ S. Das}
	5734 S.~University Avenue, Chicago, IL 60637, USA
}
\email{sayan.das@columbia.edu}
\author[C.\ Serio]{Christian Serio}
\address{C.\ Serio,
	Department of Mathematics, Stanford University,
	\newline\hphantom{\quad \ \ C. \ Serio}
	450 Jane Stanford Way,
	Stanford, CA 94305, USA
}
\email{cdserio@stanford.edu}

\begin{document}
	\begin{abstract}
             For each $\alpha \in \R, t\ge 1$, we show that there exists a unique  $\mathbb{N}$-indexed line ensemble of random continuous curves $\mathbb{R}_{\le 0} \to \mathbb{R}$ with the following properties:
            \begin{enumerate}
                \item The top curve is distributed as the time-$t$ Cole--Hopf solution to the half-space KPZ equation with narrow wedge initial condition and Neumann boundary condition with parameter $\alpha$.
                \item The line ensemble satisfies a one-sided resampling invariance property, involving softly non-intersecting Brownian motions with an attractive potential between pairs at the boundary.
            \end{enumerate}
             We call this object the \textit{half-space KPZ line ensemble}.

             For $\alpha=\mu t^{-1/3}$ with $\mu \in \R$ fixed (critical regime) and for $\alpha>0$ fixed (supercritical regime), we show that the half-space KPZ line ensemble near the boundary is tight under 1:2:3 KPZ scaling as $t\to\infty$. Moreover, all subsequential limits approximate a parabola 
             and enjoy a one-sided Brownian Gibbs property, described by non-intersecting Brownian motions with pairwise interaction at the boundary. 
             In the supercritical case, one of the main technical contributions of this paper is the first rigorous demonstration of a structure involving \textit{pairwise pinned} Brownian motions, previously conjectured in \cite{bcd}. Away from the boundary, we show that the half-space KPZ line ensemble converges to the (full-space) Airy line ensemble under 1:2:3 KPZ scaling as $t\to\infty$.
             
             {A key ingredient in our arguments, which may be of independent interest, is a novel identity relating the full- and half-space KPZ equations, which may be valuable in deducing various geometric properties of the latter.}
            \end{abstract}
	
	
	\maketitle
	{
			\hypersetup{linkcolor=black}
			\setcounter{tocdepth}{1}
			\tableofcontents
		}

\section{Introduction}

The half-space Kardar--Parisi--Zhang (HSKPZ) equation on the negative half-line\footnote{We consider the equation on the negative half-line rather than the positive one, as the one-sided Gibbs property stated in Theorem \ref{thm1}\ref{12b} is easier to formulate and visualize in this setting.} is the stochastic PDE
\begin{equation}
    \label{halfkpz}
    \begin{aligned}
    & \partial_t \mathcal{H}^\alpha = \frac12\partial_{xx}\mathcal{H}^\alpha+\frac12(\partial_x \mathcal{H}^\alpha)^2+\xi, \qquad \mathcal{H}^\alpha=\mathcal{H}^\alpha(x,t), \qquad (x,t)\in (-\infty,0]\times (0,\infty),
\end{aligned}
\end{equation}
subject to Neumann boundary condition with parameter $\alpha\in \R$:
\begin{align}\label{bdy}
    \partial_x \mathcal{H}^\alpha(x,t)\mid_{x=0}=-\alpha.
\end{align}
Here $\xi(x,t)$ is a space-time white noise.
A physically relevant notion of solution to this equation is given by the Cole--Hopf
solution $\mathcal{H}^\alpha:=\log \mathcal{Z}_\alpha$ where $\mathcal{Z}_\alpha$ solves the half-space stochastic heat equation (HSSHE) with Robin boundary condition: 
\begin{equation}
    \label{sheeq}
    \begin{aligned}
 & \partial_t\mathcal{Z}_\alpha=\tfrac12\partial_{xx}\mathcal{Z}_\alpha+\mathcal{Z}_\alpha\cdot \xi, \\
   & \partial_x\mathcal{Z}_\alpha(x,t)\mid_{x=0} = -(\alpha-\tfrac12)\mathcal{Z}_\alpha(0,t).
\end{aligned}
\end{equation}
In this paper we work with narrow wedge initial data for HSKPZ, which corresponds to setting $\mathcal{Z}_\alpha(x,0)= \delta(x)$, a Dirac delta mass at $0$. The well-posedness of HSSHE was studied in \cite{par}, and in particular a mild
solution of HSSHE under delta initial data exists, is unique under mild growth conditions, and is almost surely positive. This justifies the definition
of $\mathcal{H}^\alpha = \log \mathcal{Z}_\alpha$ as a solution of the HSKPZ equation.

The solution to \eqref{halfkpz} can be interpreted as the point-to-point free energy (i.e., the logarithm of the point-to-point partition function) of a continuum directed random polymer propagating on the half-line $(-\infty, 0]$, subject to an interaction induced by a wall at the origin. The parameter $\alpha$ governs the strength of this interaction. The long-time behavior of the half-space KPZ equation has attracted significant attention in the physics literature since the work of Kardar \cite{kar2}. Depending on the value of $\alpha$, the system is expected to exhibit three distinct behaviors:
\begin{itemize}[leftmargin=20pt]
\item Subcritical regime ($\alpha < 0$): The polymer paths are pinned to the wall and the free energy displays Gaussian fluctuations with $1/2$ fluctuation scaling exponent.
\item Supercritical regime ($\alpha > 0$): The polymer paths stay away from the wall and the free energy exhibits $1/3$ fluctuation and $2/3$ transversal scaling exponents.
\item Critical regime ($\alpha = \mu t^{-1/3}$ for some $\mu \in \mathbb{R}$): The model still exhibits $1/3$ and $2/3$ scaling exponents, but the limiting distribution and qualitative behavior of the polymer paths differ from the supercritical regime.
\end{itemize}

At the one-point level\textemdash specifically, for the distribution of the half-space KPZ equation at the origin, $\mathcal{H}^\alpha(0, t)$\textemdash substantial progress has been made over the past decade in rigorously establishing this picture, both in the physics and mathematics literature. We mention in particular the work \cite{ims22}, which established scaling limits for $\mathcal{H}^\alpha(0,t)$ given by the Tracy--Widom GSE distribution in the supercritical regime and the Baik--Rains distribution in the critical regime. We review these developments in more detail in Section \ref{sec1.3}.

\medskip

In this work, we initiate a \textbf{process-level} study of the half-space KPZ equation via a Gibbsian line ensemble framework, i.e., we investigate scaling limits of the entire process $\mathcal{H}^\alpha(\cdot,t)$ rather than just $\mathcal{H}^\alpha(0,t)$. One of the main results of this paper is the demonstration that, for each fixed $t$, the function $\mathcal{H}^\alpha(\cdot, t)$ can be viewed as the top curve of a line ensemble (in all three regimes), which we call the \textit{half-space KPZ line ensemble}. This ensemble satisfies a certain one-sided Gibbs property, involving softly non-intersecting Brownian motions with a pairwise attraction at the boundary. We describe the line ensemble, its Gibbs property, and related results in the following subsection.

{In view of the above discussion regarding scaling exponents of the HSKPZ equation, it is natural to study process-level scaling limits of the entire HSKPZ line ensemble. The general program of establishing such weak scaling limits for line ensembles, dating back to \cite{ble}, typically involves two distinct parts: (i) establishing tightness, and (ii) proving uniqueness of limit points. In the present paper, we take up item (i) for the scaled HSKPZ line ensemble. Item (ii) is beyond the scope of this work and, due to the difficulties of positive temperature, will likely require the development of new characterization results for half-space line ensembles; we discuss this further in Remark \ref{rem:limit} and Section \ref{sec1.3}. Our tightness results for the line ensemble are detailed in Section \ref{sec1.1} below, and as a special case we deduce process-level tightness for the half-space KPZ equation under 1:2:3 KPZ scaling in the critical and supercritical regimes.
}

 \begin{theorem}[Part of Theorem \ref{kpz123}] \label{thms} Let $\mathcal{H}^\alpha(\cdot,\cdot)$ be the Cole--Hopf solution to the HSKPZ equation in \eqref{halfkpz} with narrow wedge initial data and Neumann boundary condition \eqref{bdy} with parameter $\alpha \in \R$. We define the 1:2:3 scaled KPZ equation as
     \begin{align*}
      \mathfrak{H}^{t,\alpha}(x):= \frac{\mathcal{H}^\alpha(xt^{2/3},t)+t/24}{t^{1/3}}, \quad (x,t)\in(-\infty,0]\times[1,\infty).
     \end{align*}
     \begin{enumerate}[label=(\alph*),leftmargin=20pt]
         \item For each $\alpha>0$, as $t\to\infty$ varies, $\mathfrak{H}^{t,\alpha}$ is tight in the space $\mathcal{C}(\R_{\le 0})$ of continuous functions on $\R_{\le 0}$ equipped with the topology of uniform convergence on compact sets.
         \item For each $\mu \in \R$, as $t\geq 1$ varies, $\mathfrak{H}^{t,\mu t^{-1/3}}$ is tight in the space $\mathcal{C}(\R_{\le 0})$ with the uniform-on-compact topology.
     \end{enumerate}   
 \end{theorem}

Theorem \ref{thms} is a component of Theorem \ref{kpz123}, where we establish the tightness of the entire half-space KPZ line ensemble. This ensemble is expected to converge to the half-space Airy line ensemble\textemdash recently constructed in \cite{halfairy} in the critical case and, after our paper was posted, in \cite{dsy26} in the supercritical case using some ideas from our work. We identify several properties of the subsequential limits in Theorem \ref{subseq}. We believe these properties are sufficient to characterize the limiting distribution (see Remark~\ref{rem:limit}).

\subsection{Main results}   \label{sec1.1} 

\subsubsection{Fixed time results} The starting point of the proof of Theorem \ref{thms} is to show that the HSKPZ equation can be embedded as the top curve of a line ensemble that satisfies a novel one-sided Gibbs property. This line ensemble is described in the following result.

\begin{theorem}[HSKPZ line ensemble] \label{thm1} For all $t\ge 1$ and $\alpha \in \R$, there exists a unique collection of random curves $\{\mathcal{H}_k^\alpha(x,t) : k\in \mathbb{N},\, x\in (-\infty,0]\}$, which we call the half-space KPZ line ensemble (HSKPZLE), such that:
\begin{enumerate}[label=(\alph*),leftmargin=20pt]
    \item\label{12a} The top curve, $\mathcal{H}_1^\alpha(\cdot,t)$, is equal in distribution to $\mathcal{H}^\alpha(\cdot,t)$, the time $t$ Cole--Hopf solution to the narrow wedge initial data HSKPZ equation \eqref{halfkpz} with Neumann boundary condition \eqref{bdy}.

    \item\label{12b} The line ensemble enjoys a certain one-sided Brownian Gibbs property, which can be described as follows. For each $k\in \mathbb{N}$ and $A<0$, the law of the curves $(\mathcal{H}_i^\alpha(\cdot,t))_{i=1}^k$ on the interval $[A,0]$ conditioned on $(\mathcal{H}_i^\alpha(A,t))_{i=1}^k$ and $\mathcal{H}_{k+1}^\alpha(\cdot,t)$ is that of $k$ independent Brownian motions $({B}_i)_{i=1}^k$  started from $(A,\mathcal{H}_i^\alpha(A,t))_{i=1}^{k}$ and with drifts $(-1)^i \alpha$, reweighed by a Radon--Nikodym (RN) derivative  proportional to
\begin{align}\label{twosideg}
    \exp\left(-\sum_{i=1}^k \int_{A}^{0} e^{B_{i+1}(x)-B_i(x)} \,dx \right),
\end{align}
where $B_{k+1}:=\mathcal{H}_{k+1}^\alpha(\cdot,t)$. By the Cameron--Martin theorem, the drifts can be incorporated inside the RN derivative and the conditional law can be viewed as the law of $k$ independent Brownian motions $({B}_i)_{i=1}^k$  started from $(A,\mathcal{H}_i^\alpha(A,t))_{i=1}^{k}$ (with no drift) reweighed by an RN derivative  proportional to
\begin{align}\label{onesideg}
    \exp\left(\sum_{i=1}^k (-1)^i\alpha B_i(0) -\sum_{i=1}^k\int_{A}^{0} e^{B_{i+1}(x)-B_i(x)} \,dx \right).
\end{align}
\end{enumerate}
\end{theorem}

\medskip

We prove Theorem \ref{thm1} in Sections \ref{sec5} and \ref{sec6}. The one-sided Gibbs property stated in Theorem \ref{thm1}\ref{12b} was anticipated in \cite{bcy}, based on insights from the recently constructed half-space log-gamma (HSLG) line ensemble in \cite{half1}.

In the context of the full-space KPZ equation, the (full-space) KPZ line ensemble was introduced in \cite{kpzle}. It enjoys a two-sided Gibbs property: given boundary data on both ends, the conditional law is absolutely continuous with respect to Brownian bridges (with no drift), with a Radon-Nikodym derivative similar to \eqref{twosideg}. We note that the same two-sided Gibbs property follows for the HSKPZLE from \eqref{onesideg} by further conditioning at a right boundary point. A key feature distinguishing the half-space model from its full-space counterpart is the boundary interaction potential term $\sum_{i=1}^k (-1)^i \alpha B_i(0)$, which appears in the one-sided Gibbs property \eqref{onesideg}. This additional term leads to qualitatively different large-scale behavior in the half-space setting, particularly in the supercritical regime (see Figure \ref{figO1}(B)). 

Our proof for Theorem \ref{thm1} utilizes the HSLG line ensemble structure from \cite{half1}. We establish the existence of the HSKPZ line ensemble by proving that the HSLG line ensemble is tight under intermediate disorder scaling, and that every subsequential limit satisfies Properties \ref{12a} and \ref{12b} of Theorem \ref{thm1}. The uniqueness of such a line ensemble then follows by adapting results from \cite{dimh}, which provide a characterization of ensembles satisfying a two-sided Gibbs property.

 As in the full-space setting, the availability of a Gibbsian line ensemble for the HSKPZ equation opens the door to a wide range of possibilities for extracting fine properties of the model. Indeed, Gibbsian line ensembles have proven immensely useful in studying various intricate aspects of models that exhibit a line ensemble structure\textemdash for example, in analyzing the increments of the free energy, investigating fractal properties and fine tail estimates, and understanding the geometry of geodesics and polymer paths. We refer the reader to \cite{ham2,ham4,ham3,ham1,chh19,chaos2,chaos1,ciw1,ciw2,cgh21,dg21,wuconv,dz22,dz22b,dz23,gh22,ds24,hegde} for results in this direction. Gibbsian line ensembles also play a pivotal role in the study of related scaling limits such as the Airy sheet and the directed landscape; see \cite{dv18,dov18,bgh1,bgh2,3by2,sv21,chhm,gh21,rv21,wu23}.

\medskip

{A key ingredient behind the proofs of the above results, which may be of interest in its own right, is a novel identity relating the full- and half-space KPZ equations. In the following proposition, $\mathcal{Z}_\alpha$ is the solution to the HSSHE as in \eqref{sheeq}, and $\mathcal{Z}_\alpha^{\mathrm{full},B}$ denotes the solution to the full-space SHE with certain \textit{half-Brownian} initial data of drift $-\alpha$ (see Definition \ref{defshe}\ref{hbdata}).
\begin{proposition}\label{p:iden}
For any $x\le 0$, we have
\begin{equation}\label{fhid}
    \mathcal{Z}_{\alpha}^{\mathrm{full},B}({x},t)\stackrel{d}{=}\frac12 \int_{-\infty}^{{x}} \mathcal{Z}_\alpha(y,t)\,dy.
\end{equation}
\end{proposition}
Beyond its utility in the present paper, the above identity may be used to establish tail estimates for the half-space KPZ equation in the entire supercritical regime by transferring known estimates for the full-space equation. More generally, we believe this identity may be useful in deducing various geometric properties of the half-space KPZ equation in the future.
}

\subsubsection{Scaling limit results near the boundary} We now turn towards the scaling limit of the HSKPZ line ensemble under 1:2:3 scaling. {Fix a slope parameter $\beta \ge 0$ and let 
\begin{equation}\label{def:Hrescaled}
\mathfrak{H}_i^{t,\alpha;\beta}(x)= \frac{\mathcal{H}_i^\alpha(xt^{2/3}-\beta t,t)+t/24+\beta^2t/2-\beta xt^{2/3}}{t^{1/3}}, \quad t\ge 1, \, x\le \beta t^{1/3}
\end{equation}
be the rescaled HSKPZ line ensemble.} 

\medskip

{Our next result studies $\mathfrak{H}^{t,\alpha;\beta}$ as $t$ tends to infinity in the $\beta=0$ case. When $\beta=0$, we simply write $\mathfrak{H}^{t,\alpha}=\mathfrak{H}^{t,\alpha;0}$. In this setting, the boundary interaction plays a nontrivial role; the scaling limits are different from their full-space counterparts.}

\begin{theorem}\label{kpz123}  For each fixed $\alpha>0$, as $t\to\infty$, $\mathfrak{H}^{t,\alpha}$ is tight in the space $\mathcal{C}(\mathbb{N}\times \mathbb{R}_{\le 0})$ of continuous functions on $\mathbb{N}\times\mathbb{R}_{\le 0}$ equipped with the uniform-on-compact topology. For each fixed $\mu\in \R$, as $t\to\infty$, $\mathfrak{H}^{t,\mu t^{-1/3}}$ is tight in $\mathcal{C}(\mathbb{N}\times \mathbb{R}_{\le 0})$ with the uniform-on-compact topology.  
\end{theorem}

The threshold $t\geq 1$ here is unimportant; $t$ only needs to be bounded away from 0. The above theorem does not cover $\alpha<0$; under that regime we expect to see Gaussian fluctuations for the KPZ equation.

The proof of Theorem \ref{kpz123}, presented in Section \ref{sec9}, relies crucially on the Gibbs property. There is now a substantial body of work demonstrating that the Gibbs property is a powerful tool for establishing spatial tightness of line ensembles. For instance, given one-point tightness of the top curve, the Gibbs property can be used to deduce tightness of the entire ensemble. This approach was first developed in the full-space setting in \cite{ble} and has since been widely applied in the analysis of Gibbsian line ensembles across various full-space probabilistic models; see, for example, \cite{kpzle,cd18,dnv19,wu2,bcd,xd1,dff,serio,wuconv,hegde}. However, extending this program to the half-space setting presents several unique challenges. An analogous result was obtained in the half-space context in \cite{half1} for the half-space log-gamma polymer model. In Section \ref{sec1.2}, we outline key ideas and strategies we use to address these challenges.

\smallskip

We expect that the rescaled HSKPZ line ensemble \eqref{def:Hrescaled} converges as $t\to\infty$ to a limiting object that serves as the universal scaling limit in the half-space KPZ universality class. This limiting object should be the \textit{half-space Airy line ensemble}\textemdash a natural half-space analogue of the full-space Airy line ensemble \cite{prahofer,ble}. {In the critical regime, \cite{halfairy} recently constructed a one-parameter family (depending on $\mu$) of half-space Airy line ensembles. Following the original appearance of this paper, \cite{dsy26} constructed the \textit{pinned} half-space Airy line ensemble as the $\mu\to\infty$ limit (see Figure \ref{figO1}(B)), corresponding to the supercritical regime we treat here. We discuss these works in more detail in Section \ref{sec1.3}.}

\smallskip

While we do not prove convergence to the half-space Airy line ensemble in this paper, we establish several properties of the subsequential limits (see the theorem below) that we believe are sufficient to uniquely characterize the limiting object (see Remark \ref{rem:limit} below).

\begin{theorem}[Properties of subsequential limits] \label{subseq} Fix $\alpha>0$ and $\mu\in \R$. Let $\mathfrak{H}_{\mathrm{sc}}^{\infty,\alpha}$ be any subsequential limit point of $\mathfrak{H}^{t,\alpha}$, and $\mathfrak{H}_{\mathrm{cr}}^{\infty,\mu}$ any subsequential limit point of $\mathfrak{H}^{t,\mu t^{-1/3}}$, as $t\to\infty$. We have the following:
\begin{enumerate}[label=(\alph*),leftmargin=20pt]
    \item\label{parta} $\mathfrak{H}_{\bullet}^{\infty,\bullet}$ is strictly ordered. For each $x<0$, with probability $1$ 
    $$\mathfrak{H}_{1;\bullet}^{\infty,\bullet}(x) > \mathfrak{H}_{2;\bullet}^{\infty,\bullet}(x) > \cdots$$
    \item \label{partb} $\mathfrak{H}_{1;\bullet}^{\infty,\bullet}$ approximates a parabola. For each $\e>0$, there exists $\mathfrak{R}(\e)$ such that for all $x\leq 0$,
    \begin{align*}
\Pr\bigg(\bigg|\mathfrak{H}_{1;\bullet}^{\infty,\bullet}(x)+\frac{x^2}{2}\bigg| \le \e |x|+\mathfrak{R}(\e)\bigg) \ge 1-\e.
    \end{align*}
     
    \item \label{partd}  
    For each $m\in \mathbb{N}$ and $A<0$, the law of the curves $(\mathfrak{H}_{i;\mathrm{sc}}^{\infty,\alpha}(\cdot))_{i=1}^{2m}$ on the interval $[A,0]$ conditioned on $(\mathfrak{H}_{i;\mathrm{sc}}^{\infty,\alpha}(A))_{i=1}^{2m}$ and $\mathfrak{H}_{2m+1;\mathrm{sc}}^{\infty,\alpha}(\cdot)$ is that of $2m$ independent Brownian motions $({B}_i)_{i=1}^{2m}$  started from $(A,\mathfrak{H}_{i;\mathrm{sc}}^{\infty,\alpha}(A))_{i=1}^{2m}$, conditioned not to intersect: $$B_1 \ge B_2 \ge \cdots \ge B_{2m} > \mathfrak{H}_{2m+1;\mathrm{sc}}^{\infty,\alpha} \mbox{ on }[A,0],$$ and with pairwise pinning at the origin:
    $$B_{2i-1}(0)=B_{2i}(0) \mbox{ for }i=1,2,\ldots,m.$$ 
    
    \item \label{parte}   For each $m\in \mathbb{N}$ and $A<0$, the law of the curves $(\mathfrak{H}_{i;\mathrm{cr}}^{\infty,\mu}(\cdot))_{i=1}^{2m}$ on the interval $[A,0]$ conditioned on $(\mathfrak{H}_{i;\mathrm{cr}}^{\infty,\mu}(A))_{i=1}^{2m}$ and $\mathfrak{H}_{2m+1;\mathrm{cr}}^{\infty,\mu}(\cdot)$ is that of $2m$ independent Brownian motions $({B}_i)_{i=1}^{2m}$ with drifts $(-1)^i\mu$ started from $(A,\mathfrak{H}_{i;\mathrm{cr}}^{\infty,\mu}(A))_{i=1}^{2m}$, conditioned not to intersect: $$B_1 \ge B_2 \ge \cdots \ge B_{2m}> \mathfrak{H}_{2m+1;\mathrm{cr}}^{\infty,\mu} \mbox{ on }[A,0].$$  

    \item \label{partc}  
    For each $k\in \mathbb{N}$ and $A<B<0$, the law of the curves $(\mathfrak{H}_{i;\bullet}^{\infty,\bullet}(\cdot))_{i=1}^{k}$ on the interval $[A,B]$ conditioned on $(\mathfrak{H}_{i;\bullet}^{\infty,\bullet}(A),\mathfrak{H}_{i;\bullet}^{\infty,\bullet}(B))_{i=1}^{k}$ and $\mathfrak{H}_{k+1;\bullet}^{\infty,\bullet}$ is that of $k$ independent Brownian bridges $({B}_i)_{i=1}^{k}$  started from $(A,\mathfrak{H}_{i;\bullet}^{\infty,\bullet}(A))_{i=1}^{k}$ and ending at $(B,\mathfrak{H}_{i;\bullet}^{\infty,\bullet}(B))_{i=1}^{k}$ conditioned not to intersect: $$B_1 \ge B_2 \ge \cdots \ge B_{k} > \mathfrak{H}_{k+1;\bullet}^{\infty,\bullet} \mbox{ on }[A,B].$$
\end{enumerate}
\end{theorem}

Theorem \ref{subseq} is proven in Section \ref{sec9}. The proofs of parts \ref{partd} and \ref{parte} rely crucially on establishing the convergence of the one-sided Gibbs measures described in part \ref{12b} of Theorem \ref{thm1}, in both the critical and supercritical regimes. The critical regime is significantly easier to handle; the convergence in the supercritical regime\textemdash established in Section \ref{sec7}\textemdash is one of the novel technical contributions of this paper. Indeed, the event of pairwise pinning described in part \ref{partd} has probability zero, so the conditioning is not immediately well-defined, and significant technical work is required to construct these limiting objects. We refer the reader to Section \ref{sec1.2} below for a brief discussion of how these pairwise pinned non-intersecting Brownian motions are constructed, and to Section \ref{sec7} for the full details.

\begin{remark}
 \label{rem:limit} 
Proving uniqueness of the subsequential limits is beyond the scope of this paper, but we outline one potential approach toward establishing it, which would involve different techniques from those used here to establish tightness. In the full-space setting, the recent work of Aggarwal and Huang \cite{ah23} proved a strong characterization result. They showed that any line ensemble satisfying a two-sided non-intersecting Brownian Gibbs property (cf.~Theorem \ref{subseq}\ref{partd}-\ref{partc}) and approximating a parabola (cf.~Theorem \ref{subseq}\ref{partb}) must agree with the Airy line ensemble up to an independent affine shift.
 In upcoming work \cite{adh}, Aggarwal, Huang, and the first author aim to establish a similar result in the half-space setting. 
Once such a result becomes available in the half-space setting, Theorems \ref{kpz123} and \ref{subseq} would together imply the convergence of the half-space KPZ line ensemble {to the half-space Airy line ensembles of \cite{halfairy} and \cite{dsy26}.}
\end{remark}

\subsubsection{Scaling limit results away from the boundary} 
When the slope parameter $\beta\neq 0$, the limiting statistics are given by its full space counterparts as our following result shows. 
The scaling convention we use for the Airy line ensemble $\mathcal{A}$ is that in \cite[Definition 2.4]{ah23}.

\begin{theorem}\label{kpzb123}  Fix any $\alpha \in \mathbb{R}$ and $\beta>\max\{-\alpha,0\}$. As $t\to\infty$, the line ensemble $\{\mathfrak{H}_i^{t,\alpha;\beta}(x) : i\in\mathbb{N},\,x\le \beta t^{1/3}\}$ converges weakly to $\{2^{-1/3}\mathcal{A}_i(2^{-1/3}x) - x^2/2 : i\in\mathbb{N},\,x\in\mathbb{R}\}$ in the uniform-on-compact topology on $\mathcal{C}(\mathbb{N}\times\R)$, where $\mathcal{A}$ denotes the (stationary) Airy line ensemble.
\end{theorem}

When $\alpha+\beta<0$, we expect to see Gaussian fluctuations for the top curve. However, when $\alpha+\beta=\mu t^{-1/3}$ for some $\mu\in \mathbb{R}$, the limiting distribution should be governed by Airy wanderer line ensembles. Using the Pfaffian Schur process structure, such results away from the boundary have been established for half-space geometric last passage percolation (LPP) in \cite{zhou,dzhou}. Such integrable structures are unavailable in the KPZ equation setting, and new ideas are needed to establish the corresponding result.

\subsection{Proof ideas} \label{sec1.2} In this section, we sketch the key ideas behind the proof of our main results. 

\begin{figure}[ht!]
    \begin{tikzpicture}[auto,
 		block_KPZs/.style = {rectangle, draw=black, fill=white, text width=18em, text centered, minimum height=3.5em},
 		block_KPZ/.style = {rectangle, draw=black, fill=white, text width=20em, text centered, minimum height=3.5em},
 		line/.style ={draw, thick, -latex', shorten >=0pt}]
 		\node [block_KPZ] (hslg) at (-3,-3) {Half-space log-gamma line ensemble: $\L_i^{n;\theta}(j)$};
 		\node [block_KPZs] (hskpz) at (1,-6) {Half-space KPZ line ensemble: $\mathcal{H}_i^\alpha(x,t)$};
 		\node [block_KPZs] (hsair) at (5,-9) {Half-space Airy line ensemble: $\mathfrak{H}_i^\infty(y)$};
   
 		\begin{scope}[every path/.style=line]
             \path (hslg)  -- (hskpz) node[midway,right] {\begin{tabular}{c}
                  $\theta=\frac12+\sqrt{N}$, $n = Nt/2+1$, $j=-x\sqrt{N}+1$,
                  \\ $N\to \infty$
             \end{tabular}} node[midway,left] {(Theorem \ref{thm0}) \phantom{xx}};  
             \path [dashed] (hskpz)  -- (hsair) node[midway,right] {\phantom{xx} $x=yt^{2/3}$, $t\to \infty$}node[midway,left] {(Theorems \ref{kpz123} and \ref{subseq}) \phantom{xx}};
 		\end{scope}
 	\end{tikzpicture}
 	\caption{Schematic representation of the hierarchy of models. The second arrow is dashed to indicate that tightness and properties of subsequential limits are established under 1:2:3 scaling in the present work.}
 	\label{fig:conn}
\end{figure}
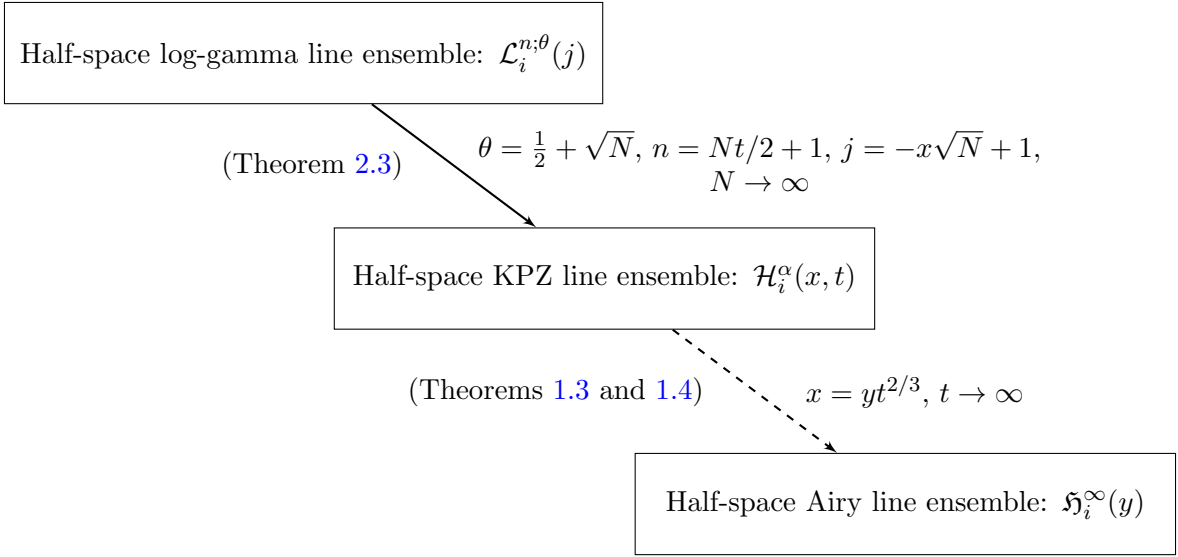

\subsubsection{Construction of the HSKPZ line ensemble} \label{sec1.2.1}

The construction of the HSKPZ line ensemble is relatively simple and broadly follows the strategy developed in \cite{kpzle}, where the authors constructed the analogous KPZ line ensemble in the full-space setting. We begin with the half-space log-gamma (HSLG) polymer model, whose free energy (log-partition function) is known to converge to the solution of the half-space KPZ (HSKPZ) equation under intermediate disorder scaling \cite{bc22}. It was recently shown in \cite{half1} that this free energy can be embedded into a Gibbsian line ensemble (the HSLG line ensemble), whose Gibbs property is a discretization of the continuous one stated in property \ref{12b} in Theorem \ref{thm1}, involving log-gamma random walks.

We show that under intermediate disorder scaling, the entire HSLG line ensemble is tight, and every subsequential limit satisfies properties \ref{12a} and \ref{12b} of Theorem \ref{thm1}. This is Theorem \ref{thm0} in the text (see also Figure \ref{fig:conn}). A similar program was carried out in the full-space setting by \cite{kpzle}, where the authors analyzed the O'Connell--Yor polymer model and established tightness of the associated line ensemble. A key input in their proof was the parabolic curvature of the full-space KPZ equation\textemdash the fact that $\mathcal{H}^{\mathrm{full}}(x,t) + \frac{x^2}{2t}$ is stationary in $x$. This fact does not hold in the half-space setting. However, by comparing the half-space KPZ equation with its full-space counterpart via chaos expansion, we derive a weaker parabolic curvature estimate (Proposition \ref{p:parabol}) that suffices for our purposes.

With this input, our proof largely follows the strategy of \cite{kpzle}, employing an inductive scheme to propagate tightness to the lower curves, with necessary technical modifications to account for the one-sided boundary and associated Gibbs property. This yields tightness of the HSLG line ensemble and ensures that any subsequential limit satisfies the desired properties. The two-sided Gibbs property of the HSKPZ line ensemble is identical to that of the full-space case. By applying a uniqueness result from \cite{dimh}, which characterizes line ensembles with such a Gibbs property, we deduce the uniqueness of subsequential limits.

\subsubsection{Tightness under KPZ scaling} \label{sec1.2.2}

The tightness of the HSKPZ line ensemble under KPZ scaling, on the other hand, is more challenging due to the lack of information about the behavior of the top curve\textemdash the HSKPZ equation\textemdash as $t \to \infty$. As mentioned earlier, \cite{ims22} established the distributional limit of $\mathfrak{H}_1^{t,\alpha}(0)$, but beyond this, little is known about the long-time behavior of the HSKPZ equation, particularly away from the boundary. {Thus, we are not able to rely on techniques from the zero temperature setting which make use of exact formulas for the top curve (see Section \ref{sec1.3}), and we must develop different methods.} Our proof strategy can be divided into two parts:

\subsubsection*{Tightness on intervals far to the left.} 

We first establish tightness of the HSKPZ line ensemble on intervals far to the left; that is, $(\mathfrak{H}_i^{t,\alpha}(x))|_{i\in \ll1,k\rr,\, x\in [-Q_1-Q_2,-Q_1]}$ is tight for all $Q_2>0$ and for all sufficiently large $Q_1>0$. This part closely follows the strategy developed in \cite{half1}. Fixing some notation, we say that a sequence $X_y^t$ is \textit{tight uniformly in $y$} if, for every $\varepsilon>0$, there exists $M(\varepsilon)>0$ such that $\Pr(|X_y^t| > M) \le \varepsilon$ for all $y$. We say $X_y^t$ is \textit{upper-tight} if $\max(X_y^t,0)$ is tight and \textit{lower-tight} if $\min(X_y^t,0)$ is tight.

{Our proof starts with the identity \eqref{fhid} relating the full- and half-space KPZ equations. We obtain this from a distributional identity of \cite{bc22} connecting the HSLG polymer partition function to a full-space log-gamma polymer with a boundary perturbation. Taking the intermediate disorder limit yields the corresponding identity \eqref{fhid} at the level of the SHE. The long-time behavior of the full-space KPZ equation with half-Brownian initial data, $\log \mathcal{Z}_{\alpha}^{\mathrm{full},B}$, has been studied in \cite{corwin2013crossover,borodin2014free}, from which we deduce that $\log \mathcal{Z}_{\alpha}^{\mathrm{full},B}$ is pointwise tight under 1:2:3 scaling and exhibits a parabolic trajectory (Lemma \ref{hbpara}):}
\begin{align*}
      \frac{\log\mathcal{Z}_{\alpha}^{\mathrm{full},B}(xt^{2/3},t)+t/24}{t^{1/3}}+\frac{x^2}2
\end{align*}
is tight uniformly in $x$. Owing to the identity \eqref{fhid}, this gives us a weak version of parabolic trajectory for $\mathfrak{H}_1^{t,\alpha}$ (Lemma \ref{p1}): for all sufficiently large $r$,
\begin{align}\label{weakparab}
    & \sup_{y\in [x-r,x]} \mathfrak{H}_1^{t,\alpha}(y)+\frac{x^2}{2} \text{ is lower-tight}, \qquad \inf_{y\in [x-r,x]} \mathfrak{H}_1^{t,\alpha}(y)+\frac{x^2}{2} \text{ is upper-tight},
\end{align}
uniformly in $x$. {With this input, we are able to establish tightness of each curve of the HSKPZ line ensemble on intervals far to the left (see Propositions \ref{0ult} and \ref{0ultp}). This involves a fusion of Gibbsian line ensemble arguments from \cite{kpzle} and \cite{half1}. We employ an inductive scheme similar to that used in \cite{kpzle} (and in Section \ref{sec6} of the present paper) to propagate tightness from the top curve down. In particular, the argument for lower-tightness hinges on the fairly well known heuristic idea that if any curve of the line ensemble were excessively low, then the top curve would follow a straight trajectory rather than a parabolic one, contradicting \eqref{weakparab}. We refer to Section \ref{sec8.3} for more details.}

\subsubsection*{Tightness near the boundary: Limits of one-sided Gibbs measures} Having established tightness on intervals far to the left, the main challenge is to show that this tightness propagates all the way to the boundary.

We sketch the proof here for the supercritical case, as the critical case is significantly easier. In the supercritical regime, $\alpha > 0$, the main difficulty reduces to understanding the limit of the one-sided Gibbs measures under 1:2:3 scaling. Under this scaling, the RN derivative in \eqref{onesideg} (taking $k=2m$) becomes
 \begin{align*}
     \mathcal{W}_t :=  \prod_{i=1}^m e^{-\alpha t^{1/3}(B_{2i-1}(0)-B_{2i}(0))} \cdot \prod_{i=1}^{2m+1} \exp\left(-t^{2/3} \int_{A}^0 e^{t^{1/3}(B_{i+1}(x)-B_i(x))} \, dx\right).
 \end{align*}
Intuitively, the second factor enforces the inequality $B_i(x) \ge B_{i+1}(x)$ for $x\in[A,0]$ as $t \to \infty$, while the first factor penalizes configurations where $B_{2i-1}(0) > B_{2i}(0)$. Consequently, we expect pairwise pinning in the limit: $B_{2i-1}(0) = B_{2i}(0)$. However, such configurations have probability zero for independent Brownian motions, which makes the analysis quite delicate.

A similar difficulty was addressed in \cite{half1}, where the authors analyzed $W_N$\textemdash a discrete analogue of $\mathcal{W}_t$\textemdash that appears in the context of the HSLG line ensemble. Very loosely speaking, their key idea was to show
\[
\mathbb{E}[W_N \ind_{\m{A}}] \le \frac{C}{N^{2/3}} \mathbb{P}(\m{A}), \qquad \mathbb{E}[W_N] \ge \frac{C^{-1}}{N^{2/3}},
\]
for suitable events $\m{A}$ (such as events controlling the modulus of continuity). In their setting, the discrete structure of $W_N$ offered certain technical advantages in computing the order of the expectations\textemdash tools that are not available to us in the continuum setting. More importantly, we do not pursue this approach because even proving analogous bounds in our case would only imply tightness and would not suffice to extract the finer properties of the subsequential limits, particularly part \ref{partd} of Theorem \ref{subseq}.

\begin{figure}[h!]
		\centering
		\begin{subfigure}[b]{0.45\textwidth}
			\centering
            \begin{overpic}[]{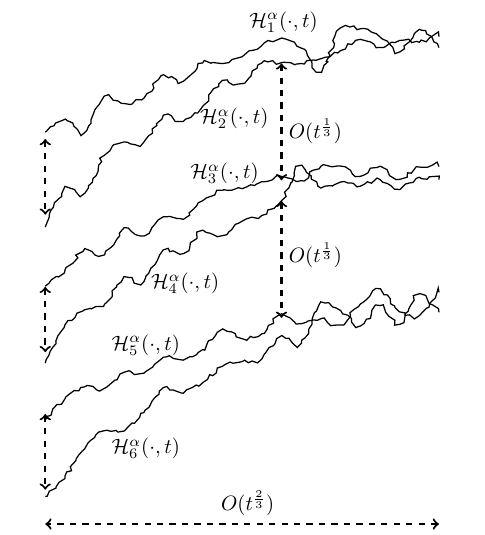}
\put(-4,67){$O(t^{\frac{1}{3}})$}
\put(-4,40){$O(t^{\frac{1}{3}})$}
\put(-4,15){$O(t^{\frac{1}{3}})$}
\put(85,90){$O(1)$}
\put(85,66){$O(1)$}
\put(85,42){$O(1)$}
\end{overpic}
	\caption{}
		\end{subfigure}\qquad\qquad
\begin{subfigure}[b]{0.45\textwidth}
		\centering
		\begin{tikzpicture}[scale=.8,line cap=round,line join=round,>=triangle 45,x=4cm,y=1.5cm]
			\draw[semithick, decorate, decoration={random steps,segment length=2pt,amplitude=2pt}] plot[domain=-2:0] (\x, {3-2*(\x)^2/5});
			\draw[semithick, decorate, decoration={random steps,segment length=2pt,amplitude=1.5pt}] plot[domain=-2:0] (\x, {2.98-2*(\x)^2/3});
			\draw[semithick, decorate, decoration={random steps,segment length=2pt,amplitude=1.5pt}] plot[domain=-2:0] (\x, {1-2*(\x)^2/5});
			\draw[semithick, decorate, decoration={random steps,segment length=2pt,amplitude=1.5pt}] plot[domain=-2:0] (\x, {.98-2*(\x)^2/3});
			\draw[semithick, decorate, decoration={random steps,segment length=2pt,amplitude=1.5pt}] plot[domain=-2:0] (\x, {-1-2*(\x)^2/5});
			\draw[semithick, decorate, decoration={random steps,segment length=2pt,amplitude=1.5pt}] plot[domain=-2:0] (\x, {-1.02-2*(\x)^2/3});
			\draw (-1,3) node[anchor=west] {$\mathfrak{H}_1^\infty(\cdot)$};
			\draw (-1,2) node[anchor=west] {$\mathfrak{H}_2^\infty(\cdot)$};
			\draw (-1,1) node[anchor=west] {$\mathfrak{H}_3^\infty(\cdot)$};
			\draw (-1,0) node[anchor=west] {$\mathfrak{H}_4^\infty(\cdot)$};
			\draw (-1,-1) node[anchor=west] {$\mathfrak{H}_5^\infty(\cdot)$};
			\draw (-1,-2) node[anchor=west] {$\mathfrak{H}_6^\infty(\cdot)$};
		\end{tikzpicture}
			\caption{}
		\end{subfigure}
		
		\caption{(A) Half-space KPZ line ensemble under KPZ scaling in the supercritical regime. (B) Half-space Airy line ensemble in the supercritical regime.}			
\label{figO1}
	\end{figure}
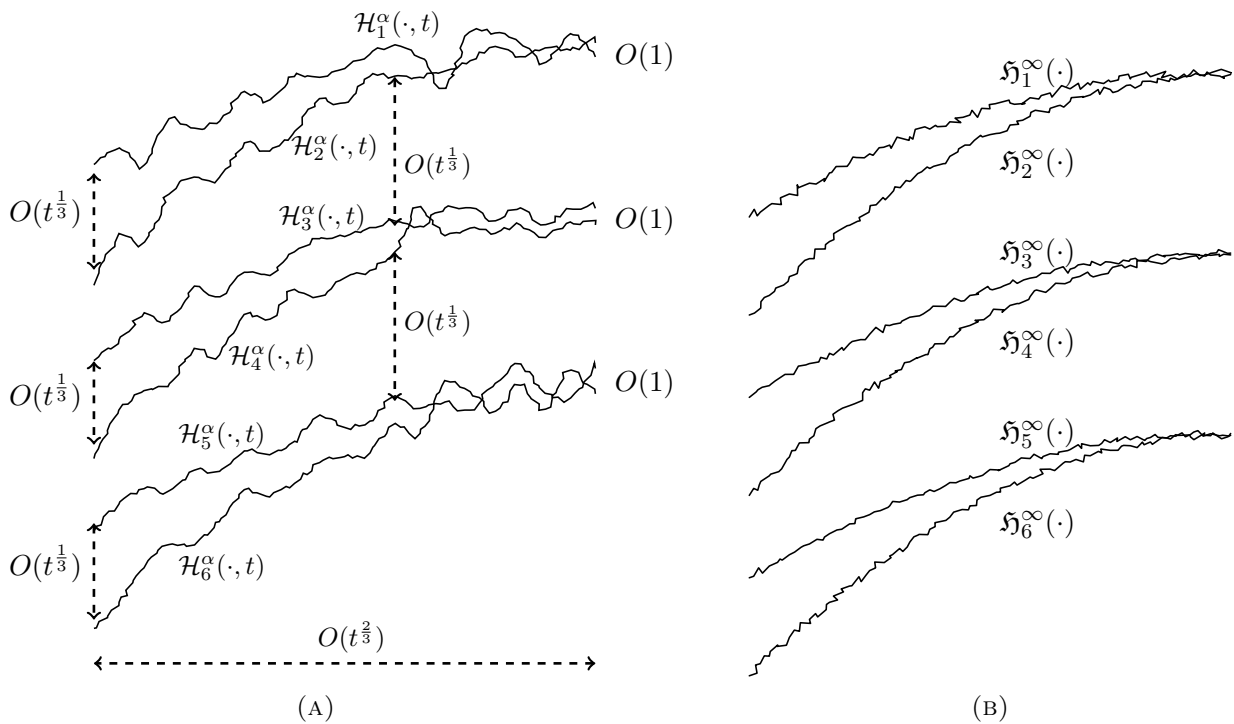

\medskip

We adopt a different approach that allows us to establish \textit{weak convergence} of the one-sided Gibbs measures. To illustrate the main idea, consider first the two-path version of the Gibbs measure: namely, the law of two independent Brownian motions started from $a_1, a_2$ (with $a_1 > a_2$), reweighed by the Radon--Nikodym derivative proportional to
\begin{align*}
     e^{-\alpha t^{1/3}(B_1(0) - B_2(0))} \cdot \exp\left(-t^{2/3} \int_{A}^0 e^{t^{1/3}(B_2(x) - B_1(x))} \, dx\right).
\end{align*}
We introduce the sum and difference processes: $U = B_1 + B_2$ and $V = B_1 - B_2$. Since the RN derivative depends only on the difference, it follows from properties of Brownian motion that $U$ and $V$ are independent, with $U$ being a Brownian motion started from $a_1 + a_2$ with diffusion coefficient $2$. The law of $V$ is absolutely continuous w.r.t.~a Brownian motion started from $a_1 - a_2$ with diffusion coefficient $2$, with RN derivative proportional to
\[
e^{-\alpha t^{1/3} V(0)} \cdot \exp\left(-t^{2/3} \int_{A}^0 e^{-t^{1/3} V(x)} \, dx\right).
\]
The process $V$ can be viewed as a diffusively scaled version of a drifted Brownian motion with a soft barrier. We refer to Section~\ref{sec:dbmsft} for the full proof of its weak convergence. Briefly, our strategy relies on showing that $V$ satisfies several forms of stochastic monotonicity\textemdash with respect to the starting point, the drift parameter, and the barrier (see Lemma~\ref{stm}). These monotonicity properties allow us to bypass precise expectation computations involving the RN derivative and instead directly establish the weak convergence of $V$ (Theorem~\ref{wconvsf}). The limit of $V$ may be described as a time reversal of a Brownian meander conditioned to end at $a_1-a_2$ (see Remark \ref{rem:tran}). As $B_1 = \frac12(U+V)$ and $B_2 = \frac12(U-V)$, this yields the weak convergence of the two-path one-sided Gibbs measures. 

The general $2m$-path case then follows relatively easily by interpreting it as a system of $m$ independent two-path components, reweighed by an additional RN derivative (whose expectation remains bounded away from zero in the limit). Full details are presented in Section~\ref{sec:multi}.

\medskip

The stationary measure for the open KPZ equation on an interval admits a similar two-path Gibbsian description \cite{ck,bld01,bld2,bkww,bcy,himwich}, and the tools developed in Section \ref{sec7} may be used to study diffusive scaling limits of these stationary measures. We mention that our proof relies crucially on the fact that the sum and difference of two independent Brownian motions are themselves independent\textemdash a property not satisfied in general by analogous discrete models. It would be interesting to develop an invariance principle showing that all suitable discretizations of the one-sided Gibbs measures converge to the continuum Gibbs measures described in part~\ref{parte} of Theorem~\ref{subseq}. Such a result would constitute a key step toward establishing that any subsequential limit of the HSLG line ensemble under 1:2:3 scaling satisfies the property described in part~\ref{parte}, thereby potentially leading to a full proof of convergence of the HSLG line ensemble to the half-space Airy line ensemble. We leave this direction for future work.

\subsection{Related works} \label{sec1.3}
Half-space polymers are a natural variant of full-space polymers (introduced in \cite{huse, imb, bol}) that have been extensively studied in the literature; see, for example, \cite{comets, timo, batesch, bcd} and the references therein. These models arise naturally in the context of wetting phenomena \cite{phy1, phy2, phy3}, where one studies directed polymers interacting with a wall. They have attracted significant attention due to the presence of a phase transition\textemdash commonly referred to as the depinning transition\textemdash and a rich phase diagram for limiting distributions depending on the diagonal strength, as discussed earlier before Theorem \ref{thms}. However, aside from a few exactly solvable models, rigorous mathematical results remain scarce.

\subsubsection*{Zero temperature}
At the one-point level, the phase diagram was first rigorously established for geometric last passage percolation (LPP)\textemdash a discrete polymer model at zero temperature\textemdash in a series of works by Baik and Rains \cite{br1, br01, br3}. Multi-point fluctuations were subsequently analyzed in \cite{sis}, and similar results were later proven for exponential LPP in \cite{bbcs0, bbcs} using the framework of Pfaffian Schur processes.  For further recent works on half-space LPP, we refer to \cite{bete,ale1,ale2,ale3,zhang,halfairy}. 

{Among these recent works, we highlight \cite{halfairy}, which is closely related to our own. We remark here on some differences between the techniques involved in that paper and the present one. In \cite{halfairy}, the authors constructed a Gibbsian line ensemble for half-space geometric LPP and, using Pfaffian Schur formulas, established its process-level weak convergence under 1:2:3 scaling in the critical regime. The availability of exact formulas via the Pfaffian Schur process provides significantly more detailed information than what is currently accessible in the positive temperature case, enabling the authors to prove explicit finite-dimensional convergence of the underlying Pfaffian point processes. 
}

{As discussed before the statement of Theorem \ref{thms}, establishing process-level weak convergence of the line ensembles requires two distinct ingredients, (i) tightness and (ii) uniqueness of the limit. For item (i), the aforementioned finite-dimensional convergence provides pointwise tightness of all curves, and \cite{halfairy} leverages Gibbs resampling techniques to upgrade this to process-level tightness. (By comparison, in our proof of tightness discussed in Section \ref{sec1.2.2}, our parabolic curvature estimates for the top curve serve as a weak substitute for this pointwise control.) For item (ii), knowledge of the limiting finite-dimensional distributions implies uniqueness of the limiting line ensemble via \cite{dimh}, and along with tightness this implies true weak convergence. The limiting object is the one-parameter family of half-space Airy line ensembles in the critical regime. The authors further showed that it satisfies the resampling property described in Theorem~\ref{subseq}\ref{parte}, and that its restriction to finitely many vertical lines forms Pfaffian point processes with crossover kernels matching those obtained in \cite{bbcs}.
}

{Lastly, we mention the recent works \cite{dsy26,ddy26}, which appeared after the initial posting of the present paper. \cite{dsy26} established convergence of the critical half-space Airy line ensembles of \cite{halfairy} as their drift parameter diverges, to an explicit object known as the \textit{pinned} half-space Airy line ensemble (see Figure \ref{figO1}(B)). The key innovation of this work is to prove convergence of the underlying Pfaffian point processes from \cite{halfairy} to a singular limit with doubling of points at the boundary. As above, once tightness is established separately, the availability of such exact formulas for the prelimiting objects permits a proof of full weak convergence. Following the methods of the present paper, \cite{dsy26} shows that the limiting pinned line ensemble satisfies the resampling property described in Theorem~\ref{subseq}\ref{partd}. Subsequently, \cite{ddy26} proves the convergence of half-space geometric LPP to the pinned half-space Airy line ensemble, extending the ideas of \cite{halfairy,dsy26}.
}

\subsubsection*{Positive temperature} For the HSLG polymer model, the first rigorous proof of the depinning transition appeared in \cite{bw}, where the authors established precise fluctuation results, including the BBP phase transition \cite{bbap} for the point-to-line free energy. Regarding the point-to-point free energy, the limit theorem and the Baik--Rains phase transition were initially conjectured in \cite{bbc20} based on an uncontrolled steepest descent analysis of formulas derived from half-space Macdonald processes. This conjecture was recently resolved in \cite{ims22} using novel ideas that connect the half-space polymer model to a free boundary variant of the Schur process.

The study of half-space Gibbsian line ensembles was initiated in \cite{half1}, where the HSLG line ensemble was constructed and tightness of the free energy process under 1:2:3 scaling was established. Building on this line ensemble structure, the subcritical regime of the HSLG polymer was studied in \cite{dz23}, and stationary measures were analyzed in \cite{ds24}. We also note in passing that the subcritical regime exhibits markedly different behavior and can be addressed for more general polymer models, as demonstrated in \cite{vic}.

Apart from arising as an intermediate disorder limit of the HSLG polymer, the HSKPZ equation can also be obtained as a scaling limit of the weakly asymmetric simple exclusion process (WASEP) in half-space, or studied directly from the stochastic PDE perspective. This has sparked a surge of new results recently in both the mathematics \cite{hai,cs1, bbcw, bbc20, par, par2, bc22, ims22,rantao} and physics literature \cite{gld,bbc,ito,de,kr,bkld2,bld1,bkld}.

Among these works, as mentioned previously, our study is closely related to \cite{ims22}. By taking limits of the Laplace transform formulas for the HSLG log-partition function derived therein, the authors obtained explicit formulas for the Laplace transform of the height function $\mathcal{H}^\alpha(0,t)$. These formulas are amenable to asymptotic analysis, allowing them to prove long-time fluctuation results across all three regimes and rigorously establish the Baik--Rains phase transition.

In addition to polymer models, the Baik--Rains phase transition has recently been rigorously proven for half-space ASEP and six-vertex models in the works \cite{he1, he2}.

\subsection*{Organization} The rest of the paper is organized as follows.

Sections \ref{sec:hslg}–\ref{sec6} are devoted to the construction of the HSKPZ line ensemble. In Section \ref{sec:hslg}, we introduce the HSLG polymer and its associated line ensemble, and show how the top curve of this ensemble converges to the HSKPZ equation. In Section \ref{sec2}, we establish the parabolic trajectory of the HSKPZ equation. Section \ref{sec4} defines the relevant Gibbs properties for both the HSLG and HSKPZ line ensembles. In Section \ref{sec5}, we prove Theorems \ref{thm0} and \ref{thm1}, assuming upper and lower tightness for each pre-limiting curve; the proofs of these tightness results are deferred to Section \ref{sec6}.

Sections \ref{sec7}–\ref{sec9} focus on the HSKPZ line ensemble under KPZ scaling. In Section \ref{sec7}, we study diffusive limits of the HSKPZ Gibbs measures. Section \ref{sec8} establishes preliminary estimates for the tightness of the HSKPZ line ensemble. The proofs of Theorem \ref{kpz123} and Theorem \ref{subseq} are given in Section \ref{sec9}.

\subsection*{Notation} Throughout this paper we use $C = C(\alpha, \beta, \gamma, \ldots) > 0$ to denote a generic deterministic positive finite constant that may change from line to line, but dependent on the designated variables $\alpha, \beta, \gamma, \cdots$. We use blackboard bold fonts for probability and expectation ($\mathbb{P}$, $\mathbb{E}$) when the underlying system involves discrete objects such as log-gamma polymers, random walks, or bridges. We use bold fonts ($\mathbf{P}$, $\mathbf{E}$) when the underlying system involves continuous objects such as the HSKPZ line ensemble and its limits, Brownian motions, or Brownian bridges. We use sans-serif fonts such as $\m{A},\m{B}, \m{C}, \ldots$ to denote events. We use $\neg \m{A}$ to denote the complement of an event $\m{A}$. 

We use $\mathbb{N}$ to denote the positive integers. For integers $k\leq \ell$, we write $\ll k,\ell \rr := [k,\ell]\cap\Z$. We often use $\R_{\ge 0}$ to denote $[0,\infty)$ and $\R_{\le 0}$ to denote $(-\infty,0]$. For $x,y\in \R$, we write $x\vee y : =\max \{x,y\}$ and $x\wedge y:= \min \{x,y\}$. Lastly, in writing we often use the abbreviations ``l.h.s.'' and ``r.h.s.'' for ``left-hand side'' and ``right-hand side'' respectively, and ``w.r.t.'' for ``with respect to.'' 

\subsection*{Acknowledgments} C.S and S.D.~thank Ivan Corwin and Evgeni Dimitrov for useful discussions and comments on an earlier draft of the paper. S.D.~thanks Amol Aggarwal for bringing the reference \cite{borodin2014free} to their attention. C.S.~acknowledges support from the Northern California Chapter of the ARCS Foundation and the National Science Foundation MSPRF.

\subsection*{AI disclosure} The first version of this paper was posted in June 2025 and was written without the use of AI. The present version (Version 2) includes convergence to the Airy line ensemble (Theorem \ref{kpzb123}). The intellectual input of this theorem is due to the authors. AI was used to check and edit some of the technical details of both the new and the previous proofs, and to check the grammar. The authors take full responsibility for the mathematical content and correctness of this paper.

\section{The half-space log-gamma (HSLG) polymer and line ensemble} \label{sec:hslg}

In this section, we introduce the half-space log-gamma (HSLG) polymer and the corresponding line ensemble. Fix any $\theta$ large enough so that $\alpha>-\theta$. Consider a family of independent random variables $(\mathscr W_{i,j})_{(i,j)\in \mathcal{I}}$, with $\mathcal{I}:=\{(i,j) \in \mathbb{N}^2 :  j\le i\}$, such that
\begin{align}\label{eq:wt0} \mathscr
W_{i,j}\sim \operatorname{Gamma}^{-1}(\alpha+\theta)  \textrm{ for } i=j \qquad \textrm{and } \quad \mathscr W_{i,j}\sim \operatorname{Gamma}^{-1}(2\theta) \textrm{ for } j<i,
	\end{align}
where $X\sim \operatorname{Gamma}^{-1}(\beta)$ means  $X$ is a random variable with density $\ind_{x>0}\Gamma^{-1}(\beta)x^{-\beta-1}e^{-1/x}$. 

A directed lattice path $\pi=\big((x_i,y_i)\big)_{i=1}^k$ confined to the half-space index set $\mathcal{I}$ is an up-right path with all $(x_i,y_i)\in \mathcal{I}$ which only makes unit steps in the coordinate directions, that is, $(x_{i+1},y_{i+1})=(x_i,y_i)+(0,1)$ or $(x_{i+1},y_{i+1})=(x_i,y_i)+(1,0)$.   
Given $(m,n)\in\mathcal{I}$, we let $\Pi_{m,n}$ denote the set of all directed paths from $(1,1)$ to $(m,n)$ confined to $\mathcal{I}$. Given the random variables from \eqref{eq:wt0}, we define the weight of a path $\pi$ and the point-to-point partition function of the HSLG polymer by
	\begin{align}\label{def:hslg}
		 Z(m,n):=\sum_{\pi\in \Pi_{m,n}} \prod_{(i,j)\in \pi} \mathscr W_{i,j}.
	\end{align}

In \cite{half1}, the authors demonstrated that the free energy process $(\log Z(n+k-1,n-k+1))_{k\ge 1}$ along the line $x+y=2n$ of the HSLG polymer, can be realized as the top curve of a Gibbsian line ensemble $(\L_i^{n;\theta}(\cdot))_{i\in \ll1,N\rr}$,  called the HSLG line ensemble. This ensemble consists of log-gamma increment random walks interacting through a soft non-intersection condition and subject to pairwise pinning at the boundary.  
This remarkable line ensemble structure comes from the geometric RSK correspondence \cite{cosz14, osz14, nz17,bz19} and the half-space Whittaker process \cite{bbc20}.

The construction of the HSLG line ensemble is based on the multi-path point-to-point partition functions corresponding to full-space polymers under \textit{symmetrized environments}. These are defined in \eqref{zsymr} as sums over multiple non-intersecting paths on the full quadrant $\mathbb{N}^2$ (not just the octant $\mathcal{I}$) of products of symmetrized versions of the weights from \eqref{eq:wt0}:
	\begin{align}
		\label{eq:symwt}
		\til{\mathscr W}_{i,j} := \begin{cases}
		\tfrac{1}{2} \mathscr W_{i,j} & \mbox{when }i=j, \\
			\mathscr W_{j,i} & \mbox{when }j>i, \\
			\mathscr W_{i,j} & \mbox{when }j<i.
		\end{cases}
	\end{align}
For $m,n,r\in \mathbb{N}$ with $n\ge r$, let $\Pi_{m,n}^{(r)}$ be the set of $r$-tuples of non-intersecting upright paths in $\mathbb{N}^2$ starting from $(1,r), (1,r-1),\cdots, (1,1)$ and going to $(m,n),(m,n-1), \ldots, (m,n-r+1)$ respectively. We define the multi-path point-to-point symmetrized partition function as
	\begin{align}\label{zsymr}
		Z_{\operatorname{sym}}^{(r)}(m,n):=\sum_{(\pi_1,\ldots,\pi_r)\in \Pi_{m,n}^{(r)}}  \prod_{(i,j)\in \pi_1\cup \cdots \cup \pi_r} \til{\mathscr W}_{i,j},
	\end{align}
with the convention that $Z_{\operatorname{sym}}^{(0)}(m,n)\equiv 1$ for all $m,n\in \mathbb{N}$. One can recover the HSLG partition function from the symmetrized partition function via the following identity. For each $(p,q)\in \mathcal{I}$ we have
\begin{align}\label{symi}
    2Z_{\m{sym}}^{(1)}(p,q)=Z(p,q),
\end{align} 
where $Z$ is the HSLG polymer partition function. The identity \eqref{symi} appears in Section 2.1 of \cite{bw} and follows from the symmetry of the weights. We stress that this is an exact equality, not just in distribution.

	\begin{definition}[Half-space log-gamma line ensemble] \label{l:nz}
		Fix $n\in \mathbb{N}$. For each $i\in \ll1,n\rr$ and $j\in \ll1,2n-2i+2\rr$, we set		
		\begin{equation*}
			\L_i^{n;\theta}(j)=\log\left(\frac{2 Z_{\operatorname{sym}}^{(i)}(p,q)}{ Z_{\operatorname{sym}}^{(i-1)}(p,q)}\right),
		\end{equation*}
		where $p:=n+\lfloor j/2 \rfloor$ and $q:=n-\lceil j/2\rceil+1$. We call the collection of random variables
		\begin{align*}
			\big(\L_i^{n;\theta}(j): i\in \ll1,n\rr, j\in \ll1,2n-2i+2\rr\big)
		\end{align*}
the half-space log-gamma (HSLG) line ensemble with parameters $(\alpha,\theta)$.
	\end{definition}
Owing to the identity \eqref{symi}, we have
 \begin{align*}
     (\L_i^{n;\theta}(2j-1))_{j\in \ll1,n\rr} \stackrel{d}{=} (\log Z(n+j-1,n-j+1))_{j\in \ll1,n\rr}.
 \end{align*}
We note that the above definition omits the centering term considered in \cite{half1} and \cite{ds24}, as that centering is only relevant for analyzing the HSLG line ensemble under 1:2:3 scaling. In contrast, as outlined in the introduction, our focus will be on the intermediate disorder scaling. We now introduce this scaling.

\begin{definition}[Scaled HSLG line ensemble]
Fix $t\ge 1, N\in \mathbb{N}$. Set $\theta=\frac12+\sqrt{N}$, and assume $N$ is large enough so that $\alpha>-\theta$. Fix $i\in \ll1, \lfloor Nt/2\rfloor+1\rr$.  Whenever $x\sqrt{N} \in \ll 2i-2\lfloor Nt /2\rfloor-3,0\rr$, we define
\begin{equation}
\label{hinxt}
\begin{aligned}
\hh_i^N(x,t) := \mathcal{L}_i^{\lfloor Nt/2\rfloor + 1;\theta}(-x\sqrt{N}+1) + \left(\frac{2\lfloor Nt/2\rfloor +i +1+ \mathbf{1}_{x\sqrt{N}\;\mathrm{odd}}}{2}\right)\log N,
\end{aligned}
\end{equation}
For other values of $x\le 0$, we define $\hh_i^N(x,t)$ by linear interpolation between the specified lattice points and constant extension to their left, so that $\hh_i^N(\cdot,t)$ is a continuous function on $\mathbb{R}_{\leq 0}$ for each $t$. We call $\{\hh_i^N(\cdot,t)\}_{i\in \ll1,\lfloor Nt/2\rfloor +1\rr}$ the \textit{scaled HSLG line ensemble}. 
\end{definition}

\begin{theorem}\label{thm0}  For each fixed $t\ge 1$, $k\in \mathbb{N}$, we have the following weak convergence:
    \begin{align*}
        (\hh_i^N(\cdot,t))_{i=1}^k \stackrel{d}{\longrightarrow} (\mathcal{H}_i^\alpha(\cdot,t))_{i=1}^k
    \end{align*}
    as $N\to\infty$ in the uniform-on-compact topology, where $(\mathcal{H}_i^\alpha(\cdot,t))_{i=1}^k$ is the HSKPZ line ensemble coming from Theorem \ref{thm1}.
\end{theorem}

As mentioned in Section~\ref{sec1.2.1}, the proofs of Theorems~\ref{thm1} and~\ref{thm0} are closely intertwined, and we will establish them simultaneously in Section~\ref{sec5}.

The convergence of the top curve $\hh_1^N(\cdot,t)$ to the HSKPZ equation $\mathcal{H}^\alpha(\cdot,t)$ essentially follows from results established in \cite{wu,bc22}. However, some care is needed in applying these results due to the fact that the line ensemble is defined along a staircase path. For completeness, we conclude this section with a short proof of finite-dimensional convergence of the top curve that invokes \cite{wu,bc22} as key inputs. This finite-dimensional convergence will be needed in the proof of Theorem \ref{thm0} to apply the characterization result of \cite{dimh}; of course, Theorem \ref{thm0} in particular upgrades the convergence to the process level.

\begin{proposition}\label{thma} For each fixed $t\geq 1$, we have the following convergence of finite-dimensional distributions:
    \begin{align*}
        \hh_1^N(\cdot,t) \stackrel{\mathrm{f.d.}}\longrightarrow \mathcal{H}^\alpha(\cdot,t),
    \end{align*}
    where $\mathcal{H}^\alpha(\cdot,t)$ is the\footnote{The uniqueness in law of $\mathcal{H}^\alpha$ is justified by \cite[Proposition 4.3]{par}.} Cole--Hopf solution to the HSKPZ equation \eqref{halfkpz} with Neumann boundary condition \eqref{bdy} and narrow wedge initial data.
\end{proposition}

\begin{proof} We begin by considering a slight variant of $\hh_1^N(x,t)$ from \eqref{hinxt}: whenever $x\sqrt N\in 2\Z\cap[-2\lfloor Nt/2\rfloor,0]$, we define
\begin{equation*}
\begin{aligned}
\hat{\hh}^N(x,t) := \mathcal{L}_1^{\lfloor Nt/2\rfloor + 1;\theta}(-x\sqrt{N}+1) + \left(\lfloor Nt/2\rfloor + 1\right)\log N.
\end{aligned}
\end{equation*}
For other values of $x\le 0$, we define $\hat{\hh}^N(x,t)$ by linear interpolation and constant extension to the left, so that $\hat{\hh}^N(\cdot,t)$ is a continuous function on $\mathbb{R}_{\leq 0}$ for each $t$.

The difference between $\hat{\hh}^N$ and $\hh_1^N$ lies in the interpolation: $\hat{\hh}^N$ is obtained by interpolating only the odd  points of the top curve of the line ensemble, whereas $\hh_1^N$ uses all points. The advantage of considering $\hat{\hh}^N$ is that it only involves  partition functions $Z(m,n)$ with $m+n=2\lfloor Nt/2\rfloor$, which have been well studied in the literature. In particular invoking \cite[Theorem 2.2]{wu} or \cite[Theorem 5.4]{bc22}, we get that
\begin{align}\label{Hweak}
    \hat{\hh}^N(\cdot,t) \stackrel{d}{\longrightarrow} \mathcal{H}^\alpha(\cdot,t)
\end{align}
as processes in $x$. To extend the convergence to $\hh_1^N$, we observe the following identity from Definition \ref{l:nz} and \eqref{symi}:
\begin{equation}
    \label{evenodd}
    \begin{aligned}
 \mathcal{L}_1^{n;\theta}(2j+2)+\frac12\log N -\mathcal{L}_1^{n;\theta}(2j+1) & = \left[\log 2\mathscr W_{n+j+1,n-j}+\frac12\log N\right] \\ & \hspace{2cm}+ \log\left[\frac12\left(1+e^{\L_1^{n;\theta}(2j+3)-\L_1^{n;\theta}(2j+1)}\right)\right].   
\end{aligned}
\end{equation}
We claim that under the intermediate disorder scaling ($n=\lfloor Nt/2 \rfloor+1$ and $j= - x\sqrt{N}/2 $), the r.h.s. of \eqref{evenodd} goes to zero in probability. Indeed, when $\theta=\frac12+\sqrt{N}$, $\mathscr W_{n+j+1,n-j}$ is distributed as $\operatorname{Gamma}^{-1}(1+2\sqrt{N})$, which implies
$$\log 2\mathscr W_{n+j+1,n-j} +\frac12\log N \stackrel{p}{\longrightarrow} 0$$
by the law of large numbers (see e.g. \eqref{loggamma} below). Thanks to the process-level convergence of $\hat{\hh}^N$, the second term on the right of \eqref{evenodd} goes to $0$ in probability (under intermediate disorder scaling). Note that there is an additional $\frac12\log N$ in the centering in \eqref{hinxt} whenever $x\sqrt{N}$ is odd. Thus, $\hat{\hh}^N(x,t)-\hh_1^N(x,t) \stackrel{p}{\longrightarrow} 0$. In combination with the weak convergence of $\hat{\hh}_1^N$ in \eqref{Hweak}, this implies finite-dimensional convergence of $\hh_1^N$.
\end{proof}

\section{Comparing the full-space and half-space KPZ equations} \label{sec2}

The goal of this section is to prove a parabolic trajectory for the HSKPZ equation \eqref{halfkpz}.

\begin{proposition}\label{p:parabol} Fix $t\geq 1$ and $\e\in (0,1)$. There exists $K=K(\alpha,\e,t)>0$ such that for all $x\le 0$,
    \begin{align}\label{e:parabol}
        \P\bigg(\bigg|\H^\alpha(x,t)+\frac{x^2}{2t}\bigg| \ge K\bigg) \le \e.
    \end{align}
\end{proposition}

We prove Proposition~\ref{p:parabol} by first showing that the half-space KPZ equation is close to the full-space KPZ equation, using chaos expansions for the stochastic heat equation (SHE). We then leverage the spatial stationarity of the full-space KPZ equation after parabolic shift. Towards this end, we introduce the SHE on the full-space:

\begin{definition}[Full-space SHE] \label{defshe} Consider the SPDE
\begin{align}
    \label{fshe}
    \partial_t\mathcal{Z} =\tfrac12\partial_{xx}\mathcal{Z}+\mathcal{Z}\cdot \xi, \qquad \mathcal{Z}=\mathcal{Z}(x,t), \qquad (x,t)\in \mathbb{R}\times (0,\infty).
\end{align}
We consider mild solutions\footnote{We refer to e.g. \cite{corwin2018exactly} for the definition of mild solutions, as we will not need it here.} for two types of initial data:
\begin{enumerate}[label=(\alph*),leftmargin=20pt]
    \item Dirac delta: For $y\in\R$ and $\mathcal{Z}(x,0)=\delta(x-y)$, we denote the solution by $\mathcal{Z}^{\mathrm{full}}(y,0;x,t)$. These are known as propagators (from $(y,0)$ to $(x,t)$) in the literature. When $y=0$, we simply write $\mathcal{Z}^{\mathrm{full}}(x,t)$ for $\mathcal{Z}^{\mathrm{full}}(0,0;x,t)$.
    
    \item \label{hbdata}  Half-Brownian: For $\mathcal{Z}(x,0)=\exp\big(B(x)-\alpha x\big)\ind_{x\ge 0}$ where $B$ is a Brownian motion independent of $\xi$, we denote the solution by $\mathcal{Z}_{\alpha}^{\mathrm{full},B}(x,t)$. 
\end{enumerate}
\end{definition}

The half-Brownian initial data in \ref{hbdata} arises as the scaling limit for the weakly asymmetric simple exclusion process at the edge of the rarefaction fan \cite{corwin2013crossover}. The one-point fluctuations of $\mathcal{Z}_{\alpha}^{\mathrm{full},B}(x,t)$ and its generalizations have been studied in \cite{corwin2013crossover,borodin2014free}. The following convolution formula is well known:
\begin{align}\label{caldec}
    \mathcal{Z}_{\alpha}^{\mathrm{full},B}(x,t)= \int_{0}^\infty e^{B(y)-\alpha y}\mathcal{Z}^{\mathrm{full}}(y,0;x,t)\,dy.
\end{align}
(See for instance Proposition 5.3 in \cite{bc22} for the analogous statement in half space; the convergence of the integral on the right can be justified using chaos expansions as below, cf.~the proof of Proposition 2.9 in \cite{cgh21}.)
We investigate properties of $\mathcal{Z}_{\alpha}^{\mathrm{full},B}(x,t)$ and how it relates to the half-space SHE in  Section \ref{sec8.1}.

\medskip

The proof of Proposition \ref{p:parabol} involves writing out the mild solutions to the half- and full-space SHE as chaos expansions (see \cite[Theorem 2.2]{corwin2018exactly} for the full-space case and \cite[Proposition 5.3]{bc22} for the half-space case). Recall $\mathcal{Z}_\alpha$ denotes the HSSHE \eqref{sheeq}. Fix a space-time white noise $\{\xi(x,t) : x\in\R,\,t\geq 0\}$. For any $x\ge 0$, we may write the solutions $\mathcal{Z}_\alpha$ and $\mathcal{Z}^{\mathrm{full}}$ as 
\begin{equation}
    \label{chaos}
    \begin{aligned}
        &  \frac{\calZ_\alpha(-x,t)}{p_t^\alpha(0,x)}=1+\sum_{k=1}^{\infty} I_k^\alpha(x,t)   \mbox{ where }   I_k^\alpha(x,t):=\frac{1}{p_t^\alpha(0,x)}\int_{\Delta_k(t)} \int_{\R_{\ge 0}^k} \prod_{i=1}^{k+1} p_{t_{i}-t_{i-1}}^\alpha(x_{i-1},x_i)\xi^{\otimes k}(d\mathbf{x},d\mathbf{t}), \\ &  \frac{\calZ^{\mathrm{full}}(x,t)}{\pp_t(x)}=1+\sum_{k=1}^{\infty} I_k(x,t)  \mbox{ where }  I_k(x,t):=\frac{1}{\pp_t(x)}\int_{\Delta_k(t)} \int_{\R^k} \prod_{i=1}^{k+1} \pp_{t_{i}-t_{i-1}}(x_{i-1}-x_i)\xi^{\otimes k}(d\mathbf{x},d\mathbf{t}),
    \end{aligned}
\end{equation} 
where $\Delta_k(t):=\{(t_1,t_2,\ldots,t_k) : 0=t_0< t_1< t_2<\ldots<t_{k}<t_{k+1}=t\}$, $x_0=0$, and $x_{k+1}=x$. {Note we have coupled the two equations with the \textit{same} white noise.} Here 
\begin{align}
    \label{heat}
    \p_t(x)=\frac1{\sqrt{2\pi t}}e^{-x^2/2t}
\end{align}
is the standard heat kernel, and $p_t^\alpha(x,y)$ is the \textit{Robin heat kernel}. Before analyzing the chaos expansions, we recall the definition of $p_t^\alpha$ and state an estimate associated with it. It is the unique solution to the heat equation on $\mathbb{R}_{\ge 0}$ with Robin boundary condition:
\begin{align*}
  &  \partial_tp_t^\alpha(x,y)=\frac12\partial_{xx} p_t^\alpha(x,y), \quad \partial_x p_t^\alpha(0,y)=(\alpha -\tfrac12)p_t^\alpha(0,y),
\end{align*}
and delta initial condition, i.e., $p_t^\alpha(x,y) \to \delta(x-y)$ as $t\downarrow 0$ weakly as measures. By Lemma 4.4 in \cite{cs1}, this solution is given by
\begin{align}\label{Robin}
    p_t^\alpha(x,y)=\pp_t(x+y)+\pp_t(x-y)-2(\alpha-\tfrac12)\int_0^{\infty} \pp_t(x+y+z)e^{-(\alpha-\frac12)z}dz.
\end{align}
The following lemma shows that $p_t^\alpha(x,y)$ can be approximated well by $\p_t(x-y)$.

\begin{lemma}\label{l:heatbd} Fix any $T>1$. There exists a constant $C=C(\alpha,T)>0$ such that for all $x,y\ge 0$ and $t\in [1,T]$ we have
\begin{align}\label{e:heatbd}
C^{-1}\pp_t(x-y) \le p_t^\alpha(x,y) \le C \pp_t(x-y).
\end{align}
Furthermore,
\begin{align}\label{e:heatasym}
    p_t^\alpha(x,y)=\pp_t(x-y)\big(1+e^{-2xy/t}+O(x^{-1})\big),
\end{align}
where the $O$ term depends only on $\alpha,T$ and can be chosen independent of $y$. 
\end{lemma}

\begin{proof} It suffices to prove the inequality when $x,y$ are large enough (say $x+y \ge \alpha T+4T$).  We have $$0 \le \pp_t(x+y+z)e^{-(\alpha-\frac12)z} \le \pp_t(x+y)e^{-z(x+y)/t-(\alpha-\frac12)z}.$$ Assuming $x+y \ge \alpha T+4T$, integrating the above inequality w.r.t.~$z$ over $[0,\infty)$, we get 
\begin{align*}
    0 \le \int_{0}^\infty \pp_t(x+y+z)e^{-(\alpha-\frac12)z}dz \le \frac{t \pp_t(x+y)}{x+y+(\alpha-\frac12)t}.
\end{align*}
Using this inequality in \eqref{Robin}, and the fact that $0 \le \pp_t(x+y) \le \pp_t(x-y)$ (as $x,y\ge 0$), we get the desired bounds in \eqref{e:heatbd}. The precise asymptotic result in \eqref{e:heatasym} also follows from the above estimates (as $\pp_t(x+y)=\pp_t(x-y)e^{-2xy/t}$). 
\end{proof}

We now turn towards the proof of Proposition \ref{p:parabol}.

\begin{proof}[Proof of Proposition \ref{p:parabol}] It suffices to prove the result for large enough $|x|$, {as the restriction of $\mathcal{H}^\alpha(x,t)+x^2/2t$ to any compact interval is continuous and therefore a.s.~bounded}. We claim that {under the coupling \eqref{chaos}}, as $x\to\infty$,
\begin{align}\label{e:closeness}
    \E\left[\left|\frac{\calZ_\alpha(-x,t)}{p_t^\alpha(0,x)}-\frac{\calZ^{\mathrm{full}}(x,t)}{\pp_t(x)}\right|^2\right] \longrightarrow 0.
\end{align}
Recall \cite[Proposition 1.4]{amir2011probability} that $\log\mathcal{Z}^{\mathrm{full}}(x,t)$ is stationary in space after a parabolic shift:
\begin{align*}
    \frac{\calZ^{\mathrm{full}}(x,t)}{\pp_t(x)} \stackrel{d}{=} \frac{\calZ^{\mathrm{full}}(0,t)}{\pp_t(0)}.
\end{align*}
Given this, \eqref{e:closeness} and Markov's inequality imply the desired parabolic trajectory \eqref{e:parabol} for $\mathcal{H}^\alpha = \log \mathcal{Z}_\alpha$. We therefore focus on showing \eqref{e:closeness} using the chaos expansions \eqref{chaos}. The covariance structure of $\xi$ implies that $I_j^\alpha(x,t)-I_j(x,t)$ and $I_k^\alpha(x,t)-I_k(x,t)$ are orthogonal when $j\neq k$, so
 \begin{align*}
     \E\left[\left|\frac{\calZ_\alpha(-x,t)}{p_t^\alpha(0,x)}-\frac{\calZ^{\mathrm{full}}(x,t)}{\pp_t(x)}\right|^2\right] = \sum_{k=1}^{\infty} \E[(I_k^\alpha(x,t)-I_k(x,t))^2].
 \end{align*}
  Using Lemma \ref{l:heatbd} we have $\E[(I_k^\alpha(x,t))^2] \le C^k\cdot \E[(I_k(x,t))^2]$. Furthermore, $$\sum_{k=r}^{\infty} C^k\E[(I_k(x,t))^2] = \sum_{k=r}^{\infty} C^k\frac{\sqrt{\pi}t^{k/2}}{2^k\Gamma((k+1)/2)} \longrightarrow 0$$
    as $r\to \infty$ (see \cite[Lemma 2.4]{corwin2018exactly} for example).  Thus it suffices to show
    \begin{align}\label{tshw}
        \E[(I_k^\alpha(x,t)-I_k(x,t))^2] \longrightarrow 0
    \end{align}
    as $x\to \infty$ for each fixed $k \in \mathbb{N}$. By Ito's isometry, we have
  \begin{align*}
       & \E\big[(I_k^\alpha(x,t)-I_k(x,t))^2\big] 
 = \int_{\Delta_k(t)} \int_{\R^k} V^2(\mathbf{x},\mathbf{t}) \,d\mathbf{x}\,d\mathbf{t}, \mbox{ where} \\
 & V(\mathbf{x},\mathbf{t}):=\frac{\prod_{i=1}^{k+1} p_{t_{i}-t_{i-1}}^\alpha(x_{i-1},x_i)}{p_t^\alpha(0,x)}{\prod_{i=1}^k\ind_{x_i\ge 0}}-\frac{\prod_{i=1}^{k+1} \pp_{t_{i}-t_{i-1}}(x_{i-1}-x_i)}{\pp_t(x)}.
    \end{align*}
  Fix any $M$ large enough and divide the integral into two parts: letting $\R_{M}^k := [M,\infty)^k$, we write
    \begin{align}\label{intM}
         \int_{\Delta_k(t)} \int_{\R^k} V^2(\mathbf{x},\mathbf{t})\,d\mathbf{x}\,d\mathbf{t}= \int_{\Delta_k(t)} \int_{\mathbb{R}_{ M}^k}V^2(\mathbf{x},\mathbf{t})\,d\mathbf{x}\,d\mathbf{t} + \int_{\Delta_k(t)} \int_{(\mathbb{R}_{M}^k)^c}V^2(\mathbf{x},\mathbf{t})\,d\mathbf{x}\,d\mathbf{t}.
    \end{align}
    By \eqref{e:heatasym}, the first integral is at most $O(M^{-1})\cdot \E[(I_k(x,t))^2]$. But $\E[(I_k(x,t))^2]$ is independent of $x$, so in fact the first integral in \eqref{intM} is bounded by $CM^{-1}$. For the second integral we use the estimate \eqref{e:heatbd} to obtain
    \begin{align*}
       \int_{\Delta_k(t)} \int_{(\mathbb{R}_{ M}^k)^c}V^2(\mathbf{x},\mathbf{t})\,d\mathbf{x}\,d\mathbf{t} \le C\int_{\Delta_k(t)} \int_{(\mathbb{R}_{ M}^k)^c}\frac{\prod_{i=1}^{k+1} \pp_{t_{i}-t_{i-1}}^2(x_{i-1}-x_i)}{\pp_t^2(x)}\,d\mathbf{x}\,d\mathbf{t}. 
    \end{align*}
    Since $\pp_t^2(x)=\frac{1}{\sqrt{4\pi t}}\pp_{t/2}(x)$, 

    \begin{align}\label{intB}
         \int_{(\mathbb{R}_{ M}^k)^c}\frac{\prod_{i=1}^{k+1} \pp_{t_{i}-t_{i-1}}^2(x_{i-1}-x_i)}{\pp_t^2(x)} = \left(\frac{t^{1/2}}{(4\pi)^{k/2}}\prod_{i=1}^{k+1}\frac{1}{\sqrt{t_i-t_{i-1}}}\right)\P\left(\min_{1\le i\le k} B(t_i/2) < M\right),
    \end{align}
    where $B$ is Brownian bridge on $[0,t/2]$ from 0 to $x$. Note that $\til{B}(s):=B(s)-2sx/t$ is a Brownian bridge on $[0,t/2]$ from 0 to 0. If $t_1 \ge 2Mt/x$, then $$\P\left(\min_{1\le i\le k} B(t_i/2) < M\right) \le \P\left(\inf_{s\in [0,t/2]} \til B(s) < -M\right) = e^{-4M^2/t}.$$
    The last equality follows from the reflection principle, see for instance \cite[Proposition 12.3.3]{dudley}. Thus if we set $A=\{\mathbf{t}\in \Delta_k(t) : t_1\ge 2Mt/x\}$, then
    \begin{align*}
        \int_{A} \int_{(\mathbb{R}_{ M}^k)^c}\frac{\prod_{i=1}^{k+1} \pp_{t_{i}-t_{i-1}}^2(x_{i-1}-x_i)}{\pp_t^2(x)}\,d\mathbf{x}\,d\mathbf{t} \le e^{-M^2/C} \int_{\Delta_k(t)}\frac{t^{1/2}}{(4\pi)^{k/2}}\prod_{i=1}^{k+1}\frac{1}{\sqrt{t_i-t_{i-1}}}\,d\mathbf{t}.
    \end{align*}
    The last integral is finite. On the other hand, \eqref{intB} implies
    \begin{align*}
        \int_{A^c \cap \Delta_k(t)} \int_{(\mathbb{R}_{ M}^k)^c}\frac{\prod_{i=1}^{k+1} \pp_{t_{i}-t_{i-1}}^2(x_{i-1}-x_i)}{\pp_t^2(x)}\,d\mathbf{x}\,d\mathbf{t} \le C\int_0^{2Mt/x} \frac{dt_1}{\sqrt{t_1}} =O(\sqrt{M/x}).
    \end{align*}
    Combining the above estimates, we find
    \begin{align*}
        \E[(I_k^\alpha(x,t)-I_k(x,t))^2] \le C\big(M^{-1}+e^{-M^2/C}+\sqrt{M/x}\big).
    \end{align*}
    Taking $M=\sqrt{x}$ and letting $x\to \infty$ leads to \eqref{tshw}.
\end{proof}

\section{HSLG and HSKPZ Gibbs properties} \label{sec4}

In this section, we discuss the Gibbs property for the scaled HSLG and HSKPZ line ensembles. 

\begin{definition}[Scale-$N$ log-gamma random walks and bridges] Fix a scale $N\in \mathbb{N}$ with $\sqrt{N}+\frac12>|\alpha|$. Set $\mathbb{Z}_N = \{x : -x\sqrt{N} \in \mathbb{N} \cup \{0\}\}$. 
We say $X$ is distributed as $\log\operatorname{Gamma}(\beta)$ if $e^X$ is distributed as $\operatorname{Gamma}(\beta)$.
    Suppose $(X(i))_{i\ge 0}$ are independent~random variables distributed as $\log\operatorname{Gamma}(\frac12+(-1)^{i}\alpha+\sqrt{N})$. For convenience we record that
    \begin{equation}\label{loggamma}
    \begin{split}
       & \Ex[X(i)] = \psi(\tfrac12+(-1)^{i}\alpha+\sqrt{N})=\log\sqrt{N}+(-1)^{i}\frac{\alpha}{\sqrt{N}}+O(N^{-1}), \\ & \operatorname{Var}(X(i))=\psi'(\tfrac12+(-1)^{i}\alpha+\sqrt{N})=\frac{1}{\sqrt{N}}-(-1)^{i}\frac{\alpha}{N}+O(N^{-3/2}).
    \end{split}
    \end{equation}
    Fix $a,b\in \R$, and $A_1,A_2 \in \mathbb{Z}_{N}$ with $A_1 \le A_2$. Set $S^N(A_1)=a$ and for $x\in \Z_N\cap (A_1,A_2]$ set
    $$S^N(x)-S^N(x-N^{-1/2}):=(-1)^{x\sqrt{N}}\big(X(-x\sqrt{N})-\log \sqrt{N}\big)$$
     We linearly interpolate $S^N$ in between. We refer to the distribution of $S^N$ as the scale-$N$ log-gamma random walk on $[A_1,A_2]$ started from $a$ with parameter $\alpha$. We refer to the distribution of $S^N$ conditioned on $S^N(A_2)=b$ as the scale-$N$ log-gamma random bridge on $[A_1,A_2]$ from $a$ to $b$.
\end{definition}

 HSLG Gibbs measures are absolutely continuous w.r.t.~the above random walks/bridges with an explicit Radon--Nikodym (RN) derivative which we define next.
 
\begin{definition}[Discrete RN derivative] Suppose $N\ge 1, \ell\geq k\geq 1$ are integers  and $A_1,A_2\in \Z_{N}$ with $A_1\le A_2$.  Let $(S_i^N)_{i=k-1}^{\ell+1} :[A_1,A_2]\to \R^{\ell-k+3}$ be a given function. Set $S_i(p):=\infty$ for $i\in \llbracket k-1,\ell+1 \rrbracket$, $p\notin [A_1,A_2]$. We define the discrete RN derivative as
\begin{align}\label{def:W}
    W_{N;f,g}^{k,\ell;A_1,A_2}(S^N) :=  \prod_{i=k-1}^{\ell} \exp\bigg(-\frac1{\sqrt{N}}\ \sum_{j\in [A_1,A_2]\cap 2\Z_N}  \sum_{r\in \{+1,-1\}} e^{S_{i+1}^N(j)-S_i^N(j+rN^{-1/2})} \bigg),
\end{align}
where $f:=S_{k-1}^N$ and $g:=S_{\ell+1}^N$. We often just write $W_{N;f,g}^{k,\ell;A_1,A_2}$ when the law of $S^N$ is clear from context.
\end{definition}

We now proceed define the one-sided and two-sided HSLG Gibbs measures.

\begin{definition}[Two-sided HSLG Gibbs measures] \label{def:twogibbs}
Suppose $N\ge 1$ and $\ell\geq k\geq 1$ are integers and $A_1,A_2\in \Z_N$ with $A_1\le A_2$.  Let $(S_{i}^N)_{i=k}^{\ell}$ be $\ell-k+1$ independent log-gamma random walk bridges on $[A_1,A_2]$ from $(a_i)_{i=k}^{\ell}$ to $(b_i)_{i=k}^{\ell}$. Let  $f,g : [A_1,A_2] \to \R\cup \{\pm \infty\}$ be two given functions. We define the  scale-$N$ two-sided HSLG Gibbs measure by
$$\Pr_{N;f,g}^{k,\ell;A_1,A_2;\vec{a},\vec{b}}(\m{A})=\frac{\Ex[W_{N;f,g}^{k,\ell;A_1,A_2}(S^N)\ind_{\m{A}}]}{\Ex[W_{N;f,g}^{k,\ell;A_1,A_2}(S^N)]}.$$
\end{definition}

 We will often refer to $f$ as the ceiling and $g$ as the floor in the context of the above definition (see Remark \ref{rem:inter} for interpretation). Let us take a moment to explain the notation. We use blackboard fonts (e.g., $\mathbb{P}, \mathbb{E}$) when the underlying system involves log-gamma random walks or bridges. The subscript indicates the scale parameter $N$ and includes the floor/ceiling data, separated by a semicolon. The superscript encodes the first and last indices of the curves, their domain, and the starting and ending boundary data\textemdash again separated by semicolons.

\begin{definition}[One-sided HSLG Gibbs measures]\label{def:halfgibbs} Suppose $N\ge 1$ and $\ell\geq k\geq 1$ are integers and $A\in \Z_N$ with $A<0$.  Let $(S_{i}^N)_{i=k}^{\ell}$ be $\ell-k+1$ independent scale-$N$ log-gamma random walks on $[A,0]$ started from $(a_i)_{i=k}^{\ell}$ with parameters $(-1)^{i}\alpha$. Let  $f,g : [A,0] \to \R\cup \{\pm\infty\}$ be two given functions. We define the scale-$N$ one-sided HSLG Gibbs measure by
$$\Pr_{N;f,g}^{k,\ell;A;\vec{a},\star}(\m{A})=\frac{\Ex[W_{N;f,g}^{k,\ell;A,\0}(S^N)\ind_{\m{A}}]}{\Ex[W_{N;f,g}^{k,\ell;A,\0}(S^N)]}.$$
\end{definition}

In the case of one-sided HSLG Gibbs measures, the right boundary data is free, so we denote it by replacing $\vec{b}$ with $\star$. Since the domains in this setting are always intervals of the form $[A, 0]$, we omit the right endpoint from the notation for brevity. When $k=1$, $f$ will always be $\infty$. In that case, we drop $k$ and $f$ from the notation and simply write $\Pr_{N;g}^{\ell;A_1,A_2;\vec{a},\vec{b}}$ or $\Pr_{N;g}^{\ell;A;\vec{a},\star}$. If $g\equiv -\infty$, we drop $g$ from the notation as well and write $\Pr_{N}^{\ell;A_1,A_2;\vec{a},\vec{b}}$ or $\Pr_{N}^{\ell;A;\vec{a},\star}$. The one-sided HSLG Gibbs measure depends on $\alpha$ but we have suppressed it from the notation.

\begin{definition}[External $\sigma$-algebra]\label{def:ext}
    Fix a line ensemble $(\mathcal{L}_i(x) : i\in\mathbb{N},\, x\in\Lambda)$, where $\Lambda \subset \mathbb{R}$ is an interval (see e.g. Section 2.1 in \cite{dff} for a general definition of line ensembles). For integers $1\leq k\leq \ell$, and an interval $I \subset \Lambda$, we write $\mathcal{F}_{\m{ext}}(\ll k,\ell\rr \times I)$ to denote the $\sigma$-algebra generated by $(\mathcal{L}_i(x) : i\notin \ll k,\ell\rr \mbox{ or } x\notin I)$. 
\end{definition}

The scaled HSLG line ensemble defined in \eqref{hinxt} enjoys a Gibbs property that is described in the next lemma.

\begin{lemma}[Scaled HSLG Gibbs property]\label{lem:gibbs} Fix $N\ge 1$, $t>0$, $\ell \ge k \ge 1$, and $A_1,A_2\in \Z_N$ with $A_1 \le A_2$. Consider the scaled HSLG line ensemble $(\hh_i^N(\cdot,t))_{i\in\ll k,\ell\rr}$ from \eqref{hinxt}. Assume $N$ is large enough that the reference gamma parameters are positive, $\ell+1\leq\lfloor Nt/2\rfloor+1$, and $[A_1,0]$ lies in the genuine lattice domains of these curves and their floor in \eqref{hinxt}.
\begin{enumerate}[label=(\alph*),leftmargin=20pt]
    \item The law of $(\hh_{i}^N(\cdot,t))_{i\in \ll k,\ell\rr}$ on $[A_1,\0]$ conditioned on $\mathcal{F}_{\m{ext}}(\ll k,\ell\rr \times (A_1,0])$ is $\Pr_{N;f,g}^{k,\ell;A_1,\star;\vec{a},\star}$ with $f=\hh_{k-1}^N(\cdot,t)$, $g=\hh_{\ell+1}^N(\cdot,t)$, and $a_i = \hh_{i}^N(A_1)$ for $i\in \llbracket k,\ell \rrbracket$.

    \item The law of $(\hh_{i}^N(\cdot,t))_{i\in \ll k,\ell\rr}$ on $[A_1,A_2]$ conditioned on $\mathcal{F}_{\m{ext}}(\ll k,\ell\rr \times (A_1,A_2))$ is $\mathbb{P}_{N;f,g}^{k,\ell;A_1,A_2;\vec{a},\vec{b}}$, with $f=\hh_{k-1}^N(\cdot,t)$, $g=\hh_{\ell+1}^N(\cdot,t)$, $a_i = \hh_{i}^N(A_1,t)$, and $b_i = \hh_{i}^N(A_2,t)$ for $i\in \llbracket k,\ell \rrbracket$.
\end{enumerate}
\end{lemma}

\begin{proof} This is a consequence of Lemmas 2.2 and 2.5 in \cite{half1}, where the one-sided and two-sided boundary Gibbs property are established for the unscaled HSLG line ensemble.
\end{proof}

We next define the continuous analogues of the above Gibbs measures.

\begin{definition}[Continuous RN derivative] Suppose $A_1\le A_2\le 0$ and $L\ge 1$. Let $B:=(B_i)_{i=k-1}^{\ell+1} :[A_1,A_2]\to \R^{\ell-k+3}$ be a given function. We define the continuous RN derivative as
\begin{align}\label{def:Wcont}
    \mathcal{W}_{L;f,g}^{k,\ell;A_1,A_2}(B):= \prod_{i=k-1}^{\ell} \exp\left(-L\int_{A_1}^{A_2} e^{\sqrt{L}(B_{i+1}(x)-B_i(x))} \,dx \right),
\end{align}
where $f:=B_{k-1}$ and $g:=B_{\ell+1}$. We often just write $\mathcal{W}_{L;f,g}^{k,\ell;A_1,A_2}$ when the law of $B$ is clear from context.
\end{definition}

\begin{definition}[Two-sided HSKPZ Gibbs measures] \label{kpzgibbs2}  Suppose $A_1\le A_2\in\R$ and $L\geq 1$. Let $(B_{i})_{i=k}^{\ell} : [A_1,A_2]\to \R^{\ell-k+1}$ be $\ell-k+1$ many independent Brownian bridges from $(a_i)_{i=k}^{\ell}$ to $(b_i)_{i=k}^{\ell}$. Fix any two functions $f,g : [A_1,A_2] \to \R$, and define the scale-$L$ two-sided HSKPZ Gibbs measure by
$$\mathbf{P}_{L;f,g}^{k,\ell;A_1,A_2;\vec{a},\vec{b}}(\m{A})=\frac{\E[\mathcal{W}_{L;f,g}^{k,\ell;A_1,A_2}(B)\ind_{\m{A}}]}{\E[\mathcal{W}_{L;f,g}^{k,\ell;A_1,A_2}(B)]}.$$
\end{definition}

\begin{definition}[One-sided HSKPZ Gibbs measures] \label{kpzgibbs} Suppose $A\le 0$ and $L\ge 1$. Let $(B_{i})_{i=k}^{\ell} : [A,0]\to \R^{\ell-k+1}$ be $\ell-k+1$ independent Brownian motions started from $(a_i)_{i=k}^{\ell}$ and with drifts $(-1)^i \alpha \sqrt{L}$. We denote their joint law by $\E_\alpha$. Fix any two functions $f,g : [A,0] \to \R$, and define
the scale-$L$ one-sided HSKPZ Gibbs measure as
$$\mathbf{P}_{L;f,g}^{k,\ell;A;\vec{a},\star}(\m{A})=\frac{\E_{\alpha}\left[\mathcal{W}_{L;f,g}^{k,\ell;A,0}(B)\ind_{\m{A}}\right]}{\E_{\alpha}\left[\mathcal{W}_{L;f,g}^{k,\ell;A,0}(B)\right]}.$$
Thanks to Cameron--Martin theorem, we may incorporate the drifts within the RN derivative and have the following equivalent description (with no drift):
\begin{align}
    \label{rnd}
    \mathbf{P}_{L;f,g}^{k,\ell;A;\vec{a},\star}(\m{A})=\frac{\E_{0}\left[\mathcal{W}_{L;f,g}^{k,\ell;A,0}(B)\prod\limits_{i=k}^\ell e^{(-1)^i\alpha \sqrt{L} B_i(0)} \ind_{\m{A}}\right]}{\E_{0}\left[\mathcal{W}_{L;f,g}^{k,\ell;A,0}(B)\prod\limits_{i=k}^\ell e^{(-1)^i\alpha \sqrt{L} B_i(0)}\right]}.
\end{align}
\end{definition}

We use bold fonts (e.g., $\mathbf{P}, \mathbf{E}$) in the above two definitions as the underlying system involves Brownian motions or Brownian bridges. The same conventions for superscripts and subscripts described below Definitions \ref{def:twogibbs} and \ref{def:halfgibbs} apply in this context as well. The scaling parameter $L$ in the above definitions acts as the diffusive scaling parameter. Indeed, if $B$ is distributed as $\mathbf{P}_{1;f\sqrt{L},g\sqrt{L}}^{k,\ell;A_1L,A_2L;\vec{a}\sqrt{L},\vec{b}\sqrt{L}}$ (resp.~$\mathbf{P}_{1;f\sqrt{L},g\sqrt{L}}^{k,\ell;AL;\vec{a}\sqrt{L},\star}$), then $B_L(x):=L^{-1/2}B(xL)$ is distributed as $\mathbf{P}_{L;f,g}^{k,\ell;A_1,A_2;\vec{a},\vec{b}}$ (resp.~$\mathbf{P}_{L;f,g}^{k,\ell;A;\vec{a},\star}$). We shall study the $L\to \infty$ limit of scale-$L$ one-sided HSKPZ Gibbs measures in Section \ref{sec7}.

\begin{remark}[Soft non-intersection and pinning]\label{rem:inter}
    Note that the Radon–Nikodym (RN) derivatives in the one- and two-sided HSLG and HSKPZ Gibbs measures heavily penalize adjacent curves crossing one another. Consequently, these measures can be interpreted as softened versions of non-intersecting random walks/bridges or Brownian motions/bridges. Accordingly, we will often refer to them as log-gamma random walks and bridges (in the case of HSLG Gibbs measures) or Brownian motions and bridges (in the case of HSKPZ Gibbs measures) conditioned ``softly'' not to intersect. The functions $f$ and $g$ serve as additional soft barriers: the RN derivatives heavily penalize configurations where the curves cross above $f$ or below $g$. Accordingly, we will loosely refer to $f$ as a soft ceiling and $g$ as a soft floor. When we take the diffusive scaling limit of HSKPZ Gibbs measures (as we will do in Section~\ref{sec7}), the soft non-intersection and barrier constraints become hard\textemdash meaning the curves are conditioned to strictly non-intersect and to stay below $f$ and above $g$. 
    
    In the one-sided case, and specifically in the supercritical regime $\alpha >0$, the exponential factors in \eqref{rnd} on the other hand penalize configurations in which $B_{2j-1}(0) > B_{2j}(0)$ for any $j$. In combination with the soft non-intersection penalty, this roughly enforces $B_{2j-1}(0) - B_{2j}(0) = O(1)$. In the diffusive limit $L\to\infty$ this will lead to the pairwise pinned non-intersecting Brownian motions mentioned in the introduction, as we show in Section \ref{sec7}.
\end{remark}

The following theorem shows that the scaled HSLG Gibbs measures converge to the HSKPZ Gibbs measures.

\begin{theorem}[Invariance principle] \label{lem:invprin} Fix $N\ge 1$ and $\ell\ge k\ge 1$. Suppose $A_{1,N},A_{2,N} \in \Z_N$ with $A_{1,N}\le A_{2,N}$ and  $A_{1,N}\to A_1$, $A_{2,N}\to A_2$ as $N\to \infty$ for some $A_1 < A_2 \le 0$. Suppose $\vec{a}_N, \vec{b}_N \in \R^{\ell-k+1}$ and $\vec{a}_N \to \vec{a}$, $\vec{b}_N\to\vec{b}$ as $N\to \infty$. Suppose $f_N, g_N :\R \to \R$ with $f_N\to f$ and $g_N\to g$ uniformly.
Suppose $(S_{i}^1(x) : x\in [A_{1,N},A_{2,N}])_{i\in \ll k,\ell\rr}$  and $(S_{i}^2(x) : x\in[A_{1,N},0])_{i\in \ll k,\ell\rr}$ are distributed according to $\Pr_{N;f_N,g_N}^{k,\ell;A_{1,N},A_{2,N};\vec{a}_N,\vec{b}_N}$  and $\Pr_{N;f_N,g_N}^{k,\ell;A_{1,N},\star;\vec{a}_N,\star}$ respectively. We have the convergence 
\begin{align*}
    (S_i^j(x))_{i\in \ll k,\ell \rr} \stackrel{d}{\longrightarrow} (B_i^j(x))_{i\in \ll k,\ell \rr}, \qquad j\in\{1,2\},
\end{align*}
as continuous processes on $\llbracket k,\ell \rrbracket \times [A_1,A_2]$ and $\llbracket k,\ell\rrbracket \times [A_1,0]$, where $(B_i^1(\cdot))_{i\in \ll k,\ell\rr}$ and $(B_i^2(\cdot))_{i\in \ll k,\ell\rr}$ are distributed according to $\mathbf{P}_{1;f,g}^{k,\ell;A_1,A_2;\vec{a},\vec{b}}$  and $\mathbf{P}_{1;f,g}^{k,\ell;A_1;\vec{a},\star}$ respectively.
\end{theorem}

In terms of notation, this invariance principle can be informally summarized as:
 $$\lim_{N\to\infty} \Pr_{N;f_N,g_N}^{k,\ell;A_{1,N},A_{2,N};\vec{a}_N,\vec{b}_N}=\P_{1;f,g}^{k,\ell;A_1,A_2;\vec{a},\vec{b}}, \qquad \lim_{N\to\infty} \Pr_{N;f_N,g_N}^{k,\ell;A_{1,N};\vec{a}_N,\star}=\P_{1;f,g}^{k,\ell;A_1;\vec{a},\star}.$$
\begin{remark}
    Set $\hat{A}_1=\inf_N A_{1,N}$, and $\hat{A}_2=\sup_N A_{2,N}$. All the random and non-random curves in the above statement are viewed as continuous functions on $[\hat{A}_1,0]$ or $[\hat{A}_1,\hat{A}_2]$ by taking them to be constant outside of their original domains.
\end{remark}

\begin{proof} 
Although the log-gamma random walk is not a random walk with i.i.d.~increments (because of the $(-1)^i$ sign in $X(i)$), it is not hard to prove that Donsker's invariance principle still applies for this model: $S^N$ converges weakly to a Brownian motion with drift $\alpha$ or Brownian bridge. The proof will follow easily by applying the invariance principle to the underlying log-gamma random walks/bridges and noting that the discrete RN derivatives $W_{N;f_N,g_N}^{k,\ell;A_{1,N},A_{2,N}}$ converge to the corresponding continuous one $\mathcal{W}^{k,\ell;A_1,A_2}_{1;f,g}$.

Let us focus on $j=1$; the case $j=2$ is analogous. Let $F$ be any bounded continuous functional on $\mathcal{C}([\widehat{A}_1,0], \mathbb{R}^{\ell-k+1})$. It suffices to show that
\begin{equation}\label{invconv}
\mathbb{E}_{N;f_N,g_N}^{k,\ell;A_{1,N};\vec{a}_N,\star}[F(S^1)] \longrightarrow \mathbf{E}_{1;f,g}^{k,\ell,A_1;\vec{a},\star}[F(B^1)]
\end{equation}
as $N\to\infty$. By definition,
\begin{equation}\label{invfrac}
\mathbb{E}_{N;f_N,g_N}^{k,\ell;A_{1,N};\vec{a}_N,\star}[F(S^1)] = \frac{\mathbb{E}[F(S^N) W_{N;f_N,g_N}^{k,\ell;A_{1,N},0}(S^N)]}{\mathbb{E}[W_{N;f_N,g_N}^{k,\ell;A_N,0}(S^N)]},
\end{equation}
where $S^N = (S_i^N)_{i=k}^\ell$ are independent scale-$N$ log-gamma random walks started from $(a_i^N)_{i=k}^\ell$ with parameters $(-1)^i\alpha$. As noted above, the invariance principle implies that $(S_i^N)_{i=k}^\ell$ converge weakly to independent Brownian motions $(B_i)_{i=k}^\ell$ on $[A,0]$ started at $(a_i)_{i=k}^\ell$ and with drifts $(-1)^i\alpha$. By the Skorohod representation theorem, we may pass to another probability space (for which we will use the same notation for convenience) where the convergence (in the uniform topology) holds almost surely.

Now consider the weights $W_{N;f_N,g_N}^{k,\ell;A_{1,N},0}$. In the definition \eqref{def:W}, the exponent in each factor in the product is a Riemann sum, and the uniform convergence $S_i^N \to B_i$, $f_N\to f$, and $g_N\to g$ implies
\[
\frac{1}{\sqrt{N}} \sum_{j\in[A_1,0]\cap2\mathbb{Z}_N} \sum_{r\in\{\pm 1\}} e^{S_{i+1}^N(j) - S_i^N(j+rN^{-1/2})} \longrightarrow \int_{A_1}^0 e^{B_{i+1}(x) - B_i(x)}\,dx.
\]
In view of the definition \eqref{def:Wcont} of $\mathcal{W}_{1;f,g}^{k,\ell;A_1,0}$, it follows that $W_{N;f_N,g_N}^{k,\ell;A_{1,N},0} \longrightarrow \mathcal{W}_{1;f,g}^{k,\ell;A_1,0}$ a.s. Applying the dominated convergence theorem to numerator and denominator of \eqref{invfrac} and using the continuity of $F$ implies \eqref{invconv}.
\end{proof}

We next state an important stochastic monotonicity result for the HSLG Gibbs measures. This result follows from \cite[Proposition 2.6]{half1}.

\begin{lemma}[Stochastic monotonicity] \label{lem:sm}
Fix $N\geq 1$, $\ell\geq k\leq 1$, and $A_1,A_2\in\Z_N$ with $A_1\leq A_2$. For $j=1,2$, fix $(a_i^j)_{i=k}^\ell$ and $(b_i^j)_{i=k}^\ell$ with $a_i^1 \leq a_i^2$ and $b_i^1 \leq b_i^2$ for $i\in\ll k,\ell\rr$, and let $f^j,g^j : [A_1,A_2]\to\R$ be continuous functions such that $f^1(x) \leq f^2(x)$ and $g^1(x)\leq g^2(x)$ for all $x\in[A_1,A_2]$. Then
\begin{enumerate}[label=(\alph*),leftmargin=20pt]

\item There exists a coupling of the two processes $(S_i^j)_{i=k}^\ell \sim \Pr_{N;f^j,g^j}^{k,\ell;A_1,A_2;\vec{a}^j,\vec{b}^j}$ such that $S_i^1(x) \leq S_i^2(x)$ for all $i\in\ll k,\ell\rr$ and $x\in[A_1,A_2]$ a.s.

\item There exists a coupling of the two processes $(S_i^j)_{i=k}^\ell \sim \Pr_{N;f^j,g^j}^{k,\ell;A_1;\vec{a}^j,\star}$ such that $S_i^1(x) \leq S_i^2(x)$ for all $i\in\ll k,\ell\rr$ and $x\in[A_1,0]$ a.s.
\end{enumerate}
{By the invariance principle (Theorem \ref{lem:invprin}), stochastic monotonicity holds for one-sided and two-sided HSKPZ Gibbs measures as well.} 
\end{lemma}

We conclude this section by stating an extension of the Gibbs property from Lemma \ref{lem:gibbs} for random domains, the strong Gibbs property.

\begin{definition}[Stopping domains] \label{def:sd}
    Fix  $1\leq k\leq \ell$ and an interval $\Lambda\subset \R$. Suppose $(\L_i(\cdot))_{i\in \mathbb{N}}$ is a line ensemble on $(-\infty,0]$. For two random variables $\sigma\leq\tau$, we say the random interval $[\sigma,\tau]$ is a stopping domain for $(\mathcal{L}_i(\cdot))_{i\in \ll k,\ell\rr}$ if for all $a\leq b\leq 0$, the event $\{\sigma\leq a, \tau\geq b\}$ lies in the $\sigma$-algebra $\mathcal{F}_{\m{ext}}(\ll k,\ell\rr \times (a,b))$ (recall Definition \ref{def:ext}). Similarly, we say that $[\sigma,0]$ is a stopping domain for $(\mathcal{L}_i(\cdot))_{i\in \ll k,\ell \rr}$ if for all $a\leq 0$, the event $\{\sigma\leq a\}$ lies in $\mathcal F_{\m{ext}}(\ll k,\ell\rr\times(a,0])$. In these cases, we abuse notation slightly and write $\mathcal{F}_{\m{ext}}(\ll k,\ell\rr \times (\sigma,\tau))$ for the $\sigma$-algebra generated by $\sigma,\tau$, and $(\mathcal{L}_i(x) : i\notin\ll k,\ell\rr \mbox{ or } x\notin (\sigma,\tau))$, and likewise for $\mathcal{F}_{\m{ext}}(\ll k,\ell\rr \times (\sigma,0])$.
\end{definition}

\begin{lemma}[HSLG strong Gibbs property]
 \label{lem:sg}    
 Fix $N\geq 1$, $\ell\geq k\geq 1$, and random variables $\sigma,\tau\in\Z_N$ with $\sigma\leq \tau$. Assume the finite-domain and parameter restrictions of Lemma~\ref{lem:gibbs} hold for the resampling domains.
 \begin{enumerate}[label=(\alph*),leftmargin=20pt]

     \item If $[\sigma,\tau]$ is a stopping domain for $(\hh_i^N(\cdot,t))_{i\in\llbracket k,\ell\rrbracket}$, then the law of $(\hh_i^N(\cdot,t))_{i\in\llbracket k,\ell\rrbracket}$ on $[\sigma,\tau]$ given $\mathcal{F}_{\m{ext}}(\ll k,\ell\rr \times (\sigma,\tau))$ is $\Pr_{N;f,g}^{k,\ell;\sigma,\tau;\vec{a},\vec{b}}$, with $f=\hh_{k-1}^N(\cdot,t)$, $g=\hh_{\ell+1}^N(\cdot, t)$, $a_i = \hh_i^N(\sigma,t)$, and $b_i = \hh_i^N(\tau,t)$ for $i\in\ll k,\ell \rr$. 
     
     \item If $[\sigma,0]$ is a stopping domain for $(\hh_i^N(\cdot,t))_{i\in\llbracket k,\ell\rrbracket}$, then the law of $(\hh_i^N(\cdot,t))_{i\in\llbracket k,\ell\rrbracket}$ on $[\sigma,0]$ given $\mathcal{F}_{\m{ext}}(\ll k,\ell\rr \times (\sigma,0])$ is $\Pr_{N;f,g}^{k,\ell;\sigma;\vec{a},\star}$, with $f=\hh_{k-1}^N(\cdot,t)$, $g=\hh_{\ell+1}^N(\cdot, t)$, and $a_i = \hh_i^N(\sigma,t)$ for $i\in\ll k,\ell \rr$. 
 \end{enumerate}
\end{lemma}

The proof of the strong Gibbs property is a straightforward consequence of the standard Gibbs property in Lemma \ref{lem:gibbs}; see for instance \cite[Lemma 2.5]{kpzle}.

\section{Construction of the HSKPZ line ensemble} \label{sec5}

In this section we prove our main results constructing the HSKPZ line ensemble, Theorems \ref{thm1} and \ref{thm0}. We will rely heavily on Proposition \ref{tightline}, which we state here. Its proof is deferred to the next section.

\begin{proposition}\label{tightline}
Fix $t\ge 1$. For all $n\geq 1$, $\varepsilon>0$, there exists $R_n = R_n(\varepsilon)>0$ such that for all $x_0>0$, large $N$, and $\overline{x}\in[-x_0,0]$, we have
\begin{equation*}
    \mathbb{P} \bigg(\sup_{x\in[\overline{x}-1,\overline{x}]} \left|\hh_n^N(x,t) + \frac{x^2}{2t} \right| >  R_n \bigg) < \varepsilon.
\end{equation*}
\end{proposition}

We will also need the following technical ingredient, which establishes a lower bound on the normalization constant for the one-sided HSLG Gibbs measure describing the HSLG line ensemble on an interval near the boundary.

\begin{proposition}\label{zlbd}
Fix $t\ge 1$, $k\in\mathbb{N}$, and $A<0$, and let $A_N$ be the largest element of $\Z_N$ less than $A$. Let $\vec{x}^N = (\hh_1^N(A_N,t),\dots,\hh_k^N(A_N,t))$. Let $(S_i^N)_{i=1}^k$ be $k$ independent scale-$N$ log-gamma random walks on $[A_N,0]$ started from $(x_i^N)_{i=1}^k$ with parameters $(-1)^{i}\alpha$. For any $\e>0$ there exists $\delta>0$ such that for all large $N$,
\begin{equation*}
\mathbb{P}\left( \Ex \left[W^{1,k;A_N,0}_{N;+\infty,g_t}(S^N)\right] < \delta\right) < \e,
\end{equation*}
where $W^{1,k;A_N,0}_{N;+\infty,g_t}(S^N)$ is defined in \eqref{def:W} and $g_t := \hh_{k+1}^N(\cdot,t)|_{[A_N,0]}$.
\end{proposition}

{The expectation in $\Ex[ W^{1,k;A,0}_{N;+\infty,g_t}(S^N)]$ is taken only w.r.t.~the randomness coming from the log-gamma random walks. This expectation is still random as $g_t= \hh_{k+1}^N(\cdot,t)$ and the starting data $\vec{x}_N=(\hh_i^N(A_N,t))_{i=1}^k$ are random.} 

\begin{proof}
For $M>0$, define the event
\[
\mathsf{E}_M := \bigg\{\sup_{x\in[A_N,0]} \hh_{k+1}^N(x,t) \leq M \bigg\} \cap \left\{\max_{1\leq i\leq k} |\hh_i^N(A_N,t)| \leq M \right\}.
\]
By Proposition \ref{tightline}, we can choose $M$ large enough so that $\mathbb{P}(\mathsf{E}_M) > 1-\e/2$ for large $N$. Let $\m{D}_\delta := \{\Ex [W^{1,k;A_N,0}_{N;+\infty,g_t}(S^N)] < \delta\}$. Then 
\[
\mathbb{P}(\mathsf{D}_\delta) \leq \mathbb{P}(\mathsf{D}_\delta\cap\mathsf{E}_M) + \e/2,
\]
so it suffices to show that for small enough $\delta$, $\mathbb{P}(\mathsf{D}_\delta \cap\mathsf{E}_M) \to 0$ as $N\to\infty$. On the event $\mathsf{E}_M$, by definition of $W_{N;+\infty,\bullet}^{1,k;A_N,0}$ in \eqref{def:W}, we have
\begin{equation}\label{Wmlbd}
W^{1,k;A_N,0}_{N;+\infty,g_t}(S^N) \geq W^{1,k;A_N,0}_{N;+\infty,M}(S^N).
\end{equation}
Define the new event
\[
\mathsf{F}_{M,\delta,N} := \left\{ \Ex \left[W^{1,k;A_N,0}_{N;+\infty,M}(S^N)\right] < \delta, \, \max_{1\leq i\leq k} |x_i^N| \leq M \right\},
\]
where $x_i^N$ are defined in the statement. Then $\mathbb{P}(\mathsf{D}_\delta \cap \mathsf{E}_M) \leq \mathbb{P}(\mathsf{F}_{M,\delta,N})$ by \eqref{Wmlbd}. By Fatou's lemma,
\begin{equation*}
    \limsup_{N\to\infty} \mathbb{P}(\mathsf{F}_{M,\delta,N}) \leq \mathbb{E}\left[\limsup_{N\to\infty} \mathbf{1}_{\mathsf{F}_{M,\delta,N}}\right].
\end{equation*}
We claim that the limsup inside the expectation is identically 0 if $\delta>0$ is small enough. Let $\Omega$ denote the implicit probability space supporting the HSLG line ensemble, and fix $\omega\in\Omega$. Let $\Ex [W^{1,k;A_N,0}_{N;+\infty,M}(S^N(\omega))]$ denote the corresponding realization of the random variable $\Ex [W^{1,k;A_N,0}_{N;+\infty,M}(S^N)]$ (w.r.t.~the randomness of the boundary conditions $x_i^N(\omega)$\textemdash again, the expectation is taken over the random walk paths). It suffices to show that for some $\delta>0$, any subsequence of $\{\mathbf{1}_{\mathsf{F}_{M,\delta,N}}(\omega)\}_{N\geq 1}$ has a further subsequence converging to 0 as $N\to\infty$. If $|x_i^N(\omega)| > M$ for some $i$, then $\mathbf{1}_{\mathsf{F}_{M,\delta,N}}(\omega) = 0$. So it is enough to consider $\omega$ such that for infinitely many $N$, $\max_{1\leq i\leq k} |x_i^N(\omega)| \leq M$. Then any subsequence has a further subsequence $\{N_j\}_{j\geq 1}$ such that $\vec{x}^{N_j}(\omega)$ converges to some $\vec{x}(\omega) \in [-M,M]^k$. Now for any $\vec{x}\in\R^k$ let $B^{\vec{x}} = (B_i^{x_i})_{i=1}^k$ denote $k$ independent Brownian motions on $[A,0]$ started from $x_i$ with drifts $(-1)^{i}\alpha$, and let $\mathcal{W}_{1;+\infty,M}^{1,k;A,0}(B^{\vec{x}})$ be as in \eqref{def:Wcont}. Then it follows by the argument in the proof of Theorem \ref{lem:invprin} that
\begin{equation}\label{Zconv}
\Ex \left[W^{1,k;A_N,0}_{N_j;+\infty,M}(S^{N_j}(\omega))\right] \longrightarrow \E  \left[\mathcal{W}_{1;+\infty,M}^{1,k;A,0}(B^{\vec{x}(\omega)})\right]
\end{equation}
as $j\to\infty$. Define 
\[
\delta := \frac{1}{2} \inf_{\vec{x}\in[-M,M]^k} \E \left[\mathcal{W}^{1,k;A,0}_{1;+\infty,M}(B^{\vec{x}})\right].
\]
This $\delta$ exists and is positive, since $\E[\mathcal{W}^{1,k;A,0}_{+\infty,M}(B^{\vec x})]$ is clearly a positive continuous function of $\vec{x}\in\R^k$ and $[-M,M]^k$ is compact. By \eqref{Zconv}, we can choose $j_0 = j_0(\omega)$  so that for all $j\geq j_0$ we have $\Ex [W_{+\infty,M}^{1,k;A_N,0}(S^{N_j}(\omega))] \geq \delta$, thus $\mathbf{1}_{\mathsf{F}_{M,\delta,N_j}}(\omega)=0$ for $j\geq j_0$. So indeed $\mathbf{1}_{\mathsf{F}_{M,\delta,N_j}}(\omega)\to 0$ as $j\to\infty$, and we are done.
\end{proof}

We are now equipped to construct the HSKPZ line ensemble.

\begin{proof}[Proof of Theorems \ref{thm1} and \ref{thm0}]
We will split the proof into two steps for clarity. We first prove tightness of the rescaled HSLG line ensemble; we then establish the Gibbs property for subsequential limits and appeal to a uniqueness result of \cite{dimh} to obtain convergence to a unique limit, which we identify as the HSKPZ line ensemble.

\medskip

\noindent\textbf{Step 1.} We prove tightness as follows. For $A<0$, a function $\mathbf{f} = (f_1,\dots,f_k) : \ll 1, k\rr \times [A,0] \to \mathbb{R}$, and $\delta > 0$, define the modulus of continuity
\begin{align}\label{moc}
    \mc(\mathbf{f},\delta) := \max_{1\leq i\leq k} \sup_{\substack{x,y\in [A,0], \\ |x-y|<\delta}} |f_i(x)-f_i(y)|.
\end{align}

Then it suffices to prove the following two statements:
\begin{enumerate}
    \item For all $k\in\mathbb{N}$, and $t,\e>0$, there exists $M>0$ so that
    \[
    \limsup_{N\to\infty} \mathbb{P}\left( |\hh^N_k(0,t)| > M \right) < \e.
    \]

    \item For all $k\in\mathbb{N}$, $A<0$, and $t,\e,\rho>0$, there exists $\delta>0$ so that
    \[
    \limsup_{N\to\infty} \mathbb{P}\left(\mc(\hh^N(\cdot,t)|_{\ll 1, k\rr \times [A_N,0]}, \delta) > \rho\right) < \e.
    \]
\end{enumerate}
Here $A_N$ denotes the largest element of $\Z_N$ less than $A$. Statement (1) follows immediately from Proposition \ref{tightline}, so it remains to prove statement (2). For $\delta,\rho,\eta,M>0$, define the two events
\begin{align*}
\mathsf{A}_N(\delta,\rho) &:= \left\{\mc(\hh^N(\cdot,t)|_{\ll 1, k\rr \times [A_N,0]},\delta) > \rho \right\}, \quad \mathsf{B}_N(\eta,M) := \left\{\mathbb{E}\left[W_{N;+\infty,\hh_{k+1}^N(\cdot,t)}^{1,k;A_N,0}(S^N)\right] > \eta \right\}.
\end{align*}
Here $S^N = (S_i^N)_{i=1}^k$ are independent scale-$N$ log-gamma random walks on $[A_N,0]$ started from $(\hh_i^N(A_N,t))_{i=1}^k$. By Proposition \ref{zlbd}, we may choose $\eta,M>0$ so that
\begin{equation}\label{BNetaM}
    \liminf_{N\to\infty} \mathbb{P}(\mathsf{B}_N(\eta,M)) > 1 - \e/2.
\end{equation}
The one-sided Gibbs property (Lemma \ref{lem:gibbs}) implies that
\begin{equation}\label{ANBN}
\begin{split}
    \mathbb{P}(\mathsf{A}_N(\delta,\rho) \cap \mathsf{B}_N(\eta,M)) &= \mathbb{E} \left[\mathbf{1}_{\mathsf{B}_N(\eta,M)} \mathbb{P}\left(\mathsf{A}_N(\delta,\rho) \mid \mathcal{F}_{\mathrm{ext}}(\ll 1, k\rr \times (A_N,0]) \right) \right]\\
    &= \mathbb{E} \left[\mathbf{1}_{\mathsf{B}_N(\eta,M)} \mathbb{P}_{N;\hh_{k+1}^N(\cdot,t)}^{k;A_N;\vec{a},\star}(\mathsf{A}_N(\delta,\rho)) \right],
    \end{split}
\end{equation}
where we have written $\vec{a} = (\hh_1^N(A_N,t),\dots,\hh_k^N(A_N,t))$. Note on the event $\mathsf{B}_N(\eta,M)$, where the normalization constant for the law $\mathbb{P}_{N;\hh_{k+1}^N(\cdot,t)}^{k;A;\vec{a},\star}$ is bounded below by $\eta$, we may upper bound the probability in the second line of \eqref{ANBN} by
\begin{equation}\label{tightprod}
\mathbb{P}_{N;\hh_{k+1}^N(\cdot,t)}^{k;A_N;\vec{a},\star}(\mathsf{A}_N(\delta,\rho)) \leq \eta^{-1} \sum_{i=1}^k \mathbb{P}\left(\mc(S^N_i,\delta)>\rho\right),
\end{equation}
where $S_i^N$ are independent scale-$N$ log-gamma random walks on $[A,0]$ with parameters $(-1)^i\alpha$ started from $(a_i)_{i=1}^k$. Note the modulus of continuity of $S^N_i$ is independent of $a_i$ (this just amounts to a vertical shift of the entire walk), so we may assume that in fact each $S_i^N$ starts at 0. Then by Donsker's invariance principle, $S_i^N$ converge weakly as $N\to\infty$ to independent Brownian motions $B^i$ on $[A,0]$ started from 0 with drifts $(-1)^i\alpha$. The modulus of continuity of each $B^i$ is stochastically bounded above by $\mathbf{w}(B,\delta)+|\alpha|\delta$, where $B$ is a standard Brownian motion on $[A,0]$ with no drift. Therefore we may choose $N$ large enough so that the r.h.s.~of \eqref{tightprod} is at most $2\eta^{-1}k \cdot \P(\mathbf{w}(B,\delta)+|\alpha|\delta>\rho)$. By L\'{e}vy's modulus of continuity for Brownian motion, this probability can be made arbitrarily small by choosing $\delta$ small depending on $\rho,A,\alpha$. In particular, we can take $N$ large and $\delta$ small so that the r.h.s.~of \eqref{tightprod} is at most $\e/2$, uniformly over the event $\m{B}_N(\eta,M)$. In view of \eqref{BNetaM} and \eqref{ANBN}, statement (2) above follows.

\medskip

\noindent\textbf{Step 2.} We now show that any subsequential limit of $\hh^N(\cdot,t)$ must satisfy the Brownian Gibbs property described in Theorem \ref{thm1}(b). Let $\mathcal{H} = (\mathcal{H}_i(x,t) : i\in\mathbb{N}, x\in\mathbb{R})$ be any such subsequential limit; without loss of generality we may assume that in fact $\hh^N(\cdot,t) \to \mathcal{H}(\cdot,t)$. Applying the Skorohod representation theorem again, we may pass to another probability space where the convergence is almost sure. From now on we will omit the variable $t$ from the notation as it doesn't play a role in the argument.

We will show that for any $k\in\mathbb{N}$ and $A<0$, any bounded continuous functional $F$ on $\mathcal{C}(\ll1,k\rr\times[A,0])$, and any bounded continuous cylinder functional $G$ of the full exterior of $\ll1,k\rr\times(A,0]$, we have
\begin{equation}\label{eqgibbs}
\E \left[F(\mathcal{H}|_{\ll 1,k\rr \times [A,0]}) G(\mathcal{H}) \right] = \E\left[ \mathbf{E}_{1;g}^{k;A;\vec{a},\star}[F] \cdot G(\mathcal{H})\right],
\end{equation}
where $\vec{a} = (\mathcal{H}_1(A),\dots,\mathcal{H}_k(A))$ and $g = \mathcal{H}_{k+1}|_{[A,0]}$. By a standard monotone class argument (see e.g. the proof of \cite[Lemma 2.3]{dm21} for a slightly different phrasing), this implies the desired Gibbs property. 

In the remainder of the proof we will abuse notation by omitting the restrictions, i.e., we will write $F(\mathcal{H})$ and $G(\mathcal{H})$ to mean the expressions in \eqref{eqgibbs}. By the assumed almost sure convergence $\hh^N \to \mathcal{H}$ and the dominated convergence theorem,
\[
\E \left[F(\mathcal{H}) G(\mathcal{H}) \right]  = \lim_{N\to\infty} \mathbb{E} \left[F(\hh^N) G(\hh^N) \right] .
\]
Let $A_N$ be the smallest element of $[A,0] \cap \mathbb{Z}_N$. The Gibbs property for $\hh^N$ (Lemma \ref{lem:gibbs}) implies that
\begin{equation}\label{gibbssplit}
\begin{split}
\mathbb{E} \left[F(\hh^N) G(\hh^N) \right] &= \mathbb{E} \left[ \mathbb{E}[F(\hh^N) \mid \mathcal{F}_{\mathrm{ext}}(\ll 1, k\rr \times (A_N,0])] \cdot G(\hh^N)\right]\\
&= \mathbb{E}\left[\mathbb{E}_{N;\hh_{k+1}^N}^{k;A_N;\vec{a}_N,\star}[F] \cdot G(\hh^N) \right],
\end{split}
\end{equation}
where $\vec{a}_N = (\hh_1^N(A_N),\dots,\hh_k^N(A_N))$. In evaluating $F$ under the resampling law, keep the original paths on $[A,A_N]$; this segment vanishes in the limit. The uniform convergence $\hh^N\to\mathcal{H}$ and the continuity of $F$ imply, by Theorem \ref{lem:invprin}, that 
\[
\mathbb{E}_{N;\hh_{k+1}^N}^{k;A_N;\vec{a}_N,\star}[F] \longrightarrow \mathbf{E}_{\mathcal{H}_{k+1}}^{k;A;\vec{a},\star}[F]
\]
a.s. as $N\to\infty$. Returning to \eqref{gibbssplit} and using the continuity of $G$, the dominated convergence theorem now implies \eqref{eqgibbs}.

To conclude the actual weak convergence of the sequence $\{\hh^N\}_{N\geq 1}$, we now appeal to the characterization of $H$-Brownian line ensembles in \cite{dimh}. In the language of that paper, our argument above shows that any subsequential limit $\mathcal{H}$ of $\hh^N$ is an $H$-Brownian Gibbs line ensemble on the interval $(-\infty,0]$, with $H = e^x$. (Although we only verified the one-sided Gibbs property, an exactly analogous argument verifies the two-sided Gibbs property.) It follows from \cite[Theorem 1.1]{dimh} that $\mathcal{H}$ is determined by the finite-dimensional distributions of its top curve $\mathcal{H}_1$. But by Proposition \ref{thma}, $\mathcal{H}_1$ is necessarily equal in finite-dimensional distributions to the HSKPZ equation $\mathcal{H}^\alpha$. Therefore there is only one possible subsequential limit $\mathcal{H}$. Along with the tightness verified in Step 1, this shows that every subsequence of $\{\hh^N\}_{N\ge 1}$ has a further subsequence converging weakly to $\mathcal{H}$. This implies that in fact $\hh^N\to \mathcal{H}$ weakly.
\end{proof}

{As a byproduct of the above proof we see that the continuous 
analogues of Lemmas \ref{lem:gibbs} and \ref{lem:sg} are valid for HSKPZ line ensemble as well. We state them here for completeness.}

\begin{lemma}[HSKPZ Gibbs property]\label{lem:gibbs2} Fix  $\ell \ge k \ge 1$,  $A_1 \le A_2 \le 0$, and $t\ge 1$. Consider the HSKPZ line ensemble $(\mathcal{H}_i^\alpha(\cdot,t))_{i\in \mathbb{N}}$. 
\begin{enumerate}[label=(\alph*),leftmargin=20pt]
    \item The law of $(\mathcal{H}_{i}^\alpha(\cdot,t))_{i\in \ll k,\ell\rr}$ on $[A_1,\0]$ conditioned on $\mathcal{F}_{\m{ext}}(\ll k,\ell\rr \times (A_1,0])$ is $\P_{1;f,g}^{k,\ell;A_1;\vec{a},\star}$ with $f=\mathcal{H}_{k-1}^\alpha(\cdot,t)$, $g=\mathcal{H}_{\ell+1}^\alpha(\cdot,t)$, and $a_i = \mathcal{H}_{i}^\alpha(A_1,t)$ for $i\in \llbracket k,\ell \rrbracket$.

    \item The law of $(\mathcal{H}_{i}^\alpha(\cdot,t))_{i\in \ll k,\ell\rr}$ on $[A_1,A_2]$ conditioned on $\mathcal{F}_{\m{ext}}(\ll k,\ell\rr \times (A_1,A_2))$ is $\mathbf{P}_{1;f,g}^{k,\ell;A_1,A_2;\vec{a},\vec{b}}$, with $f=\mathcal{H}_{k-1}^\alpha(\cdot,t)$, $g=\mathcal{H}_{\ell+1}^\alpha(\cdot,t)$, $a_i = \mathcal{H}_{i}^\alpha(A_1,t)$, and $b_i = \mathcal{H}_{i}^\alpha(A_2,t)$ for $i\in \llbracket k,\ell \rrbracket$.
\end{enumerate}
\end{lemma}

\begin{lemma}[HSKPZ strong Gibbs property]
 \label{lem:sg2}    
 Fix $\ell\geq k\geq 1$, $t\ge 1$, and random variables $\sigma,\tau$ with $\sigma\leq \tau$. Consider the HSKPZ line ensemble $(\mathcal{H}_i^\alpha(\cdot,t))_{i\in \mathbb{N}}$. Recall Definition \ref{def:sd} of stopping domains.
 \begin{enumerate}[label=(\alph*),leftmargin=20pt]
     \item If $[\sigma,0]$ is a stopping domain for $(\mathcal{H}_i^\alpha(\cdot,t))_{i\in\llbracket k,\ell\rrbracket}$, then the law of $(\mathcal{H}_i^\alpha(\cdot,t))_{i\in\llbracket k,\ell\rrbracket}$ on $[\sigma,0]$ given $\mathcal{F}_{\m{ext}}(\ll k,\ell\rr \times (\sigma,0])$ is $\P_{1;f,g}^{k,\ell;\sigma;\vec{a},\star}$, with $f=\H_{k-1}^\alpha(\cdot,t)$, $g=\H_{\ell+1}^\alpha(\cdot, t)$, and $a_i = \H_i^\alpha(\sigma,t)$ for $i\in\ll k,\ell \rr$. 
     
     \item If $[\sigma,\tau]$ is a stopping domain for $(\mathcal{H}_i^\alpha(\cdot,t))_{i\in\llbracket k,\ell\rrbracket}$, then the law of $(\mathcal{H}_i^\alpha(\cdot,t))_{i\in\llbracket k,\ell\rrbracket}$ on $[\sigma,\tau]$ given $\mathcal{F}_{\m{ext}}(\ll k,\ell\rr \times (\sigma,\tau))$ is $\P_{1;f,g}^{k,\ell;\sigma,\tau;\vec{a},\vec{b}}$, with $f=\H_{k-1}^\alpha(\cdot,t)$, $g=\H_{\ell+1}^\alpha(\cdot, t)$, $a_i = \H_i^\alpha(\sigma,t)$, and $b_i = \H_i^\alpha(\tau,t)$ for $i\in\ll k,\ell \rr$. 
 \end{enumerate}
\end{lemma}
By Brownian scaling, the rescaled ensemble in \eqref{def:Hrescaled} satisfies the same Gibbs properties with $L=t^{2/3}$ when $\beta=0$. For $\beta>0$, the two-sided property is unchanged by the common affine tilt; the one-sided reference drifts and wall location change.

\section{Tightness of HSLG line ensemble under intermediate disorder scaling} \label{sec6}

In this section, we prove Proposition \ref{tightline}. Following the strategy of \cite{kpzle}, we prove the proposition using three key lemmas, which can be viewed as analogues of Propositions 6.1, 6.2, and 6.3 of \cite{kpzle}. 

\begin{lemma}\label{lowertight1}
For all $n\geq 1$, $\varepsilon>0$, there exists $R_n = R_n(\varepsilon)>0$ such that for all $x_0>0$, large $N$, and $\overline{x}\in[-x_0,0]$, we have
\begin{equation*}
    \mathbb{P} \left(\inf_{x\in[\overline{x}-1,\overline{x}]} \left(\hh_n^N(x,t) + \frac{x^2}{2t} \right) < - R_n \right) < \varepsilon.
\end{equation*}
\end{lemma}

\begin{lemma}\label{lowertight2}
Fix any $\delta_0>0$. For all $n\geq 1$, $\varepsilon>0$, $\delta \in (0,\delta_0)$, there exists $T_0>0$ such that for all $x_0 > T_0$, large $N$, $T\in[T_0,x_0]$, and $y_0\in [-x_0+T,0]$, we have
\begin{equation*}
    \mathbb{P}\left(\inf_{x\in[y_0-T,y_0]} \left(\hh_n^N(x,t) + \frac{x^2}{2t}\right) < -\delta T^2 \right) < \varepsilon.
\end{equation*}
\end{lemma}

\begin{lemma}\label{uppertight}
For all $n\geq 1$, $\varepsilon>0$, there exists $\hat{R}_n = \hat{R}_n(\varepsilon) > 0$ such that for all $x_0>0$, large $N$, and $\overline{x} \in [-x_0,0]$, we have
\begin{equation*}
\Pr\bigg(\sup_{x\in[\overline{x}-1,\bar{x}]} \left(\hh_n^N(x,t) + \frac{x^2}{2t}\right) > \hat{R}_n \bigg) < \varepsilon.
\end{equation*}
\end{lemma}

\begin{proof}[Proof of Proposition \ref{tightline}] Proposition \ref{tightline} is a direct consequence of Lemmas \ref{lowertight1} and \ref{uppertight}.
\end{proof}

The proof of the three key lemmas will proceed by the same induction scheme as in \cite{kpzle}. We will write Lemma \ref{lowertight1}$(n)$ to denote the statement of Lemma \ref{lowertight1} at index $n$, and similarly for the other two lemmas. Lemma \ref{lowertight1}$(n)$ will be proven using the statements of all three lemmas for index $n-1$. Lemma \ref{lowertight2}$(n)$ will then be proven using Lemma \ref{lowertight1}$(n)$ and Lemma \ref{lowertight2}($n-1$). Lastly, Lemma \ref{uppertight}$(n)$ will be proven using Lemma \ref{lowertight1} for indices $n$, $n-1$, and $n-2$, and Lemma \ref{uppertight}$(n-1)$.

\subsection{Base case of the induction} In this subsection, we prove the three key lemmas for $n=1$.

\begin{proof}[Proof of Lemmas \ref{lowertight1} and \ref{lowertight2}: $n=1$ case.] We prove the base case for Lemma \ref{lowertight1} only; the proof of Lemma~\ref{lowertight2} below also covers its base case, with ceiling $+\infty$.
Thanks to Propositions \ref{thma} and  \ref{p:parabol}, we can find $R_0(\e)>0$ such that for all $x \in (-\infty,0]$,
\begin{equation}\label{w1}
   \limsup_{N\to\infty} \mathbb{P} \left( \hh_1^N(x,t) + \frac{x^2}{2t}  < - R_0 \right) < \varepsilon/8.
\end{equation}
Let us fix $\overline{x}\in [-x_0,0]$. Without loss of generality we assume $\overline{x}\in \Z_N$. Take $\overline{x}'\in [\overline{x}-2,\overline{x}-1]\cap \Z_N$.
Let $M(\e)>0$ be such that $e^{-M^2} \le \e/4$. Set $R_1:=R_0+M+2$, and define the events
\begin{align*}
    \m{A}:= \bigcap_{x\in \{\overline{x}',\overline{x}\}}\left\{\hh_1^N(x,t) + \frac{x^2}{2t}  \ge - R_0\right\}, \quad \m{B}:=\left\{\inf_{x\in[\overline{x}',\overline{x}]} \left(\hh_1^N(x,t) + \frac{x^2}{2t} \right) \ge - R_1\right\}.
\end{align*}
Let $\mathcal{F} = \mathcal{F}_{\operatorname{ext}}\big(\{1\}\times (\overline{x}',\overline{x})\big)$.  Note that $\m{A}\in\mathcal{F}$. By a union bound and the tower property of conditional expectation we have
\begin{align*}
    \Pr(\neg\m{B}) \le \Pr(\neg\m{A})+\Ex\left[\ind_{\m{A}}\Ex\left[\ind_{\neg\m{B}} | \mathcal{F} \right]\right].
\end{align*}
Thanks to \eqref{w1}, $\Pr(\neg\m{A}) \le \e/2$ for large enough $N$. Hence it suffices to show 
\begin{align}\label{w5}
    \limsup_{N\to \infty} \Ex\left[\ind_{\m{A}}\Ex\left[\ind_{\neg\m{B}} | \mathcal{F} \right]\right] \le \e/4.
\end{align}
By the HSLG Gibbs property (Lemma \ref{lem:gibbs}), $\Ex\left[\ind_{\neg\m{B}} | \mathcal{F} \right] = \Pr_{N;f,g}^{1;\overline{x}',\overline{x};a,b}(\neg\m{B})$, 
where $f=\infty, g=\hh_2^N(\cdot,t)$, $a=\hh_1^N(\overline{x}',t)$, and $b=\hh_1^N(\overline{x},t)$.
We observe that on the event $\m{A}$, $a \ge -R_0-\overline{x}'^2/2t$ and $b\ge -R_0-\overline{x}^2/2t$. By stochastic monotonicity (Lemma \ref{lem:sm}), letting $g\downarrow -\infty$,  $a\downarrow -R_0-\overline{x}'^2/2t$, and $b\downarrow -R_0-\overline{x}^2/2t$ only increases the probability. Thus,
\begin{align*} \ind_{\m{A}}\cdot\Pr_{N;f,g}^{1;\overline{x}',\overline{x};a,b}(\neg\m{B}) \le \Pr\bigg( \inf_{x\in [\overline{x}',\overline{x}] } S^N(x)+\frac{x^2}{2t} \le -R_1\bigg),
\end{align*}
where $S^N$ is a scale-$N$ log-gamma random bridge on $[\overline{x}',\overline{x}]$ from $-R_0-\overline{x}'^2/2t$ to $-R_0-\overline{x}^2/2t$. In view of the above inequality, an application of Fatou's lemma leads to
\begin{align}\label{w4}
    \limsup_{N\to \infty} \Ex\left[\ind_{\m{A}}\Ex\left[\ind_{\neg\m{B}} | \mathcal{F} \right]\right] \le  \limsup_{N\to\infty}\Pr\bigg( \inf_{x\in [\overline{x}',\overline{x}] } S^N(x)+\frac{x^2}{2t} \le -R_1\bigg).
\end{align}
Let $(u_N,v_N) \in \Z_{N}$ be a pair that maximize the probability on the r.h.s.~of \eqref{w4} over all (finitely many) possible choices of $(\overline{x}',\overline{x}) \in \{(z,w)\in \Z_N^2 : w\in [-x_0,0], w-z\in [1,2]\}$. Then
\begin{align*}
    \limsup_{N\to\infty}\Pr\bigg( \inf_{x\in [\overline{x}',\overline{x}] } S^N(x)+\frac{x^2}{2t} \le -R_1\bigg) \le \limsup_{N\to\infty}\Pr\bigg( \inf_{x\in [u_N,v_N] } S^N(x)+\frac{x^2}{2t} \le -R_1\bigg).
\end{align*}
We focus on the subsequence along which the limsup on the right hand side is attained. Passing through another subsequence (still denoted as the original sequence) we may assume  $u_N\to u$ and $v_N\to v$. Then by the invariance principle (Theorem \ref{lem:invprin}), 
\begin{align}\label{w3}
\limsup_{N\to\infty}\Pr\bigg( \inf_{x\in [u_N,v_N] } S^N(x)+\frac{x^2}{2t} \le -R_1\bigg) \le \mathbf{P}\bigg( \inf_{x\in [u,v]} B(x)+\frac{x^2}{2t}\le -R_1\bigg),
\end{align}
where $B$ is Brownian bridge on $[u,v]$ from $-R_0-u^2/2t$ to $-R_0-v^2/2t$. Note that $v-u \in [1,2]$. By for instance \cite[Proposition 12.3.3]{dudley} (and by the choice of $M$) we have
 \begin{align}\label{w2}
     \mathbf{P}\left(\inf_{x\in [u,v]} (B(x)-\E[B(x)]) \le -M\right) \le e^{-2M^2/(v-u)} \le \e/4.
 \end{align}
On the other hand, for $x\in [u,v]$, a simple computation shows that 
$$\E[B(x)]+\frac{x^2}{2t}=-R_0-\frac1{2t}(x-u)(v-x) \ge -R_0-2=-R_1+M.$$ 
In view of \eqref{w2}, it follows that r.h.s.~of \eqref{w3} $\le \e/4$. Plugging this bound back into \eqref{w4} verifies \eqref{w5}, completing the proof.
\end{proof}
\begin{proof}[Proof of Lemma \ref{uppertight}: $n=1$ case.] 
It suffices to bound the supremum on $[\bar x-\frac12,\bar x]$, and then cover $[\bar x-1,\bar x]$ by two such intervals. We assume the endpoints below lie in $\Z_N$; rounding outward changes the following bounds by $o(1)$, uniformly for $\bar x\in[-x_0,0]$.
Set $a=\bar x-2$, $z=\bar x-1$, and let $\tau_N$ be the last point of $[\bar x-\frac12,\bar x]\cap\Z_N$ where $\hh_1^N(x,t)+x^2/(2t)>\hat R_1$, using $\bar x-\frac12$ if there is no such point. Let $\m A$ denote the event that there is such a point, and put
\[
\m B=\{\hh_1^N(a,t)+a^2/(2t)\geq-R_1\},\qquad
\m C=\{\hh_1^N(z,t)+z^2/(2t)>\hat R_1/2-R_1-1\}.
\]
By Lemma \ref{lowertight1} for $n=1$ (which we recall has been proven already), we may choose $R_1$ large enough so that $\Pr(\neg\mathsf{B}) < \e/2$. By Propositions \ref{thma} and \ref{p:parabol}, choose $\hat R_1$ so that $\Pr(\m C)<\e/8$ for large $N$. We claim that 
\begin{align}\label{twc}
 \Pr(\m{C}) \ge \tfrac14 \cdot \Pr(\m{A}\cap \m{B}),
\end{align}
for all large enough $N$. This implies $\Pr(\m{A})\le \Pr(\neg\m{B})+4\Pr(\m{C}) <\e$. We thus focus on proving \eqref{twc}. Towars this end, note that
\begin{align*}
    \Pr(\m{C}) &\geq \Pr(\m{A} \cap \m{B} \cap \m{C}) = \Ex\left[ \ind_{\m{A} \cap \m{B}} \Ex \left[\ind_{\m{C}} \mid \mathcal{F}_{\mathrm{ext}}(\{1\} \times (a,\tau_N))\right]\right].
\end{align*}
By the strong Gibbs property (Lemma \ref{lem:sg}),
\begin{align*}
    \ind_{\m{A} \cap \m{B}} \cdot \Ex \left[\ind_{\m{C}} \mid \mathcal{F}_{\mathrm{ext}}(\{1\} \times (a,\tau_N))\right] = \ind_{\m{A} \cap \m{B}} \cdot \Pr_{N; \hh_2^N(\cdot,t)}^{1;a,\tau_N;\hh_1^N(a,t),\hh_1^N(\tau_N,t)}(\m{C}).
\end{align*}
On the event $\m{A}$, $\hh_1^N(\tau_N,t) \ge v:= \widehat{R}_1 - \tau_N^2/2t$ and on the event $\m{B}$, $\hh_1^N(a,t) \ge u := -R_1 - a^2/2t$. Thus, by stochastic monotonicity (Lemmas \ref{lem:sm}) we may write
\begin{align*}
    \ind_{\m{A} \cap \m{B}} \cdot \Pr_{N; \hh_2^N(\cdot,t)}^{1;a,\tau_N;\hh_1^N(a,t),\hh_1^N(\tau_N,t)}(\m{C}) 
    &\geq \ind_{\m{A} \cap \m{B}} \cdot \Pr_{N; -\infty}^{1;a,\tau_N;u,v} \left(S^N(z) > \frac{u+v}{2} \right).
\end{align*}
The law in the probability expression on the right hand side of the above equation is simply the law of a random walk bridge with roughly the midpoint requiring to be above the midpoint of the endpoints of the bridge. Then by invariance principle (Theorem \ref{lem:invprin}), for sufficiently large $N$, the above probability on the right hand side is at least $\frac14$. This yields the second inequality in \eqref{twc}.
\end{proof}

\subsection{Proof of Lemma \ref{lowertight1}: Induction step}
We assume that $n\geq 2$ and that all lemmas have been proven up to index $n-1$. Let $T_0$ be as in Lemma \ref{lowertight2}$(n-1)$. Let $\hat{R}_{n-1}$ be as in Lemma \ref{uppertight}$(n-1)$. Set  $u=\hat{R}_{n-1}-(\overline{x}-2T)^2/2t$, $v=\hat{R}_{n-1}-(\overline{x}-T)^2/2t$.  We choose $T$ large enough depending on $\hat{R}_{n-1},\e,t$ such that
\begin{align}\label{tcond}
   & -\frac{T^2}{16t}-\frac{1}{2t}\bigg(\bar{x}-\frac{3T}{2}\bigg)^2 \ge T+\frac{u+v}{2}, \qquad \exp\left(\sqrt{2\pi t}\,e^{-T}\right)\cdot \frac{e^{-2T}}{1-e^{-2T}} < \frac{\e}{2}.
\end{align}
holds for all $\bar{x}\le 0$. The first inequality above holds due to strict concavity of the function $-x^2/2t$. Set $M:=2T+\frac{T^2}{8t}$ and define the event
\begin{equation*}
    \mathsf{E}_n^N := \left\{\sup_{x\in[\overline{x}-2T,\overline{x}-T]} \left(\hh_n^N(x,t) + \frac{x^2}{2t}\right) > -M \right\}.
\end{equation*}

The proof of Lemma \ref{lowertight1}$(n)$ will follow directly from the next two lemmas. 
\begin{lemma}\label{lowtightE}
    For all $x_0>0$, $N \geq N_0(x_0,\varepsilon)$, and $\overline{x}\in[-x_0,0]$,
    \begin{equation*}
        \mathbb{P}(\neg\mathsf{E}_n^N) \le 4\varepsilon.
    \end{equation*}
\end{lemma}

\begin{lemma}\label{lowtightEinf}
    For all $x_0>0$, $N\geq N_0(x_0,\varepsilon)$, and $\overline{x}\in[-x_0,0]$, 
    \begin{equation*}
        \mathbb{P}\left(\left\{\inf_{x\in[\overline{x}-1,\overline{x}]} \left(\hh_n^N(x,t) + \frac{x^2}{2t}\right) < - R_n \right\} \cap \mathsf{E}_n^N \right) < C\varepsilon.
    \end{equation*}
\end{lemma}

\begin{proof}[Proof of Lemma \ref{lowtightE}]
Define the event
\begin{equation*}
    \mathsf{H}_{n-1}^N := \left\{\hh_{n-1}^N(x,t) + \frac{x^2}{2t} < \hat{R}_{n-1} \mbox{ at } x = \overline{x} -T \mbox{ and } x = \overline{x} - 2T \right\}.
\end{equation*}
Then $\mathbb{P}(\neg\mathsf{H}_{n-1}^N) < 2\varepsilon$ for large $N$ by Lemma \ref{uppertight}($n-1$), so it is enough to prove the following claim: 
\begin{equation}\label{lowtightEclaim}
\mathbb{P}(\neg\mathsf{E}_n^N \cap \mathsf{H}_{n-1}^N) < 2\varepsilon.
\end{equation}
Let 
\[
\mathsf{A} := \left\{\hh_{n-1}^N \left(\overline{x} - \frac{3T}{2},t\right) + \frac{(\overline{x}-\frac{3T}{2})^2}{2t} < - \frac{T^2}{16t}\right\}.
\]
By Lemma \ref{lowertight2}$(n-1)$, $\mathbb{P}(\mathsf{A})<\varepsilon$ for large $N$. Therefore to prove \eqref{lowtightEclaim} it suffices to show that
\begin{equation}\label{tofs}
    \mathbb{P}(\neg\mathsf{E}_n^N \cap \mathsf{H}_{n-1}^N \cap \neg\mathsf{A}) < \varepsilon.
\end{equation}
for all large enough $N$.
To do so, we write
\begin{equation*}
    \mathbb{P}(\neg\mathsf{E}_n^N \cap \mathsf{H}_{n-1}^N \cap \neg\mathsf{A}) = \mathbb{E}\left[\mathbf{1}_{\neg\mathsf{E}_n^N \cap \mathsf{H}_{n-1}^N} \mathbb{E} \left[\mathbf{1}_{\neg\mathsf{A}} \mid \mathcal{F}_{\mathrm{ext}}\left(\{n-1\}\times(\overline{x}-2T, \overline{x}-T)\right)\right]\right].
\end{equation*}
We now apply the Gibbs property (Lemma \ref{lem:gibbs}) to the inner expectation to write
\begin{align}\nonumber
     & \mathbf{1}_{\neg\mathsf{E}_n^N \cap \mathsf{H}_{n-1}^N} \cdot \mathbb{E} \left[\mathbf{1}_{\neg\mathsf{A}} \mid \mathcal{F}_{\mathrm{ext}}\left(\{n-1\}\times(\overline{x}-2T, \overline{x}-T)\right)\right] \\ & \hspace{1cm}= \mathbf{1}_{\neg\mathsf{E}_n^N \cap \mathsf{H}_{n-1}^N} \cdot \mathbb{P}_{N;\hh^N_{n-2}(\cdot,t),\hh^N_{n}(\cdot,t)}^{n-1,n-1;\hh_{n-1}^N(a,t),\hh_{n-1}^N(b,t)} \bigg(S^N \left(\overline{x}-\tfrac{3T}{2}\right) +\frac{(\overline{x}-\frac{3T}{2})^2}{2t} > -\frac{T^2}{16t}\bigg), \label{etht}
\end{align}
where $a=\overline{x}-2T$, $b=\overline{x}-T$. (Without loss of generality we assume $a,b\in \Z_N$.)
Note by stochastic monotonicity that on $\neg\m{E}_n^N \cap \m{H}_{n-1}^N$, if we replace the boundary conditions with $(u,v) := (\hat{R}_{n-1}-(\overline{x}-2T)^2/2t,\hat{R}_{n-1}-(\overline{x}-T)^2/2t)$ and the ceiling $\hh^N_{n-2}$ with $+\infty$, the event $\neg\mathsf{A}$ becomes more likely. Moreover, on the event $\neg\mathsf{E}_n^N$, $\hh_n^N(x)$ does not exceed $-M_1(x):=-M-x^2/2t$, so replacing $\hh_n^N(x,t)$ by $-M_1(x)$ makes the event more likely. Thus (reindexing the walk from $n-1$ to $1$ and omitting the ceiling of $+\infty$ in the notation) we have
\begin{align*}
 \eqref{etht}  &  \le \mathbf{1}_{\neg\mathsf{E}_n^N \cap \mathsf{H}_{n-1}^N} \cdot \mathbb{P}_{N;-M_1}^{1;a,b;u,v}\bigg(S^N \left(\overline{x}-\tfrac{3T}{2}\right) +\frac{(\overline{x}-\frac{3T}{2})^2}{2t} > -\frac{T^2}{16t}\bigg) \\ & \le \mathbb{P}_{N;-M_1}^{1;a,b;u,v}\left(S^N\left(\overline{x}-\tfrac{3T}{2}\right) \ge T+\frac{u+v}{2}\right),
\end{align*}
where the last inequality follows from the first condition on $T$ in \eqref{tcond}. Thus, by Fatou's lemma
\begin{align} \nonumber
    \limsup_{N\to \infty} \mathbb{P}(\neg\mathsf{E}_n^N \cap \mathsf{H}_{n-1}^N \cap \neg\mathsf{A}) &  \le \limsup_{N\to \infty} \mathbb{P}_{N;-M_1}^{1;a,b;u,v}\left(S^N \left(\overline{x}-\tfrac{3T}{2}\right) \ge T+\frac{u+v}{2}\right) \\ \label{invs} & =\mathbf{P}_{1;-M_1}^{1;a,b;u,v}\left(B(\bar{x}-3T/2) > T+\frac{u+v}{2}\right) \\ & = \mathbf{P}_{1;-\hat{M}}^{1;a,b;0,0}\left(B(\bar{x}-3T/2) > T\right), \label{convb}
\end{align}
where 
\begin{align}\label{affineshift}
    \hat{M}(x) & :=M_1(x)+\frac{u(b-x)+v(x-a)}{b-a} 
    = M+ \hat{R}_{n-1} - \frac{T^2}{8t}+\frac{1}{2t}\left(x-\left(\bar{x}-\tfrac{3T}{2}\right)\right)^2. 
\end{align}
Here the equality in \eqref{invs} follows from the invariance principle (Theorem \ref{lem:invprin}) and  the equality in \eqref{convb} follows via an affine shift of the Brownian bridge. The second equality in \eqref{affineshift} follows from the definitions of $a,b,u,v$ and simple algebra.
Next we note that by Definition \ref{kpzgibbs}, we have 
\begin{align*}
    & \mathbf{P}_{1;-\hat{M}}^{1;a,b;0,0}\left(B(\bar{x}-3T/2) > T\right) 
    \le \frac{\mathbf{P}\left(B(\bar{x}-3T/2) > T\right)}{\mathbf{E}\left[\exp\left(-\int_{\bar{x}-2T}^{\bar{x}-T} e^{-\hat{M}(x)-B(x)}dx\right)\right]}, 
\end{align*}
where the process $B : [\bar{x}-2T,\bar{x}-T]\to \R$ is a Brownian bridge from $0$ to $0$ under $\mathbf{P}$. By standard Brownian bridge tail estimates, e.g. \cite[Proposition 12.3.3]{dudley}, we have $\mathbf{P}\left(B(\bar{x}-3T/2) > T\right) \le e^{-2T}$. Let us thus bound the denominator of the above fraction. To do this, we note that on the event $\{\inf_{x\in [\bar{x}-2T,\bar{x}-T]} B(x) \ge  -T\}$, we have
\begin{align*}
  \int_{\bar{x}-2T}^{\bar{x}-T} e^{-\hat{M}(x)-B(x)}dx \le \int_{\R} e^{T-\hat{M}(x)}dx 
  \le \sqrt{2\pi t}\, e^{-T},
\end{align*}
where the last inequality follows by Gaussian integration and using the definition $M=2T+\frac{T^2}{8t}$. Thus,
\begin{align*}
    \mathbf{E}\left[\exp\left(-\int_{\bar{x}-2T}^{\bar{x}-T} e^{-\hat{M}(x)-B(x)}dx\right)\right] & \ge \exp\left(-\sqrt{2\pi t}\,e^{-T}\right)\cdot \mathbf{P}\left(\inf_{x\in [\bar{x}-2T,\bar{x}-T]} B(x) \ge  -T\right) \\ & = \exp\left(-\sqrt{2\pi t}\,e^{-T}\right)\cdot (1-e^{-2T}),
\end{align*}
where the last equality is again by \cite[Proposition 12.3.3]{dudley}. Thus overall we have
\begin{align*}
    \mathbf{P}_{1;-\hat{M}}^{1;a,b;0,0}\left(B(\bar{x}-3T/2) > T\right) \le \exp\left(\sqrt{2\pi t}\,e^{-T}\right)\cdot \frac{e^{-2T}}{1-e^{-2T}} <\frac{\e}{2},
\end{align*}
where the last inequality follows from the second condition in \eqref{tcond}. Thus by \eqref{convb},
\begin{align*}
    \limsup_{N\to \infty}  \mathbb{P}(\neg\mathsf{E}_n^N \cap \mathsf{H}_{n-1}^N \cap \neg\mathsf{A}) \le \frac{\e}{2}, 
\end{align*}
proving \eqref{tofs}.
\end{proof}

\begin{proof}[Proof of Lemma \ref{lowtightEinf}]
{Introduce $J=[\bar x-2T,0]$ if $\bar x>-2T$, and $J=[\bar x-2T,\bar x+2T]$ otherwise, and the events}
\begin{align*}
    \mathsf{F}_{n-1}^N &:= \left\{{\inf_{x\in J}} \left(\hh_{n-1}^N(x,t) + \frac{x^2}{2t}\right) \geq -M + 2\sqrt{T} \right\},\\
    \mathsf{G} &:= \left\{\inf_{x\in[\overline{x}-T,\overline{x}]}\left(\hh_n^N(x,t) + \frac{x^2}{2t}\right) < -R_n \right\}.
\end{align*}
{By Lemma \ref{lowertight2}$(n-1)$, applied on $J$ with sufficiently small $\delta=\delta(t)$ and using $|J|\leq4T$, $\mathbb{P}(\neg\mathsf{F}_{n-1}^N)<\varepsilon$ for large $T,N$.} We claim that
\begin{equation*}
    {\mathbb{P}(\mathsf{E}_n^N \cap \mathsf{F}_{n-1}^N \cap \mathsf{G}) < 5\varepsilon.}
\end{equation*}
We split the proof into two cases, depending on the size of $\overline{x}$. We will only present the argument in detail in the case $\overline{x} >-2T$. 
Introduce the random time
\[
\sigma_n^N := \inf\left\{ x\in[\overline{x}-2T, \overline{x}-T]\cap \mathbb{Z}_N : \hh_n^N(x,t) + \frac{x^2}{2t} \geq -M \right\}.
\]
{If no such $x$ exists, set $\sigma_n^N=\bar x-T$ (with grid rounding understood). Keep the event $\m E_n^N$ explicitly: on this event a hit occurs. For $\bar x\leq-2T$, use also the last analogous hit in $[\bar x+T,\bar x+2T]$. Its failure has probability at most $4\e$ by the proof of Lemma \ref{lowtightE}, applied to that interval. Resample between the two hits, with the ceiling controlled on $J$. After subtracting the chord, the reference bridge has length at most $4T$, and its mean plus $x^2/(2t)$ lies between $-M-2T^2/t$ and $-M$. The reflection-principle estimates below then give the same bound. We give the one-sided case in detail.
}We write
\begin{equation*}
    \mathbb{P}(\mathsf{E}_n^N \cap \mathsf{F}_{n-1}^N \cap \mathsf{G}) = \mathbb{E} \left[ \mathbf{1}_{\mathsf{E}_n^N \cap \mathsf{F}_{n-1}^N} \mathbb{E} \left[\mathbf{1}_{\mathsf{G}}  \mid \mathcal{F}_{\mathrm{ext}}(\{n\}\times(\sigma_n^N,0])\right]\right].
\end{equation*}
By the strong Gibbs property (Lemma \ref{lem:sg}) we note that
\begin{align*}
    & \mathbf{1}_{\mathsf{E}_n^N \cap \mathsf{F}_{n-1}^N} \cdot \mathbb{E} \left[\mathbf{1}_{\mathsf{G}}  \mid \mathcal{F}_{\mathrm{ext}}(\{n\}\times(\sigma_n^N,0])\right] \\ & \hspace{2cm} = \mathbf{1}_{\mathsf{E}_n^N \cap \mathsf{F}_{n-1}^N} \cdot  \mathbb{P}_{N;\hh_{n-1}^N(\cdot,t),\hh_{n+1}^N(\cdot,t)}^{n,n;\sigma_n^N;\hh_n^N(\sigma_n^N,t),\star} \left(\inf_{x\in[\overline{x}-T,\overline{x}]} \left( S^N(x) + \tfrac{x^2}{2t} \right) < - R_n \right).
\end{align*}
By stochastic monotonicity (Lemma \ref{lem:sm}), the above probability will increase if we replace $\hh_{n-1}^N(x,t)$ with $-M+2\sqrt{T}-x^2/2t$ (on the event $\m{F}_{n-1}^N$), $\hh_{n+1}^N$ with $-\infty$, and (by definition of $\sigma_n^N$) replace $\hh_n^N(\sigma_n^N,t) + (\sigma_n^N)^2/2t$ by $-M$ (on the event $\m{E}_n^N$). Thus we have an upper bound of
\begin{equation}\label{sigmaprob}
\mathbf{1}_{\mathsf{E}_n^N \cap \mathsf{F}_{n-1}^N} \cdot \mathbb{P}_{N;f,-\infty}^{n,n;\sigma_n^N;c,\star}\left(\inf_{x\in[\overline{x}-T,\overline{x}]} \left( S^N(x) + \tfrac{x^2}{2t} \right) < - R_n \right),
\end{equation}
where $c=-M-(\sigma_n^N)^2/2t$ and $f(x)=-M+2\sqrt{T}-x^2/2t$. 

To estimate \eqref{sigmaprob}, we will first translate to the continuum. For each $N$, let $\tau^N \in [\overline{x}-2T,\overline{x}-T]$ be chosen so that the probability in \eqref{sigmaprob} is maximized when $\sigma_n^N$ is replaced by $\tau^N$ and $c$ is replaced by $-M-(\tau^N)^2/2t$. This can be done since there are finitely many choices of $\sigma_n^N\in \Z_N \cap [\bar{x}-2T,\bar{x}-T]$ on the event $\m{E}_n^N$ and thus of $\tau^N$. Since the sequence $(\tau^N)_{N\geq 1}$ is bounded, we can pass to a subsequence along which $\tau^N\to\tau \in[\overline{x}-2T,\overline{x}-T]$. Then by Theorem \ref{lem:invprin}, the limsup of the probability in \eqref{sigmaprob} is bounded above, uniformly over the event $\mathsf{E}_n^N \cap \mathsf{F}_{n-1}^N$, by
\begin{align}\label{Bparab}
& \mathbf{P}_{1;f,-\infty}^{n,n;\tau;-M-\tau^2/2t,\star} \left(\inf_{x\in[\overline{x}-T,\overline{x}]} (B(x)+\tfrac{x^2}{2t}) < -R_n \right). 
\end{align}
{Using the definition of one-sided HSKPZ Gibbs measures (Definition \ref{kpzgibbs})} we get
\begin{align}\label{Bparab1}
    \eqref{Bparab} = \frac{\E\left[\exp\left(-\int_{\tau}^0 e^{B(x)-f(x)}\,dx\right) \mathbf{1}\left\{\inf_{x\in[\overline{x}-T,\overline{x}]} (B(x)+\tfrac{x^2}{2t}) < -R_n \right\}\right]}{\E\left[\exp\left(-\int_{\tau}^0 e^{B(x)-f(x)}\,dx\right)\right]},
\end{align}
where $B$ is a Brownian motion on $[\tau,0]$ with drift $(-1)^n\alpha$ started from $-M-\tau^2/2t$. We have
\begin{align*}
    \mbox{r.h.s.~of \eqref{Bparab1}} \le \frac{\P\left(\inf_{x\in[\tau,0]} (B(x)+\tfrac{x^2}{2t}) < -R_n \right)}{\E\left[\exp\left(-\int_{\tau}^0 e^{B(x)-f(x)}\,dx\right)\right]}.
\end{align*}
We write this probability in terms of a standard Brownian motion. 
Note that $B(x) \stackrel{d}{=} \Br(x)+(-1)^n \alpha (x-\tau)-M-\tau^2/2t$ where $\Br$ is a standard Brownian motion on $[\tau,0]$ (started at 0). Hence the last expression is equal to
\begin{align}\label{6.xxfrac}
    \frac{\P\left(\inf_{x\in[\tau,0]} (\Br(x)+(-1)^n \alpha (x-\tau)-M-\tau^2/2t+\tfrac{x^2}{2t}) < -R_n \right)}{\E\left[\exp\left(-\int_{\tau}^0 e^{\Br(x)+(-1)^n \alpha (x-\tau)-\tau^2/2t+x^2/2t-2\sqrt{T}}\,dx\right)\right]}.
\end{align}
We first lower bound the denominator. Since the parabola $-x^2/2t$ is concave, it lies above the line segment connecting $(\tau,-\tau^2/2t)$ and $(0,0)$ of slope $|\tau|/2t$. We may therefore bound $(-1)^n\alpha(x-\tau)-\tau^2/2t+x^2/2t \leq (|\alpha|-|\tau|/2t)(x-\tau)$. Recalling $\tau\in[-4T,-T]$, for $T>2|\alpha|t$ this last expression is negative for all $x\in[\tau,0]$. Using the reflection principle and the fact that $\tau\in[-4T,-T]$, we have $$\P\bigg(\sup_{x\in[\tau,0]} \Br(x) \le \sqrt{T}\bigg) = \P(|\Br(0)|\le \sqrt{T})   \ge \int_{-1/4}^{1/4} \p_1(x)\,dx=:c_1.$$ (Recall $\mathfrak{p}_1$ is the standard heat kernel \eqref{heat}.) Then for $T>2|\alpha|t$, we obtain a lower bound on the denominator of \eqref{6.xxfrac} of 
\begin{equation}\label{6.xxdenom}
    c_1 \exp\left(-\int_\tau^0 e^{\sqrt{T}-2\sqrt{T}}\,dx\right) \geq c_1\exp(-4Te^{-\sqrt{T}}) \geq c_1c_2,
\end{equation}
where $c_2 = \min_{T\geq 0} \exp(-4Te^{-\sqrt{T}}) > 0$. We conclude by upper bounding the numerator of \eqref{6.xxfrac}. Let us choose {$R_n = M+4|\alpha|T+8T^2/t+K\sqrt{T}$}, for a constant $K$ to be determined.
We can lower bound {$(-1)^n\alpha(x-\tau)-M-\tau^2/(2t)+x^2/(2t)\geq-4|\alpha|T-M-8T^2/t$}. Then  the numerator of \eqref{6.xxfrac} is less than or equal to
$$\P\bigg(\inf_{x\in[\tau,0]}\Br(x) \leq -K\sqrt{T}\bigg)= 2\P(\Br(0) \le -K\sqrt{T}) \le 2e^{-K^2T/2|\tau|}\le 2e^{-K^2/8},$$  where the equality above is again via the reflection principle. Combining with the lower bound from \eqref{6.xxdenom}, we see that \eqref{6.xxfrac} is bounded above (for $T>2|\alpha|t$) by $2c_1^{-1}c_2^{-1}e^{-K^2/8}$. Taking $K$ large depending on $c_1,c_2$, this can be made arbitrarily small, and we are done.
\end{proof}

\subsection{Proof of Lemma \ref{lowertight2} : Induction step}
{We assume that Lemma \ref{lowertight1}$(n)$ has been proven, and, when $n\geq2$, that Lemma \ref{lowertight2}$(n-1)$ has been proven; for $n=1$ take the ceiling to be $+\infty$. It suffices to treat $0<\delta\leq1/16$.} Taking $R_n = R_n(\e\delta/3)$ from Lemma \ref{lowertight1}$(n)$, we have for all $x_0>0$, $x\in[-x_0,0]$, and large $N$,
\begin{equation}\label{ed/3}
    \mathbb{P}\left(\hh_n^N(x,t) + \frac{x^2}{2t} < -R_n \right) \leq \frac{\e\delta}{3}.
\end{equation}
Fix $y_0\in [T-x_0,0]$. Instead of working with the interval $[y_0-T,y_0]$, we shall work with the slightly modified interval $I_N=[a_N,b_N]:=[{N^{-1/2}\lfloor\sqrt{N}(y_0-T)\rfloor}, N^{-1/2}\lceil\sqrt{N}y_0\rceil]$. Introduce the event
\[
\mathsf{C} := \left\{\inf_{x\in I_N}\left(\hh_{n-1}^N(x,t) + \frac{x^2}{2t} \right) \geq - \frac{1}{2}\delta T^2\right\} \cap \left\{\inf_{x\in\{a_N,b_N\}} \left(\hh_n^N(x,t) + \frac{x^2}{2t}\right) \geq - R_n\right\}.
\]
Note the control on $\hh_n^N$ is only imposed at the boundary. By \eqref{ed/3} and Lemma \ref{lowertight2}$(n-1)$, we can assume there exist $T_0>0$, $N_0\in\mathbb{N}$ such that for all $x_0\geq T_0$, $T\in[T_0,x_0]$, $y_0\in[-x_0+T,0]$, and $N\geq N_0$ we have
\begin{equation}\label{CyT}
    \mathbb{P}(\mathsf{C}) \geq 1-\e/3.
\end{equation}
We also assume $T_0$ is large enough so that $R_n \leq \frac{5}{8}\delta T_0^2$. We want to show that for large enough $N$ and $T$,
\begin{equation}\label{EyT}
    \mathbb{P}(\mathsf{E}) \leq \e, \mbox{ where } \mathsf{E} := \left\{\inf_{x\in I_N} \left(\hh_n^N(x,t) + \frac{x^2}{2t}\right) \leq -\delta T^2\right\}.
\end{equation}
Set
\begin{align*}
    {A := \{x\in [-x_0-1,0] \cap \Z_N} : \lceil x\rceil -x < N^{-1/2} \}.
\end{align*}
Say $x\in A$ is \textit{bad} if $\hh_n^N(x,t) + x^2/2t < -R_n$ and \textit{good} otherwise. Let 
\begin{align*}
    X:=\sum_{x\in A\cap I_N} \ind_{x\;\mathrm{is}\;\mathrm{bad}}, \qquad \m{B}:=\{X \le ({T+2})\delta\}.
\end{align*}
By \eqref{ed/3}, the probability that  any given $x$ is bad is at most $\e\delta/3$. So, by Markov's inequality, noting that $\#(A\cap I_N) \leq {T+2}$, we have
\begin{equation}\label{ByT}
    \mathbb{P}(\mathsf{B}) =1 -\Pr(X > ({T+2})\delta) \geq 1-\frac{\Ex[X]}{({T+2})\delta} \ge 1- \frac{\e}{3}.
\end{equation}
In view of \eqref{CyT} and \eqref{ByT}, in order to prove \eqref{EyT} it suffices to show that
\begin{equation*}
\mathbb{P}(\mathsf{E} \cap \mathsf{B} \cap \mathsf{C}) \leq \frac{\e}{3}.
\end{equation*}
We condition to rewrite this as
\begin{align*}
    \mathbb{P}(\mathsf{E} \cap \mathsf{B}\cap \mathsf{C}) &= \mathbb{E} \left[\mathbf{1}_{\mathsf{C}} \mathbb{E}\left[\mathbf{1}_{\mathsf{E}\cap\mathsf{B}} \mid \mathcal{F}_{\mathrm{ext}}(\{n\}\times (a_N,b_N))\right] \right].
\end{align*}
By the  Gibbs property (Lemma \ref{lem:gibbs}), 
we have that
\begin{align}\label{q1}
\mathbf{1}_{\mathsf{C}}\mathbb{E}\left[\mathbf{1}_{\mathsf{E}\cap\mathsf{B}} \mid \mathcal{F}_{\mathrm{ext}}(\{n\}\times (a_N,b_N))\right] =\mathbf{1}_{\mathsf{C}} \cdot\mathbb{P}_{N;\hh_{n-1}^N(\cdot,t),\hh_{n+1}^N(\cdot,t)}^{n,n;a_N,b_N;\hh_n^N(a_N,t),\hh_n^N(b_N,t)}(\mathsf{E}\cap\mathsf{B}).
\end{align}
{On $\m C$, condition further on the exact set $B$ of good mesh points and on their heights, under the original Gibbs law in \eqref{q1}. Include $a_N,b_N$ as good points and remove duplicates. On $\m B$, adjacent good points $p_N<q_N$ have distance at most $(T+2)\delta+2+2N^{-1/2}\leq2\delta T$ for large $T,N$. There are at most $T+3$ intervening intervals.
For each such interval the event of dipping below $-\delta T^2-x^2/(2t)$ is decreasing. Conditional monotonicity allows us to lower the endpoints to
$r(x):=-3\delta T^2/4-x^2/(2t)$, the soft ceiling to $f(x):=-\delta T^2/2-x^2/(2t)$, and the floor to $-\infty$, retaining the hard upper constraints at the bad mesh points. Lower these hard constraints further to
$h(x):=-5\delta T^2/8-x^2/(2t)$ throughout the interval. Thus
\begin{equation}\label{ETBTGibbs}
\Pr(\m E\cap\m B\cap\m C)
\leq (T+3)\sup_{p_N,q_N} Q_{N;f,h}^{p_N,q_N;r(p_N),r(q_N)}(\m E_{p_N,q_N}),
\end{equation}
where $Q_{N;f,h}$ retains both the soft ceiling $f$ and the hard ceiling $h$, and $\m E_{p,q}$ is the dip event on $[p,q]$.

By the invariance principle, the limiting comparison law is a Brownian bridge from $r(p)$ to $r(q)$, conditioned below $h$, and weighted by $W_f=\exp(-\int_p^q e^{B(x)-f(x)}dx)$. On the hard-ceiling event, with $d=q-p$,
\begin{equation}\label{ETprod}
e^{-d e^{-\delta T^2/8}}\leq W_f\leq1.
\end{equation}
Hence we may remove the soft weight at a cost at most $\exp(2\delta T e^{-\delta T^2/8})$.
For the unconditioned bridge with mean $\ell$, concavity gives
\[
-\frac34\delta T^2-\frac{d^2}{8t}\leq\ell(x)+\frac{x^2}{2t}\leq-\frac34\delta T^2.
\]
Since $t\geq1$, $\delta\leq1/16$, and $d\leq2\delta T$, the distance from this mean to the dip level is at least $\delta T^2/8$, and its distance below $h$ is at least $\delta T^2/8$. The reflection principle therefore gives, for large $T$,
\begin{equation}\label{goodints}
Q_{f,h}^{p,q;r(p),r(q)}(\m E_{p,q})
\leq 2\exp\left(2\delta T e^{-\delta T^2/8}-\delta T^3/64\right)
<\frac{\e}{6(T+3)}.
\end{equation}
The bound is uniform for $p,q$ in the fixed compact interval under consideration. The invariance principle transfers it to the discrete bridges for all large $N$ (if $q-p\to0$, the same bridge tail estimate gives zero). Combining \eqref{ETBTGibbs}--\eqref{goodints} proves the required bound $\e/3$ and completes the induction.

}
\subsection{Proof of Lemma \ref{uppertight} : Induction step} Fix $n\ge 2$. Assume that Lemma \ref{lowertight1}$(k)$ holds for $k=n-2,n-1,n$ and Lemma \ref{uppertight}$(n-1)$ holds. Let $K_n, M_n, \hat{R}_n>0$ be constants depending on $n$ and $\e$ that will be chosen appropriately later. Fix any $x_0>0$ and take $\bar{x}\in [-x_0,0]$.
We assume that $\bar{x}, \bar{x}-\frac12, \bar{x}-2 \in \Z_{N}$ (otherwise we can tweak these numbers slightly so that they lie in $\Z_N$).
    Define $\chi=\chi_N := \sup\{x\in [\bar{x}-\frac12,\bar{x}]\cap \Z_N : \hh_n^N(x,t) + x^2/2t \ge \hat{R}_n\}$, or $\chi = \overline{x}-\frac12$ if this set is empty. Define the events
    \begin{align*}
        \m{E} & : = \left\{\sup_{x\in [\bar{x}-\frac12,\bar{x}]} \left(\hh_n^N(x,t) + \frac{x^2}{2t}\right) \ge \hat{R}_n \right\}, \\
        \m{Q} & := \left\{\inf_{x\in [\bar{x}-2,\bar{x}]} \left(\hh_{n-2}^N(x,t) + \frac{x^2}{2t}\right) > -K_{n} \right\}, \\
        \m{A} & := \left\{\hh_{n-1}^N(\chi,t)+\chi^2/2t \ge -M_{n}\right\} \cap \bigcap_{j=n-1}^n  \left\{\hh_{j}^N(\bar{x}-2,t)+(\bar{x}-2)^2/2t \ge -M_n\right\}, \\
        \m{B} & := \left\{\sup_{x\in [\bar{x}-2,\chi]} \left(\hh_{n-1}^N(x,t) + \frac{x^2}{2t}\right) \ge \hat{R}_{n-1} \right\}.
    \end{align*}
Here $\hat{R}_{n-1}$ comes from Lemma \ref{uppertight}$(n-1)$; $K_n, M_n,$ and $\hat{R}_n$ are yet to be specified. Note that $[\bar{x}-2,\chi]$ forms a stopping domain. By the tower property of conditional expectation we have
    \begin{align}\label{r3}
        & \Pr\big(\m{E}\cap \m{Q} \cap \m{A} \cap \m{B} \big) = \Ex\left[\ind_{\m{E}\cap \m{Q} \cap \m{A}}\Ex\left[\ind_{\m{B}} \mid \mathcal{F}_{\mathrm{ext}}(\{n-1,n\}\times(\bar{x}-2,\chi))\right]\right].
    \end{align}
    By the strong Gibbs property (Lemma \ref{lem:sg}),
    \begin{align*} 
\Ex\left[\ind_{\m{B}} \mid \mathcal{F}_{\mathrm{ext}}(\{n-1,n\}\times(\bar{x}-2,\chi))\right] = \Pr_{N;\hh_{n-2}^N(\cdot,t),\hh_{n+1}^N(\cdot,t)}^{n-1,n;\bar{x}-2,\chi;\vec{a},\vec{b}}(\m{B}),
    \end{align*}
    where $\vec{a}=(\hh_{n-1}^{N}(\bar{x}-2),\hh_{n}^{N}(\bar{x}-2))$ and $\vec{b}=(\hh_{n-1}^{N}(\chi),\hh_{n}^{N}(\chi))$.
    On the event $\m{E}\cap \m{Q} \cap \m{A}$, we have
    \begin{align*}
       & \hh_{n-1}^N(\bar{x}-2) \ge -M_n -(\bar{x}-2)^2/2t, \quad \hh_{n}^N(\bar{x}-2) \ge -M_n -(\bar{x}-2)^2/2t, \\
      &  \hh_{n-1}^N(\chi) \ge -M_n -\chi^2/2t, \quad \hh_{n}^N(\chi) \ge \hat{R}_n -\chi^2/2t, \\
      &  \hh_{n-2}^N(x) \ge -K_n-x^2/2t, \mbox{ for } x\in [\bar{x}-2,\chi].
    \end{align*}
 Thus by stochastic monotonicity (Lemma \ref{lem:sm}) we have
    \begin{align}\label{r1}
        \ind_{\m{E}\cap \m{Q} \cap \m{A}} \cdot \Pr_{N;\hh_{n-2}^N(\cdot,t),\hh_{n+1}^N(\cdot,t)}^{n-1,n;\bar{x}-2,\chi;\vec{a},\vec{b}}(\m{B}) \ge \ind_{\m{E}\cap \m{Q} \cap \m{A}} \cdot \Pr_{N;-K_n-x^2/2t,-\infty}^{n-1,n;\bar{x}-2,\chi;\vec{u},\vec{v}}(\m{B}),
    \end{align}
    where $\vec{u}:=(-M_n-(\bar{x}-2)^2/2t, \,-M_n-(\bar{x}-2)^2/2t)$ and $\vec{v}:=(-M_n-\chi^2/2t, \, \hat{R}_n -\chi^2/2t)$. We may lower bound the above probability using the following lemma.
    
    \begin{lemma}\label{pchi}
       There exist $\delta, M^0>0$ and functions $K^0(M)>0$, $\hat{R}^0(M,K)>0$ such that for all $M\ge M^0$, $K\ge K^0(M)$, $\hat{R} \ge \hat{R}^0(M,K)$, and all $\bar{x}\in [-x_0,0]\cap \Z_N$, $\chi\in [\bar{x}-\frac12,\bar{x}]\cap \Z_N$, and large enough $N$ we have $p_N(\bar{x},\chi) \ge \frac14$, where $p_N(\bar{x},\chi)$ is defined to be
       \begin{align*}
           \Pr_{N;-K-x^2/2t,-\infty}^{n-1,n;\bar{x}-2,\chi;-M-(\bar{x}-2)^2/2t,-M-(\bar{x}-2)^2/2t;-M-\chi^2/2t,\hat{R}-\chi^2/2t}\bigg(\sup_{x\in [\bar{x}-2,\chi]} \left(S_{n-1}^N(x) + \frac{x^2}{2t}\right) \ge \frac12\delta \hat{R} \bigg).
       \end{align*}
    \end{lemma}

 The Brownian analogue of the above lemma was established in \cite[Proposition 7.6]{kpzle}. Lemma \ref{pchi} follows easily from their result by applying the invariance principle (Theorem \ref{lem:invprin}). We shall provide a short proof of it in a moment. Let us assume Lemma \ref{pchi} for now and use the parameters therein to specify our constants. 

 \medskip

  We assume $K_n \ge R_{n-2}(\e)$ and $M_n \ge R_{n-1}(\e)\vee R_n(\e)$, where $R_{n-2}, R_{n-1},R_n$ come from Lemma \ref{lowertight1}. Then by the induction hypothesis we have 
  \begin{equation}\label{QAB}
  \Pr(\neg(\m{Q} \cap \m{A})) \le 2\e, \quad \Pr(\m{B}) \le \e
  \end{equation}for all large enough $N$. We further assume our constants satisfy $M_n \ge M^0$, $K_n \ge K^0(M_n)$, $\hat{R}_n \ge \hat{R}^0(M_n,K_n)$, and $\hat{R}_n \ge 2\hat{R}_{n-1}/\delta$ where $\delta,M^0,K^0, \hat{R}^0$ are in Lemma \ref{pchi}. Then r.h.s.~of \eqref{r1} $\ge \frac14 \ind_{\m{E}\cap \m{Q} \cap \m{A}}$. Plugging this bound back in \eqref{r3}, we get 
 \begin{align*}
     \Pr(\m{B}) \ge \Pr\big(\m{E}\cap \m{Q} \cap \m{A} \cap \m{B} \big) \ge \frac14\Pr\big(\m{E}\cap \m{Q} \cap \m{A} \big) \ge \frac14\Pr(\m{E})- \frac14\Pr(\neg(\m{Q} \cap \m{A})).
 \end{align*}
 In view of \eqref{QAB} we thus have $\Pr(\m{E}) \le 6\e$ for large $N$. Adjusting $\e$, we get the desired result.

\begin{proof}[Proof of Lemma \ref{pchi}]  Let us take $\bar{x}_N^* \in [-x_0,0]\cap\Z_N$ and $\chi_N^* \in [\bar{x}-1/2,\bar{x}]\cap\Z_N$ that minimize $p_N(\cdot,\cdot)$ defined in Lemma \ref{pchi}. 
We claim that $\liminf_{N\to \infty} p_N(\bar{x}_N^*,\chi_N^*)$ has a lower bound. Consider any subsequence $\{N_k\}_{k\geq 1}$ along which we have
\begin{align*}
    \lim_{k\to \infty} p_{N_k}(\bar{x}_N^*,\chi_{N_k}^*) = \liminf_{N\to\infty} p_N(\bar{x}_N^*,\chi_N^*).
\end{align*}
    We can focus on a further subsequence of $N_k$ (still denoted by $N_k$) along which $\bar{x}_{N_k}^*,\chi_{N_k}^*$ converge to $\bar{x}_0,\chi_0$. Then by Theorem \ref{lem:invprin}, we know that
    \begin{align*}
         \lim_{k\to \infty} p_{N_k}(\bar{x}_{N_k}^*,\chi_{N_k}^*) = p_{}(\bar{x}_0,\chi_{0}),
    \end{align*}
    where $p_{}(\bar{x}_0,\chi_{0})$ equals
\begin{align*}
    \mathbf{P}_{1;-K-x^2/2t,-\infty}^{n-1,n;\bar{x}_0-2,\chi_0;-M-(\bar{x}_0-2)^2/2t,-M-(\bar{x}_0-2)^2/2t;-M-\chi_0^2/2t,\hat{R}-\chi_0^2/2t}\bigg(\sup_{x\in [\bar{x}_0-2,\chi_0]} \left(B_{n-1}(x) + \frac{x^2}{2t}\right) \ge \frac12\delta \hat{R} \bigg).
\end{align*}
    Proposition 7.6 in \cite{kpzle}, after decreasing $\delta$ and increasing the thresholds if necessary, gives $p_{}(\bar x_0,\chi_0)\geq1/2$. Thus for large enough $N$, we can conclude Lemma \ref{pchi} holds.
\end{proof}

\section{Drifted Brownian motion with soft barrier} \label{sec:dbmsft} In this and next sections, we study the diffusive limits of one-sided HSKPZ Gibbs measures, as defined in Definition \ref{kpzgibbs}. These results form a crucial piece in proving our main results, Theorems \ref{kpz123} and \ref{subseq}.  In this section, we analyze a particular class of the HSKPZ Gibbs measures in Definition \ref{kpzgibbs} under diffusive scaling.

Throughout this subsection, we fix $\alpha>0$ (supercritical regime), $A<0$, and suppose $(a_L)_{L\ge 1}$ is a sequence of numbers such that $a_L\to a \ge 0$ as $L\to\infty$. 
Let $B_L:[A,0]\to \R$ be a process distributed as $\mathbf{P}_{L;0}^{1;A;a_L,\star}$ (recall Definition \ref{kpzgibbs}). Unpacking the definition we see that $B_L$ is absolutely continuous w.r.t.~Brownian motion with drift $-\alpha\sqrt{L}$ with RN derivative proportional to
\begin{align*}
    \mathcal{W}_{L;+\infty,0}^{1,1;A,0}(B_L)=\exp\left(-L\int_{A}^0 e^{-\sqrt{L}B_L(x)}dx\right).
\end{align*}
Thus, $B_L$ can be though of as drifted Brownian motion conditioned softly to stay above a barrier at zero (recall the interpretation of Gibbs measures discussed in Remark \ref{rem:inter}). The main result in this subsection is the weak convergence of $B_L$.
\begin{theorem} \label{wconvsf} The process $B_L$ defined above converges weakly to $\Lambda^{+}$, a Brownian bridge on $[A,0]$ from $a$ to $0$ conditioned to stay positive on $(A,0)$ (see Remark \ref{rem:tran} below for a precise definition).
\end{theorem}

{We will only need the above theorem for $a>0$, but we include the $a=0$ case for completeness.}

\begin{remark}[The limiting process] \label{rem:tran} When $a=0$, $\Lambda^{+}$ is defined to be a Brownian excursion on $[A,0]$, see e.g. \cite{durrigle}. For $a>0$, we define $\Lambda^{+}$ as a time-reversal of a Brownian meander conditioned to end at $a$, as follows. First let $\Br^+$ be a standard Brownian meander on $[0,1]$, i.e., $\Br^+(x) := (1-\theta)^{-1/2}|\Br(\theta+(1-\theta)x)|$, where $\Br$ is a standard Brownian motion and $\theta :=\sup\{x\in[0,1]: \Br(x)=0\}$. The transition densities of $\Br^+$ are known from \cite[p. 613]{iglehart}: they are given for $0<x_1<x_2\leq 1$ and $y_1,y_2>0$ by
\begin{equation}\label{BMtr}
\begin{split}
p(0,0;x_1,y_1) &= \sqrt{2\pi}\,\frac{y_1}{x_1} \mathfrak{p}_{x_1}(y_1) \Psi(y_1/\sqrt{1-x_1}),\\
p(x_1,y_1;x_2,y_2) &= \left[\mathfrak{p}_{x_2-x_1}(y_1-y_2)-\mathfrak{p}_{x_2-x_1}(y_1+y_2)\right] \frac{\Psi(y_2/\sqrt{1-x_2})}{\Psi(y_1/\sqrt{1-x_1})},
\end{split}
\end{equation}
where $\p_t(y)$ is the standard heat kernel \eqref{heat} and $\Psi(x) := \sqrt{2/\pi}\int_0^x e^{-y^2/2}\,dy$. For $a>0$, we may define a bridge process $B^+$ which is $\Br^+$ conditioned on the event $\{\Br^+(1)=b\}$ with $b:=a/\sqrt{|A|}$; the conditional density of $B^+(x)$ for $x\in(0,1)$ is simply $f(y) := p(0,0;x,y)p(x,y;1,b)/p(0,0;1,b)$. We then set $\Lambda^+(x) := \sqrt{|A|}\,B^+(x/A)$ for $x\in[A,0]$. Using the transition densities of $W^+$ in \eqref{BMtr} (see also Lemma 4.2 in \cite{dz22}) it is straightforward to compute the joint density of $\Lambda^{+}$ at $k$ times. For $A=x_0<x_1<\cdots<x_k <0$, $y_0=a$, and $y_1,\dots,y_k>0$ we have
\begin{equation}\label{density}
    \P\left(\bigcap_{i=1}^k \{\Lambda^+(x_i) \in dy_i\}\right) = \frac{A}{x_k}\frac{y_k}{a}\frac{\mathfrak{p}_{-x_k}(y_k)}{\mathfrak{p}_{-A}(a)} \prod_{i=1}^{k}\left[\mathfrak{p}_{x_i-x_{i-1}}(y_i-y_{i-1}) - \mathfrak{p}_{x_i-x_{i-1}}(y_i+y_{i-1})\right]dy_i. 
\end{equation}
\end{remark}

\medskip

To prove Theorem \ref{wconvsf}, we begin by defining a generalization of the process $\Lambda_L$ considered therein. 

\begin{definition}
Given $\beta,\kappa\geq 0$, and $L>0$, let $B$ be a Brownian motion on $[A,0]$ started from $a_L$. Define 
\begin{align*}
     {\mathcal{W}}_{\beta,L}^{\e,\kappa}(B) &  :=\exp\left(-\beta (B(0)\vee (-\e))-L \int_A^0 e^{-\sqrt{L}(B(x)+\kappa)}dx\right) \\
     \hat{\mathcal{W}}_{\beta,L}^{\e,\kappa}(B) & :=\exp\left(-\beta ((B(0)+\e)\vee 0)-L \int_A^0 e^{-\sqrt{L}(B(x)+\kappa)}dx\right)=e^{-\beta \e}{\mathcal{W}}_{\beta,L}^{\e,\kappa}(B).
\end{align*}
Define the corresponding law
\begin{align}\label{ekbl}
\P_{\beta,L}^{\e,\kappa}(\m{A})=\P_{\beta,L}^{a_L;\e,\kappa}(\m{A}):=\frac{\E\left[\mathcal{W}_{\beta,L}^{\e,\kappa}(B)\ind_{\m{A}}\right]}{\E\left[\mathcal{W}_{\beta,L}^{\e,\kappa}(B)\right]}=\frac{\E\left[\hat{\mathcal{W}}_{\beta,L}^{\e,\kappa}(B)\ind_{\m{A}}\right]}{\E\left[\hat{\mathcal{W}}_{\beta,L}^{\e,\kappa}(B)\right]}.
\end{align}
(The last equality holds since $\mathcal{W}$ and $\hat{\mathcal{W}}$ are constant multiples.) We write $$\mathcal{W}_{\beta,L} := \lim_{\e\uparrow \infty} \mathcal{W}_{\beta,L}^{\e,0} = \exp\left(-\beta B(0) - L\int_A^0 e^{-\sqrt{L}B(x)}\,dx\right)$$ and $\P_{\beta,L}^{a_L}=\P_{\beta,L}$ for the law of a Brownian motion on $[A,0]$ started from $a_L$ with RN derivative proportional to $\mathcal{W}_{\beta,L}$. When $a>0$, we define the law $\P_{\beta,\infty}^{a}=\P_{\beta,\infty}$ as a Brownian motion  with drift $-\beta$ started from $a$ conditioned to stay positive on $(A,0]$. Finally, the law $\P_{\infty,\infty}^{a}=\P_{\infty,\infty}$ is defined to be  Brownian bridge from $a$ to $0$  conditioned to stay positive on $(A,0)$. 
\end{definition}

\medskip

Roughly, as $L\to\infty$ the weight ${\mathcal{W}}_{\beta,L}^{\e,\kappa}$ enforces $B(0) \leq -\e$ and $B(x) > -\kappa$ for all $x\in(A,0)$, which is a positive probability event if $\kappa>\e>0$. The purpose of introducing $\hat{\mathcal{W}}_{\beta,L}^{\e,\kappa}$ is that it is at most 1, which will be convenient for various upper bounds in the arguments below. Note that by Cameron--Martin theorem, $B_L$ in Theorem \ref{wconvsf} is distributed as $\P_{\alpha \sqrt{L},L}$. Thus Theorem \ref{wconvsf} asserts
\begin{align*}
    \mathbf{P}_{\alpha \sqrt{L},L} \to \mathbf{P}_{\infty,\infty}.
\end{align*}
This is intuitively clear from the form of the Radon–Nikodym derivatives. However, since the conditioning event has zero probability in the limit, care is required to rigorously justify the statement. To circumvent this issue, we consider the law $\P_{\beta,L}$ and take the iterated limit: we first take $L \to \infty$ to obtain $\P_{\beta,\infty}$, then take $\beta \to \infty$ to obtain $\P_{\infty,\infty}$. To justify that this iterated limit indeed coincides with the limit of $\P_{\alpha \sqrt{L},L}$ as $L \to \infty$, we rely on various forms of stochastic monotonicity of the laws $\P_{\beta,L}^{a;\e,\kappa}$ with respect to the initial data $a$ and the parameters $\beta$ and $\kappa$.

\begin{lemma}\label{stm} Fix $L>0$.
\begin{enumerate}[label=(\alph*),leftmargin=20pt]
    \item \label{paro} Fix $a_1\geq a_2$ and $\beta\geq 0$. There exists a probability space that supports two processes $B_1, B_2$ on $[A,0]$ such that the marginal law of $B_i$ is given by $\P_{\beta,L}^{a_i}$ and with probability $1$, for all $x\in [A,0]$ we have $B_1(x) \ge B_2(x) \ge B_1(x)-(a_1-a_2)$. 
    \item \label{para} Fix $\beta_1\leq \beta_2$ and $a\geq 0$. There exists a probability space that supports two processes $B_1, B_2$ on $[A,0]$ such that the marginal law of $B_i$ is given by $\P_{\beta_i,L}^a = \P_{\beta_i,L}$ and with probability $1$, for all $x\in [A,0]$ we have $B_1(x) \ge B_2(x)$. 
    \item \label{parb} Fix $\kappa_1\leq\kappa_2$, $a\geq 0$, $\beta\geq 0$, and $\e>0$. There exists a probability space that supports two processes $B_1, B_2$ on $[A,0]$ such that the marginal law of $B_i$ is given by $\P_{\beta,L}^{a;\e,\kappa_i} = \P_{\beta,L}^{\e,\kappa_i}$ and with probability $1$, for all $x\in [A,0]$ we have $B_1(x) \ge B_2(x)$. 
\end{enumerate}
\end{lemma}

\begin{proof}
The proof follows a Glauber dynamics argument to construct a monotone coupling of the two laws in each case. The argument is quite standard in the literature and follows ideas similar to \cite{ble,kpzle,serio,halfairy}. We begin with part \ref{parb}. The argument requires a discretization of the Gibbs measure. For $N\geq 1$, write $A_N := \lceil AN\rceil$, and for $\beta,\kappa \geq 0$, let $S_{\beta,L}^{\e,\kappa}$ denote a simple random walk $S$ on $\ll A_N, 0\rr$ with $S(A_N) = a_N := \lfloor a\sqrt{N}\rfloor$, reweighed by a RN derivative proportional to
\begin{equation}\label{Wbldisc}
W_{\beta,L}^{\e,\kappa}(S) := \exp\bigg(-\beta\left(\left(\frac{S(0)}{\sqrt{N}} + \e\right) \vee 0\right) - \frac{L}{N}\sum_{k=A_N}^{-1} e^{-\sqrt{L}(N^{-1/2}S(k)+\kappa)}\bigg).
\end{equation}
It is easy to see by an argument very similar to the proof of Theorem \ref{lem:invprin} that $(N^{-1/2}S_{\beta,L}^{\e,\kappa}(Nx))_{x\in[A,0]}$ (defined by linear interpolation) converges in law as $N\to\infty$ to $\P_{\beta,L}^{\e,\kappa}$. It therefore suffices to construct a coupling under which $S_{\beta,L}^{\e,\kappa_1} \geq S_{\beta,L}^{\e,\kappa_2}$ a.s.

 We construct two Markov chains $(S_i^n(\cdot))_{n\geq 0}$, on the state space $\Omega^{a_N}$ of simple random walk paths $\ell : \ll A_N,0\rr \to \Z$ such that $\ell(A_N) = a_{N} $ and $|\ell(i+1)-\ell(i)|=1 \mbox{ for } i\in\ll A_N,-1\rr$, as follows. For $n=0$, simply set $S_1^0$ and $S_2^0$ equal to the same arbitrary path. The dynamics proceed as follows. At time $n$, we uniformly sample a pair $(k,\sigma)\in\ll A_N+1,0\rr \times \{\pm 1\}$, and we let $U^{(k,\sigma)} \sim \mathrm{Unif}(0,1)$. We define a candidate $\til S_i^n$ for $S_i^{n+1}$ by setting $\til S_i^n (k) = S_i^n(k) + 2\sigma$ and $\til S_i^n(j) = S_i^n(j)$ for $j\neq k$. If $\til S_i^n \in \Omega^{a_N}$ and if we have
\begin{equation}\label{Rdef}
R_{n,i}^{(k,\sigma)} := \frac{W_{\beta,L}^{\e,\kappa_i}(\til S_i^n)}{W_{\beta,L}^{\e,\kappa_i}(S_i^n)} \geq U^{(k,\sigma)},
\end{equation}
then we update $S_i^{n+1}$ to $\til S_i^n$. Otherwise we set $S_i^{n+1} = S_i^n$.

It is straightforward to check that for $i=1,2$ the above dynamics are stationary (in fact reversible) for the laws of $S_{\beta,L}^{\e,\kappa_i}$. Moreover, the two Markov chains $(S_i^n(\cdot))_{n\geq 0}$ are irreducible and aperiodic. Therefore $S_i^n \to S_{\beta,L}^{\e,\kappa_i}$ in law as $n\to\infty$. We will argue that for all $n$, $S_1^n(k) \geq S_2^n(k)$ for all $k\in\ll A_N,0\rr$ a.s., which upon taking $n\to\infty$ provides the desired monotone coupling $S_{\beta,L}^{\e,\kappa_1} \geq S_{\beta,L}^{\e,\kappa_2}$.

At time 0 we have $S_1^0 \geq S_2^0$ by assumption. Assume $S_1^n \geq S_2^n$ for some $n\geq 0$. Note because $S_1^0 = S_2^0$, for all $k$ and $n$ the difference $S_1^n(k) - S_2^n(k)$ must be even. Thus the only way the ordering can be violated at time $n+1$ is if for some $k$, $S_1^n(k) = S_2^n(k)$, and either (i) $S_2^n(k)$ is flipped upwards but $S_1^n(k)$ is not, or (ii) $S_1^n(k)$ is flipped downwards but $S_2^n(k)$ is not. One way case (i) could occur is if the candidate $\til S_1^{n}$ obtained by flipping $S_1^n(k)$ upwards fails to lie in $\Omega^{a_N}$, i.e., $\min(S_1^n(k-1),S_1^n(k+1)) = S_1^n(k)-1$. But since $S_1^n \geq S_2^n$, this would force $\min(S_2^n(k-1),S_2^n(k+1)) = S_2^n(k)-1$ as well, preventing $S_2^n(k)$ from flipping upwards. So in fact case (i) will only occur if $R_{n,2}^{(k,+1)} \geq U^{(k,+1)} > R_{n,1}^{(k,+1)}$. Likewise, case (ii) requires $R_{n,1}^{(k,-1)} \geq U^{(k,-1)} > R_{n,2}^{(k,-1)}$. We now rule out these two scenarios.

First observe from the definition \eqref{Wbldisc} that for $k=0$, in fact have $R_{n,1}^{(0,\sigma)} = R_{n,2}^{(0,\sigma)}$ for $\sigma=\pm 1$, so it suffices to consider $k<0$. In this case, we see that
\begin{equation}\label{Req}
R_{n,i}^{(k,\sigma)} = \exp\left(\frac{L}{N}e^{-\sqrt{L}(N^{-1/2}S_i^n(k) + \kappa_i)} \big(1-e^{-2\sigma\sqrt{L/N}}\big)\right)
\end{equation}
Since $S_1^n(k) = S_2^n(k)$ and $\kappa_1\leq\kappa_2$, it is easy to see from the above that $R_{n,1}^{(k,+1)} \geq R_{n,2}^{(k,+1)}$ and $R_{n,1}^{(k,-1)} \leq R_{n,2}^{(k,-1)}$. This implies that neither of the above two scenarios can occur, so indeed the ordering is preserved for all $n$ by induction. This completes the proof of case \ref{parb}.

For parts \ref{paro} and \ref{para}, we instead consider the weights
\[
W_{\beta_i,L}(S) := \exp\bigg(-\beta_i\frac{S(0)}{\sqrt{N}} - \frac{L}{N}\sum_{k=A_N}^{-1} e^{-\sqrt{L}N^{-1/2}S(k)}\bigg).
\]
In part \ref{paro} we take $\beta_1=\beta_2=\beta$. We let $a_{i,N}$ be two integers within distance 1 of $a_1\sqrt{N}$ and $a_2\sqrt{N}$ respectively, such that $a_{1,N}-a_{2,N}$ is even. We let $S^{a_i}$ denote simple random walks on $\ll A_N, 0\rr$ started from $a_{i,N}$ reweighed by $W_{\beta,L}$. Then it is clear that $N^{-1/2}S^{a_i}(Nx)$ converges in law to $\P_{\beta,L}^{a_i}$. As above, it suffices to construct a Markov chain $(S_1^n, S_2^n)$ on $\Omega^{a_{1,N}}\times\Omega^{a_{2,N}}$ for which $(S^{a_1},S^{a_2})$ is stationary, such that $S_1^n(k) \geq S_2^n(k) \geq S_1^n(k) - (a_{1,N}-a_{2,N})$ for all $k$. We take $S_1^0$ to be an arbitrary path started at $a_{1,N}$, and $S_2^0$ to be the same path shifted vertically by $a_{2,N}-a_{1,N}$. The dynamics are exactly analogous to part \ref{parb}, and we have the same formula \eqref{Req} with $\kappa_i=0$. 

Assume the desired inequality holds at time $n$. The same argument as above shows that $S_1^{n+1} \geq S_2^{n+1}$. To check the other inequality, note that for a violation to occur we must already have $S_1^n(k) - S_2^n(k) = a_{1,N}-a_{2,N}$ for some $k$, and either (i) $S_1^n(k)$ must flip up but not $S_2^n(k)$, or (ii) $S_2^n(k)$ must flip down but not $S_1^n(k)$. Since $S_1^n(k) \geq S_2^n(k)$, we see from \eqref{Req} that $R_{n,1}^{(k,+1)} \leq R_{n,2}^{(k,+1)}$ and $R_{n,1}^{(k,-1)} \geq R_{n,2}^{(k,-1)}$, so these scenarios cannot occur according to the rule \eqref{Rdef}. For case (i), if the candidate $\til S_2^n$ obtained by flipping $S_2^n(k)$ up fails to lie in $\Omega^{a_{2,N}}$, then we must have $\min(S_2^n(k-1),S_2^n(k+1)) = S_2^n(k)-1$. Since by the inductive hypothesis $S_1^n \leq S_2^n + (a_{1,N}-a_{2,N})$, and equality holds at $k$, we must also have $\min(S_1^n(k-1),S_1^n(k+1)) = S_1^n(k) - 1$, so that $S_1^n(k)$ cannot flip upwards either. This shows that case (i) is impossible, and similar reasoning rules out case (ii). Thus $S_2^{n+1} \geq S_1^{n+1} - (a_{1,N}-a_{2,N})$, and we conclude part \ref{paro} of the lemma. Part \ref{para} is similar but easier, so we skip it for brevity.
\end{proof}		

The first step in proving Theorem \ref{wconvsf} is showing convergence at the boundary.

\begin{lemma}\label{convzero} Assume $a>0$ and suppose $B$ is distributed as $\P_{\alpha \sqrt{L},L}^{a_L} =\P_{\alpha \sqrt{L},L}$. Then $B(0) \to 0$ in probability. 
\end{lemma}

\begin{proof} Fix $\e,\beta>0$. We first claim that
\begin{align}\label{upcl}
    \lim_{L\to \infty} \P_{\alpha\sqrt{L},L}(B(0)\le \e)=1.
\end{align}
As the event $B(0)\le \e$ is decreasing, 
by Lemma \ref{stm}\ref{para}, $$\P_{\alpha \sqrt{L},L}(B(0) \le \e) \ge \P_{\beta,L}(B(0) \le \e)$$ once $L$ is large enough so that $\alpha\sqrt{L}\geq\beta$. 
Note that for each fixed $\beta$, by Cameron-Martin theorem, the law $\P_{\beta,L}$ can be viewed as Brownian motion $B$ with drift $-\beta$ and softly conditioned with RN derivative proportional to $\exp(-L\int_A^0 e^{-\sqrt{L}B(x)}dx)$. As $L\to \infty$, this RN derivative converges to $\ind_{B(x)\ge 0}$ almost surely. As the RN derivative is bounded by $1$ and $\P(B(x)-\beta (x-A) \ge 0 \mbox{ for all } x\in [A,0])>0$, we get that as $L\to \infty$, $\P_{\beta,L} \to \P_{\beta,\infty}$ where $\P_{\beta,\infty}$ is defined to be the law of a Brownian motion with drift $-\beta$ started from $a$ and conditioned to stay positive on $[A,0]$. Thus for each $\beta>0$
\begin{align*}
    \liminf_{L\to \infty} \P_{\alpha \sqrt{L},L}(B(0) \le \e) \ge \P_{\beta,\infty}(B(0)\le \e).
\end{align*}
We claim that the probability on the right tends to 1 as $\beta\to\infty$. To see this, we write
\begin{equation}\label{B(0)>e}
\begin{split}
    \P_{\beta,\infty}(B(0) > \e) 
    &= \frac{\mathbf{E}[e^{-\beta B(0) }\mathbf{1}_{B(x)\geq 0,\,x\in[A,0]} \mathbf{1}_{B(0)>\e}]}{\mathbf{E}[e^{-\beta B(0) }\mathbf{1}_{B(x)\geq 0,\,x\in[A,0]}]}\\ &\leq \frac{e^{-\beta\e}}{\mathbf{E}[e^{-\beta B(0) }\mathbf{1}_{B(x)\geq 0,\,x\in[A,0]}\mathbf{1}_{B(0)\in[\e/4,\e/2]} ]}\\ 
    &\leq \frac{e^{-\beta\e/2}}{\P(\inf_{x\in[A,0]}B(x)\geq 0, \, B(0)\in[\e/4,\e/2])}.
\end{split}
\end{equation}
To lower bound the denominator, we rewrite the probability as $\P(\inf_{x\in[A,0]} B(x)\geq 0 \mid B(0)\in [\e/4,\e/2]) \cdot \P(B(0)\in[\e/4,\e/2])$. By Gaussian tail estimates, the second factor is bounded below by a positive constant depending on $A,a,\e$. For the first factor, stochastic monotonicity allows us to replace the conditioning with $B(0)=\e/4$, i.e., consider a Brownian bridge on $[A,0]$ from $a$ to $\e/4$. The probability that this minimum is positive is bounded below by a positive constant depending on $A,a,\e$, for instance by \cite[Proposition 12.3.3]{dudley}. In particular, in \eqref{B(0)>e} we obtain an upper bound of $Ce^{-\beta\e/2}$ for some constant $C$ independent of $\beta$, and taking $\beta\to\infty$ proves \eqref{upcl}.

We now aim to show that $\P_{\alpha\sqrt{L},L}(B(0) \le -\e)\to 0$ as $L\to \infty$. Towards this end, we claim that there exists $C>0$ free of $L$ and $\e$ such that
\begin{align}\label{dncl1}
    & \E\left[\mathcal{W}_{\alpha \sqrt{L},L}\right] \ge e^{-CL}, \\ \label{dncl2} & \E\left[\mathcal{W}_{\alpha \sqrt{L},L}\ind\{B(0)\le -\e\}\right] \le e^{2\alpha^2 |A|L}\cdot \left(\exp(-L^{-1}e^{\e\sqrt{L}/2})+Ce^{-\e^2L^2/C}\right)^{1/2}.
\end{align}
Given the above estimates, taking the ratio we see that 
\begin{align*}
    \P_{\alpha\sqrt{L},L}(B(0) \le -\e) = \frac{\E\left[\mathcal{W}_{\alpha \sqrt{L},L}\ind\{B(0)\le -\e\}\right]}{\E\left[\mathcal{W}_{\alpha \sqrt{L},L}\right]} \longrightarrow 0
\end{align*}
as $L\to \infty$. We thus focus on proving \eqref{dncl1} and \eqref{dncl2}. Note that
\begin{align*}
    \E[\mathcal{W}_{\alpha \sqrt{L},L}] & \ge \E\left[\mathcal{W}_{\alpha\sqrt{L},L} \cdot \ind\left\{\inf_{x\in [A,0]} B(x)\ge 0, B(0)\in [1,2]\right\}\right] \\ & \ge \exp(-2\alpha \sqrt{L}+AL) \cdot \P\bigg(\inf_{x\in [A,0]} B(x)\ge 0, B(0)\in[1,2]\bigg).
\end{align*}
The last inequality follows by noting that  $\mathcal{W}_{\alpha\sqrt{L},L} \ge \exp(-2\alpha \sqrt{L}+AL)$ on the event $$\left\{\inf_{x\in [A,0]} B(x)\ge 0, B(0)\in[1,2]\right\}.$$ The above event has a positive probability which is uniformly bounded below as $L\to \infty$ (as $a_L\to a>0$). This verifies \eqref{dncl1}. For \eqref{dncl2}, by the Cauchy--Schwarz inequality we have
\begin{align}\nonumber
    \E[\mathcal{W}_{\alpha \sqrt{L},L}\ind\{B(0) \le -\e\}] & \le \sqrt{\E\left[e^{-2\alpha \sqrt{L} B(0)}\right]} \sqrt{\E\left[\exp\bigg(-2L\int_A^0 e^{-\sqrt{L}B(x)}dx\bigg)\ind\{B(0) \le -\e\}\right]} \\ & \le e^{2\alpha^2 |A|L}\cdot \sqrt{\E\left[\exp\bigg(-2L\int_A^0 e^{-\sqrt{L}B(x)}dx\bigg)\ind\{B(0) \le -\e\}\right]}. \label{dncl0}
\end{align}
 We focus on bounding the expectation inside the above square root. Let $\m{A}:=\{\sup_{x\in [-L^{-2},0]} |B(x)-B(0)| \le \e/2\}$. By a union bound we have
\begin{align}\nonumber
    & \E\left[\exp\bigg(-2L\int_A^0 e^{-\sqrt{L}B(x)}dx\bigg)\ind\{B(0) \le -\e\}\right] \\ & \le \E\left[\exp\bigg(-2L\int_A^0 e^{-\sqrt{L}B(x)}dx\bigg)\ind_{\m{A}\cap \{B(0) \le -\e\}}\right] + \P(\neg\m{A}). \label{dncl3} 
\end{align}
 Note that on $\m{A} \cap \{B(0)\le -\e\}$, $B(x) \le -\e/2$ for all $x\in [-L^{-2},0]$. Thus, on $\m{A} \cap \{B(0)\le -\e\}$,
\begin{align*}
    \int_A^0 e^{-\sqrt{L}B(x)}dx \ge \int_{-L^{-2}}^0 e^{-\sqrt{L}B(x)}dx \ge L^{-2} e^{\e\sqrt{L}/2}.
\end{align*}
Hence the first term in \eqref{dncl3} is at most $\exp\big(-2L^{-1}e^{\e\sqrt{L}/2}\big)$. For the second term in \eqref{dncl3}, by Brownian motion computations $\P(\neg\m{A}) \le Ce^{-\e^2L^2/C}$ for some absolute constant $C>0$.
Plugging the above estimate back in \eqref{dncl0} we arrive at \eqref{dncl2}.
\end{proof}

Having proved the above lemma, we shall now extend the convergence to finite-dimensional distributions away from the boundary.

\begin{proposition}\label{convfd} Assume $a>0$. Suppose $B$ is distributed as $\P^{a_L}_{\alpha \sqrt{L},L} = \P_{\alpha \sqrt{L},L}$. As $L\to \infty$, the finite-dimensional distributions of $B$ converge to those of a Brownian bridge from $a$ to $0$ conditioned to stay positive on $(A,0)$. 
\end{proposition}
\begin{proof}

Let us fix $x_1<x_2<\cdots<x_k < 0$ and $y_1,y_2,\ldots,y_k \in \R$. Consider
\begin{align*}
    \m{A}_k = \bigcap_{i=1}^k \left\{B(x_i) \le y_i\right\}.
\end{align*}
 As $\m{A}_k$ is a decreasing event, by the argument at the beginning of the proof of Lemma \ref{convzero}, we obtain
\begin{align}\label{lbdc0}
    \liminf_{L\to \infty} \P_{\alpha \sqrt{L},L}(\m{A}_k) \ge \P_{\beta,\infty}(\m{A}_k).
\end{align}
The advantage of passing to $\P_{\beta,\infty}$ is that its finite-dimensional distribution are accessible. Indeed, using Brownian meander transition densities as in Remark \ref{rem:tran} we see that
\begin{align}\label{bmtran}
    \P_{\beta,\infty}\left(\bigcap_{i=1}^k B(x_i) \in dy_i\right)= \frac{f_{\vec{x}}(\vec{y};\beta)}{\int_{\R^k} f_{\vec{x}}(\vec{y};\beta) d\vec{y}},
\end{align}
where $f_{\vec{x}}(\vec{y};\beta)=0$ outside $(0,\infty)^k$, and on $(0,\infty)^k$ it is equal to
\begin{align*}
  & \int_{0}^\infty \tfrac12\beta^2 e^{-\beta u}\left[\p_{-x_{k}}(u-y_{k})-\p_{-x_{k}}(u+y_{k})\right] \prod_{i=1}^{k} \left[\p_{x_i-x_{i-1}}(y_i-y_{i-1})-\p_{x_i-x_{i-1}}(y_i+y_{i-1})\right] du
\\ & = \int_{0}^\infty \tfrac12\beta e^{-v}\left[\p_{-x_{k}}(v/\beta-y_{k})-\p_{-x_{k}}(v/\beta+y_{k})\right] \prod_{i=1}^{k} \left[\p_{x_i-x_{i-1}}(y_i-y_{i-1})-\p_{x_i-x_{i-1}}(y_i+y_{i-1})\right] dv,
\end{align*}
with $y_0=a$ and $x_0=A$. (The factor of $\frac{1}{2}\beta^2$ does not arise from the law $\P_{\beta,\infty}$ but is included to obtain a limit as $\beta\to\infty$. We are free to include it since it cancels in the normalization.) It is plain to check that as $\beta\to \infty$, $$\tfrac12\beta[\p_{-x_{k}}(v/\beta-y_{k})-\p_{-x_{k}}(v/\beta+y_{k})] \longrightarrow \frac{vy_k}{-x_k} \p_{-x_k}(y_k).$$ From here, a simple dominated convergence argument (which we skip for brevity) shows that as $\beta\to \infty$ in \eqref{bmtran} we have $f_{\vec{x}}(\vec{y};\beta) \to f_{\vec{x}}(\vec{y};\infty)$ and $\int_{\R^k} f_{\vec{x}}(\vec{y};\beta)d\vec{y}\to \int_{\R^k} f_{\vec{x}}(\vec{y};\infty)d\vec{y}$, where the following formula is for $\vec y\in(0,\infty)^k$, and the density is zero elsewhere:
\begin{align*}
f_{\vec{x}}(\vec{y};\infty) & :=  \frac{y_k}{-x_k} p_{-x_{k}}(y_{k}) \prod_{i=1}^{k} \left[\p_{x_i-x_{i-1}}(y_i-y_{i-1})-\p_{x_i-x_{i-1}}(y_i+y_{i-1})\right]\int_0^\infty ve^{-v}dv \\ & =  \frac{y_k}{-x_k}\ \p_{-x_{k}}(y_{k}) \prod_{i=1}^{k} \left[\p_{x_i-x_{i-1}}(y_i-y_{i-1})-\p_{x_i-x_{i-1}}(y_i+y_{i-1})\right].
\end{align*}
This matches with the joint densities of $\Lambda^+ \sim \P_{\infty,\infty}$ in \eqref{density} (up to constants depending on $a$ and $A$, which cancel in the normalization). Thus by Scheff\'{e}'s lemma, as $\beta\to \infty$
\begin{align}\label{betinf}
    \P_{\beta,\infty}(\m{A}_k)\to \P_{\infty,\infty}(\m{A}_k),
\end{align}
and hence taking $\beta\to \infty$ in \eqref{lbdc0}, we see that 
\begin{align}\label{lbdc}
    \liminf_{L\to \infty} \P_{\alpha \sqrt{L},L}(\m{A}_k) \ge \P_{\infty,\infty}(\m{A}_k).
\end{align}

\medskip

We now focus on showing the opposite inequality. Fix $\e>0$. Consider the event $\m{B}_\e := \{B(0)\ge -\e\}$.
Note that  ${\mathcal{W}}_{\alpha\sqrt{L},L}^{\e,0} \le \mathcal{W}_{\alpha\sqrt{L},L}$ in general, and on $\m{B}_\e$, ${\mathcal{W}}_{\alpha\sqrt{L},L}^{\e,0} = \mathcal{W}_{\alpha\sqrt{L},L}$.
Thus by \eqref{ekbl},
\begin{align*}
    \P_{\alpha \sqrt{L},L}(\m{A}_k\cap \m{B}_\e)& = \frac{\E[\mathcal{W}_{\alpha\sqrt{L},L}\ind_{\m{A}_k\cap \m{B}_\e}]}{\E[\mathcal{W}_{\alpha\sqrt{L},L}]}   \le  \frac{\E[{\mathcal{W}}_{\alpha\sqrt{L},L}^{\e,0}\ind_{\m{A}_k\cap \m{B}_\e}]}{\E[{\mathcal{W}}_{\alpha\sqrt{L},L}^{\e,0}]} \\ & = \frac{\E[\hat{{\mathcal{W}}}_{\alpha\sqrt{L},L}^{\e,0}\ind_{\m{A}_k\cap \m{B}_\e}]}{\E[\hat{\mathcal{W}}_{\alpha\sqrt{L},L}^{\e,0}]}  \le  \frac{\E[\hat{\mathcal{W}}_{\alpha\sqrt{L},L}^{\e,0}\ind_{\m{A}_k}]}{\E[\hat{\mathcal{W}}_{\alpha\sqrt{L},L}^{\e,0}]} \le \frac{\E[\hat{\mathcal{W}}_{\alpha\sqrt{L},L}^{\e,2\e}\ind_{\m{A}_k}]}{\E[\hat{\mathcal{W}}_{\alpha\sqrt{L},L}^{\e,2\e}]}.
\end{align*}
The last inequality follows from Lemma \ref{stm}\ref{parb}. 
Now as $L\to \infty$, $$\hat{\mathcal{W}}_{\alpha\sqrt{L},L}^{\e,2\e} \longrightarrow \hat{\mathcal{W}}_{\infty,\infty}^{\e,2\e}:=\ind\{ B(0) \le -\e, B(x) > -2\e \mbox{ for all }x\in [A,0]\}$$
almost surely. As $\hat{\mathcal{W}}_{\alpha\sqrt{L},L}^{\e,2\e}$ are all bounded by $1$ and the probability of the event on the right hand side is positive, by dominated convergence we have
\begin{align*}
    \limsup_{L\to \infty} \P_{\alpha \sqrt{L},L}(\m{A}_k\cap \m{B}_\e) \le \frac{\E[\hat{\mathcal{W}}_{\infty,\infty}^{\e,2\e}\ind_{\m{A}_k}]}{\E[\hat{\mathcal{W}}_{\infty,\infty}^{\e,2\e}]}.
\end{align*} Since, by Lemma \ref{convzero}, $\lim_{L\to \infty} \P_{\alpha \sqrt{L},L}(\neg\m{B}_\e) =0$, we get that 
\begin{align*}
    \limsup_{L\to \infty} \P_{\alpha \sqrt{L},L}(\m{A}_k) \le \frac{\E[\hat{\mathcal{W}}_{\infty,\infty}^{\e,2\e}\ind_{\m{A}_k}]}{\E[\hat{\mathcal{W}}_{\infty,\infty}^{\e,2\e}]} .
\end{align*}
The finite-dimensional distributions on the right hand side of the above equation are again accessible using Brownian meander transition densities. An argument similar to the proof of \eqref{betinf} shows that as $\e\downarrow 0$, the right hand side goes to $\P_{\infty,\infty}(\m{A}_k)$. This combined with \eqref{lbdc} leads to finite dimensional convergence as desired.
\end{proof}

With finite-dimensional convergence established, it remains to control the modulus of continuity to obtain path tightness.

\begin{lemma}\label{tightz} Fix $R>1$ and $\e,\rho>0$. Let $B$ have law $\P^{a_L}_{\alpha \sqrt{L},L} = \P_{\alpha \sqrt{L},L}$, where we recall $a_L \to a$, and recall $\mc$ from \eqref{moc}. There exists $\delta_0$ depending only on $\e,\rho, A,R$ such that for all $\delta\le \delta_0$, if $a\in [R^{-1},R]$ we have
\begin{align}\label{moclimsup}
    \limsup_{L\to \infty}\P_{\alpha \sqrt{L},L}\big(\mc(B,\delta)\ge \rho\big) \le \e.
\end{align}
\end{lemma}

Note that Lemma \ref{tightz} provides a modulus of continuity estimate that holds uniformly over all limiting starting points $a \in [R^{-1}, R]$. This uniformity will be crucial in one of our later arguments.

\begin{proof}
We shall assume $A=-1$ for simplicity, and we can assume $L$ is large enough so that $a_L\leq 2R$. For $\gamma\in(0,1)$, we may bound
\begin{equation}\label{mocsplit}
\mc(B,\delta) \leq \mc(B|_{[-1,-\gamma]},\delta) + \mc(B|_{[-\gamma,0]},\delta) + \mc(B|_{[-\gamma-\delta,-\gamma+\delta]},\delta).
\end{equation}
We will control the three terms separately, beginning with the first. Fix $\e>0$. For each $\gamma>0$, we may find $C(\gamma,R,\e)>1$ such that $\P_{\alpha \sqrt{L},L}(1/C<B(-\gamma)<C)\ge 1-\e/2$ for all large $L$ by Proposition \ref{convfd}. Thus by the Gibbs property,
\begin{align}
  & \nonumber \P_{\alpha \sqrt{L},L}\big(\mc(B|_{[-1,-\gamma]},\delta)\ge \rho\big)  \\ &\nonumber \le \e/2+\E_{\alpha \sqrt{L},L}\left[\ind_{1/C<B(-\gamma) <C}\cdot\E_{\alpha \sqrt{L},L}\left[\ind\big\{\mc(B|_{[-1,-\gamma]},\delta)\ge \rho\big\}\mid \mathcal{F}_{\m{ext}}([-1,-\gamma))\right]\right] \\ & \le \e/2+ \E_{\alpha \sqrt{L},L}\left[\ind_{1/C<B(-\gamma) <C}\cdot\frac{\E\left[\ind\big\{\mc(\Br,\delta)\ge \rho\big\}\right]}{\E\left[\exp\big(-L\int_{-1}^{-\gamma} e^{-\sqrt{L}\Br(x)}dx\big)\right]}\right], \label{mocexp}
\end{align}
where in the last line $\Br$ denotes a Brownian bridge on $[-1,-\gamma]$ from $a_L$ to $B(-\gamma)$. Here we used that the RN derivative $\exp(-L\int_{-1}^{-\gamma}e^{-\sqrt{L}\Br(x)}dx) \leq 1$. The expectation in the denominator in \eqref{mocexp} is bounded below by $\frac{1}{2}\P(\inf_{x\in[-1,-\gamma]} \Br(x) > \frac{1}{2(C+R)})$ for large $L$, and on the event $\{B(-\gamma) > 1/C\}$, this in turn is bounded below by a positive constant $c = c(\gamma,R,\e)$ depending on $C$. For the numerator, note that on the event $\{ B(-\gamma)<C\}$, $\mc(\Br,\delta)$ is bounded above (stochastically) by $\mc(\Br^0,\delta)+(2R+C)\delta/(1-\gamma)$, where $\Br^0$ is a Brownian bridge from 0 to 0. Since $\Br^0$ is uniformly continuous a.s., $\mc(\Br^0,\delta)\to 0$ in probability as $\delta\downarrow 0$. This allows us to choose $\delta = \delta(\gamma,\e,\rho,R)$ small enough so that the expectation in the numerator in \eqref{mocexp} is at most $c\e/2$. Altogether, for this choice of $\delta$ we obtain
\begin{equation}\label{moc1g}
\limsup_{L\to\infty}\P_{\alpha \sqrt{L},L}\big(\mc(B|_{[-1,-\gamma]},\delta)\ge \rho/3\big) \leq \e.
\end{equation}

We now control the second term in \eqref{mocsplit}. By Remark \ref{rem:tran}, under the law $\P_{\infty,\infty}$ the density of $B(-\gamma)/\sqrt{\gamma}$ at $y\geq 0$ is
$$ \frac{1}{{\sqrt{\gamma}}}
\frac{y}{a}\frac{\p_1(y)}{\p_1(a)}[\p_{1-\gamma}(a-y\sqrt{\gamma})-\p_{1-\gamma}(a+y\sqrt{\gamma})].$$
If say $\gamma\in (0,1/2)$, then by the mean value theorem this density is bounded above for all $y\geq 0$ by
\[
\frac{1}{{\sqrt{\gamma}}}\frac{y}{a}\frac{\p_1(y)}{\p_1(a)} {[2y\sqrt{\gamma}\cdot \sup_{u\in \R} \p_{1-\gamma}'(u)]} = \frac{y}{a}\frac{\p_1(y)}{\p_1(a)} \frac{2y}{\sqrt{2\pi e}\,(1-\gamma)} \leq \frac{2^{3/2}}{\sqrt{\pi e}}\frac{y^2e^{-y^2/2}}{R^{-1}\p_1(R)}.
\]
The right hand side is integrable, so given any $\e>0$ we can find $M(\e,R)$ such that for all $\gamma\in(0,1/2)$, $\P_{\infty,\infty}(B(-\gamma) \ge M\sqrt{\gamma}) \le \e$. By Proposition \ref{convfd} it follows that for large $L$, $\P_{\alpha \sqrt{L},L}(B(-\gamma) \ge M\sqrt{\gamma}) \le 2\e$.

Let $\Br^M$ be a Brownian bridge on $[-\gamma,0]$ from $M\sqrt{\gamma}$ to $M\sqrt{\gamma}$ conditioned to stay positive. We can find $M'(\e,R)>0$ such that for all $\gamma>0$, the probability of the event $\{\sup_{x\in [-\gamma,0]} \Br^M(x) \le (M+M')\sqrt\gamma\}$ is at least $1-\e$. We then choose $\gamma$ such that $(M+M')\sqrt{\gamma} \le \rho/3$. For this $\gamma$, we have $\P_{\alpha \sqrt{L},L}(B(0) \ge M\sqrt{\gamma}) \le \e$ for all large $L$ by Lemma \ref{convzero}. The Gibbs property and monotonicity first compare with the soft-floor bridge having both endpoints $M\sqrt\gamma$. By dominated convergence its law tends to that of $\Br^M$. Consequently,
\begin{equation}\label{supgamma}
\begin{split}
\limsup_{L\to\infty}\P_{\alpha\sqrt{L},L}\bigg(\sup_{x\in [-\gamma,0]} B(x) \ge \rho/6 \bigg) &\le \limsup_{L\to\infty}\P_{\alpha\sqrt{L},L} \big(B(0)\vee B(-\gamma) \geq M\sqrt{\gamma}\big)\\ 
&\qquad + \P\bigg(\sup_{x\in[-\gamma,0]}\Br^M(x) \geq (M+M')\sqrt{\gamma}\bigg)\le 4\e.
\end{split}
\end{equation}
We now bound the modulus of continuity, for any $\delta>0$, via
\begin{align*}
    \mc(B|_{[-\gamma,0]},\delta) \le \sup_{x\in [-\gamma,0]} B(x) - \inf_{x\in [-\gamma,0]} B(x).
\end{align*}
The supremum is at most $\rho/6$ with limiting inferior probability at least $1-4\e$ by \eqref{supgamma}, and the infimum can be bounded below by $-\rho/6$ with probability $1-\e$ in the limit superior: the endpoint values are at least $-\rho/12$ with probability tending to one by Lemma \ref{convzero} and Proposition \ref{convfd}; remove the floor, lower both endpoints to $-\rho/12$, and use the reflection principle, shrinking $\gamma$ if necessary. All in all, for $\gamma = \gamma(\e,\rho,R)$ chosen as above and for the corresponding choice of $\delta = \delta(\gamma,\e,\rho,R)$ in \eqref{moc1g}, we obtain 
\begin{equation}\label{mocg0}
\limsup_{L\to\infty}\P_{\alpha \sqrt{L},L}\big(\mc(B|_{[-\gamma,0]},\delta)\ge \rho/3\big)\le 5\e.
\end{equation}
Finally, we must control the third term in \eqref{mocsplit}. For this we simply note that we can assume $\delta<\gamma/2$, and the argument used to prove \eqref{moc1g} also applies to the interval $[-1,-\gamma/2] \supset [-\gamma-\delta,-\gamma+\delta]$ instead of $[-1,-\gamma]$ if we  shrink $\delta$ further depending only on $\gamma$. Combining with \eqref{moc1g} and \eqref{mocg0}, the triangle inequality applied to \eqref{mocsplit} yields that $\limsup_{L\to\infty} \P_{\alpha\sqrt{L},L}(\mc(B,\delta)\geq \rho) \leq 7\e$ for this choice of $\delta$. Since $\e$ was arbitrary, \eqref{moclimsup} follows.
\end{proof}

Using the above we may finally prove the main theorem of this subsection.

\begin{proof}[Proof of Theorem \ref{wconvsf}] When $a>0$, the theorem follows from finite-dimensional convergence and tightness proven in Proposition \ref{convfd} and Lemma \ref{tightz} respectively. The case $a=0$ can be handled by a sandwich argument using stochastic monotonicity with respect to the starting value.  Fix any $\e>0$. From Lemma \ref{stm}\ref{paro}, we get a probability space supporting two processes $B_1,B_2$ on $[A,0]$ such that the marginal law of $B_1$ is given by $\P_{\alpha\sqrt{L},L}^{a_L+\e}$, the marginal law of $B_2$ is given by $\P_{\alpha\sqrt{L},L}^{a_L}$, and with probability $1$ for all $x\in [A,0]$ we have \begin{align}\label{twineq}B_1(x) \ge B_2(x) \ge B_1(x)-\e.\end{align} As $L\to \infty$, $\P_{\alpha\sqrt{L},L}^{a_L+\e}$ converges to $\P_{\infty,\infty}^{\e}$ by the $a>0$ case above, and as $\e\downarrow 0$,  $\P_{\infty,\infty}^{\e} \to \P_{\infty,\infty}^{0}$ by \cite[Theorem 4.1]{durrigle}. In view of \eqref{twineq}, the f.d.d.'s of $B_2$ must also converge to those of $\P_{\infty,\infty}^0$, and the modulus of continuity of $B_2$ is controlled by that of $B_1$, which implies tightness. Therefore  $\P_{\alpha\sqrt{L},L}^{a_L}\to \P_{\infty,\infty}^{0}$ weakly as $L\to \infty$.
\end{proof}

We conclude this section with two basic lemmas that will be used to show certain events have positive probability. The first shows that with positive probability, the Brownian bridge with law $\P_{\infty,\infty}$ can be dominated by any suitable continuous function.

\begin{lemma} \label{bmpos} Fix $a \ge 0$, $A<0$, and $\delta>0$. Consider any non-negative continuous function $f: [A,0]\to \R_{\ge 0}$ with $f(A)=a$.  We have
\begin{align*}
    \P_{\infty,\infty}\big(B(x) \le f(x)+\delta \mbox{ for all }x\in [A,0]\big) >0.
\end{align*}
Let $L(x)=ax/A$. We have
\begin{align}\label{unibd}
    \inf_{a\ge 0}\P_{\infty,\infty}^a\big(B(x) \le L(x)+\delta \mbox{ for all }x\in [A,0]\big) >0.
\end{align}
\end{lemma}
\begin{proof} Under $\P_{\infty,\infty}$, $B$ is a Brownian bridge from $a$ to $0$ conditioned to stay positive. Let $B'$ be a Brownian bridge from $a+\delta/2$ to $\delta/2$ conditioned to stay positive. By stochastic monotonicity, we may find a coupling $\P$ of $B$ and $B'$ such that $B(x) \le B'(x)$ for all $x\in [A,0]$. Thus,
\begin{align*}
    \P\big(B(x) \le f(x)+\delta \mbox{ for all }x\in [A,0]\big) & \ge \P\big(B'(x) \le f(x)+\delta \mbox{ for all }x\in [A,0]\big) \\ & \ge \P(0 < \Br(x) \le f(x)+\delta \mbox{ for all }x\in [A,0])
\end{align*}
where in the last line $\Br$ is a Brownian bridge under $\P$ from $a+\delta/2$ to $\delta/2$ (with no further conditioning). By standard properties of Brownian bridges (see e.g. Corollary 2.10 in \cite{ble}), this last probability is positive. When $f(x)=L(x)=ax/A$ we further observe that
\begin{align*}
     \P(0 < \Br(x) \le L(x)+\delta \mbox{ for all }x\in [A,0]) \ge \P(0 < \Br(x)-L(x) \le \delta \mbox{ for all }x\in [A,0]).
\end{align*}
Since $\Br(x)-L(x)$ is Brownian bridge from $\delta/2$ to $\delta/2$, the last probability is positive and is independent of $a$.
\end{proof}
It is also standard to show (see Corollary 2.10 in \cite{ble} for example) that a Brownian motion has a positive probability of staying inside a tube:
\begin{lemma}\label{brpos} Fix any $a \in \R$, $A<0$, and $\delta>0$. Let $B$ be a Brownian motion on $[A,0]$ started from $a$.  Consider any continuous function $g: [A,0]\to \R$ with $g(A)=a$. We have
\begin{align*}
    \P\big( |B(x)-g(x)| \le \delta \mbox{ for all }x\in [A,0]\big) >0.
\end{align*}
\end{lemma}

\section{Diffusive limits and modulus of continuity estimates for HSKPZ Gibbs measures} \label{sec7}  

In this section, we study limits of HSKPZ Gibbs measures that involves multiple paths.

\subsection{Multi-path convergence} \label{sec:multi}

In this subsection, we extend Theorem \ref{wconvsf} to a much broader class of HSKPZ Gibbs measures that involves multiple paths.  We first define the laws that will appear as a limit for the Gibbs measures; the notation is summarized in the following table.

\begin{definition}\label{defni}
    Fix any $a_1 > a_2$ and $A<0$. We define a law $\P_{\infty}^{2;A;(a_1,a_2),\star}$ of a line ensemble $(B_1, B_2) : [A,0]  \to \R^{2}$, which we refer to as \textit{non-intersecting pinned Brownian motions} ($\m{PBM}$) with $2$ paths starting from $(a_1,a_2)$, as follows. Let $U$ be a Brownian motion on $[A,0]$ started at $a_1+a_2$ with diffusion coefficient 2, and let $V$ be an independent Brownian bridge on $[A,0]$ from $a_1-a_2$ to 0 with diffusion coefficient 2 conditioned to stay positive on $(A,0)$ (see Remark \ref{rem:tran}). Then set $B_1 = \frac{1}{2}(U+V)$ and $B_2 = \frac{1}{2}(U-V)$. 
\end{definition}

The motivation for this definition comes from the following lemma, which shows that the pair $(B_1,B_2)$ can be seen as a weak limit as $\e\downarrow 0$ of two non-intersecting Brownian motions conditioned to end within distance $\e$ of each other at time 0.

\begin{lemma}\label{6.7}  Fix any $a_1>a_2$. Let $(B_1^\e,B_2^\e)$ denote two independent Brownian motions on $[A,0]$ started at $(a_1,a_2)$, conditioned on non-intersection $B_1^\e(x) > B_2^\e(x)$ for all $x\in (A,0)$ and $\e$-pinning $|B_1^\e(0) - B_2^\e(0)| < \e$. Set $U^\e(x):=B_1^\e(x)+B_2^\e(x)$ and $V^\e(x):=B_1^\e(x)-B_2^\e(x)$. Then $U^\e,V^\e$ are independent, $U^\e$ is a Brownian motion started from $a_1+a_2$ with diffusion coefficient $2$, and $V^\e$ is a Brownian motion started from $a_1-a_2$ with diffusion coefficient $2$ conditioned on $V^\e(x)>0$ for all $x\in(A,0)$ and $|V^\e(0)|<\e$.
\end{lemma}

\begin{proof} It is well-known that the sum and difference of two independent Gaussians are independent; in particular, the sum and difference of two independent Brownian motions are themselves independent Brownian motions with diffusion coefficient 2. Thus prior to the conditioning, $ U^\e$ and $ V^\e$ are independent Brownian motions with diffusion coefficient 2. Now the conditioning event can be written as $\{ V^\e(x)>0\mbox{ for all }x\in(A,0), \, | V^\e(0)| < \e\}$. Since this depends only on $ V^\e$, the conditional law of $U^\e$ given this event is unchanged.
\end{proof}

It is easy to see by arguments similar to those in the previous section that the process $V^\e$ above converges weakly as $\e\downarrow 0$ to a Brownian bridge from $a_1-a_2$ to 0 conditioned to stay positive on $(A,0)$, so writing $B_1^\e = \frac12(U^\e+V^\e)$ and $B_2^\e=\frac12(U^\e-V^\e)$ and taking $\e\downarrow 0$ recovers the process $(B_1,B_2)$ in Definition \ref{defni}. (We will not actually use this fact beyond motivating the definition of pinned Brownian motions, so we do not provide a proof here.)

For more than two paths, we define the corresponding law as follows.

\begin{definition}[Limiting law in the supercritical case]\label{defni2} Fix any $a_1 > a_2 > \cdots > a_{2m}$, $A<0$, and a continuous function $g:[A,0]\to\mathbb{R}\cup\{-\infty\}$ with $g(A) < a_{2m}$. We define a law $\P_{\otimes,\infty}^{2m;A;\vec{a},\star}$ of a line ensemble $B = (B_1, B_2, \ldots, B_{2m}) : [A,0]  \to \R^{2m}$, where $(B_{2i-1},B_{2i})\sim\P_{\infty}^{2;A;(a_{2i-1},a_{2i}),\star}$, $1\leq i\leq m$, are independent $\m{PBM}$'s with $2$ paths in the sense of Definition \ref{defni}. We then define $\P_{\infty;g}^{2m;A;\vec{a},\star}$ to be the law of $B$ conditioned on non-intersection $ B_{2i}(x) >  B_{2i+1}(x)$ for $i\in \ll1,m-1\rr$ and $B_{2m}(x) > g(x)$ for $x\in (A,0)$. We call this process \textit{non-intersecting pairwise pinned Brownian motions} with $2m$ paths starting from $(a_1,a_2,\ldots,a_{2m})$ and with a hard floor $g$, or $m$-$\m{PBM}$ for short.
\end{definition}

Note that in Definition \ref{defni2}, the non-intersection between pairs $B_{2i-1}$ and $B_{2i}$ is implicit in the definition of the 2-path $\mathsf{PBM}$'s, so we have $B_1 > \cdots > B_{2m} > g$ on $(A,0)$. The non-intersection event between different pairs and the hard floor $g$ has positive probability (see the proof of Theorem \ref{6.11} below), so the conditioning is well-defined.

\begin{definition}[Limiting law in the critical case] \label{defcr}
    Let $\vec{a} \in \R^{2m}$ with $a_1>a_2>\cdots>a_{2m}$ and $A<0$. Fix $g\in [A,0]\to \R\cup \{-\infty\}$ with $g(A)<a_{2m}$. Let $\bar{\P}_{\otimes}^{2m;A;\vec{a},\star}$ denote the law of independent Brownian motions $(B_i)_{i=1}^{2m}$ with drifts $(-1)^i\mu$ on $[A,0]$ started from $\vec{a}$. Let $\bar{\P}_{\infty;g}^{2m;A;\vec{a},\star}$ denote the law $\overline{\P}_{\otimes}^{2m;A;\vec{a},\star}$, conditioned on non-intersection $B_1(x)>\cdots>B_{2m}(x)>g(x)$ for all $x\in[A,0]
    $. If $g\equiv -\infty$, we drop $g$ from the subscript.
\end{definition}

\begin{definition}[Limiting law in the two-sided case] \label{defni3}  {Fix any $a_1 > a_2 > \cdots > a_{k}$, $b_1 > b_2 > \cdots > b_{k}$, $A<B$, and a continuous function $g:[A,B]\to\mathbb{R}\cup\{-\infty\}$ with $g(A) < a_{k}$ and $g(B)<b_k$. We define a law $\P_{\infty;g}^{k;A,B;\vec{a},\vec{b}}$ of a line ensemble $B = (B_1, B_2, \ldots, B_{k}) : [A,B]  \to \R^{k}$, where $B_i$'s are independent Brownian bridges from $a_i$ to $b_i$ for $i\in \ll1,k\rr$ conditioned on non-intersection $B_i(x)>B_{i+1}(x)$ for $i\in \ll1,k-1\rr$ and $B_{k}(x)>g(x)$ for $x\in [A,B]$.  We call this process \textit{non-intersecting Brownian bridges} with $k$ paths starting from $(a_1,a_2,\ldots,a_{k})$ to $(b_1,b_2,\ldots,b_k)$ and with a hard floor $g$.}
\end{definition}

\medskip

We now give the main result of this section, establishing convergence of scale-$L$ one-sided HSKPZ Gibbs measure (Definition \ref{kpzgibbs}) in the supercritical regime to $m$-$\m{PBM}$'s (as $L\to\infty$).

\begin{theorem}[Supercritical regime] \label{6.11} Fix any $\vec{a}\in \R^{2m}$ with $a_1>a_2>\cdots>a_{2m}$ and $A<0$. Let $g:[A,0]\to \R\cup \{-\infty\}$ be a continuous function such that $a_{2m+1}:=g(A)<a_{2m}$. Let $(\vec{a}_L)_{L\ge 1}$ be a sequence of vectors in $\R^{2m}$ with the property that $\vec{a}_L \to \vec{a}$. Let $g_L:[A,0]\to \R\cup \{-\infty\}$ be a sequence of continuous function such that $g_L(x) \to g(x)$ uniformly on $[A,0]$. Recall the one-sided HSKPZ Gibbs measures $\P_{L;f,g}^{k,\ell;A;\vec{a},\star}$ introduced in Definition \ref{kpzgibbs}.  Let $(B_i)_{i=1}^{2m} : [A,0]\to \R^{2m}$ be distributed as $\P_{L;g_L}^{2m;A;\vec{a}_L,\star}$. As $L\to \infty$, $(B_i)_{i=1}^{2m}$ converges weakly to non-intersecting pairwise pinned Brownian motions with $2m$ paths started from $\vec{a}$ and with a hard floor $g$. 
\end{theorem}

The proof of Theorem \ref{6.11} relies on the following simple observation about one-sided HSKPZ Gibbs measures with $2$ paths.

\begin{lemma}\label{6.9} Suppose $a_1>a_2$. Suppose $(B_1,B_2) : [A,0]\to \R^{2}$ is distributed as $\mathbf{P}_{L}^{2;A;(a_1,a_2),\star}$. 
The processes $U := B_1+B_2$, $V := B_1-B_2$ are independent. Furthermore, $U$ is a Brownian motion started from $(a_1+a_2)$ with diffusion coefficient $2$ and $V$ is distributed as $\P_{\alpha\sqrt{L},L}$ started from $(a_1-a_2)$ with diffusion coefficient $2$.
\end{lemma}

The proof is the same as that of Lemma \ref{6.7}, as the RN derivative in \eqref{rnd} only depends on the difference. Here the variance-$2$ version of $\P_{\alpha\sqrt L,L}$ means weighting a variance-$2$ Brownian motion by $\exp(-\alpha\sqrt L\,V(0)-L\int_A^0e^{-\sqrt L V(x)}dx)$. Theorem \ref{wconvsf} applies to $V/\sqrt2+(\log2)/\sqrt{2L}$ at scale $2L$, so its conclusion holds with diffusion coefficient $2$ as well.

\begin{proof}[Proof of Theorem \ref{6.11}] Observe that 
\begin{align*}
    \mathbf{P}_{L;g_L}^{2m;A;\vec{a}_L,\star}(\m{A}) = \frac{\E\left[\mathcal{W}_{\m{inter}}\mathcal{W}_{\m{intra}}\ind_{\m{A}}\right]}{\E\left[\mathcal{W}_{\m{inter}}\mathcal{W}_{\m{intra}}\right]}
\end{align*}
where
\begin{align*}
    \mathcal{W}_{\m{inter}}=\prod_{i=1}^m \exp\left(\alpha \sqrt{L}(B_{2i}(0)-B_{2i-1}(0))-L \int_A^0 e^{\sqrt{L}(B_{2i}(x)-B_{2i-1}(x))}dx\right),
\end{align*}
and
\begin{align*}
    \mathcal{W}_{\m{intra}}=\prod_{i=1}^{m} \exp\left(-L \int_A^0 e^{\sqrt{L}(B_{2i+1}(x)-B_{2i}(x))}dx\right),
\end{align*}
where $B_{2m+1}:=g_L$. (Here the expectation is with respect to independent Brownian motions $B_1,\dots,B_{2m}$.)
By Lemmas \ref{6.7}, \ref{6.9} and Theorem \ref{wconvsf} we see that the law
\begin{align*}
    \mathbf{P}_{\otimes,L}^{2m;A;\vec{a}_L,\star}(\m{A}) := \frac{\E\left[\mathcal{W}_{\m{inter}}\ind_{\m{A}}\right]}{\E\left[\mathcal{W}_{\m{inter}}\right]}
\end{align*}
converges to $m$ many independent non-intersecting pairwise pinned Brownian motions $( B_{2i-1}, B_{2i})$ with $2$ paths starting from $(a_{2i-1},a_{2i})$. Let us denote this limiting law as $\mathbf{P}_{\otimes,\infty}^{2m;A;\vec{a},\star}$. We also see that $$\mathcal{W}_{\m{intra}} \longrightarrow \ind\{B_{2i}(x) > B_{2i+1}(x), B_{2m}(x)>g(x) \mbox{ for all }x\in [A,0], \mbox{ and } i\in \ll1,m-1\rr\}$$
almost surely (and in $L^1$ as all random variables involved are in $[0,1]$), and the latter has positive probability under $\mathbf{P}_{\otimes,\infty}^{2m;A;\vec{a},\star}$. Indeed, let us write $U_i := B_{2i-1}+B_{2i}$ and $V_i := B_{2i-1}-B_{2i}$ for $i\in\llbracket 1,m\rrbracket$. By Definition \ref{defni}, the variables $U_i,V_i$, $i\in\llbracket 1,m\rrbracket$ are mutually independent, $U_i$ are Brownian motions started at $a_{2i-1}+a_{2i}$, and $V_i$ are Brownian bridges from $a_{2i-1}-a_{2i}$ to 0 conditioned to remain nonnegative. 
Set $\delta:=\frac1{16}\min_{1\le i\le 2m} (a_i-a_{i+1}).$ For $1\le i\le m$ define
\begin{align*}
    g_{2i-1}(x):=\max\{a_{2i-1},g(x)+8(m-i+1)\delta\}, \quad g_{2i}(x):=\max\{a_{2i},g(x)+8(m-i+1)\delta-4\delta\}.
\end{align*}
By Lemmas \ref{bmpos} and \ref{brpos}, there is a positive probability that the following holds for all $x\in[A,0]$ and $i\in\llbracket 1,m\rrbracket$:
\begin{align*}
   \left| U_i-g_{2i-1}(x)-g_{2i}(x)\right| \le \delta, \qquad  V_i \le g_{2i-1}(x)-g_{2i}(x)+\delta.
\end{align*}
It follows that
\begin{align*}
B_{2i-1}(x) &= \frac{1}{2}(U_i(x)+V_i(x)) \le g_{2i-1}(x)+\delta,\\ 
B_{2i}(x) &= \frac{1}{2}(U_i(x)-V_i(x)) \ge g_{2i}(x)-\delta 
\end{align*}
holds for all $x\in [A,0]$ and $i\in\llbracket 1,m\rrbracket$ with positive probability under $\mathbf{P}_{\otimes,\infty}^{2m;A;\vec{a},\star}$. In particular, $B_{2i}(x) > B_{2i+1}(x), B_{2m}(x)>g(x)$ for all $x\in[A,0]$ and $i\in \ll1,m-1\rr$ holds with positive probability under $\mathbf{P}_{\otimes,\infty}^{2m;A;\vec{a},\star}$.
In view of Definition \ref{defni2} for the $2m$-path measure, we conclude the desired convergence.
\end{proof}

The following theorem is  the analogue of Theorem \ref{6.11} in the critical regime.

\begin{theorem}[Critical regime] \label{6.11c} Fix any $\mu\in \R$, $\vec{a}\in \R^{2m}$ with $a_1>a_2>\cdots>a_{2m}$ and $A<0$. Suppose $\alpha=\mu/\sqrt{L}$. Let $g:[A,0]\to \R\cup \{-\infty\}$ be a continuous functions such that $a_{2m+1}:=g(A)<a_{2m}$. Let $(\vec{a}_L)_{L\ge 1}$ be a sequence of vectors in $\R^{2m}$ with the property that $\vec{a}_L \to \vec{a}$. Let $g_L:[A,0]\to \R\cup \{-\infty\}$ be a sequence of continuous function such that $g_L(x) \to g(x)$ uniformly on $[A,0]$. Recall the one-sided HSKPZ Gibbs measures $\P_{L;f,g}^{\ell;A;\vec{a},\star}$ introduced in Definition \ref{kpzgibbs}.  Let $(B_i)_{i=1}^{2m} : [A,0]\to \R^{2m}$ be distributed as $\P_{L;g_L}^{2m;A;\vec{a}_L,\star}$. As $L\to \infty$, $(B_i)_{i=1}^{2m}$ converges in law to $\bar{\P}_{\infty;g}^{2m;A;\vec{a},\star}$ from Definition~\ref{defcr}. 
\end{theorem}

\begin{proof} The proof is much simpler than that of Theorem \ref{6.11} as the RN derivative in \eqref{def:Wcont} converges to $\ind\{B_1(x)> B_2(x) > \cdots > B_{2m}(x) > g(x) \mbox{ for all }x\in [A,0]\}$, which has positive probability under the law $\overline{\P}_{\otimes}^{2m;A;\vec{a},\star}$ of independent Brownian motions with alternating drifts (see for instance Lemma 2.4 in \cite{halfairy}). The convergence now follows by an argument similar to the proof of Theorem \ref{lem:invprin}. We skip the details for brevity.
\end{proof}

{In a similar spirit one can prove convergence for the two-sided HSKPZ Gibbs measures introduced in Definition \ref{kpzgibbs2}. We record this below.}

\begin{theorem}[Two-sided Gibbs convergence] \label{6.11d} {Fix any $\vec{a},\vec{b}\in \R^{k}$ with $a_1>a_2>\cdots>a_{k}$, $b_1>b_2>\cdots>b_{k}$ and $A<B$. Let $g:[A,B]\to \R\cup \{-\infty\}$ be a continuous function such that $a_{k+1}:=g(A)<a_{k}$ and $b_{k+1}:=g(B)<b_k$. Let $(\vec{a}_L)_{L\ge 1}, (\vec{b}_L)_{L\ge 1}$  be sequences of vectors in $\R^{k}$ with the property that $\vec{a}_L \to \vec{a}$ and $\vec{b}_L\to \vec{b}$. Let $g_L:[A,B]\to \R\cup \{-\infty\}$ be a sequence of continuous function such that $g_L(x) \to g(x)$ uniformly on $[A,B]$. Recall the two-sided HSKPZ Gibbs measures $\P_{L;f,g}^{\ell;A,B;\vec{a},\vec{b}}$ introduced in Definition \ref{kpzgibbs2}.  Let $(B_i)_{i=1}^{k} : [A,B]\to \R^{k}$ be distributed as $\P_{L;g_L}^{k;A,B;\vec{a}_L,\vec{b}_L}$. As $L\to \infty$, $(B_i)_{i=1}^{k}$ converges in law to ${\P}_{\infty;g}^{k;A,B;\vec{a},\vec{b}}$ from Definition~\ref{defni3}.} 
\end{theorem}

\subsection{Modulus of continuity estimates for HSKPZ Gibbs measures} 
 In this subsection, we state modulus continuity type results on HSKPZ Gibbs measures. We first state results for ordered boundary conditions.
 
\begin{lemma}\label{unimodc} Fix $R\ge 1$, $A<0$, $\rho>0$ and $\e\in (0,1)$. Recall $\mc$ from \eqref{moc}. There exists $\delta_0(m,\e,\rho,A,R)>0$ such that for all $\delta\le \delta_0$ and all $\vec{a}\in \R^{2m}$ with $\norm{\vec{a}}\le R$ and $\min_{1\le i\le 2m-1} (a_i-a_{i+1}) \ge R^{-1}$, we have
\begin{align*}
\P_{\infty}^{2m;A;\vec{a},\star}\big(\mc(B,\delta)\ge \rho\big) \le \e.
\end{align*}
\end{lemma}
\begin{proof} Observe that
    \begin{align}\label{rati}
\P_{\infty}^{2m;A;\vec{a},\star}\big(\mc(B,\delta)\ge \rho\big) \le \frac{\P_{\otimes,\infty}^{2m;A;\vec{a},\star}\big(\mc(B,\delta)\ge \rho\big)}{\P_{\otimes,\infty}^{2m;A;\vec{a},\star}\big(B_{2i}(x) > B_{2i+1}(x) \mbox{ for all }x\in [A,0], \mbox{ and } i\in \ll1,m-1\rr\big)}.
    \end{align}
Under $\P_{\otimes,\infty}^{2m;A;\vec{a},\star}$, $(B_{2i-1},B_{2i})_{i=1}^{m}$ are independent. Let $U_i := B_{2i-1}+B_{2i}$ and $V_i := B_{2i-1}-B_{2i}$ for $i\in\llbracket 1,m\rrbracket$. By Definition \ref{defni}, the variables $U_i,V_i$, $i\in\llbracket 1,m\rrbracket$ are mutually independent, $U_i$ are Brownian motions started at $a_{2i-1}+a_{2i}$, and $V_i$ are Brownian bridges from $a_{2i-1}-a_{2i}$ to 0 conditioned to remain nonnegative.

We shall proceed by showing the denominator of the r.h.s.~of \eqref{rati} is bounded from below and the numerator can be made arbitrarily small. Let $L_i(x)=(a_{2i-1}-a_{2i})x/A$. One can check that
    \begin{align*}
        & \big\{B_{2i}(x) > B_{2i+1}(x) \mbox{ for all }x\in [A,0], \mbox{ and } i\in \ll1,m-1\rr\big\} \\ & \hspace{1cm}\supset \bigcap_{i=1}^m \big\{ |U_i(x) - a_{2i-1}-a_{2i}|\le R^{-1}/4 \mbox{ for all }x\in [A,0]\big\} \\ & \hspace{5cm}\cap \bigcap_{i=1}^m \big\{V_i \le L_i(x)+R^{-1}/4 \mbox{ for all }x\in [A,0]\big\}.
    \end{align*}
    By properties of Brownian motion and Lemma \ref{bmpos}, specifically \eqref{unibd}, the probabilities
   \begin{align*}
       & \P\big( |U_i(x) - a_{2i-1}-a_{2i}|\le R^{-1}/4 \mbox{ for all }x\in [A,0]\big),  \\ &\P\big(V_i \le L_i(x)+R^{-1}/4 \mbox{ for all }x\in [A,0]\big)
   \end{align*}
are bounded below by positive constants uniformly over all choices of $\vec{a}$ satisfying the hypotheses. Thus the denominator of the r.h.s.~of \eqref{rati} is bounded below by some $c=c(m,A,R)>0$. For the numerator, note that
\begin{align*}
\P_{\otimes,\infty}^{2m;A;\vec{a},\star}\big(\mc(B,\delta)\ge \rho\big) \le \sum_{i=1}^{m} \left[\P\big(\mc(U_i,\delta)\ge \rho\big)+\P\big(\mc(V_i,\delta)\ge \rho\big)\right].
\end{align*}
For Brownian motions $U_i$, the law of $\mc(U_i,\delta)$ does not depend on the starting point. Thus one can choose $\delta_1(m,\e,A,\rho)>0$ such that for all $\delta\le \delta_1$, $\P\big(\mc(U_i,\delta)\ge \rho\big) \le c\e/2m$. Appealing to Lemma \ref{tightz} (along with Theorem \ref{wconvsf}), one can choose $\delta_2(m,\e,A,\rho,R)>0$ such that for all $\delta\le \delta_2$, $\P\big(\mc(V_i,\delta)\ge \rho\big) \le c\e/2m$ uniformly over $\vec{a}$. Combining we get the desired result.
\end{proof}

\begin{lemma}\label{unimodc2} Fix $\mu\in \R$, $R\ge 1$, $A<0$, $\rho>0$ and $\e\in (0,1)$.  Recall $\mc$ from \eqref{moc}. There exists $\delta_0(m,\e,\rho,A,R)>0$ such that for all $\delta\le \delta_0$ and all $\vec{a}\in \R^{2m}$ with $\norm{\vec{a}}\le R$ and $\min_{1\le i\le 2m-1} (a_i-a_{i+1}) \ge R^{-1}$, we have
\begin{align*}
\bar{\P}_{\infty}^{2m;A;\vec{a},\star}\big(\mc(B,\delta)\ge \rho\big) \le \e.
\end{align*}
\end{lemma}

\begin{proof} Observe that
\begin{align*}
    \bar{\P}_{\infty}^{2m;A;\vec{a},\star}\big(\mc(B,\delta)\ge \rho\big) & \le \frac{\bar{\P}_{\otimes}^{2m;A;\vec{a},\star}\big(\mc(B,\delta)\ge \rho\big)}{\bar{\P}_{\otimes}^{2m;A;\vec{a},\star}\big(B_1(x)> B_2(x) > \cdots > B_{2m}(x) \mbox{ for all }x\in [A,0]\big)} \\ & \le \frac{\sum_{i=1}^{2m}\bar{\P}_{\otimes}^{2m;A;\vec{a},\star}\big(\mc(B_i,\delta)\ge \rho\big)}{ \bar{\P}_{\otimes}^{2m;A;\vec{a},\star}\bigg(\bigcap_{i=1}^{2m}\left\{\sup_{x\in [A,0]}|B_i(x)-a_i| \le R^{-1}/4 \right\}\bigg)}.
\end{align*}
The law of $\sup_{x\in [A,0]}|B_i(x)-a_i|$ and $\mc(B_i,\delta)$ under $\bar{\P}_{\otimes}^{2m;A;\vec{a},\star}$ does not depend on the starting point $a_i$. Thus the last expression above is independent of $\vec{a}$. Furthermore, the denominator is positive by Lemma 2.4 in \cite{halfairy}, and the numerator can be made arbitrarily small by choosing $\delta$ small enough (thanks to L\'{e}vy's modulus of continuity). This completes the proof.
\end{proof}

We now turn towards modulus of continuity estimate for unordered boundary data. Towards this end, we first show that with high probability the curves become ordered in the bulk.

\begin{lemma}\label{ord}  Fix $M>0$, $k\ge 2$, $\e\in (0,1)$, $A<0$, $R\le 0$, and a scale $L\ge 1$. Let $A^-=A/2+A/(4k+4)$ and $A^+=A/2-A/(4k+4)$.  Suppose $\vec{a}_L, \vec{b}_L\in \R^{k}$ with $\norm{\vec{a}_L}, \norm{\vec{b}_L} \le M$ and $g_L : [A,0]\to\mathbb{R}\cup\{-\infty\}$ is continuous with $\sup_{x\in[A,0]} g_L(x) \leq M$. There exists $\rho>0$ depending on $A,M,k,\e$ such that for all large enough $L$,
\begin{equation}
    \label{sepbd2}
    \begin{aligned}    & \mathbf{P}_{L;g_L(\cdot-R)}^{k;A+R,R;\vec{a}_L,\vec{b}_L}\bigg(\min_{1\le i\le k-1}\inf_{\,x \in [A^-+R,A^++R]} B_i(x)-B_{i+1}(x) \ge \rho \bigg) \\ & \hspace{3cm}=\mathbf{P}_{L;g_L}^{k;A,0;\vec{a}_L,\vec{b}_L}\bigg(\min_{1\le i\le k-1}\inf_{x \in [A^-,A^+]} B_i(x)-B_{i+1}(x) \ge \rho \bigg) \ge 1-\e.
\end{aligned}
\end{equation}

Suppose $k=2m$ is even and $\alpha>0$ is fixed or $\alpha=\mu/\sqrt{L}$ for $\mu\in \R$ fixed. Then
\begin{align}\label{sepbd}    \mathbf{P}_{L;g_L}^{2m;A;\vec{a}_L,\star}\bigg(\min_{1\le i\le 2m-1} \inf_{x\in [A^-,A^+]}B_i(x)-B_{i+1}(x) \ge \rho \bigg) \ge 1-\e.
\end{align}
\end{lemma}

The proof of Lemma \ref{ord} follows via applying results of \cite{wuconv}, where an analogous separation estimate is established for the full-space KPZ line ensemble. Indeed, the result essentially follows from \cite[Proposition 3.3]{wuconv}, but as that result is stated for the KPZ line ensemble itself rather than the finite-volume Gibbs measures, it does not immediately apply to our setup. For the sake of completeness, we explain how we can translate to the setup of \cite{wuconv} to establish \eqref{sepbd} in Appendix~\ref{app}.

Given the above ordering lemma, we obtain control on the modulus of continuity for scale-$L$ HSKPZ Gibbs measures even when the boundary conditions are unordered.

\begin{theorem}\label{moclemma} Fix $A<B$, $\rho>0$, $L\ge 1$ and $M,\gamma>0$.  Suppose $\vec{a}_L, \vec{b}_L \in \R^k$ with $\norm{\vec{a}_L},\norm{\vec{b}_L}\le M$, and recall $\mc$ from \eqref{moc}. 
\begin{itemize}[leftmargin=20pt]
    \item[a)] Let
\begin{align*}
    \hat{A}:=(A+B)/2 - (B-A)/(4k+4), \quad \hat{B}:=(A+B)/2 + (B-A)/(4k+4).
\end{align*}
We have
\begin{align}\label{twow}
    \lim_{\delta \downarrow 0}\limsup_{L\to \infty}\mathbf{P}_{L}^{k;A,B;\vec{a}_L,\vec{b}_L}\big(\mc(B|_{\ll 1,k\rr \times [\hat{A},\hat{B}]},\delta)\ge \rho\big)=0.
\end{align}
\item[b)] Suppose $\alpha>0$ is fixed or $\alpha=\mu/\sqrt{L}$ for $\mu\in \R$ fixed. Suppose additionally $A<0$ and $k=2m$. We have
\begin{align}\label{onew}
    \lim_{\delta \downarrow 0}\limsup_{L\to \infty}\mathbf{P}_{L}^{2m;A;\vec{a}_L,\star}\big(\mc(B|_{\ll 1,2m\rr \times [A/2,0]},\delta)\ge \rho\big)=0.
\end{align}
\end{itemize}
\end{theorem}

\begin{proof} \textbf{Proof of \eqref{twow}.} For $M',\eta>0$ define the events
    \begin{align*}
        & \m{D}_1 := \bigcap_{x\in \{\hat{A},\hat{B}\}}\left\{\min_{1\leq i\leq k-1} B_i(x) - B_{i+1}(x) \geq \eta \right\}, \\ & \m{D}_2:=\bigcap_{x\in\{\hat A,\hat B\}}\left\{\max_{1\leq i\leq k}|B_i(x)|\leq M'\right\}.
    \end{align*} 
 Using Lemma \ref{ord}, we may choose $\eta$ small enough so that $\P_L^{k;A,B;\vec{a}_L,\vec{b}_L}(\neg \m{D}_1) \le \e/2$. For the upper bound event in $\m{D}_2$, we apply stochastic monotonicity (Lemma \ref{lem:sm}) to raise the boundary conditions until they are separated, e.g., replace $a_{L,i}$ with $a_i' := (M+k-i+1)$ and $b_{L,i}$ with $b_i':=(M+k-i+1)$. Then for any $i$ $$\P_{L}^{k;A,B;\vec{a}_L,\vec{b}_L}(B_i(x) \leq M') \geq \P_{L}^{k;A,B;\vec{a}',\vec{b}'}(B_i(x) \leq M')\xrightarrow[]{L\to\infty} \P_{\infty}^{k;A,B;\vec{a}',\vec{b}'}(B_i(x)\le M'),$$ where the last convergence follows via Theorem \ref{6.11d} as the components of $\vec{a}',\vec{b}'$ are separated. The last probability can be made arbitrarily close to 1 for large $M'$. By lowering the boundary conditions instead, a similar argument shows that $\P_{L}^{k;A,B;\vec{a}_L,\vec{b}_L}(B_i(x) \leq -M')$ can be made arbitrarily small for large $M'$. Thus we may choose $M'$ large enough so that $\P_{L}^{k;A,B;\vec{a}_L,\vec{b}_L}(\neg\m{D}_2) \le \e/2$.

Set $\m{D}=\m{D}_1\cap \m{D}_2$.
Thus, for any $\e>0$ we can choose $M'$ large and $\eta$ small so that $\P_{L}^{k;A,B;\vec{a}_L,\vec{b}_L}(\neg\m{D}) < \e$, uniformly over $\lVert \vec{a}_L\rVert \leq M$. Thus by the Gibbs property, for large $L$
    \begin{equation}\label{mocA/21}
    \P_{L}^{k;A,B;\vec{a}_L,\vec{b}_L}\big(\mc(B|_{\ll 1,k\rr \times [\hat{A},\hat{B}]},\delta)\ge \rho\big) \leq  \E_{L}^{k;A,B;\vec{a}_L,\vec{b}_L}\left[\mathbf{1}_{\m{D}} \cdot \P_{L}^{k;\hat{A},\hat{B};\vec{a}_L',\vec{b}'_L}\big(\mc(B,\delta)\geq \rho\big)\right] + \e,
    \end{equation}
    where $\vec{a}_L' = (B_1(\hat{A}),\dots,B_{k}(\hat{A}))$ and $\vec{b}_L' = (B_1(\hat{B}),\dots,B_{k}(\hat{B}))$. Now for each $L,\delta$, there are deterministic values $\vec{a}_L^{\,0} = \vec{a}_L^{\,0}(\delta)$ and $\vec{b}_L^{\,0}=\vec{b}_L^{\,0}(\delta)$ of $\vec{a}_L'$ and $\vec{b}_L'$ with $\lVert \vec{a}_L^{\,0}\rVert \leq M'$, $\min(b_{L,i}^0-b_{L,i+1}^0) \geq \eta$ and $\lVert \vec{b}_L^{\,0}\rVert \leq M'$,$\min(b_{L,i}^0-b_{L,i+1}^0) \geq \eta$ which maximize the probability inside the expectation on the right of \eqref{mocA/21}. (The set of possible $\vec{a}_L',\vec{b}_L'$ is compact and the probability clearly depends continuously on $\vec{a}_L',\vec{b}_L'$.) Then \eqref{mocA/21} implies
    \begin{equation*}
    \P_{L}^{k;A,B;\vec{a}_L,\vec{b}_L}\big(\mc(B|_{\ll 1,k\rr \times [\hat{A},\hat{B}]},\delta)\ge \rho\big) \leq \P_{L}^{k;\hat{A},\hat{B};\vec{a}_L^{\,0},\vec{b}_L^{\,0}}\big(\mc(B,\delta)\geq \rho\big) + \e.
    \end{equation*}
    By passing to a subsequence we can assume that $\vec{a}_L^{\,0} \to \vec{a}^{0}$ $\vec{b}_L^{\,0} \to \vec{b}^{0}$ as $L\to\infty$, for some $\vec{a}^0=\vec{a}^0(\delta), \vec{b}^0 = \vec{b}^0(\delta)\in\R^{k}$  with $\lVert \vec{a}^0\rVert \leq M$, $a_i^0-a_{i+1}^0 \geq \eta$ and $\lVert \vec{b}^0\rVert \leq M$, $b_i^0-b_{i+1}^0 \geq \eta$. By Theorem \ref{6.11d} it follows that 
    \begin{equation*}
        \limsup_{L\to\infty} \mathbf{P}_{L}^{k;A,B;\vec{a}_L,\vec{b}_L}\big(\mc(B|_{\ll 1,k\rr \times [\hat{A},\hat{B}]},\delta)\ge \rho\big) \leq \P_{\infty}^{k;\hat{A},\hat{B};\vec{a}^0,\vec{b}^0}\big(\mc(B,\delta)\geq \rho\big) + \e.
    \end{equation*}
Let $L_i:[\hat{A},\hat{B}]\to \mathbb{R}$ be the linear map with $L_i(\hat{A})=a_i^0$ and $L_i(\hat{B})=b_i^0$ for $i\in \ll1,k\rr$. Let $\P_{\otimes}^{k;\hat{A},\hat{B};\vec{a}^0,\vec{b}^0}$ denote the law of $k$ independent Brownian bridges from $\vec{a}^0$ to $\vec{b}^0$. Note that
\begin{align*}
\P_{\infty}^{k;\hat{A},\hat{B};\vec{a}^0,\vec{b}^0}\big(\mc(B,\delta)\geq \rho\big) & =  \frac{\P_{\otimes}^{k;\hat{A},\hat{B};\vec{a}^0,\vec{b}^0}\big(\mc(B,\delta)\ge \rho\big)}{\P_{\otimes}^{k;\hat{A},\hat{B};\vec{a}^0,\vec{b}^0}\big(B_1(x)>B_2(x)>\cdots>B_k(x) \mbox{ for }x\in [\hat{A},\hat{B}]\big)} \\ & \le \frac{\sum_{i=1}^k\P_{\otimes}^{k;\hat{A},\hat{B};\vec{a}^0,\vec{b}^0}\big(\mc(B_i-L_i,\delta)+{2M'\delta/(\hat B-\hat A)}\ge \rho\big)}{\P_{\otimes}^{k;\hat{A},\hat{B};\vec{a}^0,\vec{b}^0}\big({|B_i(x)-L_i(x)|\leq\eta/4\ \text{for all }i\in\ll1,k\rr,\ x\in[\hat A,\hat B]}\big)}
\end{align*}
By the affine invariance of Brownian bridges, the last expression in the above equation is independent of the starting and ending points. Thus this expression can be made arbitrarily small by taking $\delta\downarrow 0$. Thus
\begin{equation*}
       \lim_{\delta\downarrow0} \limsup_{L\to\infty} \mathbf{P}_{L}^{k;A,B;\vec{a}_L,\vec{b}_L}\big(\mc(B|_{\ll 1,k\rr \times [\hat{A},\hat{B}]},\delta)\ge \rho\big) \leq \e.
    \end{equation*}
As $\e$ is arbitrary, this verifies \eqref{twow}.

\medskip

\noindent\textbf{Proof of \eqref{onew}.}  For $M',\eta>0$ define the events
    \begin{align*}
        & \m{E}_1 := \left\{\min_{1\leq i\leq 2m-1} B_i(A/2) - B_{i+1}(A/2) \geq \eta \right\}, \\ & \m{E}_2:=\left\{\max_{1\leq i\leq 2m} B_i(A/2) \le M'\right\}\cap \left\{\min_{1\leq i\leq 2m} B_i(A/2) \ge -M'\right\}.
    \end{align*} 
   Using Lemma \ref{ord}, we may choose $\eta$ small enough so that $\P_L^{2m;A;\vec{a}_L,\star}(\neg \m{E}_1) \le \e/2$. For the upper bound event in $\m{E}_2$, we apply stochastic monotonicity (Lemma \ref{lem:sm}) to raise the boundary conditions until they are separated, e.g., replace $a_{i}$ with $a_i' := (M+2m-i+1)$. Then for any $i$, in the supercritical case, $$\P_{L}^{2m;A;\vec{a},\star}(B_i(A/2) \leq M') \geq \P_{L}^{2m;A;\vec{a}',\star}(B_i(A/2) \leq M')\xrightarrow[]{L\to\infty} \P_{\infty}^{2m;A;\vec{a}',\star}(B_i(A/2)\le M'),$$ where the last convergence follows via Theorem \ref{6.11} as the components of $\vec{a}'$ are separated. In the critical case, thanks to  Theorem \ref{6.11c}, the above equation is true with $\P_{\infty}^{2m;A;\vec{a}',\star}$ replaced by $\bar{\P}_{\infty}^{2m;A;\vec{a}',\star}$ where the latter is defined in Definition \ref{defcr}. In both regimes, the last probability can be made arbitrarily close to 1 for large $M'$. By lowering the boundary conditions instead, a similar argument shows that $\P_{L}^{2m;A;\vec{a},\star}(B_i(A/2) \leq -M')$ can be made arbitrarily small for large $M'$. Thus we may choose $M'$ large enough so that $\P_{L}^{2m;A;\vec{a},\star}(\neg\m{E}_2) \le \e/2$. 

Set $\m{E}=\m{E}_1\cap \m{E}_2$.
Thus, for any $\e>0$ we can choose $M'$ large and $\eta$ small so that $\P_{L}^{2m;A;\vec{a}_L,\star}(\neg\m{E}) < \e$, uniformly over $\lVert \vec{a}_L\rVert \leq M$. Thus by the Gibbs property, for large $L$
    \begin{equation}\label{mocA/2}
    \P_{L}^{2m;A;\vec{a}_L,\star}\big(\mc(B|_{\ll 1,2m\rr \times [A/2,0]},\delta)\ge \rho\big) \leq  \E_{L}^{2m;A;\vec{a}_L,\star}\left[\mathbf{1}_{\m{E}} \cdot \P_{L}^{2m;A/2;\vec{b}_L,\star}\big(\mc(B,\delta)\geq \rho\big)\right] + \e,
    \end{equation}
    where $\vec{b}_L = (B_1(A/2),\dots,B_{2m}(A/2))$. Now for each $L,\delta$, there is a deterministic value $\vec{b}_L^{\,0} = \vec{b}_L^{\,0}(\delta)$ of $\vec{b}_L$ with $\lVert \vec{b}_L^{\,0}\rVert \leq M'$ and $\min(b_{L,i}^0-b_{L,i+1}^0) \geq \eta$ which maximizes the probability inside the expectation on the right of \eqref{mocA/2}. (The set of possible $\vec{b}_L$ is compact and the probability clearly depends continuously on $\vec{b}_L$.) Then \eqref{mocA/2} implies
    \begin{equation*}
    \P_{L}^{2m;A;\vec{a}_L,\star}\big(\mc(B|_{\ll 1,2m\rr \times [A/2,0]},\delta)\ge \rho\big) \leq \P_{L}^{2m;A/2;\vec{b}_L^{\,0},\star}\big(\mc(B,\delta)\geq \rho\big) + \e.
    \end{equation*}
    By passing to a subsequence we can assume that $\vec{b}_L^{\,0} \to \vec{b}$ as $L\to\infty$, for some $\vec{b} = \vec{b}(\delta)\in\R^{2m}$  with $\lVert \vec{b}\rVert \leq M$ and $b_i-b_{i+1} \geq \eta$. In the supercritical case, by Theorem \ref{6.11} it follows that 
    \begin{equation*}
        \limsup_{L\to\infty} \mathbf{P}_{L}^{2m;A;\vec{a}_L,\star}\big(\mc(B|_{\ll 1,2m\rr \times [A/2,0]},\delta)\ge \rho\big) \leq \P_{\infty}^{2m;A/2;\vec{b},\star}\big(\mc(B,\delta)\geq \rho\big) + \e.
    \end{equation*}
In the critical case, thanks to Theorem \ref{6.11c} the above equation is still true with $\P_{\infty}^{2m;A/2;\vec{b},\star}$ replaced by $\bar{\P}_{\infty}^{2m;A/2;\vec{b},\star}$. By Lemmas \ref{unimodc} and \ref{unimodc2}, the last probability (in both regimes) can be made arbitrarily small uniformly over $\vec{b}$ by taking $\delta$ small depending on $\e,\rho,M,\eta,m$. Since $\e$ was arbitrary we are done.
\end{proof}

We conclude with a lemma establishing that there is a positive probability that the scale-$L$ one-sided HSKPZ Gibbs measure will be very high on the right half of the interval.

\begin{lemma} \label{2high} Fix $A<B$ and $M_1,M_2,\gamma>0$, and $L\ge 1$. Fix $k\in\mathbb{N}$ and $\vec{a}_L,\vec{b}_L\in \mathbb{R}^k$ with $\lVert \vec{a}_L\rVert, \lVert \vec{b}_L\rVert \le M_1$. Let $g_L :[A,B] \to \R\cup \{-\infty\}$ be continuous.

\begin{itemize}[leftmargin=20pt]
    \item[a)] Suppose $(B_i)_{i=1}^{k}$ is distributed as $\P_{L}^{k;A,B;\vec{a}_L,\vec{b}_L}$.  Then there exist $L_0(k,A,B,M_1,M_2)>0$ and $\phi(k,A,B,M_1,M_2)$ such that for all $L\ge L_0$,

\begin{align*}
\mathbf{P}_{L;g_L}^{k;A,B;\vec{a}_L,\vec{b}_L}\Bigg(\inf_{\substack{ 1\le i\le k \\ x \in [A+(B-A)/4,B-(B-A)/4]}} B_{i}(x) \ge M_2 \Bigg) \ge \phi.
\end{align*}

\item[b)] Suppose $k=2m$ with $m\in \mathbb{N}$ and $B=0$. Suppose $\alpha>0$ is fixed or $\alpha=\mu/\sqrt{L}$ for $\mu\in \R$ fixed.  Suppose $(B_i)_{i=1}^{2m}$ is distributed as $\P_{L}^{2m;A;\vec{a}_L,\star}$. There exist $L_0(m,A,M_1,M_2,\alpha)>0$ and $\phi(m,A,M_1,M_2,\alpha)>0$ (with $\alpha$ replaced by $\mu$ in the critical case) such that for all $L\ge L_0$,
\begin{align*}
\mathbf{P}_{L;g_L}^{2m;A;\vec{a}_L,\star}\Bigg(\inf_{\substack{ 1\le i\le 2m \\ x \in [A/2,0]}} B_{i}(x) \ge M_2 \Bigg) \ge \phi.
\end{align*}
\end{itemize}
\end{lemma}

\begin{proof}
We first prove part~(b) in the supercritical case. Let $b_i=-(M_1+i)$. By stochastic monotonicity (Lemma \ref{lem:sm}) we may reduce $\vec{a}_L$ to $\vec{b}$ and $g_L$ to $-\infty$. This will only decrease the probability of the event in question. Thus, 
\begin{align}\label{reto1}
\mathbf{P}_{L;g_L}^{2m;A;\vec{a}_L,\star}\Bigg(\inf_{\substack{ 1\le i\le 2m \\ x \in [A/2,0]}} B_{i}(x) \ge M_2 \Bigg) & \ge \mathbf{P}_{L}^{2m;A;\vec{b},\star}\Bigg(\inf_{\substack{ 1\le i\le 2m \\ x \in [A/2,0]}} B_{i}(x) \ge M_2 \Bigg) \\ & {\xrightarrow{L\to\infty}} \ \mathbf{P}_{\infty}^{2m;A;\vec{b},\star}\left( \inf_{x\in [A/2,0]} B_{2m}(x) \ge M_2\right), \label{reto}
\end{align}
where the above convergence is via Theorem \ref{6.11}. It suffices to show the last probability is positive. To see this, one can for instance ask the curves to remain in certain corridors as follows. 
Let $f_i : [A,0] \to \R$ denote the piecewise linear function which interpolates between $b_i$ at $A$, $M_2+2m-i+1$ at $A/2$, and $M_2+2m-i+1$ at $0$. Define the event 
\begin{align*}
   \m{Crd} & :=\bigcap_{i=1}^{m}\bigcap_{x\in [A,0]} \big\{|B_{2i-1}(x)+B_{2i}(x)-f_{2i-1}(x)-f_{2i}(x)| \le 1/4\big\} \\ & \hspace{3cm}\cap \big\{0\le B_{2i-1}(x)-B_{2i}(x)\le f_{2i-1}(x)-f_{2i}(x)+1/4\big\}. 
\end{align*}
Note that $\m{Crd}$ implies $B_1(x)\geq\cdots\geq B_{2m}(x)$ on $[A,0]$, and $B_{2m}(x)>M_2$ on $[A/2,0]$.  Thus,
\begin{align*}
\mathbf{P}_{\infty}^{2m;A;\vec{b},\star}\left( \inf_{x\in [A/2,0]} B_{2m}(x) \ge M_2\right) \ge \mathbf{P}_{\otimes,\infty}^{2m;A;\vec{b},\star}\left(\m{Crd}\right),
\end{align*}
where $\mathbf{P}_{\otimes,\infty}^{2m;A;\vec{b},\star}$ is as in Definition \ref{defni2}. By definition of the 2-path $\m{PBM}$, and by Lemmas \ref{bmpos} and \ref{brpos}, the probability on the right is positive. This proves the lemma when $\alpha>0$ is fixed. In the critical regime, \eqref{reto1} and \eqref{reto} hold with ${\P}_{\infty}^{2m;A;\vec{b},\star}$ replaced by $\bar{\P}_{\infty}^{2m;A;\vec{b},\star}$ (from Defintion \ref{defcr}) by applying Theorem \ref{6.11c} instead. Note that
\begin{align*}
    \bar{\P}_{\infty}^{2m;A;\vec{b},\star}\left( \inf_{x\in [A/2,0]} B_{2m}(x) \ge M_2\right) \ge \bar{\P}_{\otimes}^{2m;A;\vec{b},\star}\left( \bigcap_{i=1}^{2m} \left\{\sup_{x\in [A,0]} |B_{i}(x)-f_i(x)|\le 1/4\right\}\right)
\end{align*}
where the law  $\bar{\P}_{\otimes}^{2m;A;\vec{b},\star}$ is also defined in Definition~\ref{defcr}. The last probability is positive via Lemma 2.4 in \cite{halfairy}. This proves part~(b) in the critical case.

We implement a similar strategy for part~(a). We appeal to stochastic monotonicity (Lemma \ref{lem:sm}) and Theorem \ref{6.11d} to conclude
\begin{align*}
& \mathbf{P}_{L;g_L}^{k;A,B;\vec{a}_L,\vec{b}_L}\Bigg(\inf_{\substack{ 1\le i\le k \\ x \in [A+(B-A)/4,B-(B-A)/4]}} B_{i}(x) \ge M_2 \Bigg) \\ & \ge \mathbf{P}_{L}^{k;A,B;\vec{b},\vec{b}}\Bigg(\inf_{\substack{ 1\le i\le k \\ x \in [A+(B-A)/4,B-(B-A)/4]}} B_{i}(x) \ge M_2 \Bigg) \\ & {\xrightarrow{L\to\infty}} \ \mathbf{P}_{\infty}^{k;A,B;\vec{b},\vec{b}}\left( \inf_{x\in [A+(B-A)/4,B-(B-A)/4]} B_{k}(x) \ge M_2\right).
\end{align*}
Let $g_i : [A,B] \to \R$ denote the piecewise linear function which interpolates between $b_i$ at $A$, $M_2+k-i+1$ at $A+(B-A)/4$ and $B-(B-A)/4$, and $b_i$ at $B$. We have
\begin{align*}
   \mathbf{P}_{\infty}^{k;A,B;\vec{b},\vec{b}}\left( \inf_{x\in [A+(B-A)/4,B-(B-A)/4]} B_{k}(x) \ge M_2\right) & \ge  \mathbf{P}_{\infty}^{k;A,B;\vec{b},\vec{b}}\left(\bigcap_{i=1}^{k}\bigcap_{x\in [A,B]} \big\{|B_{i}(x)-g_{i}(x)| \le 1/4\big\}\right) \\ &  \ge \mathbf{P}_{\otimes}^{k;A,B;\vec{b},\vec{b}}\left(\bigcap_{i=1}^{k}\bigcap_{x\in [A,B]} \big\{|B_{i}(x)-g_{i}(x)| \le 1/4\big\}\right)
\end{align*}
where $\mathbf{P}_{\otimes}^{k;A,B;\vec{b},\vec{b}}$ denotes the law of $k$ independent Brownian bridges on $[A,B]$ from $\vec{b}$ to $\vec{b}$ without any conditioning. The last probability is positive by properties of Brownian bridges. 
\end{proof}

\section{Preliminary tightness estimates for HSKPZ line ensemble under KPZ scaling} \label{sec8}

In this section, we investigate HSKPZLE under 1:2:3 scaling. Using Gibbsian line ensemble arguments, we first develop some preliminary tightness estimates. We divide this section into three subsections. In Section \ref{sec8.1} we state a distributional identity between the half- and full-space KPZ equations and show that the corresponding full-space KPZ equation approximates a parabola under KPZ scaling. In Section \ref{sec8.2} we deduce a weak form of parabolic trajectory for the top curve and prove upper tightness of the top curve. Finally, in Section \ref{sec8.3} we prove upper and lower tightness of all curves for any point far to the left.

\subsection{A distributional identity between half and full-space KPZ equations} \label{sec8.1} {The proof of Theorem \ref{kpz123} crucially relies on the distributional identity in Proposition \ref{p:iden}, which we prove in this section.} At the level of discrete log-gamma polymers, the analogous distributional identity was discovered in \cite{bw} (Theorem 1.4). We now state their result.

Consider a family of independent random variables $(\hat{\mathscr W}_{i,j})_{(i,j)\in \mathbb{N}^2}$ such that
\begin{align*}
\hat{\mathscr W}_{i,j}\sim \operatorname{Gamma}^{-1}(\alpha+\theta)  \textrm{ for } i=1 \qquad \textrm{and } \quad \hat{\mathscr W}_{i,j}\sim \operatorname{Gamma}^{-1}(2\theta) \textrm{ for } i>1.
	\end{align*}
    We define the full-space point-to-point polymer partition function as
    \begin{align*}
        Z_{\alpha}^{\mathrm{full}}(m,n)= \sum_{\pi \in \Pi_{m,n}^{(1)}} \prod_{(i,j)\in \pi} \hat{\mathscr W}_{i,j},
    \end{align*}
    where $\Pi_{m,n}^{(1)}$ is defined around \eqref{eq:symwt}. Theorem 1.4 in \cite{bw} states that
\begin{align}\label{disciden}
    Z_{\alpha}^{\mathrm{full}}(m,n) \overset{d} = \sum_{r=m}^{m+n-1} Z(r,m+n-r),
\end{align}
where $Z(\cdot,\cdot)$ is the point-to-point HSLG polymer partition function defined in \eqref{def:hslg}. Proposition \ref{p:iden} will follow by carefully taking an intermediate disorder limit of both sides of \eqref{disciden}.

\begin{proposition}\label{inhoml} Fix $N \in \mathbb{N}$, $t>0$, and $x\in \R$ such that $Nt$ is even and $Nt/2+{x}\sqrt{N}/2$ is a positive integer. Let us take $\theta=\frac12+\sqrt{N}$, $m=Nt/2-{x}\sqrt{N}/2+1$, and $n=Nt/2+{x}\sqrt{N}/2+1$. We have
\begin{align*}
    N^{(Nt+1)/2} \cdot Z_{\alpha}^{\mathrm{full}}(m,n) \stackrel{d}{\longrightarrow} \mathcal{Z}_{\alpha}^{\mathrm{full},B}({x},t)
\end{align*}
    as $N\to\infty$, where $\mathcal{Z}_{\alpha}^{\mathrm{full},B}$ is defined in Definition \ref{defshe}.
\end{proposition} 
We remark that a general convergence result but in the half-space setting can be found in \cite{bc22}. 
Let us now briefly explain why $\mathcal{Z}_{\alpha}^{\mathrm{full},B}$ arises in the limit. Intermediate disorder scaling limits of full-space polymers in i.i.d.~environments were investigated in the seminal paper of Alberts, Khanin, and Quastel \cite{akq2}. We would like to use their results, but the environment considered in defining $Z_{\alpha}^{\mathrm{full}}$ is not quite i.i.d.; it is a one-line perturbation of an i.i.d.~environment. However, by considering the first time the polymer paths enters into the bulk, we can decompose the partition function as follows:
\begin{align}\label{inhomdec}
    Z_\alpha^{\mathrm{full}}(m,n) =\sum_{k=1}^n \prod_{j=1}^k \hat{\mathscr W}_{1,j} \cdot Z_{\theta}^{\mathrm{full}}((2,k)\to (m,n)),
\end{align}
where $Z_{\theta}^{\mathrm{full}}((2,k)\to (m,n))$ are now partition functions for an i.i.d.~$\operatorname{Gamma}^{-1}(2\theta)$ environment. Under appropriate scaling, $Z_{\theta}^{\mathrm{full}}((2,k)\to (m,n))$ converges to the solution to full-space SHE with Dirac initial data, whereas $\prod_{j=1}^k \hat {\mathscr W}_{1,j}$ converges to a Brownian motion with drift. Together, these yield the full-space SHE in the limit with half-Brownian initial data.

\begin{proof}[Proof of Proposition \ref{inhoml}]
For clarity, we split the proof into three steps.

\medskip

\noindent \textbf{Step 1.} 
 We fix an $M>0$ and study the limit of a truncated version of \eqref{inhomdec}:
\begin{align}
   \nonumber \hat{Z}_N(m,n;M) & := N^{(Nt+1)/2}\cdot \sum_{k=1}^{M\sqrt{N}} \prod_{j=1}^k \hat{\mathscr W}_{1,j} \cdot Z_{\theta}^{\mathrm{full}}((2,k)\to (m,n)) \\ & =\frac{1}{\sqrt{N}}\sum_{k=1}^{M\sqrt{N}} \left(N^{k/2}\prod_{j=1}^k \hat{\mathscr W}_{1,j} \right)\cdot  \left(N^{\frac{Nt-k+2}2}Z_{\theta}^{\mathrm{full}}((2,k)\to (m,n))\right). \label{trunc}
\end{align}
We claim that that given any $\e>0$, we can choose $M(\e)>0$ such that 
\begin{align}\label{claim2}
   & \Pr\Big(\big|N^{(Nt+1)/2}\cdot Z_{\alpha}^{\mathrm{full}}(m,n)-\hat{Z}_N(m,n;M)\big| \ge \e\Big) \le \e
\end{align}
for all large $N$. We will prove this in the next step. Assuming the above claim, it thus suffices to study the limit of $\hat{Z}_N(m,n;M)$ as $N\to \infty$ and then take $M\to \infty$.  Thinking of $k\sim y\sqrt{N}$, we view the sum in \eqref{trunc} as a Riemann sum. Note that by Donsker's invariance principle and in view of the log-gamma mean and variance in \eqref{loggamma}, as a process in $y$,
\begin{align*}
  \log\left(N^{\lceil y\sqrt{N} \rceil/2}\prod_{j=1}^{\lceil y\sqrt{N} \rceil} \hat{\mathscr W}_{1,j}\right)=  \sum_{j=1}^{\lceil y\sqrt{N} \rceil} \log \hat{\mathscr W}_{1,j} +\lceil y\sqrt{N} \rceil\log \sqrt{N} 
\end{align*}
converges to $B(y)-\alpha y$ where $B$ is a standard Brownian motion. On the other hand, since $Z_{\theta}^{\mathrm{full}}((2,k)\to (m,n))$ involves an i.i.d.~environment, intermediate disorder limit results from \cite{akq2} apply. In particular, translating Theorem 2.2 in \cite{akq2} to our setting leads to the following weak convergence:
\begin{align*}
   N^{\frac{Nt-\lceil y\sqrt{N} \rceil+2}2}Z_{\theta}^{\mathrm{full}}((2,\lceil y\sqrt{N} \rceil)\to (m,n)) \longrightarrow \mathcal{Z}^{\mathrm{full}}(y,0;x,t)
\end{align*}
as processes in $y$, where $\mathcal{Z}^{\mathrm{full}}(y,0;x,t)$ is given in Definition \ref{defshe}. Combining these two limits and using continuity, we see that the truncated partition function $\hat Z_N(m,n;M)$  in \eqref{trunc} converges to
\begin{align*}
    \int_0^M e^{B(y)-\alpha y}\mathcal{Z}^{\mathrm{full}}(y,0;x,t)\,dy.
\end{align*}
Owing to the relation \eqref{caldec}, we see that as $M\to \infty$, the above quantity converges {in distribution} to $\mathcal{Z}_{\alpha}^{\mathrm{full},B}(x,t)$. 
Thus the proposition follows modulo \eqref{claim2}.

\medskip

\noindent\textbf{Step 2.} By Markov's inequality, to prove \eqref{claim2} it suffices to show that
\begin{align*}
    \Ex\left[ N^{(Nt+1)/2} \cdot Z_{\alpha}^{\mathrm{full}}(m,n)-\hat{Z}_N(m,n;M)\right]
\end{align*}
can be made arbitrarily small (uniformly in $N$) by taking $M$ large enough. To this end, notice that the above expectation simplifies to
\begin{align*}
   \frac1{\sqrt{N}}\sum_{k=M\sqrt{N}+1}^{n} N^{k/2}(\Ex[\hat{\mathscr W}_{1,1}])^k\cdot  \left(N^{\frac{Nt-k+2}2}\Ex\left[Z_{\theta}^{\mathrm{full}}((2,k)\to (m,n))\right]\right).
\end{align*}
We claim that for $1 \leq k \leq n$ uniformly in large $N$ depending on $x,\alpha$, we have 
\begin{align}\nonumber
    & N^{\frac{Nt-k+2}2} \Ex\left[Z_{\theta}^{\mathrm{full}}((2,k)\to (m,n))\right] \\ & \hspace{2cm}= N^{\frac{Nt-k+2}2}(\Ex\hat{\mathscr W}_{2,1})^{m+n-k-1}\cdot |\Pi_{m-1,n-k+1}^{(1)}| \le Ce^{-(-x\sqrt{N}+k)^2/CN},\label{stirling} \\&
    N^{k/2}(\Ex\hat{\mathscr W}_{1,1})^k \le Ce^{Ck /\sqrt{N}},\label{lgex}
\end{align}
for some constant $C>0$ depending on $t,\alpha$.
We will prove these two bounds in the next step. Assuming them, we may estimate (for $x$ fixed in \eqref{stirling} and assuming e.g. $M \ge 2x$)
\begin{align*}
    \Ex\left[ N^{(Nt+1)/2} \cdot Z_{\alpha}^{\mathrm{full}}(m,n)-\hat{Z}_N(m,n;M)\right] &\le \frac1{\sqrt{N}}\sum_{k=M\sqrt{N}+1}^{n} C^2\exp\left(-\frac{k^2}{CN} + \frac{Ck}{\sqrt{N}}\right)\\
    &\le \int_M^\infty C^2\exp(-u^2/C+C u)\,du,
\end{align*}
where the last inequality holds for $M>C^2$ so that the summand is decreasing in $k$ (and we changed variables in the integral). The last integral can be made arbitrarily small (uniformly in $N$) by taking $M$ large enough. 

\medskip

\noindent\textbf{Step 3.} To conclude we verify \eqref{stirling} and \eqref{lgex}. By a straightforward application of Stirling's formula (see the De Moivre--Laplace theorem), for all large $N$ and $1\leq k\leq n$,
\begin{equation}\label{Pibound}
\begin{split}
|\Pi^{(1)}_{m-1,n-k+1}| = \binom{Nt-k}{\frac{Nt+x\sqrt{N}}{2}-k+1} &\leq C\cdot \frac{2^{Nt-k}}{\sqrt{\pi(Nt-k)/2}} \exp\bigg(-\frac{(-x\sqrt{N}+k-2)^2}{2(Nt-k)}\bigg)\\
&\leq C\cdot\frac{2^{Nt-k}}{\sqrt{N}} \exp\bigg(-{\frac{(-x\sqrt{N}+k)^2}{CN}}\bigg).
\end{split}
\end{equation}
On the other hand, recalling $\widehat{{\mathscr W}}_{2,1} \sim \operatorname{Gamma}^{-1}(1+2\sqrt{N})$, we have
\[
(\Ex \widehat{{\mathscr W}}_{2,1})^{m+n-k-1} = (2\sqrt{N})^{-(m+n-k-1)} = 2^{-(Nt-k+1)}N^{-\frac{Nt-k+1}{2}}.
\]
Combining with \eqref{Pibound} implies \eqref{stirling}. For \eqref{lgex}, since $\widehat{\mathscr W}_{1,1} \sim \operatorname{Gamma}^{-1}(\alpha+\frac{1}{2}+\sqrt{N})$ we have
\begin{align*}
    N^{k/2}(\Ex\widehat{\mathscr W}_{1,1})^k = N^{k/2}(\alpha-\tfrac{1}{2}+\sqrt{N})^{-k} = \bigg(1+\frac{\alpha-\tfrac{1}{2}}{\sqrt{N}}\bigg)^{-k} \leq \bigg(1-\frac{C}{\sqrt{N}}\bigg)^{-k},
\end{align*}
for $C>|\alpha-\tfrac12|$. For large $N$, $-\log(1-C/\sqrt N)\leq2C/\sqrt N$, so the last expression is at most $e^{2Ck/\sqrt N}$ for every $k\geq1$. Enlarging $C$ if necessary, we obtain \eqref{lgex} for all $1\le k\le n$.
\end{proof}

\begin{proof}[Proof of Proposition \ref{p:iden}] The proof will proceed by showing that the two random variables stochastically dominate one another. For clarity, we split the proof into three steps.

\medskip

\noindent\textbf{Step 1.} In this step we show that $2\calZ_{\alpha}^{\mathrm{full},B}(x,t)$ stochastically dominates $\int_{-\infty}^{x} \mathcal{Z}_{\alpha}(y,t)\,dy$. To this end, let us fix $x'<x$, $t>0$, and $N \in \mathbb{N}$ such that $Nt$ is even and $Nt/2+y\sqrt{N}/2$ is a positive integer for $y\in \{x,x'\}$. Let us take $\theta=\frac12+\sqrt{N}$, $m=Nt/2-x\sqrt{N}/2+1$, $m'=Nt/2-x'\sqrt{N}/2+1$, $n=Nt/2+x\sqrt{N}/2+1$, and $n'=Nt/2+x'\sqrt{N}/2+1$. Observe by \eqref{disciden} that 
\begin{align*}
\Pr\left(2N^{(Nt+1)/2}Z_{\alpha}^{\mathrm{full}}(m,n) \le y \right) & =  \Pr\left(\frac{\2}{\sqrt{N}} \sum_{r=m}^{m+n-1} N^{(Nt+2)/2}Z(r,m+n-r) \le y\right) \\ & \le   \Pr\left( \frac{\2}{\sqrt{N}} \sum_{r=m}^{m'-1} N^{(Nt+2)/2}Z(r,m+n-r) \le y\right).
\end{align*}
In view of Theorem \ref{thm0} and Proposition \ref{inhoml}, taking $N\to \infty$ in the above yields
\begin{align*}
\P\left(2\calZ_{\alpha}^{\mathrm{full},B}(x,t) \le y \right) \le  \P\left(\int_{x'}^x \mathcal{Z}_\alpha(y,t)\,dy \le y\right) 
\end{align*}
for every $x'<x\le 0$. Taking $x'\downarrow -\infty$, we get that  $2\calZ_{\alpha}^{\mathrm{full},B}(x,t)$ is stochastically larger than $\int_{-\infty}^{x} \mathcal{Z}_{\alpha}(y,t)dy$.

\medskip

\noindent\textbf{Step 2.} In this step, we verify that $\int_{-\infty}^{x} \mathcal{Z}_{\alpha}(y,t)\,dy$ stochastically dominates $2\calZ_{\alpha}^{\mathrm{full},B}(x,t)$. To this end, we claim that for each $\e>0$, there exists $x'_0<0$ such that for all $x'\leq x'_0$,
\begin{align}\label{aclaim}
    \liminf_{N\to\infty}\Pr\left( \frac2{\sqrt{N}}\sum_{r=m'}^{m+n-1} N^{(Nt+2)/2}Z(r,m+n-r) \le \e \right) \ge 1-\e.
\end{align}
We prove \eqref{aclaim} in the next step.
Assuming the claim we observe that for $x'\le x'_0$ and large $N$,
\begin{align*}
    \Pr\left(\frac{\2}{\sqrt{N}} \sum_{r=m}^{m'-1} N^{(Nt+2)/2}Z(r,m+n-r) \le y\right)  & \le \e+ \Pr\left(\frac{\2}{\sqrt{N}} \sum_{r=m}^{m+n-1} N^{(Nt+2)/2}Z(r,m+n-r) \le y+\e \right) \\ & = \e+ \Pr\left(2N^{(Nt+1)/2}Z_{\alpha}^{\mathrm{full}}(m,n) \le y+\e \right),
\end{align*}
where the second equality follows via \eqref{disciden}.
In view of Theorem \ref{thm0} and Proposition \ref{inhoml}, taking $N\to \infty$ in the above we get
\begin{align*}
    \P\left(\int_{x'}^x \mathcal{Z}_\alpha(y,t)\,dy \le y\right)  & \le  \e+ \P\left(2\calZ_{\alpha}^{\mathrm{full},B}(x,t) \le y+\e \right).
\end{align*}
Taking $x'\downarrow -\infty$ yields 
\begin{align*}
    \P\left(\int_{-\infty}^x \mathcal{Z}_\alpha(y,t)\,dy \le y\right)  & \le \e+  \P\left(2\calZ_{\alpha}^{\mathrm{full},B}(x,t) \le y+\e\right).
\end{align*}
As $\e$ was arbitrary, this shows that $\int_{-\infty}^{x} \mathcal{Z}_{\alpha}(y,t)\,dy$ is stochastically larger than $2\calZ_{\alpha}^{\mathrm{full},B}(x,t)$.

\medskip

\noindent\textbf{Step 3.} To prove \eqref{aclaim}, we compute the first moment of the underlying random variable.
We observe the following chain of equalities and inequalities: for $x'$ fixed and $N$ sufficiently large,
\begin{align*}
    & \Ex\left[\frac2{\sqrt{N}}\sum_{r=m'}^{m+n-1} N^{(Nt+2)/2}Z(r,m+n-r)\right] = \Ex\left[2 N^{(Nt+1)/2}Z_{\alpha}^{\mathrm{full}}(m',n')\right] \\ &\qquad =\frac2{\sqrt{N}}\sum_{k=1}^{n'} N^{k/2}(\Ex[\hat{\mathscr W}_{1,1}])^k\cdot  \left(N^{\frac{Nt-k+2}2}\Ex\left[Z_{\theta}^{\mathrm{full}}((2,k)\to (m',n'))\right]\right) \\ &\qquad \le \frac2{\sqrt{N}}\sum_{k=1}^n C^2\exp\bigg(-\frac{(-x'\sqrt{N}+k)^2}{CN}+\frac{Ck}{\sqrt{N}}\bigg). 
\end{align*}
The first equality above is via \eqref{disciden} (noting that $m'+n'=m+n$), the second equality is via the decomposition in \eqref{inhomdec}, and the last inequality is due to \eqref{stirling} and \eqref{lgex}. Now assuming $x' \le -2C^2$, it is easy to check that the summand in the last line is decreasing in $k$. This implies that the sum in the last line is bounded above by the integral
\begin{align*}
    \frac{2}{\sqrt{N}}\int_0^\infty C^2\exp\bigg(-\frac{(-x'\sqrt{N}+y)^2}{CN} + \frac{Cy}{\sqrt{N}}\bigg)\,dy
    &= 2C^2e^{Cx'}\int_{-x'}^\infty e^{-y^2/C+Cy}\,dy.
\end{align*}
The expression on the right goes to 0 as $x'\to -\infty$, verifying \eqref{aclaim}.
\end{proof}

We end this section with a lemma showing that the full-space KPZ equation with half-Brownian initial data, i.e., $\log \mathcal{Z}_{\alpha}^{\mathrm{full},B}$,  has weak parabolic trajectory. 

\begin{lemma} \label{hbpara} Suppose $\alpha>0$ is fixed (supercritical) or $\alpha=\mu t^{-1/3}$ with $\mu\in \R$ fixed (critical). 
    For each $\e\in (0,1)$, there exists $K=K(\mu,\e)>0$ (the $\mu$ dependence is only in the critical case) such that for all $x\le 0$ we have
    \begin{align}\label{eq:hbpara}
       \limsup_{t\to \infty} \P\bigg(\bigg|\frac{\log\mathcal{Z}_{\alpha}^{\mathrm{full},B}(xt^{2/3},t)+t/24}{t^{1/3}}+\frac{x^2}2\bigg| \ge K\bigg) \le \e.
    \end{align}
\end{lemma}
\begin{proof} This essentially follows from results in \cite{corwin2013crossover,borodin2014free}.
Recall that $\mathcal{Z}_{\alpha}^{\mathrm{full},B}$ is the partition function with half-Brownian initial data where the Brownian motion has drift $-\alpha$. It is known (see \cite[Remark 1.14]{borodin2014free} for example) that this partition function at $(x,t)$ is equal in distribution to $e^{-x^2/2t}$ times the partition function at $(0,t)$ with the drift shifted to $-\alpha+x/t$, i.e.,
\begin{align*}
    \mathcal{Z}_{\alpha}^{\mathrm{full},B}(x,t) \stackrel{d}{=} e^{-x^2/2t} \cdot \mathcal{Z}_{\alpha-x/t}^{\mathrm{full},B}(0,t)
\end{align*}
holds pointwise. Under 1:2:3 scaling this translates to 
\begin{align}\label{indeg}
   \frac{\log\mathcal{Z}_{\alpha}^{\mathrm{full},B}(xt^{2/3},t)+t/24}{t^{1/3}}+\frac{x^2}2 \stackrel{d}{=} \frac{\log \mathcal{Z}_{\alpha-xt^{-1/3}}^{\mathrm{full},B}(0,t)+t/24}{t^{1/3}}.
\end{align}
We now study the weak limit of the right-hand side of \eqref{indeg} in critical and supercritical regimes. 

\begin{itemize}[leftmargin=25pt]
    \item (Critical regime) From Corollary 8 in \cite{corwin2013crossover} (see also Corollary 1.15(b) in \cite{borodin2014free}), when $\alpha=\mu t^{-1/3}$ with $\mu\in \R$ fixed we have
\begin{align*}
    2^{1/3}t^{-1/3}\big(\log \mathcal{Z}_{\mu t^{-1/3}-xt^{-1/3}}^{\mathrm{full},B}(0,t)+t/24\big) \stackrel{d}{\longrightarrow} \rev{\mathrm{BBP}_{2^{-1/3}(x-\mu)}},
\end{align*}
where ${\mathrm{BBP}_{2^{-1/3}(x-\mu)}}$ denotes a random variable with the Baik--Ben Arous--P{\'e}ch{\'e} distribution with a single boundary perturbation of strength ${2^{-1/3}(x-\mu)}$ (again recall that $\mathcal{Z}_{\alpha}^{\mathrm{full},B}$ has drift $-\alpha$). {Indeed, Corollary~1.15(b) of \cite{borodin2014free} uses
drift $(2/t)^{1/3}b$, so our drift $(x-\mu)t^{-1/3}$ corresponds
to the spike $b=2^{-1/3}(x-\mu)$.} The BBP distribution arises in spiked random matrix theory \cite{bbap} and we refer to \cite{bbap,baik} for its precise definition (which is not important for our purposes). It is known \cite[Eq.~(2.36)]{br01} that ${\mathrm{BBP}_{2^{-1/3}(x-\mu)}}$ converges weakly to $\mathrm{TW}_{\mathrm{GUE}}$ as $x\to -\infty$, where $\mathrm{TW}_{\mathrm{GUE}}$ is a Tracy--Widom GUE distributed random variable. 
    \item (Supercritical regime) When $\alpha>0$ is fixed, for every $x\le 0$ we claim that
\begin{align}\label{scnv}
    2^{1/3}t^{-1/3}\big(\log \mathcal{Z}_{\alpha-xt^{-1/3}}^{\mathrm{full},B}(0,t)+t/24\big) \stackrel{d}{\longrightarrow} \mathrm{TW}_{\mathrm{GUE}}.
\end{align}
This is not quite proved in \cite{corwin2013crossover,borodin2014free} but can be deduced from their results (along with fluctuation results for $\log\mathcal{Z}^{\mathrm{full}}(x,t)$) using soft arguments. Fix any $\mu>0$. For all large enough $t$ we have
\begin{align*}
    2^{1/3}t^{-1/3}\big(\log \mathcal{Z}_{\alpha-xt^{-1/3}}^{\mathrm{full},B}(0,t)+t/24\big) \le 2^{1/3}t^{-1/3}\big(\log \mathcal{Z}_{\mu t^{-1/3}-xt^{-1/3}}^{\mathrm{full},B}(0,t)+t/24\big).
\end{align*}
The latter converges to $\mathrm{BBP}_{2^{-1/3}(x-\mu)}$. Thus 
\begin{align*}
    \liminf_{t\to \infty} \P\left(2^{1/3}t^{-1/3}\big(\log \mathcal{Z}_{\alpha-xt^{-1/3}}^{\mathrm{full},B}(0,t)+t/24\big) \le y\right) \ge \P\left(\mathrm{BBP}_{2^{-1/3}(x-\mu)} \le y\right).
\end{align*}
Since this holds for every $\mu>0$, taking $\mu\to \infty$ we get that
\begin{align}\label{scnv1}
    \liminf_{t\to \infty} \P\left(2^{1/3}t^{-1/3}\big(\log \mathcal{Z}_{\alpha-xt^{-1/3}}^{\mathrm{full},B}(0,t)+t/24\big) \le y\right) \ge \P\left(\mathrm{TW}_{\mathrm{GUE}} \le y\right).
\end{align}
Let us now show that the opposite inequality is also true. Fix any $\beta>0$. Observe that
\begin{align*}
    \mathcal{Z}_{\beta}^{\mathrm{full},B}(0,t) & =\int_0^\infty e^{B(y)-\beta y}\mathcal{Z}^{\mathrm{full}}(y,0;0,t)dy \\ & \ge \int_0^{t^{-1}} e^{B(y)-\beta y}\mathcal{Z}^{\mathrm{full}}(y,0;0,t)dy \\ & \ge t^{-1} \inf_{y\in [0,t^{-1}]} e^{B(y)-\beta t^{-1}} \inf_{y\in [0,t^{-1}]}\mathcal{Z}^{\mathrm{full}}(y,0;0,t).
\end{align*}
Taking logarithms we thus get that
\begin{align}\label{lwf}
   \log \mathcal{Z}_{\beta}^{\mathrm{full},B}(0,t)  \ge -\log t +\inf_{y\in [0,t^{-1}]} B(y)-\beta t^{-1} +\inf_{y\in [0,t^{-1}]}\log \mathcal{Z}^{\mathrm{full}}(y,0;0,t).
\end{align}
Let us take the KPZ scaling limit of the right hand side of the above equation. Firstly, note that 
\begin{align}\label{lwf1}
   2^{1/3}t^{-1/3}\left(-\log t +\inf_{y\in [0,t^{-1}]} B(y)-\beta t^{-1}\right) \stackrel{p}{\longrightarrow} 0.
\end{align}
Next, note that $\log \mathcal{Z}^{\mathrm{full}}(\cdot,0;0,t)$ is equal in distribution to the Cole--Hopf solution of the full-space KPZ equation started from narrow wedge initial data, which has been extensively studied in the literature. The spatial increments of $\log \mathcal{Z}^{\mathrm{full}}(\cdot,0;0,t)$ can be controlled using available tail estimates and the Brownian Gibbs property of the associated line ensemble. In particular, invoking Proposition 4.1 from \cite{dg21} (with $\alpha=2,\kappa=1,\beta=t^{-5/6}, s=t^{1/6}$ therein) yields
\begin{align}\label{lwf2}
    \inf_{y\in [0,t^{-1}]} 2^{1/3}t^{-1/3}\left(\log\mathcal{Z}^{\mathrm{full}}(y,0;0,t) - \log\mathcal{Z}^{\mathrm{full}}(0,0;0,t) \right) \stackrel{d}{\longrightarrow} 0.
\end{align}
Finally, from the one-point distributional convergence for the full-space KPZ equation established in  \cite{amir2011probability}, we have
\begin{align}\label{lwf3}
  2^{1/3}t^{-1/3}\left(\log\mathcal{Z}^{\mathrm{full}}(0,0;0,t) + t/24\right) \stackrel{d}{\longrightarrow} \mathrm{TW}_{\mathrm{GUE}}.
\end{align}
Adding \eqref{lwf1}, \eqref{lwf2}, and  \eqref{lwf3}, and taking $\beta=\alpha-x t^{-1/3}$ in \eqref{lwf}, we obtain
\begin{align*}
    \limsup_{t\to \infty} \P\left(2^{1/3}t^{-1/3}\big(\log \mathcal{Z}_{\alpha-x t^{-1/3}}^{\mathrm{full},B}(0,t)+t/24\big) \le y\right) \le \P\left(\mathrm{TW}_{\mathrm{GUE}} \le y\right).
\end{align*}
Combining the above inequality with \eqref{scnv1} leads to \eqref{scnv}.
\end{itemize}

\medskip

In both regimes, we see that the limiting distribution of the right hand side of \eqref{indeg} is a tight sequence as $x$ varies over $(-\infty,0]$. Thus the $K$ in \eqref{eq:hbpara} can be chosen free of $x$.
\end{proof}

\begin{lemma} \label{hbpara2} Fix $\alpha,\beta\in\R$ with $\beta>(-\alpha)\vee 0$. For each  $x \in \R$ we have the following pointwise convergence
    \begin{align}\label{eq:hbpara2}
       \frac{\log\mathcal{Z}_{\alpha}^{\mathrm{full},B}(xt^{2/3}-\beta t,t)+\frac{t}{24}+\frac{\beta^2t}{2} -\beta xt^{2/3}}{t^{1/3}}+\frac{x^2}2 \stackrel{d}{\longrightarrow} 2^{-1/3}\mathrm{TW}_{\mathrm{GUE}}.
    \end{align}
\end{lemma}
\begin{proof} Taking $x\mapsto x-\beta t^{1/3}$ in \eqref{indeg} we get
\begin{align*}
    \frac{\log\mathcal{Z}_{\alpha}^{\mathrm{full},B}(xt^{2/3}-\beta t,t)+\frac{t}{24}}{t^{1/3}} +\frac{(x-\beta t^{1/3})^2}{2}  \stackrel{d}{=} \frac{\log\mathcal{Z}_{\alpha+\beta-xt^{-1/3}}^{\mathrm{full},B}(0,t)+t/24}{t^{1/3}},
\end{align*}
    which simplifies to the relation
    \begin{align*}
    \frac{\log\mathcal{Z}_{\alpha}^{\mathrm{full},B}(xt^{2/3}-\beta t,t)+\frac{t}{24}+\frac{\beta^2t}{2} - \beta xt^{2/3}}{t^{1/3}}+\frac{x^2}2  \stackrel{d}{=} \frac{\log\mathcal{Z}_{\alpha+\beta+xt^{-1/3}}^{\mathrm{full},B}(0,t)+t/24}{t^{1/3}},
\end{align*}
By \eqref{scnv} the right hand side converges to $2^{-1/3}\mathrm{TW}_{\mathrm{GUE}}$ as long as $\alpha+\beta>0$ (though \eqref{scnv} therein assumes $x\le 0$; the proof there works for any $x\in \R$). 
\end{proof}

\subsection{Properties of the top curve} \label{sec8.2}
In this section, we first deduce a weak parabolic trajectory and upper tightness of the top curve $\mathfrak{H}_1^{t,\alpha;\beta}$ (recall the definition in \eqref{def:Hrescaled}) utilizing the parabolic trajectory of $\log \mathcal{Z}_\alpha^{\mathrm{full},B}$ and the identity in Proposition \ref{inhoml}.  For convenience, we drop $\alpha$ from the notation in this subsection and the next and simply write
\begin{align*}
    \calZ=\calZ_\alpha, \quad \calZ^{\mathrm{full},B} = \calZ_\alpha^{\mathrm{full},B}, \quad \H_i=\H_i^\alpha, \quad \mathfrak{H}_i^{t;\beta}=\mathfrak{H}_i^{t,\alpha;\beta}, \quad \mathfrak{H}_i^{t}=\mathfrak{H}_i^{t,\alpha}.
\end{align*}
\begin{lemma}\label{p1} Fix any $\e\in (0,1)$. Assume $\alpha>0$ is fixed or $\alpha=\mu t^{-1/3}$ with $\mu\in\R$ fixed. There exists $M_0(\e)>0$ such that for all $r,M_1,M_2\ge M_0$, $x<0$, and large $t$ we have
    \begin{align} \label{p1e}
        \P\bigg(\sup_{y\in [x-r,x]} \mathfrak{H}_1^{t}(y)+\frac{x^2}{2} \ge -M_1\bigg) \ge 1-\e. 
    \end{align}
    \begin{align} \label{p2e}
        \P\bigg(\inf_{y\in [x-2t^{-2/3},x]} \mathfrak{H}_1^t(y)+\frac{x^2}2 \le M_2\bigg) \ge 1-\e. 
    \end{align}
\end{lemma}

\begin{proof} We write that an event $\m{A}$ happens $\wp(\e)$  if $\P(\m{A})\ge 1-\e$ for all large enough $t$. By Proposition \ref{p:iden} and Lemma \ref{hbpara}, $\wp(\e)$ we have 
\begin{align}\label{wpe}
t^{-1/3}\left(\log \int_{-\infty}^{xt^{2/3}} \mathcal{Z}(y,t)\,dy+ t/24\right)\stackrel{d}{=}   t^{-1/3}(\log (2\mathcal{Z}^{\mathrm{full},B}(xt^{2/3},t)+t/24)) \ge -K-\frac{x^2}{2}
\end{align}
 and $$t^{-1/3}\left(\log \int_{-\infty}^{(x-r)t^{2/3}} \mathcal{Z}(y,t)\,dy+ t/24\right)\stackrel{d}{=} t^{-1/3}(\log (2\mathcal{Z}^{\mathrm{full},B}((x-r)t^{2/3},t))+t/24) \le K-\frac{(x-r)^2}{2},$$
 for some $K=K(\mu,\e)$.
Let us take $r \ge \sqrt{4K+2\log 2}$, so that $K-\frac{(x-r)^2}{2}+\log2 \le -K-\frac{x^2}{2}$. For this choice of $r$, a union bound (and assuming $t\ge 1$) gives us that $\wp(2\e)$
\begin{align*}
    t^{-1/3}\left(\log \int_{-\infty}^{xt^{2/3}} \mathcal{Z}(y,t)\,dy+t/24\right) \ge t^{-1/3}\left(\log \int_{-\infty}^{(x-r)t^{2/3}} \mathcal{Z}(y,t)\,dy+t/24\right)+t^{-1/3}\log 2,
\end{align*}
which simplifies to $$\int_{-\infty}^{xt^{2/3}} \mathcal{Z}(y,t)\,dy \ge 2\int_{-\infty}^{(x-r)t^{2/3}} \mathcal{Z}(y,t)\,dy.$$ Rearranging we thus get that $\wp(2\e)$,
\begin{align}\label{wpe2}
    \int_{(x-r)t^{2/3}}^{xt^{2/3}} \mathcal{Z}(y,t)\,dy \ge \frac12\int_{-\infty}^{xt^{2/3}} \mathcal{Z}(y,t)\,dy.
\end{align}
On the other hand,
$$rt^{2/3} \sup_{y\in [x-r,x]} \mathcal{Z}(yt^{2/3},t) \ge \int_{(x-r)t^{2/3}}^{xt^{2/3}} \mathcal{Z}(y,t)\,dy.$$
Thus, taking \eqref{wpe} and \eqref{wpe2} into account we get that $\wp(3\e)$, (recalling $\mathcal{H}_1 \overset{d}= \log\mathcal{Z}$)
\begin{align*}
    \sup_{y\in [x-r,x]} \mathcal{H}_1(yt^{2/3},t)+t/24 &  \ge -t^{1/3}(K+x^2/2)-\log (2rt^{2/3}).
\end{align*}
As for large enough $t$, we have $t^{-1/3}\log (2rt^{2/3}) \le 1$, rearranging terms we find that $\wp(3\e)$,
\begin{align*}
    \sup_{y\in [x-r,x]} \mathfrak{H}_1^t(y)+\frac{x^2}{2} \ge -K-1.
\end{align*}
Adjusting $\e\mapsto \e/3$ and setting $M=K(\mu,\e/3)+1$, we conclude \eqref{p1e}. For the second bound, \eqref{p2e}, note that for any $r>0$ we have
\begin{align*}
   \log [rt^{2/3}] + \inf_{y\in [x-r,x]}  \mathcal H_1(yt^{2/3},t)   & \le \log \int_{(x-r)t^{2/3}}^{xt^{2/3}} \calZ(y,t)\,dy \\ & \le \log \int_{-\infty}^{xt^{2/3}} \calZ(y,t)\,dy \stackrel{d}{=} \log\mathcal Z^{\mathrm{full},B}(xt^{2/3},t) +\log \2.
\end{align*}
where the last distributional equality is due to Proposition \ref{p:iden}.
Take $r=2t^{-2/3}$ in the above to get that $\inf_{y\in [x-2t^{-2/3},x]}  \mathfrak H_1^t(y)$  is stochastically smaller than
\begin{align*}
       t^{-1/3}(\log \mathcal Z^{\mathrm{full},B}(xt^{2/3},t) +t/24)
\end{align*}
which in turn is less than $K-x^2/2$ $\wp(\e)$ by Lemma \ref{hbpara}. Taking $M_2=K(\mu,\e)$ completes the proof. 
\end{proof}

\begin{lemma}\label{app1} Fix any $x\in \mathbb{R}$. Fix any $\alpha \in \mathbb{R}$ and $\beta > \max\{-\alpha,0\}$. Assume $t$ is so large that $xt^{2/3}-\beta t \le 0$. Set 
\begin{align}
    U_t(x) & :=t^{-1/3}\left(\log\int_{-\infty}^{xt^{2/3}-\beta t}\mathcal{Z}(y,t)dy+t/24+\beta^2t/2-\beta xt^{2/3}\right)+\frac{x^2}{2}.
\end{align}
{Set $w_t=(1-e^{-\beta t^{5/12}})/\beta$.}
\begin{enumerate}[leftmargin=20pt,label=\alph*)]
    \item Almost surely we have
\begin{align*}
    \inf_{y\in [x-t^{-1/4},x]}  {\mathfrak H}_1^{t;\beta}(y)+\frac{x^2}{2} \le  U_t(x) - {t^{-1/3}\log w_t}.
\end{align*}
\item Fix any $\e\in (0,1)$. There exists $t_0(x,\e)>0$ such that for all $t\ge t_0$,
\begin{align*}
    \P\left( \sup_{y\in [x-t^{-1/4},x]} {\mathfrak{H}}_1^{t;\beta}(y)+\frac{x^2}{2}  \ge U_t(x)-{t^{-1/3}\log(2w_t)}\right) \ge 1-\e/2.
\end{align*}
\end{enumerate}
\end{lemma}

\begin{proof}
{Let $r=t^{-1/4}$. By the definitions of $\mathfrak H_1^{t;\beta}$ and $U_t$,
\begin{align*}
e^{t^{1/3}(U_t(x)-x^2/2)}
& =t^{2/3}\int_{-\infty}^x e^{t^{1/3}\mathfrak H_1^{t;\beta}(y)}e^{\beta t^{2/3}(y-x)}\,dy \\ & \ge t^{2/3}\int_{x-r}^x e^{t^{1/3}\mathfrak H_1^{t;\beta}(y)}e^{\beta t^{2/3}(y-x)}\,dy \\ & \ge \inf_{y\in [x-r,x]} e^{t^{1/3}\mathfrak H_1^{t;\beta}(y)} \cdot t^{2/3}\int_{x-r}^x e^{\beta t^{2/3}(y-x)}\,dy.
\end{align*}
We have
$t^{2/3}\int_{x-r}^x e^{\beta t^{2/3}(y-x)}dy=w_t$.
Taking logarithms proves part (a).
}For part (b), By Proposition \ref{p:iden} and Lemma \ref{hbpara2}, with probability at least $1-\e/4$ we have 
\begin{align}\label{apwpe}
& t^{-1/3}\left(\log \int_{-\infty}^{xt^{2/3}-\beta t} \mathcal{Z}(y,t)\,dy+ t/24+\beta^2t/2\right) \\ & \stackrel{d}{=}   t^{-1/3}({\log(2\mathcal{Z}^{\mathrm{full},B}(xt^{2/3}-\beta t,t))+t/24+\beta^2t/2) \ge -K}-\frac{x^2}{2}+\beta x t^{1/3},
\end{align}
 and with probability at least $1-\e/4$ we have
 \begin{align}
   & t^{-1/3}\left(\log \int_{-\infty}^{(x-r)t^{2/3}-\beta t} \mathcal{Z}(y,t)\,dy+ t/24+\beta^2t/2\right)\\ & \stackrel{d}{=} t^{-1/3}({\log(2\mathcal{Z}^{\mathrm{full},B}((x-r)t^{2/3}-\beta t,t))+t/24+\beta^2t/2) \le K}-\frac{(x-r)^2}{2}+\beta(x-r)t^{1/3},  
 \end{align}
 {for some $K=K(\e)$. The estimate at the moving point $x-r$ follows by the drift-shift identity \eqref{indeg} and monotonicity in the drift, bounding $x-r$ between $x-1$ and $x$.} By our choice of $r$, for large enough $t$, we have $$K-\frac{(x-r)^2}{2}+\beta(x-r)t^{1/3}+t^{-1/3}\log2 \le -K-\frac{x^2}{2}+\beta x t^{1/3}.$$ For this choice of $r$, a union bound (and assuming $t\ge 1$) gives us that with probability at least $1-\e/2$,
\begin{align*}
    & t^{-1/3}\left(\log \int_{-\infty}^{xt^{2/3}-\beta t} \mathcal{Z}(y,t)\,dy+t/24+\beta^2 t/2\right) \\ & \ge t^{-1/3}\left(\log \int_{-\infty}^{(x-r)t^{2/3}-\beta t} \mathcal{Z}(y,t)\,dy+t/24+\beta^2 t/2\right)+t^{-1/3}\log 2,
\end{align*}
which simplifies to 
\begin{align*}
    \int_{-\infty}^{xt^{2/3}-\beta t} \mathcal{Z}(y,t)\,dy & \ge 2\int_{-\infty}^{(x-r)t^{2/3}-\beta t} \mathcal{Z}(y,t)\,dy \\ & =2\int_{-\infty}^{xt^{2/3}-\beta t} \mathcal{Z}(y,t)\,dy-2\int_{(x-r)t^{2/3}-\beta t}^{xt^{2/3}-\beta t} \mathcal{Z}(y,t)\,dy.
\end{align*}
Rearranging we thus get that  with probability at least $1-\e/2$ we have
\begin{align*}
    \int_{(x-r)t^{2/3}-\beta t}^{xt^{2/3}-\beta t} \mathcal{Z}(y,t)\,dy \ge \frac12\int_{-\infty}^{xt^{2/3}-\beta t} \mathcal{Z}(y,t)\,dy.
\end{align*}
{On this event, the same weighted integral over $[x-r,x]$ is at least $\frac12 e^{t^{1/3}(U_t(x)-x^2/2)}$ and is at most $w_t\exp(t^{1/3}\sup_{y\in[x-r,x]}\mathfrak H_1^{t;\beta}(y))$. Taking logarithms proves part (b).
}\end{proof}

\begin{corollary}\label{cor:b3} Fix any $\e\in(0,1)$ and $x\in\R$. There exists $M_0(\e)>0$ such that, for every $M\geq M_0$, there exists $t_0(M,x,\e)>0$
    \begin{align} \label{app1e}
        \P\bigg(\sup_{y\in [x-t^{-1/4},x]} \mathfrak{H}_1^{t;\beta}(y)+\frac{x^2}{2} \ge -M\bigg) \ge 1-\e,
    \end{align}
    \begin{align} \label{app2e}
        \P\bigg(\inf_{y\in [x-t^{-1/4},x]} \mathfrak{H}_1^{t;\beta}(y)+\frac{x^2}2 \le M\bigg) \ge 1-\e. 
    \end{align} 
\end{corollary}
\begin{proof} By Proposition \ref{p:iden},
\begin{align}
    U_t(x)\stackrel{d}{=} t^{-1/3}\left(\log\left(2\mathcal{Z}^{\mathrm{full},B}(xt^{2/3}-\beta t,t)\right)+t/24+\beta^2t/2-\beta xt^{2/3}\right)+\frac{x^2}{2}.
\end{align}
    By Lemma \ref{hbpara2}, the right hand side of the above equation weakly converges to $2^{-1/3}\mathrm{TW}_{\mathrm{GUE}}$ (pointwise). Corollary \ref{cor:b3} now follows from Lemma \ref{app1}.
\end{proof}

Given Lemma \ref{p1} and tightness of $t^{-1/3}(\log\mathcal{Z}^{\mathrm{full},B}(0,t)+t/24)$ from Lemma \ref{hbpara}, we can deduce upper tightness of the supremum of the top curve on any interval.

\begin{proposition} \label{1ult} Fix any $\e\in (0,1)$ and $L>0$. For sufficiently large $t$ we have the following. 

\begin{itemize}[leftmargin=20pt]
    \item[a)] Assume $\alpha>0$ is fixed or $\alpha=\mu t^{-1/3}$ with $\mu\in \mathbb{R}$ fixed. There exists $R(\mu, L,\e)>0$ such that
\begin{align*}
        \P\bigg(\sup_{y\in [-L,0]} \mathfrak{H}_1^{t;0}(y) \le R\bigg) \ge 1-\e. 
    \end{align*}
    \item[b)] Assume $\alpha\in \mathbb{R}$ and $\beta>\max\{-\alpha,0\}$. There exists $R( L,\e)>0$ such that
    \begin{align*}
        \P\bigg(\sup_{y\in [-L,L]} \mathfrak{H}_1^{t;\beta}(y) \le R\bigg) \ge 1-\e. 
    \end{align*}
    \end{itemize}
    
\end{proposition}

\begin{proof} \textbf{Proof of Part(a)}. Fix $\e>0$. Consider $M_0(\e)$ and $K(\e)$ from Lemma \ref{p1} and Lemma \ref{hbpara} respectively. Set $R:=8[M_0+(4L)^2]+9K$, and define the events
\begin{align*}
    \m{A} & :=\left\{\sup_{y\in [-4L-M_0,-4L]} \mathfrak{H}_1^{t;0}(y)+\frac{(4L)^2}{2} \ge -M_0 \right\}, \\ \quad \m{B} & :=\left\{\sup_{y\in [-L,0]} \mathfrak{H}_1^{t;0}(y) \ge R\right\}, \quad \m{C}:=\left\{\inf_{y\in [-2L,-L]} \mathfrak{H}_1^{t;0}(y) \ge R/8\right\}.
\end{align*}
Note that on the event $\m{C}$, by Proposition \ref{p:iden} we have (stochastically)
\begin{align*}
    \mathcal Z^{\mathrm{full},B}(0,t) \ge \frac12\int_{-2Lt^{2/3}}^{-Lt^{2/3}} \mathcal Z(y,t)\,dy \ge \frac12 Lt^{2/3}\exp(Rt^{1/3}/8-t/24).
\end{align*}
For large $t$ this forces
\begin{align*}
    t^{-1/3}[\log \mathcal Z^{\mathrm{full},B}(0,t) +t/24] \ge R/9 \ge K,
\end{align*}
which by Lemma \ref{hbpara} happens with probability at most $\e$. Thus $\P(\m{C}) \le \e$. We claim that
\begin{align}\label{pcla}
    \P(\m{A}\cap \m{B}) \le 2  \cdot \P(\m{C})
\end{align}
for all large enough $t$. This implies $\P(\m{A}\cap \m{B}) \le 2\e$. By Lemma \ref{p1}, $\P(\neg \m{A}) \le \e$ so $\P(\m{B}) \le 3\e$, and adjusting $\e$ we get part~(a).

We now focus on proving \eqref{pcla}. Let $\omega$ be the first time in $[-4L-M_0,-4L]$ where $\mathfrak{H}_1^t(\omega)$ is larger than $-M_0-(4L)^2/2$. Let $\tau$ be the last time on $[-L,0]$ where $\mathfrak{H}_1^t(y)$ is larger than $R$.  Note that $[\omega,\tau]$ is a stopping domain (see Definition \ref{def:sd}).  On $\m{A}\cap \m{B}$ we have $[(\tau+\omega)/2,\tau] \supset [-2L,-L]$, so
\begin{equation}
\label{tfrd}
    \begin{aligned}
    \P(\m{C}) \ge \P(\m{A} \cap \m{B} \cap \m{C}) &\ge \P\bigg(\m{A}\cap \m{B} \cap \left\{\inf_{y\in [(\tau+\omega)/2,\tau]} \mathfrak{H}_1^{t;0}(y) \ge R/8\right\} \bigg) \\ & =\E\left[\ind_{\m{A}\cap \m{B}} \cdot \E\left[\ind_{\inf_{y\in [(\tau+\omega)/2,\tau]} \mathfrak{H}_1^{t;0}(y) \ge R/8} \mid \mathcal{F}_{\m{ext}}(\{1\}\times (\omega,\tau)) \right]\right].
\end{aligned}
\end{equation}
Applying the strong Gibbs property Lemma \ref{lem:sg2} and stochastic monotonicity, we see that
\begin{align}
    \label{trf}
    \ind_{\m{A}\cap\m{B}}\E\left[\ind_{\inf_{y\in [(\tau+\omega)/2,\tau]} \mathfrak{H}_1^t(y) \ge R/8} \mid \mathcal{F}_{\m{ext}}(\{1\}\times (\omega,\tau)) \right] \ge \ind_{\m{A}\cap\m{B}} \P\bigg(\inf_{y\in [(\tau+\omega)/2,\tau]} B(y) \ge R/8\bigg),
\end{align}
where $B$ is a Brownian bridge on $[\omega,\tau]$ from $-M_0-(4L)^2/2$ to $R$. As $R \ge 2[M_0+(4L)^2/2]$, a straightforward computation shows that $\inf_{y\in [(\tau+\omega)/2,\tau]} \E[B(y)] \ge R/4$, and by Lemma 2.11 in \cite{kpzle} (using $\tau-\omega \le M_0+4L$) we have 
\begin{align*}
    \P\bigg(\inf_{y\in [(\tau+\omega)/2,\tau]} \big(B(y)-\E[B(y)]\big) \ge -R/8\bigg) \ge 1-\exp\big(-2(R/8)^2/(M_0+4L)\big) \ge \frac12. 
\end{align*}
Thus, the probability on the right hand side of \eqref{trf} is bounded below by $\frac12$. Plugging this bound into \eqref{tfrd} implies $\P(\m{C}) \ge \frac12\P(\m{A}\cap \m{B})$, verifying \eqref{pcla}.

\medskip\noindent\textbf{Proof of part (b).}
Fix $L>0$, and for $D,K>0$ define the events
\begin{align*}
    &\mathsf{A} := \left\{\sup_{y\in[-L-3,-L-2]} \mathfrak{H}_1^{t;\beta}(y) \ge - D\right\}, \quad \mathsf{B} := \left\{\sup_{y\in[-L,L]}\mathfrak{H}_1^{t;\beta}(y) \ge R\right\},\\
    &\mathsf{C} := \left\{\inf_{y\in[-L-1-t^{-1/4}, -L-1]} \mathfrak{H}_1^{t;\beta}(y) \ge K\right\}.
\end{align*}
By Corollary~\ref{cor:b3}, we may choose $D,K$ large enough so that for all large $t$,
\begin{equation}\label{eq:pace}
 \P\left(\mathsf{A}\right)\ge1-\e,
 \qquad
 \P\left(\mathsf{C}\right)\le\e.
\end{equation}
On the event $\m A\cap\m B$, let $\sigma := \inf\{y\in[-L-3,-L-2] : \mathfrak{H}_1^{t;\beta}(y) \ge - D\}$ and $\tau := \sup\{y\in[-L,L] : \mathfrak{H}_1^{t;\beta}(y) \ge R\}$. Let $B$ denote a Brownian bridge on $[\sigma,\tau]$ from $-D$ to $R$. Then it follows from the strong two-sided Gibbs property for the stopping domain $[\sigma,\tau]$ and stochastic monotonicity that
\begin{equation}\label{eq:pabck}
    \P(\m C) \ge \P(\m A \cap \m B \cap \m C) \ge \E\left[\mathbf{1}_{\m A \cap \m B} \cdot \P\left(\inf_{y\in[-L-1-t^{-1/4},-L-1]} B(y) \ge K\right)\right].
\end{equation}
For large $t$, every $y\in[-L-1-t^{-1/4},-L-1]$ satisfies
\begin{equation*}
 \frac{y-\sigma}{\tau-\sigma}\ge\frac1{2(2L+3)},\qquad
 \tau-\sigma\le2L+3.
\end{equation*}
Thus the mean of the Brownian bridge $B$ is at least $R/(4L+6)-D$ on the interval $[-L-1-t^{-1/4},-L-1]$.
We may then choose $R$ large enough depending on $L,D,K$ so that the Brownian bridge $B$ stays uniformly above $K$ on this interval with probability at least $1/2$, uniformly in $\sigma,\tau$. (This follows from the reflection principle, see for instance \cite[Proposition 12.3.3]{dudley}.) 

It now follows from \eqref{eq:pabck} that
$\P(\m C)\ge\frac12\P(\m A\cap\m B)$, and then by \eqref{eq:pace} we obtain
$\P(\m B)\le3\e$. Replacing $\e$ by $\e/3$ proves part (b).
\end{proof}

\subsection{Properties of lower curves}  \label{sec8.3} 
In this section, we investigate the tightness for lower curves of $(\mathfrak{H}_k^{t,\alpha;\beta})_{k\ge 1}$. As before, we work under the following two assumptions on the parameters:
\begin{itemize}
    \item[a)] Assume $\alpha>0$ is fixed or $\alpha=\mu t^{-1/3}$ with $\mu\in \R$ fixed. Set $\beta=0$.
    \item[b)] Assume $\alpha\in \R$ and $\beta>\max\{-\alpha,0\}$.
\end{itemize}

We will suppress the dependency on the parameters and simply write $(\mathfrak{H}_k^{t})_{k\ge 1}$ for the line ensemble. 

All estimates below are asymptotic in $t$,
with the parameters of the regime fixed. In case (b) the two-sided Gibbs
property is unchanged by the common affine tilt in \eqref{def:Hrescaled}:
Brownian bridges retain their variance and all curve differences retain
the same interaction $t^{2/3}e^{t^{1/3}(B_{i+1}-B_i)}$.
We will also use the reflected ensemble $h_i^t(x)=\mathfrak H_i^{t;\beta}(-x)$
in case (b). Its two-sided Gibbs property is identical, and
Corollary~\ref{cor:b3} provides its weak parabolic bounds on shrinking
intervals $[x,x+t^{-1/4}]$. No one-sided resampling
is used for either bulk ensemble.

\medskip

We have the following three results:

\begin{proposition}\label{p3} Fix any $\e\in (0,1)$, $k\in\mathbb{N}$, and $x_0>0$. There exists $R_0(k,x_0,\e)>x_0$ such that for all $R\ge R_0$ and large $t$ we have
    \begin{align*}
        \P\bigg(\sup_{y\in [-R,-x_0]} \mathfrak{H}_k^t(y) \ge -R^2\bigg) \ge 1-\e. 
    \end{align*}
\end{proposition}

\begin{proposition} \label{0ult} Fix any $\e\in (0,1)$ and $k\in \mathbb{N}$. There exists $Q_0(k,\e)>0$ such that the following holds. For all $Q_1\ge Q_0$ and $Q_2>0$, there exists $M_k(Q_1,Q_2,\e)>0$ such that for all large enough $t$ we have
\begin{align} \P\bigg(\inf_{y\in [-Q_1-Q_2,-Q_1]} \mathfrak{H}_k^t(y) \ge -M_k\bigg) \ge 1-\e.\label{inf1}
    \end{align}
\end{proposition}
\begin{proposition} \label{0ultp} Fix any $\e\in (0,1)$, $k\in \mathbb{N}$, and $Q_2>0$. There exists $\hat Q_0(k,Q_2,\e)>0$ such that the following holds. For all $Q_1\ge \hat{Q}_0$, there exists $\hat M_k(Q_1,Q_2,\e)>0$ such that for all large enough $t$ we have
    {\begin{align} \P\bigg(\sup_{y\in [-Q_1-Q_2,-Q_1]} \mathfrak{H}_k^t(y) \le \hat M_k\bigg) \ge 1-\e. \label{suo1}
    \end{align}}
\end{proposition}

We prove the above three propositions via induction on $k$. Defining $\mathfrak{H}_0^t\equiv \infty$, the $k=0$ case is trivial for Proposition \ref{0ult}. The $k=1$ case for Propositions \ref{p3} and \ref{0ultp} is covered in Lemmas \ref{p1}, Corrollary \ref{cor:b3}, and Proposition \ref{1ult}. We perform the induction scheme in the ensuing subsections.

\subsubsection{Propositions \ref{p3}$(k)$ and \ref{0ult}$(k-1)$ imply Proposition \ref{0ult}$(k)$}  Fix any $Q_2>0$. Assume that \eqref{inf1} holds for $\mathfrak{H}_{k-1}^t$ in place of $\mathfrak{H}_k^t$, with parameters $Q_0(k-1,\cdot)$, $M_{k-1}(\cdot,Q_2,\cdot)$.  Also recall the function $R_0(k-1,\cdot,\cdot)$ from Proposition \ref{p3}. Set
\begin{align*}
   & \hat{Q}:=Q_0(k-1,\e/4), \quad
    Q_1 \geq Q_0(k,\e) := R_0(k,\hat{Q},\e/4) \\
   & R:=R_0(k,Q_1+Q_2,\e/4), \quad
    M_{k-1}:=M_{k-1}(\hat{Q},R,\e/4).
\end{align*}
Define the events
\begin{align*}
    \m{A} & :=\left\{\sup_{y\in [-Q_1,-\hat{Q}]} \mathfrak{H}_{k}^t(y) \ge -Q_1^2\right\}, \quad \m{B}:=\left\{\sup_{y\in [-R,-Q_1-Q_2]} \mathfrak{H}_{k}^t(y) \ge -R^2\right\}, \\
    \m{C} & :=\left\{\inf_{y\in [-\hat Q-R,-\hat{Q}]} \mathfrak{H}_{k-1}^t(y) \ge -M_{k-1}\right\}
\end{align*}
Let $\omega$ be the first time in $[-R,-Q_1-Q_2]$ such that $\mathfrak{H}_k^t(\omega) \ge -R^2$ and $\tau$ be the last time in $[-Q_1,-\hat Q]$ such that $\mathfrak{H}_k^t(\tau) \ge -Q_1^2$. Then $[\omega,\tau]$ is a stopping domain (see Definition \ref{def:sd}). Let us take $\Upsilon \ge Q_1^2+R^2$. By the strong Gibbs property and stochastic monotonicity,
\begin{align*}
  &  \ind_{\m{A}\cap \m{B} \cap \m{C}} \cdot \E\left[\ind\bigg\{\inf_{y\in [-Q_1-Q_2,-Q_1]} \mathfrak{H}_k^t(y) \le -M_k\bigg\} \, \bigg| \, \mathcal{F}_{\m{ext}}(\{k\}\times (\omega,\tau))\right] \\ & \le \ind_{\m{A}\cap \m{B}\cap \m{C}} \cdot \frac{\P \left(\inf_{y\in [\omega,\tau]} B(y) \le -M_k\right)}{\E \left[\exp(-t^{2/3}\int_{\omega}^{\tau} e^{t^{1/3}(B(y)+M_{k-1})}dy)\right]}
\end{align*}
where $B$ is a Brownian bridge on $[\omega,\tau]$ from $-\Upsilon$ to $-\Upsilon$. On $\m{A}\cap\m{B}$, $\tau-\omega\le R$. Thus by Brownian motion estimates one can first choose $\Upsilon$ large enough (depending on $M_{k-1}, R$) so that the expectation in the denominator above is at least $\frac12$. Then one can choose $M_k$ (depending on $\Upsilon$, $R$ and $\e$) such that the probability in the numerator above is less than $\e/4$. Thus taking expectations we see that
\begin{align*}
    \P\bigg(\m{A}\cap \m{B} \cap\m{C}\cap\left\{\inf_{y\in [-Q_1-Q_2,-Q_1]} \mathfrak{H}_k^t(y) \le -M_k\right\}\bigg) \le \e/4.
\end{align*}
By Proposition \ref{p3}($k$), $\P(\m{A}\cap \m{B}) \ge 1-\e/2$. By our choices of parameters, we have, by the induction hypothesis, that $\P(\m{C})\ge 1-\e/4$. Thus \eqref{inf1} holds for index $k$ with above the choices of parameters. 

\subsubsection{Propositions \ref{0ult}$(k,k-1,k-2)$ and \ref{0ultp}$(k-1)$ imply Proposition \ref{0ultp}$(k)$} The proof is very similar to the induction step of Lemma \ref{uppertight}. We include the details for completeness. Let $K_k, D_k, \hat{M}_k>0$ be {constants depending on $k,Q_1,Q_2,\e$} that will be chosen appropriately later. 
{Fix $Q_2>0$ and put $\e'=\e/[6(\lceil2Q_2\rceil+1)]$. Choose
\[
Q_1\geq\max\{Q_0(k-2,\e'),Q_0(k-1,\e'),Q_0(k,\e'),\hat Q_0(k-1,Q_2+2,\e')\},
\]
where the zero-index condition is omitted. All intermediate error bounds below are taken with $\e'$ in place of $\e$ when choosing the constants.} Let $\bar{x}\in [-Q_1-Q_2,-Q_1]$.  Let $\chi = \sup\{x\in[\bar{x}-\frac12,\bar{x}] : \mathfrak{H}_k^t(x) \ge \hat{M}_k\}$, or $\chi=\overline{x}-\frac12$ if this set is empty. Consider the following events:
    \begin{align*}
        \m{E} & : = \left\{\sup_{x\in [\bar{x}-1/2,\bar{x}]} \mathfrak{H}_k^t(x)  \ge \hat{M}_k \right\}, \quad
        \m{Q}  := \left\{\inf_{x\in [\bar{x}-2,\bar{x}]} \mathfrak{H}_{k-2}^t(x) > -K_{k} \right\}, \\ \m{B} & := \left\{\sup_{x\in [\bar{x}-2,\chi]} \mathfrak{H}_{k-1}^t(x)  \ge \hat{M}_{k-1} \right\}, \quad
        \m{A}  := \left\{\mathfrak{H}_{k-1}^t(\chi) \ge -D_{k}\right\} \cap \bigcap_{j=k-1}^k  \left\{\mathfrak{H}_{j}^t(\bar{x}-2) \ge -D_k\right\}.
    \end{align*}
Here $\hat{M}_{k-1}$ comes from Proposition \ref{0ultp}$(k-1)$; $K_k, D_k,$ and $\hat{M}_k$ are yet to be specified. Note that {$[\bar{x}-2,\chi]$ forms a stopping domain}. By the tower property of conditional expectation we have
    \begin{align}\label{r31}
        & \P\big(\m{E}\cap \m{Q} \cap \m{A} \cap \m{B} \big) = \E\left[\ind_{\m{E}\cap \m{Q} \cap \m{A}}\E\left[\ind_{\m{B}}|\mathcal{F}_{\mathrm{ext}}(\{k-1,k\}\times(\bar{x}-2,\chi))\right]\right].
    \end{align}
    Note that by the strong Gibbs property (Lemma \ref{lem:sg2}) we have
    \begin{align*} 
\E\left[\ind_{\m{B}}|\mathcal{F}_{\mathrm{ext}}(\{k-1,k\}\times(\bar{x}-2,\chi))\right] = \P_{t^{2/3};\mathfrak{H}_{k-2}^t,\mathfrak{H}_{k+1}^t}^{k-1,k;\bar{x}-2,\chi;\vec{a},\vec{b}}(\m{B}),
    \end{align*}
    where $\vec{a}=\big(\mathfrak{H}_{k-1}^{t}(\bar{x}-2),\mathfrak{H}_{k}^{t}(\bar{x}-2)\big)$ and $\vec{b}=\big(\mathfrak{H}_{k-1}^{t}(\chi),\mathfrak{H}_{k}^{t}(\chi)\big)$.
    On the event $\m{E}\cap \m{Q} \cap \m{A}$, we have
    \begin{align*}
       & \mathfrak{H}_{k-1}^t(\bar{x}-2) \ge -D_k, \quad \mathfrak{H}_{k}^t(\bar{x}-2) \ge -D_k , \quad
        \mathfrak{H}_{k-1}^t(\chi) \ge -D_k , \quad \mathfrak{H}_{k}^t(\chi) \ge \hat{M}_k , \\
      &  \mathfrak{H}_{k-2}^t(x) \ge -K_k \mbox{ for } x\in [\bar{x}-2,\chi].
    \end{align*}
 Thus by stochastic monotonicity (Lemma \ref{lem:sm}) we have
    \begin{align}\label{r10}
        \ind_{\m{E}\cap \m{Q} \cap \m{A}} \cdot \P_{t^{2/3};\mathfrak{H}_{k-2}^t,\mathfrak{H}_{k+1}^t}^{k-1,k;\bar{x}-2,\chi;\vec{a},\vec{b}}(\m{B}) \ge \ind_{\m{E}\cap \m{Q} \cap \m{A}} \cdot \P_{t^{2/3};-K_k,-\infty}^{k-1,k;\bar{x}-2,\chi;\vec{u},\vec{v}}(\m{B}),
    \end{align}
    where $\vec{u}:=(-D_k,-D_k)$ and $\vec{v}:=(-D_k, \hat{M}_k\big)$. We may lower bound the above probability using the following lemma proven in \cite{kpzle}.
    
    \begin{lemma}\label{pchi1}
       There exist $\delta, D^0>0$, and functions $K^0(D)>0$, $\hat{M}^0(D,K)>0$ such that for all $D\ge D^0$, $K\ge K^0(D)$, $\hat{M} \ge \hat{M}^0(D,K)$ and all $\bar{x}\le 0$, $\chi\in [\bar{x}-1/2,\bar{x}]$, $L\ge 1$ we have  
       \begin{align*}
        \P_{L;-K,-\infty}^{k-1,k;\bar{x}-2,\chi;(-D,-D),(-D,\hat{M})}\bigg({\sup_{x\in [\bar{x}-2,\chi]} B_{k-1}(x)} \ge \delta \hat{M} \bigg) \ge \frac14.
       \end{align*}
    \end{lemma}

The above lemma is a variant of Proposition 7.6 in \cite{kpzle} where the parabola has been removed. {This non-parabola variant is also proved in \cite[Section 7.4.1]{kpzle}. To account for our extra prefactor $L$, put $c_L=(\log L)/\sqrt L\leq2/e$ and replace the two paths by $C_1=B_{k-1}$, $C_2=B_k+c_L$. The interaction then has the normalization of that reference, with ceiling $-K-c_L$. Lowering this ceiling to $-K-2/e$ and the shifted lower endpoints to their unshifted values gives the quoted bound, after adjusting $K^0$ and $\hat M^0$.} Let us now use the parameters in the above lemma to specify our constants. 

 \medskip

  We assume $K_k \ge M_{k-2}(\e)$ and $D_k \ge M_{k-1}(\e)\vee M_k(\e)$, where $M_{k-2}, M_{k-1},M_k$ come from Proposition \ref{0ult}. Then by the induction hypothesis we have $\P(\neg\m{Q} \cup \neg\m{A}) \le 2\e$ and $\P(\m{B}) \le \e$ for all large enough $t$. We further assume our constants satisfy $D_k \ge D^0$, $K_k \ge K^0(D_k)$, {$\hat{M}_k \ge \hat{M}^0(D_k,K_k)$}, and $\hat{M}_k \ge 2\hat{M}_{k-1}/\delta$, where $\delta,D^0,K^0, \hat{M}^0$ come from Lemma \ref{pchi1}. Then by Lemma \ref{pchi1}, we have that r.h.s.~of \eqref{r10} $\ge \frac14 \ind_{\m{E}\cap \m{Q} \cap \m{A}}$. Plugging this bound back into \eqref{r31}, we get 
 \begin{align*}
     \P(\m{B}) \ge \P\big(\m{E}\cap \m{Q} \cap \m{A} \cap \m{B} \big) \ge \frac14\P\big(\m{E}\cap \m{Q} \cap \m{A} \big) \ge \frac14\P(\m{E})- \frac14\P(\neg\m{Q} \cup \neg\m{A}).
 \end{align*}
 {Owing to our choices of constants, the preceding bound with $\e'$ gives $\P(\m E)\leq6\e'=\e/(\lceil2Q_2\rceil+1)$ for all large $t$, uniformly over $\bar x\in[-Q_1-Q_2,-Q_1]$. Covering this interval by at most $\lceil2Q_2\rceil+1$ intervals of length $1/2$ and taking a union bound proves \eqref{suo1}.

} 

 \subsubsection{Propositions \ref{0ult}$(1,\dots,k)$ and \ref{0ultp}$(1,\dots,k)$ imply Proposition \ref{p3}$(k+1)$}

If the lower curves are too low, then due to the Gibbs property the top curve behaves like a Brownian motion and hence should not have any parabolic trajectory which is a contradiction to Lemma \ref{p1}. 
The proof is similar to that of Theorem 3.3 in \cite{half1} and Proposition 7.4 in \cite{ds24}, where an analogous statement has been proven for the HSLG line ensemble using the above idea. 

{Fix $\e,k,x_0$ as in the statement of the proposition. Let $M_0=M_0(\e/18)$ be as in Lemma \ref{p1}, enlarged to control both bounds there. First choose $Q_2=Q_2(k,\e)$ large enough for the Brownian bridge estimate below. Then choose
\[
Q_1\geq x_0\vee\max_{1\leq i\leq k}\{Q_0(i,\e/(6k)),\hat Q_0(i,Q_2,\e/(6k))\}.
\]
Set $R=Q_1+Q_2$, $M_i=M_i(Q_1,Q_2,\e/(6k))$, $\hat M_i=\hat M_i(Q_1,Q_2,\e/(6k))$, and $\hat M=\max_{1\leq i\leq k}(M_i\vee\hat M_i)$.

} Let {$A=-Q_1-Q_2/2-Q_2/(8k+8)$ and $B=-Q_1-Q_2/2+Q_2/(8k+8)$}. 
 Let $C =-Q_1-Q_2/2= (A+B)/2$ be their midpoint. Assume $Q_2$ is large enough depending on $\e$ so that $[C-M_0,C] \subset [A,B]$. We consider the following three events:
\begin{align*}
    \mathsf{A} &:= \left\{\sup_{y\in[-R,-Q_1]} \mathfrak{H}^t_{k+1}(y) < -R^2 \right\}, \quad \m{B} := \left\{\max_{x\in\{A,B\}} \inf_{y\in[x-2t^{-2/3},x]}\mathfrak{H}_1^t(y) + \frac{x^2}{2} \leq M_0 \right\} ,\\
    \m{C} &:= \left\{\sup_{y\in[C-M_0,C]} \mathfrak{H}_1^t(y) + \frac{C^2}{2} \geq -M_0 \right\}.
\end{align*}
Since $Q_1 \geq x_0$, in order to show \eqref{p3} for index $k+1$, it suffices to show that $\P(\m{A}) < \e$ for large $t$. The idea is essentially to argue that although $\m{B}$ and $\m{C}$ are likely, $\m{C}$ is unlikely given $\m{A}\cap\m{B}$, which implies $\m{A}$ is unlikely.

Define $\sigma := \inf\{y\in[A-2t^{-2/3},A] : \mathfrak{H}_1^t(y) + A^2/2 \leq M_0\}$ or $\sigma:=A$ if this set is empty, and $\tau := \sup\{y\in[B-2t^{-2/3},B] : \mathfrak{H}_1^t(y) + B^2/2 \leq M_0\}$ or $\tau := B-2t^{-2/3}$ if this set is empty. We introduce the two auxiliary events
\[
\m{D} := \left\{{\min_{1\leq i\leq k-1}} \min_{x\in\{\sigma,\tau\}} \left(\mathfrak{H}_i^t(x)-\mathfrak{H}_{i+1}^t(x)\right) > 0 \right\}, \quad \m{E} := \left\{\max_{1\leq i\leq k} |\mathfrak{H}_i^t(-R)| \vee |\mathfrak{H}_i^t(-Q_1)| \leq \hat{M}\right\}. 
\]
By Lemma \ref{p1} and Propositions \ref{0ult}$(1,\dots,k)$ and \ref{0ultp}$(1,\dots,k)$, and the choice of parameters above, for large $t$ we have 
\begin{equation}\label{PE}
\P(\neg(\m{B}\cap\m{C})) < \e/6, \qquad \P(\neg\m{E}) <\e/3.
\end{equation}
By the Gibbs property,
\begin{align*}
    \P(\m{A}\cap\m{D}\cap\m{E}) &= \E\left[\ind_{\m{A}\cap\m{E}} \cdot \E\left[\ind_{\m{D}} \mid \mathcal{F}_{\m{ext}}(\ll 1,k\rr \times (-R,-Q_1))\right]\right].
\end{align*}
{For $k=1$, take $\m D$ to be the whole space. For $k\geq2$, the inset choice of $A,B$ ensures that both stopping windows lie inside the central interval of Lemma \ref{ord} for all large $t$. By \eqref{sepbd2}, on $\m A\cap\m E$} the inner conditional expectation above is at least $1/2$ for large $t$. Thus $\P(\m{A}\cap \m{D} \cap\m{E}) \geq \frac12\P(\m{A}\cap\m{E})$. Combining with the second inequality in \eqref{PE}, we have
\begin{equation}\label{2ADe}
    \P(\m{A}) \leq 2\P(\m{A}\cap\m{D}\cap\m{E}) + \e/3 \leq 2\P(\m{A}\cap\m{D}) + \e/3.
\end{equation}

It remains to upper bound $\P(\m{A}\cap\m{D})$. By the first inequality in \eqref{PE},
\begin{equation}\label{ABCDe6}
    \P(\m{A}\cap\m{D}) \leq \P(\m{A}\cap\m{B}\cap\m{C}\cap\m{D}) + \e/6.
\end{equation}
Note that $[\sigma,\tau]$ is a stopping domain, and by the tower property of conditional expectation we may write
\begin{align}\label{ABCD}
    \P(\m{A}\cap\m{B}\cap\m{C}\cap\m{D}) &\le \E\left[\ind_{\m{A}\cap\m{B}\cap\m{D}} \cdot \E \left[\ind_{\m{C}} \mid \mathcal{F}_{\m{ext}}(\ll 1,k\rr \times(\sigma,\tau))\right]\right].
\end{align}
On the event $\m{B}\cap\m{D}$, we have $M_0-A^2/2 \geq \mathfrak{H}_1^t(\sigma) > \cdots > \mathfrak{H}_k^t(\sigma)$ and $M_0-B^2/2 \geq \mathfrak{H}_1^t(\tau) > \cdots > \mathfrak{H}_k^t(\tau)$. Let us introduce the raised boundary conditions
\[
x_i := M_0-\frac{A^2}{2} + (k-i+1)\sqrt{B-A}, \quad y_i := M_0-\frac{B^2}{2} + (k-i+1)\sqrt{B-A}, \quad 1\leq i\leq k.
\]
Then $x_i \geq \mathfrak{H}_i^t(\sigma)$ and $y_i\geq \mathfrak{H}_i^t(\tau)$ for $1\leq i\leq k$. Moreover, on $\m{A}$ we have $\sup_{y\in[-R,-Q_1]} \mathfrak{H}_{k+1}^t(y) < -R^2 < M_0-A^2/2$. Therefore by the strong Gibbs property and stochastic monotonicity,
\begin{align}\label{ABD}
    \ind_{\m{A}\cap\m{B}\cap\m{D}} \cdot \E \left[\ind_{\m{C}} \mid \mathcal{F}_{\m{ext}}(\ll 1,k\rr \times(\sigma,\tau))\right] &\leq \P^{k;\sigma,\tau;\vec{x},\vec{y}}_{t^{2/3};M_0-\sigma^2/2}\left(\sup_{y\in[C-M_0,C]} B_1(y) + \frac{C^2}{2} \geq -M_0 \right).
\end{align}
(Recall the notation for the Gibbs measures in Definition \ref{kpzgibbs2}.) As $t\to\infty$, $\sigma\to A$ and $\tau\to B$, and the Gibbs measure on the right-hand side of \eqref{ABD} converges to the law of $k$ non-intersecting Brownian bridges on $[A,B]$ with boundary conditions $\vec{x},\vec{y}$ and floor $M_0-A^2/2$. Since $x_i-x_{i+1} = y_i-y_{i+1} = \sqrt{B-A}$ for $1\le i\le k-1$, and $\min(x_k-(M_0-A^2/2), y_k-(M_0-A^2/2)) \geq \sqrt{B-A}$, the probability of independent Brownian bridges with boundary data $\vec{x},\vec{y}$ not intersecting each other or the floor at $M_0-A^2/2$ is bounded below by a positive constant $c_0 = c_0(k)$. Hence for large $t$, 
\begin{equation}\label{Wparab}
\mbox{r.h.s.~of \eqref{ABD}} \leq \frac{2}{c_0} \cdot \P\left(\sup_{y\in[C-M_0,C]} \Br(y) + \frac{C^2}{2} \geq -M_0\right),
\end{equation}
where $\Br$ is a single Brownian bridge on $[A,B]$ from {$x_1$ to $y_1$}. Let $L$ denote the line segment connecting {$(A,x_1)$ and $(B,y_1)$}. By concavity of the parabola $-y^2/2$, $L(y) < -C^2/2$ for $y\in[C-M_0,C]$, and $L(y)+C^2/2$ is maximized at $y=C$. Recalling that $C = (A+B)/2$, a simple computation shows that 
\begin{equation}\label{L(C)}
L(C)+ \frac{C^2}{2} = -\frac18(B-A)^2 + M_0 + k\sqrt{B-A}.
\end{equation}
Now recall that {$B-A=Q_2/(4k+4)$}. Fix any $K>0$. By choosing $Q_2$ sufficiently large depending on $M_0$ (which only depends on $\e$), $k$, and $K$, we may ensure from \eqref{L(C)} that $L(y) + C^2/2 \leq -M_0-K\sqrt{B-A}$ for all $y\in[C-M_0,C]$. It follows that
\begin{equation}
\begin{split}\label{supW}
\mbox{r.h.s.~of \eqref{Wparab}} &\leq \frac{2}{c_0} \cdot \P\left(\sup_{y\in[C-M_0,C]} \Br(y) - L(y) \geq K\sqrt{B-A}\right)\\
&\leq \frac{2}{c_0} \cdot \P\left(\sup_{y\in[A,B]} \Br(y)-L(y) \geq K\sqrt{B-A} \right).
\end{split}
\end{equation}
By affine invariance and diffusive scaling of Brownian bridges, {the last expression is $2/c_0$ times the probability that a standard Brownian bridge on $[0,1]$ has maximum at least $K$}, which goes to 0 as $K\to\infty$. We may therefore take $K$ large enough depending on $\e$ and $c_0(k)$ so that the second line of \eqref{supW} is at most $\e/6$. For the corresponding choice of $Q_2$ in \eqref{supW} (again, depending only on $k,\e$), we have shown that the limsup as $t\to\infty$ of the l.h.s.~of \eqref{ABD} is a.s.~at most $\e/6$. Applying Fatou's lemma to the r.h.s.~of \eqref{ABCD}, we get that $\limsup_{t\to\infty} \P(\m{A}\cap\m{B}\cap\m{C}\cap\m{D}) \leq \e/6$. In view of \eqref{2ADe} and \eqref{ABCDe6}, we conclude that $\limsup_{t\to\infty} \P(\m{A}) \leq \e$, as desired. 

{In regime (b), use Corollary~\ref{cor:b3} in place of Lemma~\ref{p1}: replace the endpoint windows $[x-2t^{-2/3},x]$ by $[x-t^{-1/4},x]$ and the central high-point window by $[C-t^{-1/4},C]$. Choose $M_0$ large enough for that corollary. These windows lie in the same inset intervals for large $t$, and the stopping points still tend to $A,B$. For the reflected ensemble use right-pointing windows and first/last hits in the same order. The central window then lies in $[C,C+t^{-1/4}]$; its chord differs from $L(C)$ by $o(1)$ for the fixed $Q_1,Q_2$. Thus the same negative quadratic term in \eqref{L(C)} proves the claim in both bulk cases.

}
\section{Tightness and subsequential limits of the HSKPZ line ensemble} \label{sec9}

In this section we prove our main results about the HSKPZ line ensemble under diffusive scaling, Theorems \ref{kpz123} and \ref{subseq}. 

\subsection{Proof of tightness} 

We begin with pointwise tightness. We first give a result for $\beta=0$ that extends the lower tightness result in Proposition \ref{0ult} to all intervals, and extends the upper tightness in Proposition \ref{1ult} to all curves.
\begin{proposition}\label{p10} Fix any $\e\in (0,1)$, $k\in \mathbb{N}$, and $M>0$. We assume either $\alpha>0$ is fixed (supercritical) or $\alpha=\mu t^{-1/3}$ with $\mu\in\R$ fixed (critical). There exists $K(k,M,\e)>0$ such that for all large $t$,
\begin{align}\label{e10} 
        \P\bigg(\sup_{y\in [-M,0]} \mathfrak{H}_k^{t}(y) \le K\bigg) \ge 1-\e, \quad  \P\bigg(\inf_{y\in [-M,0]} \mathfrak{H}_{k}^{t}(y) \ge -K\bigg) \ge 1-\e.
    \end{align}
\end{proposition}

\begin{proof} Let us take $m=\lceil k/2 \rceil$. Thanks to Proposition \ref{0ult}, we can find $M'(m,\e) \ge M$ and $K'(M,m,\e)$ such that
\begin{align}\label{aepr}
\P(\m{A}) \ge 1-\e/2, \mbox{ where }    \m{A}:=\left\{\min_{1\le j\le 2m}\mathfrak{H}_j^{t}(-2M') \ge -K'\right\}. 
\end{align}
By the Gibbs property (Lemma \ref{lem:gibbs2}) and stochastic monotonicity (Lemma \ref{lem:sm})
\begin{align*}
   & \ind_{\m{A}}\cdot \E\left[\ind\bigg\{\inf_{ 1\le j \le 2m,\, y\in [-M',0]} \mathfrak{H}_{j}^{t}(y) \le -K\bigg\} \, \bigg|\, \mathcal{F}_{\m{ext}}(\ll 1,2m\rr \times (-2M',0])\right] \\ & \le \mathbf{P}_{t^{2/3}}^{2m;-2M';\vec{a},\star}\left(\inf_{ 1\le j \le 2m, \, y\in [-M',0]} B_j(y) \le -K\right),
\end{align*}
where $\vec{a}=(-K'-1,-K'-2,\ldots,-K'-2m) \in \mathbb{R}^{2m}$. By Theorems \ref{6.11} and \ref{6.11c}, the comparison laws converge weakly, so we may choose $K$ large enough (depending on $\e,M',K'$) such that the above probability is less than $\e/2$ for all large enough $t$. Thus taking expectation we get that
\begin{align*}
    \P\left(\m{A} \cap\left\{\inf_{ 1\le j \le 2m,\, y\in [-M',0]} \mathfrak{H}_{j}^{t}(y) \le -K\right\}\right) \ge 1-\e/2.
\end{align*}
Combining this with \eqref{aepr} we obtain
\begin{align*}
    \P\left(\inf_{ 1\le j \le 2m,\, y\in [-M',0]} \mathfrak{H}_{j}^{t}(y) \le -K\right) \ge 1-\e.
\end{align*}
This verifies the second inequality in \eqref{e10}.

For the first inequality in \eqref{e10}, we proceed via induction. The $k=1$ case is covered in Proposition \ref{1ult}. Assume that the supremum tightness holds for $\mathfrak{H}_{k-1}^{t}$. Then, using an argument identical to the induction step of proving Proposition \ref{0ultp}, the inequality for $\sup \mathfrak{H}_k^{t}$ follows. We omit the details for brevity.
\end{proof}

We now give the analogous result for the general slope $\beta$ case.

\begin{proposition}\label{p101} Fix any $\e\in (0,1)$, $k\in \mathbb{N}$, and $M>0$. Also fix $\alpha\in\R$ and $\beta > \max\{-\alpha,0\}$. There exists $K(k,M,\e)>0$ such that for all large $t$,
\begin{align}\label{e10} 
        \P\bigg(\sup_{y\in [-M,M]} \mathfrak{H}_k^{t;\beta}(y) \le K\bigg) \ge 1-\e, \quad  \P\bigg(\inf_{y\in [-M,M]} \mathfrak{H}_{k}^{t;\beta}(y) \ge -K\bigg) \ge 1-\e.
    \end{align}
\end{proposition}

\begin{proof} We omit the $\beta$ superscript for brevity. Thanks to Proposition \ref{0ult}, we can find $M'(k,\e) \ge M$ and $K'(M,k,\e)$ such that
\begin{align}\label{aepr}
\P(\m{A}) \ge 1-\e/2, \mbox{ where }    \m{A}:=\left\{\min_{1\le j\le k}\mathfrak{H}_j^{t}(-(2k+2)M') \ge -K', \min_{1\le j\le k}\mathfrak{H}_j^{t}((2k+2)M') \ge -K'\right\}. 
\end{align}
By the Gibbs property (Lemma \ref{lem:gibbs2}) and stochastic monotonicity (Lemma \ref{lem:sm})
\begin{align*}
   & \ind_{\m{A}}\cdot \E\left[\ind\bigg\{\inf_{ 1\le j \le k,\, y\in [-M',M']} \mathfrak{H}_{j}^{t}(y) \le -K\bigg\} \, \bigg|\, \mathcal{F}_{\m{ext}}(\ll 1,k\rr \times (-(2k+2)M',(2k+2)M'))\right] \\ & \le \mathbf{P}_{t^{2/3}}^{2m;-(2k+2)M',(2k+2)M';\vec{a},\vec{a}}\left(\inf_{ 1\le j \le k, \, y\in [-M',M']} B_j(y) \le -K\right),
\end{align*}
where $\vec{a}=(-K'-1,-K'-2,\ldots,-K'-k) \in \mathbb{R}^{k}$. By Theorem \ref{6.11d}, the comparison laws converge weakly, so we may choose $K$ large enough (depending on $\e,M',K',k$) such that the above probability is less than $\e/2$ for all large enough $t$. Thus taking expectation we get that
\begin{align*}
    \P\left(\m{A} \cap\left\{\inf_{ 1\le j \le k,\, y\in [-M',M']} \mathfrak{H}_{k}^{t}(y) \le -K\right\}\right) \ge 1-\e/2.
\end{align*}
Combining this with \eqref{aepr} we obtain
\begin{align*}
    \P\left(\inf_{ 1\le j \le k,\, y\in [-M',M']} \mathfrak{H}_{k}^{t}(y) \le -K\right) \ge 1-\e.
\end{align*}
This verifies the second inequality in \eqref{e10}.

For the first inequality in \eqref{e10}, we proceed via induction. The $k=1$ case is covered in Proposition \ref{1ult}. Assume that the supremum tightness holds for $\mathfrak{H}_{k-1}^{t}$. Then, using an argument identical to the induction step of proving Proposition \ref{0ultp}, the inequality for $\sup \mathfrak{H}_k^{t}$ follows. We omit the details for brevity.
\end{proof}

We are now equipped to prove tightness. The following argument applies both to the case of $\beta=0$ (Theorem \ref{kpz123}) as well as $\beta > \max\{-\alpha,0\}$ (part of Theorem \ref{kpzb123}).

{\begin{proof}[Proof of Theorem \ref{kpz123}]
First take $\beta=0$, with $\alpha>0$ fixed or $\alpha=\mu t^{-1/3}$. Fix $m\in\mathbb N$, $A<0$, and $\gamma>0$, and put
\[
J=[A,0],\qquad I=[A/2,0],\qquad I_0=[A/4,0],\qquad L=t^{2/3}.
\]
Let $\m A=\{\mc(\mathfrak H^t|_{\ll1,2m\rr\times I_0},\delta)\geq\gamma\}$. Since $A$ is arbitrary, pointwise tightness from Proposition \ref{p10} reduces the proof to making $\P(\m A)$ small. Write $f=\mathfrak H^t_{2m+1}$ and define
\[
\m C=\{\max_{i\leq2m}|\mathfrak H_i^t(A)|\leq M,\ \sup_J f\leq M\},\qquad
\m D=\{\max_{i\leq2m}|\mathfrak H_i^t(A/2)|\leq M'\}.
\]
By Proposition \ref{p10}, choose $M,M'$ so that
\begin{equation}\label{events2}
\P(\neg\m C)+\P(\neg\m D)\leq\e
\end{equation}
for large $t$. Given the exterior of $\ll1,2m\rr\times(A/2,0]$, let $Q_I=\mathbf P_L^{2m;A/2;\vec x,\star}$ be the Gibbs law with the floor removed, where $x_i=\mathfrak H_i^t(A/2)$. Set
\[
W_I=\exp\left(-L\int_I e^{\sqrt L(f(x)-B_{2m}(x))}dx\right),\qquad q=E_{Q_I}[W_I].
\]
The Gibbs property gives
\begin{equation}\label{floorsep}
\P(\m A\mid\mathcal F_{\m{ext}}(\ll1,2m\rr\times(A/2,0]))
=\frac{E_{Q_I}[W_I\ind_{\m A}]}{q}.
\end{equation}
We claim that, for small $\eta>0$,
\begin{equation}\label{acceptbot2}
\P(\{q<\eta\}\cap\m C)\leq\e.
\end{equation}
Indeed, condition on the exterior of $\ll1,2m\rr\times(A,0]$ and let $Q_J^\circ$ retain all internal interactions and the floor interaction on $J\setminus I$, removing only the floor on $I$. Under $Q_J^\circ$, the conditional law on $I$ is $Q_I$. The original outer law has density $W_I/E_{Q_J^\circ}[W_I]$ relative to $Q_J^\circ$. Thus size-biasing gives
\begin{equation}\label{iden5}
\ind_{\m C}\P(q<\eta\mid\mathcal F_{\m{ext}}(\ll1,2m\rr\times(A,0]))
=\ind_{\m C}\frac{E_{Q_J^\circ}[q\ind_{q<\eta}]}{E_{Q_J^\circ}[q]}.
\end{equation}
By monotonicity and Lemma \ref{2high}(b), on $\m C$ the probability under $Q_J^\circ$ that $\inf_I B_{2m}\geq M+1$ is at least some $\phi>0$, uniformly for large $L$. On this event $W_I\geq\exp(-L|I|e^{-\sqrt L})\geq1/2$. Consequently,
\begin{equation}\label{rio}
\ind_{\m C}E_{Q_J^\circ}[q]=\ind_{\m C}E_{Q_J^\circ}[W_I]\geq\ind_{\m C}\phi/2.
\end{equation}
Taking expectations in \eqref{iden5} proves \eqref{acceptbot2} when $\eta\leq\e\phi/2$.

The events $\m C,\m D$ and $\{q\geq\eta\}$ are measurable with respect to the inner exterior. Since $W_I\leq1$, \eqref{floorsep} now gives
\[
\P(\m A)\leq2\e+\eta^{-1}\sup_{\max_i|x_i|\leq M'}
\mathbf P_L^{2m;A/2;\vec x,\star}\big(\mc(B|_{I_0},\delta)\geq\gamma\big).
\]
Theorem \ref{moclemma}(b) makes the last term at most $\e$ for sufficiently small $\delta$ and large $L$. This proves tightness for $\beta=0$.

For $\beta>\max\{-\alpha,0\}$ use $k$ paths and the three intervals
$J=[-2T,2T]$, $I=[-T,T]$, and $I_0=[-T/(2k+2),T/(2k+2)]$.
Define $\m C$ using the endpoint values at $\pm2T$ and the floor on $J$, and $\m D$ using the endpoint values at $\pm T$. Proposition \ref{p101}, Lemma \ref{2high}(a), and Theorem \ref{moclemma}(a) give exactly the same estimates. Varying $T$ covers every compact interval, proving the tightness assertion in Theorem \ref{kpzb123}.
\end{proof}}

\subsection{Properties of subsequential limits} In this subsection we prove Theorem \ref{subseq} describing the subsequential limits of the HSKPZ line ensemble near the boundary under KPZ scaling, as well as Theorem \ref{kpzb123} establishing convergence to the Airy line ensemble away from the boundary. The proof follows fairly quickly from results in Section \ref{sec7} on diffusive limits of the one-sided HSKPZ Gibbs measures and Section \ref{sec8} on the weak parabolic trajectory of the top curve.

We first deal with the $\beta=0$ case.

\begin{proof}[Proof of Theorem \ref{subseq}\ref{parta}] 
We will omit the indices $\bullet$, as the proof is the same for the critical and supercritical regimes. Fix any $x<0$, $m\in\mathbb{N}$, and for $\rho>0$ and $t\in[1,\infty]$ define the event
\[
\m{A}_\rho^t := \left\{\min_{1\leq i\leq 2m-1}\mathfrak{H}_i^t(x) - \mathfrak{H}^t_{i+1}(x) > \rho\right\}.
\]
We claim that for any $\e>0$ we can find $\rho>0$ small enough so that 
\begin{equation}\label{Ainf}
\P(\m{A}_\rho^\infty) > 1-\e.
\end{equation}
As
\[
\P(\mathfrak{H}^\infty_1(x) > \cdots > \mathfrak{H}_{2m}^\infty(x)) \geq \P(\m{A}_\rho^\infty)
\]
for any $\rho>0$, \eqref{Ainf} implies $\P(\mathfrak{H}^\infty_1(x) > \cdots > \mathfrak{H}_{2m}^\infty(x))=1$. Taking the countable intersection over $m$ verifies the claim in part \ref{parta}. We thus focus on proving \eqref{Ainf}. For $M,t>0$, define the event
\[
\m{B}_M^t := \left\{\max_{1 \leq i\leq 2m} |\mathfrak{H}^t_i(2x)| \leq M\right\} \cap \left\{\sup_{y\in[2x,0]} |\mathfrak{H}^t_{2m+1}(y)| \leq M\right\}.
\]
By Theorem \ref{kpz123}, we may choose $M>0$ so that 
\begin{equation}\label{10.1}
    \limsup_{t\to\infty} \P(\neg\m{B}_M^t) < \e/2.
\end{equation}
On the other hand, by the Gibbs property (Lemma \ref{lem:gibbs2})
\begin{align*}
    \P(\neg\m{A}_\rho^t \cap \m{B}_M^t) &= \E\left[ \mathbf{1}_{\m{B}_M^t} \E\left[\mathbf{1}_{\neg\m{A}_\rho^t} \mid \mathcal{F}_{\m{ext}}(\ll 1,2m\rr\times(2x,0])\right]\right]\\
    &= \E\left[\mathbf{1}_{\m{B}_M^t} \P_{t^{2/3};g^t}^{2m;2x;\vec{a}^t,\star} \left(\min_{1\leq i\leq 2m-1} B_i(x)-B_{i+1}(x) \leq \rho \right)\right],
\end{align*}
where $\vec{a}^t = (\mathfrak{H}_1^t(2x),\dots,\mathfrak{H}_{2m}^t(2x))$ and $g^t(y) = \mathfrak{H}^t_{2m+1}(y)$ for $y\in[2x,0]$. Lemma \ref{ord} allows us to find $\rho,t_0>0$ so that the probability inside the expectation in the last line is at most $\e/2$ for all $t\geq t_0$, uniformly over the event $\m{B}_M^t$. Combining with \eqref{10.1}, it follows for this choice of $\rho$ and all large $t$ we have $\P(\m{A}_\rho^t) > 1-\e$. Apply this bound with $2\rho$ in place of $\rho$. Portmanteau for the closed event that the minimum gap is at least $2\rho$ gives $\P(\m A_\rho^\infty)\geq\P(\min_i(\mathfrak H_i^\infty(x)-\mathfrak H_{i+1}^\infty(x))\geq2\rho)\geq\limsup_{t\to\infty}\P(\m A_{2\rho}^t)>1-\e$, which proves \eqref{Ainf}.
\end{proof}

\begin{proof}[Proof of \ref{partb}] Fix any $\e>0$. Consider $M_0(\e)>0$ from Lemma \ref{p1}. Assume $\mathfrak{H}^t \to \mathfrak{H}^\infty$ weakly.  Taking $t\to\infty$ in \eqref{p2e} and using continuity of $\mathfrak{H}_1^\infty$, we get that
\begin{align}\label{Hinfubd}
    \P\left( \mathfrak H_1^\infty(x)+\frac{x^2}2   \le M_0\right) \ge 1-\e.
\end{align}
We thus focus on showing the other inequality. Take $K=K(\mu,\e)$ from Lemma \ref{hbpara} and set $x_0=(2K+\log 2)/\e$.  Note that for all $x\in[-x_0-1,0]$, by Proposition \ref{p10} we may find $\mathfrak{R}_1(\e)>0$ so that
\begin{align*}
    \P\bigg(\mathfrak{H}_1^\infty(x)+\frac{x^2}{2} \ge -\mathfrak{R}_1\bigg) \ge \P\bigg(\inf_{y\in [-x_0-1,0]} \mathfrak{H}_1^\infty(y) \ge -\mathfrak{R}_1\bigg) \ge 1-\e.
\end{align*} Take any $z\le -x_0$. Repeating the argument in the proof of \eqref{p1e} with $r\mapsto \e$ and $x\mapsto z$, we see that
\begin{align*}
    \P\left(\sup_{y\in [z-\e,z]} \mathfrak{H}_1^t(y) \ge -K-\frac{z^2}2-1\right)\ge 1-3\e.
\end{align*}
Now let us choose $x \le -x_0-1$. Then $\P(\m{A})\ge 1-6\e$ where
\begin{align*}
  \m{A}:= \left\{\sup_{y\in [x+1-\e,x+1]} \mathfrak{H}_1^t(y) \ge -K-\frac{(x+1)^2}2-1\right\}\cap \left\{\sup_{y\in [x-1-\e,x-1]} \mathfrak{H}_1^t(y) \ge -K-\frac{(x-1)^2}2-1\right\}.
\end{align*}
Let $\omega$ be the first time in $[x-1-\e,x-1]$ such that $\mathfrak{H}_1^t(\omega) \ge -K-\frac{(x-1)^2}2-1$ and $\tau$ be the last time in $[x+1-\e,x+1]$ such that $\mathfrak{H}_1^t(\tau) \ge -K-\frac{(x+1)^2}2-1$. Note $[\omega,\tau]$ is a stopping domain. Fix some constant $K'>0$ to be chosen appropriately later. By the strong Gibbs property (Lemma \ref{lem:sg2}) and stochastic monotonicity (Lemma \ref{lem:sm}),
 \begin{align}
   &\ind_{\m{A}}\cdot \E\left[\ind\bigg\{\mathfrak{H}_1^t(x)
    \le -K-K'-\frac{x^2}2-\e|x|\bigg\}\,\bigg|\, \mathcal{F}_{\m{ext}}(\{1\}\times(\omega,\tau))\right] \nonumber \\
    &\qquad\qquad\qquad \le \ind_{\m{A}}\cdot \P\left(B(x) \le -K-K'-\frac{x^2}2-\e|x| \right) \label{KK'},
 \end{align}
 where $B$ is Brownian bridge on $[\omega,\tau]$ with $B(\omega)=-K-\frac{(x-1)^2}{2}-1$ and $B(\tau)=-K-\frac{(x+1)^2}{2}-1$. As $\omega \in [x-1-\e,x-1]$ and $\tau \in [x+1-\e,x+1]$ on $\m{A}$, we have $\E[B(x)] \ge -K-x^2/2-3/2-\e|x|$. Thus one may take $K'$ large enough (depending only on $\e$) so that the probability in \eqref{KK'} is less than $\e$. As $\P(\neg\m{A}) \le 6\e$, we obtain  
 \begin{align*}
     \P\left(\mathfrak{H}_1^t(x) \le -K-K'-\frac{x^2}2-\e|x|\right) \le 7\e.
 \end{align*}
 Combining with \eqref{Hinfubd} and adjusting $\e$ proves the desired bound.
\end{proof}

\begin{proof}[Proof of \ref{partd}, \ref{parte}] These follow from part \ref{parta} and Theorems \ref{6.11} (for part~\ref{partd}) and \ref{6.11c} (for part~\ref{parte}), in the same manner as Step 2 in the proof of Theorem \ref{thm0} in Section \ref{sec5}. So, we will be brief and focus only the supercritical case \ref{partd}. Without loss of generality we can assume $\mathfrak{H}^{t,\alpha} \to \mathfrak{H}_{\mathrm{sc}}^{\infty,\alpha}$ weakly, and by the Skorohod representation theorem we can assume the convergence is uniform on compact sets a.s. Recalling the notation from Definition \ref{defni2}, it suffices, as in \eqref{eqgibbs}, to test against any bounded continuous $F$ on $\mathcal C(\ll1,2m\rr\times[A,0])$ and any bounded continuous cylinder functional $G$ of the full exterior of $\ll1,2m\rr\times(A,0]$:
\begin{equation}\label{airygibbs1}
\mathbf E[F(\mathfrak H_{\mathrm{sc}}^{\infty,\alpha})G(\mathfrak H_{\mathrm{sc}}^{\infty,\alpha})]
=\mathbf E[\mathbf E_{\infty;g}^{2m;A;\vec a,\star}[F]\,G(\mathfrak H_{\mathrm{sc}}^{\infty,\alpha})],
\end{equation}
where $a=(\mathfrak{H}_{1,\mathrm{sc}}^{\infty,\alpha}(A),\dots,\mathfrak{H}_{2m,\mathrm{sc}}^{\infty,\alpha}(A))$ and $g=\mathfrak{H}_{2m+1,\mathrm{sc}}^{\infty,\alpha}|_{[A,0]}$. By assumption, and by the Gibbs property of $\mathfrak{H}^{t,\alpha}$ (Lemma \ref{lem:gibbs2}), the left hand side of the above is equal to
\begin{equation*}
\lim_{t\to\infty} \mathbf{E}\left[F(\mathfrak{H}^{t,\alpha}) G(\mathfrak{H}^{t,\alpha})\right] = \lim_{t\to\infty}\mathbb{E}\left[\E_{t^{2/3};g^t}^{2m;A;\vec{a}^t,\star}[F]\cdot G(\mathfrak{H}^{t,\alpha})\right],
\end{equation*}
where $\vec{a}^t = (\mathfrak{H}_1^{t,\alpha}(A), \dots,\mathfrak{H}_{2m}^{t,\alpha}(A))$, $g^t := \mathfrak{H}^{t,\alpha}_{2m+1}|_{[A,0]}$, and the inner expectation is as in Definition \ref{kpzgibbs}. Thanks to \ref{parta}, we know that $a_1 > \cdots > a_{2m} > g(A)$ a.s. Therefore, Theorem \ref{6.11} (with $L=t^{2/3}$) implies that the inner expectation on the right side in the last display converges as $t\to\infty$ to $\E_{\infty;g}^{2m;A;\vec{a},\star}[F]$. The dominated convergence theorem then allows us to conclude \eqref{airygibbs1}.
\end{proof}
\begin{proof}[Proof of \ref{partc}]
This is a direct consequence of the one-sided Gibbs property in \ref{partd} and \ref{parte}.
\end{proof}

We conclude with the proof of convergence for $\alpha\in\R$ and $\beta > \max\{-\alpha,0\}$.

\begin{proof}[Proof of Theorem \ref{kpzb123}] 

Tightness of $\mathfrak{H}^{t;\beta}$ was proven in the previous subsection. An argument very similar to that used above to prove Theorem \ref{kpz123}\ref{partd},\ref{parte} shows that any weak subsequential limit has the two-sided Brownian Gibbs property. Let $\mathfrak H^{\infty;\beta}$ be any such subsequential limit and fix $x\in\R$. At every continuity point $a$ of the law of $\mathfrak H_1^{\infty;\beta}(x)+x^2/2$, tightness and Lemma \ref{app1} imply
\begin{align*}
\limsup_{t\to\infty}\P\big(U_t(x)-t^{-1/3}\log w_t\leq a\big)
&\leq\P\big(\mathfrak H_1^{\infty;\beta}(x)+x^2/2\leq a\big)\\
&\leq\liminf_{t\to\infty}\P\big(U_t(x)-t^{-1/3}\log(2w_t)\leq a\big),
\end{align*}
where in the second bound the exceptional probability in Lemma \ref{app1}(b) is sent to zero. Both logarithmic corrections tend to zero.

    By Proposition \ref{p:iden} and Lemma \ref{hbpara2}, $U_t(x)$ converges weakly to $2^{-1/3}\mathrm{TW}_{\mathrm{GUE}}$. It follows that 
    \begin{equation}\label{eq:htw}
        \mathfrak{H}_1^{\infty;\beta}(x)+x^2/2\stackrel{d}{=}2^{-1/3}\mathrm{TW}_{\mathrm{GUE}}
    \end{equation}
    for every $x\in\R$. 
    
    We may now apply the strong characterization \cite[Corollary 2.11]{ah23} of the Airy line ensemble $\mathcal{A}$, with parameters $\sigma = 1/2$ and $q = 2^{1/6}\sigma^{1/3} = 2^{-1/6}$. This implies that $\mathfrak{H}^{\infty;\beta}$ is equal in law to the sum of $q^{-1}\cdot 2^{-1/2}(\mathcal{A}(q^2x)-(q^2x)^2) = 2^{-1/3}\mathcal{A}(2^{-1/3}x)-x^2/2$ and a random affine shift $\mathfrak{l}x+\mathfrak{c}$.

    It remains to show that $\mathfrak{l}=\mathfrak{c}=0$. It follows from the last distributional equality, the fact that $\mathcal{A}(x) \overset{d}= \mathrm{TW}_{\mathrm{GUE}}$ for every $x\in\R$, and \eqref{eq:htw}, that for $X\sim 2^{-1/3}\mathrm{TW}_{\mathrm{GUE}}$ independent of $\mathfrak{l}$ and $\mathfrak{c}$, we have $X \overset{d}= X + \mathfrak{l}x+\mathfrak{c}$ for every $x\in\R$. Taking $x=0$ clearly implies $\mathfrak{c}=0$, and then taking $x=1$ implies $\mathfrak{l}=0$.

    We have shown that $\mathfrak{H}^{t;\beta}$ is tight as $t\to\infty$, and that the only possible weak limit point is $2^{-1/3}\mathcal{A}(2^{-1/3}x) - x^2/2$. Together these imply weak convergence, as desired.
\end{proof}

\appendix \section{Uniform gap estimates: Proof of Lemma \ref{ord}} \label{app}

In this section, we prove Lemma \ref{ord}. For convenience, we drop the subscript $L$ from $\vec{a}_L$ and $g_L$. First note that the equality in \eqref{sepbd2} is via spatial translation invariance of two-sided HSKPZ Gibbs measures. So, we shall assume $R=0$. For the one-sided claim \eqref{sepbd}, we first aim to reduce to the case of two-sided Gibbs measures \eqref{sepbd2}. Towards this end, we claim that we can find $M'>0$ so that that 
\begin{equation}\label{M'bd}
\mathbf{P}_{L;g}^{2m;A;\vec{a},\star}(\norm{B(0)} \le M')\ge 1-\e
\end{equation}
for all large enough $L$. For the upper bound, we apply stochastic monotonicity (Lemma \ref{lem:sm}) to raise the boundary conditions until they are separated, e.g., replace $a_{i}$ with $a_i' := (M+2m-i+1)$ and $g$ with $g' \equiv M$. Then for any $i$, in the supercritical case, $$\P_{L;g}^{2m;A;\vec{a},\star}(B_i(0) \leq M') \geq \P_{L;g'}^{2m;A;\vec{a}',\star}(B_i(0) \leq M')\xrightarrow[]{L\to\infty} \P_{\infty;g'}^{2m;A;\vec{a}',\star}(B_i(0)\le M'),$$ where the last convergence follows via Theorem \ref{6.11} as the components of $\vec{a}'$ are separated and lie strictly above $g'$. In the critical case, thanks to  Theorem \ref{6.11c}, the above equation is true with $\P_{\infty;g'}^{2m;A;\vec{a}',\star}$ replaced by $\bar{\P}_{\infty;g}^{2m;A;\vec{a}',\star}$ where the latter is defined in Definition \ref{defcr}. In both regimes, the last probability can be made arbitrarily close to 1 for large $M'\geq M$. By lowering the boundary conditions instead, a similar argument shows that $\P_{L;g}^{2m;A;\vec{a},\star}(B_i(0) \leq -M')$ can be made arbitrarily small. Taking a union bound proves \eqref{M'bd}.

Conditioning on $B(0)$, it thus suffices to show \eqref{sepbd}, as this will imply \eqref{sepbd2} as well. More precisely, we will show that we can find $\rho>0$ so that uniformly over $\vec{a},\vec{b}$ with $\lVert \vec{a}\rVert, \lVert \vec{b}\rVert \leq M'$ and large $L$ we have
\begin{equation}\label{gap2side}
    \P_{L;g}^{k;A,0;\vec{a},\vec{b}}\left(\inf_{x\in [A^-,A^+]}B_i(x)-B_{i+1}(x) \geq \rho\right) \geq 1-\e,
\end{equation}
for any $1\leq i\leq k-1$.  Taking a union bound over $i$ and adjusting $\e$ implies \eqref{sepbd}. 
To prove \eqref{gap2side}, we begin by rewriting the probability in the Gibbsian framework considered in \cite{wuconv}. The difference between our setup and that of \cite{wuconv} lies in the definition of the Radon–Nikodym derivative. Instead of using the form given in \eqref{def:Wcont}, they consider the following:

\begin{align}\label{def:Wcont2}
\hat{\mathcal{W}}_{L;f,g}^{k,\ell;A_1,A_2}(B) := \prod_{i=k-1}^{\ell} \exp\left(-\int_{A_1}^{A_2} e^{\sqrt{L}(B_{i+1}(x)-B_i(x))} dx \right).
\end{align}
with $B_{k-1}=f$ and $B_{\ell+1}=g$. The distinction here is the absence of the $L$ prefactor in front of the integral, which appears in our formulation \eqref{def:Wcont}. To translate into the above setup, we just need to lift the $i$-th curve by $(i-1)\log L/\sqrt{L}$ units. Indeed, if we define $a_i':=a_i+(i-1)\log L/\sqrt{L}$, $b_i':=b_i+(i-1)\log L/\sqrt{L}$ and $g':=g+k\log L/\sqrt{L}$, then
\begin{align*}
    \P_{L;g}^{k;A,0;\vec{a},\vec{b}}\left(\inf_{x\in [A^-,A^+]}B_i(x)-B_{i+1}(x) \geq \rho\right)=\hat\P_{L;g'}^{k;A,0;\vec{a}',\vec{b}'}\left(\inf_{x\in [A^-,A^+]} B_i(x)-B_{i+1}(x) +\frac{\log L}{\sqrt{L}} \geq \rho\right)
\end{align*}
where the $\hat\P_{L;g'}^{\bullet}$ law is defined just as $\P_{L;g'}^{\bullet}$ (recall Definition \ref{kpzgibbs2}) but using $\hat{\mathcal{W}}$ from \eqref{def:Wcont2} instead of $\mathcal{W}$. Thus to prove \eqref{gap2side} it suffices to show that 
\begin{equation}\label{gap2side2}
    \hat\P_{L;g'}^{k;A,0;\vec{a}',\vec{b}'}\left(\inf_{x\in [A^-,A^+]} B_i(x)- B_{i+1}(x) \geq \rho/2\right) \geq 1-\e
\end{equation}
holds for all large enough $L$.

We focus on proving \eqref{gap2side2}.
We will condition on the $\sigma$-algebra $\mathcal{F}_{\m{ext}}(\ll 1,i+1\rr \times (A^-,A^+))$. By the Gibbs property, the resulting conditional law of $( B_1,\dots, B_{i+1})$ is $\hat\P_{L;f}^{i+1;A^-,A^+;\vec{x},\vec{y}}$, with $\vec{x} = ( B_j(A^-))_{j=1}^{i+1}$, $\vec{y} = ( B_j(A^+))_{j=1}^{i+1}$, and $f= B_{i+2}$ (when $i=k-1$, $B_{2m+1}:=g'$; $B_{k+1}:=g'$). Recalling the notation \eqref{def:Wcont2}, let $$Z^{\vec{x},\vec{y}}_f := \mathbf{E}\left[\hat{\mathcal{W}}_{L;+\infty,f}^{1,i+1;A^-,A^+}(B)\right]$$ denote the normalizing constant for this conditional measure (the expectation being taken over independent Brownian bridges with boundary data $\vec{x},\vec{y}$, conditional on $\vec{x},\vec{y},f$). For $\hat M'>0$ to be chosen, define the event
\begin{align*}
\m{G} :=\bigcap_{j=i}^{i+1}\left\{\max\left(| B_j(A^-)|, | B_j(A^+)|\right) \leq \hat M'\right\}.
\end{align*}
By a monotonicity argument (the $\hat{\P}$ laws also obey stochastic monotonicity similar to that in Lemma \ref{lem:sm}; cf.~Lemma 2.10 in \cite{wuconv}) similar to that used to prove \eqref{M'bd}, one choose $\hat{M}'$ large enough so that $\hat{\P}_{L;g'}^{k;A,0;\vec{a}',\vec{b}'}(\m{G}) \ge 1-\e/4$ for all large $L$. We claim that there exists $\delta>0$ so that for all large $L$, uniformly over $\lVert\vec{a}'\rVert,\lVert\vec{b}'\rVert \leq  M'+1$,
\begin{equation}\label{Glbd}
    \hat\P_{L;g'}^{k;A,0;\vec{a}',\vec{b}'}\left(Z^{\vec{x},\vec{y}}_f > \delta\right) > 1-\e/4.
\end{equation}
Then by a union bound, in \eqref{gap2side2} we can restrict to the event $\left\{Z^{\vec{x},\vec{y}}_f > \delta\right\} \cap \m{G}$. The proof of the desired lower bound now follows exactly the same route as \cite[Proposition 3.3]{wuconv}, conditioning on $\mathcal{F}_{\m{ext}}(\ll 1,i+1\rr \times [A^-,A^+])$ and applying estimates for independent Brownian bridges (see in particular Claims 3.4 and 3.5 therein).

We therefore focus on proving \eqref{Glbd}, for which we aim to apply \cite[Proposition 4.3]{wuconv}. Let us introduce for $1\leq i\leq i+2$ the deterministic times $\ell_j :=A/2+A(i+2-j)/(4k+4)$, $r_j := A/2-A(i+2-j)/(4k+4)$. For $M'' > 0$ to be determined, define the ``good boundary'' event $\m{GB}=\m{GB}^+\cap \m{GB}^{-}$ where
\[
\m{GB}^{\pm} := \bigcap_{j=1}^{(i+2)\wedge k} \left\{ \sup_{x\in[\ell_j,\ell_j]\cup[r_i,r_1]} \pm B_j(x) \leq M'' \right\}.
\]
We now argue that we can find $M''\ge M'+1$ large enough so that
\begin{equation}\label{GBlbd}
    \hat\P_{L;g'}^{k;A,0;\vec{a}',\vec{b}'}(\m{GB}) > 1-\e/8
\end{equation}
uniformly over $\lVert \vec{a}'\rVert, \lVert\vec{b}'\rVert \leq M'+1$, $g\le M'+1$ and large $L$. 

 By stochastic monotonicity of the $\hat{\P}$ law, we may replace each $a_j'$ and $b_j'$ with $a_j'' := M'+2(2m-j+1)K$ and $g'$ by $M'+1$. This will only decrease the probability of the event $\m{GB}^+$. Thus,
 \begin{align*}
     \hat\P_{L;g'}^{k;A,0;\vec{a}',\vec{b}'}(\m{GB}^+) \ge \hat\P_{L;M+1}^{k;A,0;\vec{a}'',\vec{a}''}(\m{GB}^+).
 \end{align*}
As the boundary data are separated and away from the floor $M'+1$, as $L\to\infty$, the law $\hat\P_{L;M+1}^{k;A,0;\vec{a}'',\vec{a}''}$ converges to $k$ Brownian bridges on $[A,0]$ from $\vec{a}''$ to $\vec{a}''$ conditioned not to intersect and to stay above $M+1$. By straightforward tube estimates for Brownian bridges (cf.~\cite[Lemma 3.14]{dff}) the probability of the event $\m{GB}^+$ under this limiting law can be made at least $1-\e/32$ for $K,M''$ large enough. This implies $\hat\P_{L;g'}^{k;A,0;\vec{a}',\vec{b}'}(\m{GB}^+) > 1-\e/16$ for all large $L$. A similar monotonicity argument leads to $\hat\P_{L;g'}^{k;A,0;\vec{a}',\vec{b}'}(\m{GB}^-) > 1-\e/16$ for large $L$, and together these imply \eqref{GBlbd}.

Finally to prove \eqref{Glbd}, given \eqref{GBlbd} it suffices to show 
\begin{align}
    \label{lastineq}
    \hat{\P}_{L;g'}^{k;A,0;\vec{a}',\vec{b}'}\big(\{Z_f^{\vec{x},\vec{y}}\le \delta\}\cap\m{GB}\big) \le \e/8.
\end{align}
To do so we condition on the $\sigma$-algebra generated by $B_j$ on $[\ell_1,\ell_j]\cup[r_j,r_1]$ for $1\leq j\leq i+2$. Let $f_j$ denote the conditioned curves $B_j |_{[\ell_1,\ell_j]\cup[r_j,r_1]}$. By the Gibbs property, the resulting conditional law, $\bar{\P}$ say, is a Gibbs measure of the type defined in \cite[Eq.~(4.3)]{wuconv} with boundary conditions given by $f_1,\dots,f_{i+2}$. If $i\le k-2$, on the event $\m{GB}$, these boundary conditions $f_j$ are ``$M''$-\textbf{Good}'' in the language of \cite[Section 4]{wuconv}, and in this case \cite[Proposition 4.3]{wuconv} applies to show that $\ind_{\m{GB}}\cdot \bar{\P}(Z_f^{\vec{x},\vec{y}} \le \delta) \le \e/8$ for some $\delta>0$ small enough.
If $i=k-1$, $f_{i+2}=f_{k+1}=f=g'$ is deterministic and by definition $Z_{f}^{\vec{x},\vec{y}}$ is decreasing as $f$ increases. Since $f=g' \le M''$, in this case, $\ind_{\m{GB}}\cdot \bar{\P}(Z_f^{\vec{x},\vec{y}} \le \delta) \le \ind_{\m{GB}}\cdot \bar{\P}(Z_{M''}^{\vec{x},\vec{y}} \le \delta) \le \e/8$ where the last inequality follows by using \cite[Proposition 4.3]{wuconv} again. In both cases, taking expectation we arrive at \eqref{lastineq}  and we are done.

    \bibliographystyle{alpha}		
	\bibliography{hskpz,oldh}

\newcommand{\etalchar}[1]{$^{#1}$}
\begin{thebibliography}{DNKLDT20}

\bibitem[Abr80]{phy2}
Douglas~B Abraham.
\newblock Solvable model with a roughening transition for a planar {I}sing ferromagnet.
\newblock {\em Phys. Rev. Lett.}, 44(18):1165, 1980.

\bibitem[ACQ11]{amir2011probability}
Gideon Amir, Ivan Corwin, and Jeremy Quastel.
\newblock Probability distribution of the free energy of the continuum directed random polymer in 1+1 dimensions.
\newblock {\em Commun. Pure Appl. Math.}, 64(4):466--537, 2011.

\bibitem[ADH25]{adh}
Amol Aggarwal, Sayan Das, and Jiaoyang Huang.
\newblock Strong characterization for the half-space {A}iry line ensemble.
\newblock {\em In preparation}, 2025.

\bibitem[AH23]{ah23}
Amol Aggarwal and Jiaoyang Huang.
\newblock Strong characterization for the {A}iry line ensemble.
\newblock {\em arXiv preprint arXiv:2308.11908}, 2023.

\bibitem[AKQ14]{akq2}
Tom Alberts, Konstantin Khanin, and Jeremy Quastel.
\newblock The intermediate disorder regime for directed polymers in dimension $1+1$.
\newblock {\em Ann. Probab.}, 42(3):1212--1256, 2014.

\bibitem[Bai06]{baik}
Jinho Baik.
\newblock {Painlevé formulas of the limiting distributions for nonnull complex sample covariance matrices}.
\newblock {\em Duke Math. J.}, 133(2):205--235, 2006.

\bibitem[BAP05]{bbap}
Jinho Baik, G{\'e}rard~Ben Arous, and Sandrine P{\'e}ch{\'e}.
\newblock Phase transition of the largest eigenvalue for nonnull complex sample covariance matrices.
\newblock {\em Ann. Probab.}, 33(5):1643--1697, 2005.

\bibitem[BBC16]{bbc}
Alexei Borodin, Alexey Bufetov, and Ivan Corwin.
\newblock Directed random polymers via nested contour integrals.
\newblock {\em Ann. Phys.}, 368:191--247, 2016.

\bibitem[BBC20]{bbc20}
Guillaume Barraquand, Alexei Borodin, and Ivan Corwin.
\newblock Half-space {M}acdonald processes.
\newblock {\em Forum Math. Pi}, 8, 2020.

\bibitem[BBCS18a]{bbcs0}
Jinho Baik, Guillaume Barraquand, Ivan Corwin, and Toufic Suidan.
\newblock Facilitated exclusion process.
\newblock In {\em The Abel Symposium}, pages 1--35. Springer, 2018.

\bibitem[BBCS18b]{bbcs}
Jinho Baik, Guillaume Barraquand, Ivan Corwin, and Toufic Suidan.
\newblock {P}faffian {S}chur processes and last passage percolation in a half-quadrant.
\newblock {\em Ann. Probab.}, 46(6):3015--3089, 2018.

\bibitem[BBCW18]{bbcw}
Guillaume Barraquand, Alexei Borodin, Ivan Corwin, and Michael Wheeler.
\newblock Stochastic six-vertex model in a half-quadrant and half-line open asymmetric simple exclusion process.
\newblock {\em Duke Math. J.}, 167(13):2457--2529, 2018.

\bibitem[BBNV18]{bete}
Dan Betea, J{\'e}r{\'e}mie Bouttier, Peter Nejjar, and Mirjana Vuleti{\'c}.
\newblock The free boundary {S}chur process and applications {I}.
\newblock {\em Ann. H. Poincar\'{e}}, 19(12):3663--3742, 2018.

\bibitem[BC20]{batesch}
Erik Bates and Sourav Chatterjee.
\newblock The endpoint distribution of directed polymers.
\newblock {\em Ann. Probab.}, 48(2):817--871, 2020.

\bibitem[BC23]{bc22}
Guillaume Barraquand and Ivan Corwin.
\newblock {Stationary measures for the log-gamma polymer and KPZ equation in half-space}.
\newblock {\em Ann. Probab.}, 51(5):1830--1869, 2023.

\bibitem[BCD23]{bcd}
Guillaume Barraquand, Ivan Corwin, and Evgeni Dimitrov.
\newblock Spatial tightness at the edge of {G}ibbsian line ensembles.
\newblock {\em Commun. Math. Phys.}, 397(3):1309--1386, 2023.

\bibitem[BCD24]{half1}
Guillaume Barraquand, Ivan Corwin, and Sayan Das.
\newblock {KPZ exponents for the half-space log-gamma polymer}.
\newblock {\em Probab. Theory Relat. Fields}, pages 1--131, 2024.

\bibitem[BCF14]{borodin2014free}
Alexei Borodin, Ivan Corwin, and Patrik Ferrari.
\newblock Free energy fluctuations for directed polymers in random media in 1+1 dimension.
\newblock {\em Commun. Pure Appl. Math.}, 67(7):1129--1214, 2014.

\bibitem[BCY24]{bcy}
Guillaume Barraquand, Ivan Corwin, and Zongrui Yang.
\newblock Stationary measures for integrable polymers on a strip.
\newblock {\em Invent. Math.}, 237(3):1567--1641, 2024.

\bibitem[BFO20]{ale1}
Dan Betea, Patrik~L Ferrari, and Alessandra Occelli.
\newblock Stationary half-space last passage percolation.
\newblock {\em Commun. Math. Phys.}, 377(1):421--467, 2020.

\bibitem[BFO22]{ale2}
Dan Betea, Patrik~L Ferrari, and Alessandra Occelli.
\newblock The half-space {A}iry stat process.
\newblock {\em Stoch. Process. Their Appl.}, 146:207--263, 2022.

\bibitem[BGH21]{bgh1}
Riddhipratim Basu, Shirshendu Ganguly, and Alan Hammond.
\newblock Fractal geometry of $\mbox{Airy}_2$ processes coupled via the {A}iry sheet.
\newblock {\em Ann. Probab.}, 49(1):485--505, 2021.

\bibitem[BGH22]{bgh2}
Erik Bates, Shirshendu Ganguly, and Alan Hammond.
\newblock Hausdorff dimensions for shared endpoints of disjoint geodesics in the directed landscape.
\newblock {\em Electron. J. Probab.}, 27:1--44, 2022.

\bibitem[BHL83]{phy3}
E~Br{\'e}zin, BI~Halperin, and S~Leibler.
\newblock Critical wetting in three dimensions.
\newblock {\em Phys. Rev. Lett.}, 50(18):1387, 1983.

\bibitem[BKLD20]{bkld2}
Guillaume Barraquand, Alexandre Krajenbrink, and Pierre Le~Doussal.
\newblock Half-space stationary {K}ardar--{P}arisi--{Z}hang equation.
\newblock {\em J. Stat. Phys.}, 181(4):1149--1203, 2020.

\bibitem[BKLD22]{bkld}
Guillaume Barraquand, Alexandre Krajenbrink, and Pierre Le~Doussal.
\newblock Half-space stationary {K}ardar--{P}arisi--{Z}hang equation beyond the {B}rownian case.
\newblock {\em J. Phys. A: Math. and Theor.}, 55(27), 2022.

\bibitem[BKWW23]{bkww}
W{\l}odek Bryc, Alexey Kuznetsov, Yizao Wang, and Jacek Weso{\l}owski.
\newblock {Markov processes related to the stationary measure for the open KPZ equation}.
\newblock {\em Probab. Theory Relat. Fields}, 185(1):353--389, 2023.

\bibitem[BLD21a]{bld1}
Guillaume Barraquand and Pierre Le~Doussal.
\newblock {K}ardar--{P}arisi--{Z}hang equation in a half space with flat initial condition and the unbinding of a directed polymer from an attractive wall.
\newblock {\em Phys. Rev. E}, 104(2):024502, 2021.

\bibitem[BLD21b]{bld2}
Guillaume Barraquand and Pierre Le~Doussal.
\newblock {Steady state of the KPZ equation on an interval and Liouville quantum mechanics}.
\newblock {\em EPL 137 61003}, 2021.

\bibitem[BLD23]{bld01}
Guillaume Barraquand and Pierre Le~Doussal.
\newblock {Stationary measures of the KPZ equation on an interval from Enaud--Derrida’s matrix product ansatz representation}.
\newblock {\em J. Phys. A: Math. Theor.}, 56(14):144003, 2023.

\bibitem[Bol89]{bol}
Erwin Bolthausen.
\newblock A note on the diffusion of directed polymers in a random environment.
\newblock {\em Commun. Math. Phys.}, 123(4):529--534, 1989.

\bibitem[BR01a]{br1}
Jinho Baik and Eric~M Rains.
\newblock Algebraic aspects of increasing subsequences.
\newblock {\em Duke Math. J.}, 109(1):1--65, 2001.

\bibitem[BR01b]{br01}
Jinho Baik and Eric~M Rains.
\newblock The asymptotics of monotone subsequences of involutions.
\newblock {\em Duke Math. J.}, 109(2):205--281, 2001.

\bibitem[BR01c]{br3}
Jinho Baik and Eric~M Rains.
\newblock Symmetrized random permutations.
\newblock In {\em Random matrix models and their applications}, volume~40 of {\em Math. Sci. Res. Inst. Publ.}, pages 1--19, 2001.

\bibitem[BW22]{bw}
Guillaume Barraquand and Shouda Wang.
\newblock An identity in distribution between full-space and half-space log-gamma polymers.
\newblock {\em Int. Math. Res. Not.}, rnac132, 2022.

\bibitem[BZ19]{bz19}
Elia Bisi and Nikos Zygouras.
\newblock Point-to-line polymers and orthogonal {W}hittaker functions.
\newblock {\em Trans. Am. Math. Soc.}, 371(12):8339--8379, 2019.

\bibitem[CD18]{cd18}
Ivan Corwin and Evgeni Dimitrov.
\newblock Transversal fluctuations of the {ASEP}, stochastic six vertex model, and {H}all-{L}ittlewood {G}ibbsian line ensembles.
\newblock {\em Commun. Math. Phys.}, 363(2):435--501, 2018.

\bibitem[CGH21]{cgh21}
Ivan Corwin, Promit Ghosal, and Alan Hammond.
\newblock {KPZ} equation correlations in time.
\newblock {\em Ann. Probab.}, 49(2):832--876, 2021.

\bibitem[CH14]{ble}
Ivan Corwin and Alan Hammond.
\newblock {Brownian Gibbs property for Airy line ensembles}.
\newblock {\em Invent. Math.}, 195:441--508, 2014.

\bibitem[CH16]{kpzle}
Ivan Corwin and Alan Hammond.
\newblock {KPZ} line ensemble.
\newblock {\em Probab. Theory Relat. Fields}, 166:67--185, 2016.

\bibitem[CHH19]{chh19}
Jacob Calvert, Alan Hammond, and Milind Hegde.
\newblock {B}rownian structure in the {KPZ} fixed point.
\newblock {\em To appear in Ast\'erisque, arXiv:1912.00992}, 2019.

\bibitem[CHHM23]{chhm}
Ivan Corwin, Alan Hammond, Milind Hegde, and Konstantin Matetski.
\newblock Exceptional times when the {KPZ} fixed point violates {J}ohansson's conjecture on maximizer uniqueness.
\newblock {\em Electron. J. Probab.}, 28:1--81, 2023.

\bibitem[CIW18]{ciw2}
Pietro Caputo, Dmitry Ioffe, and Vitali Wachtel.
\newblock Tightness and line ensembles for {B}rownian polymers under geometric area tilts.
\newblock In {\em International Conference on Statistical Mechanics of Classical and Disordered Systems}, pages 241--266. Springer, 2018.

\bibitem[CIW19]{ciw1}
Pietro Caputo, Dmitry Ioffe, and Vitali Wachtel.
\newblock Confinement of {B}rownian polymers under geometric area tilts.
\newblock {\em Electron. J. Probab.}, 24:1--21, 2019.

\bibitem[CK24]{ck}
Ivan Corwin and Alisa Knizel.
\newblock {Stationary measure for the open KPZ equation}.
\newblock {\em Commun. Pure Appl. Math.}, 77(4):2183--2267, 2024.

\bibitem[Com17]{comets}
Francis Comets.
\newblock {\em Directed polymers in random environments}.
\newblock Springer, 2017.

\bibitem[Cor18]{corwin2018exactly}
Ivan Corwin.
\newblock Exactly solving the {KPZ} equation.
\newblock {\em arXiv preprint arXiv:1804.05721}, 2018.

\bibitem[COSZ14]{cosz14}
Ivan Corwin, Neil O’Connell, Timo Sepp{\"a}l{\"a}inen, and Nikolaos Zygouras.
\newblock Tropical combinatorics and {W}hittaker functions.
\newblock {\em Duke Math. J.}, 163(3):513--563, 2014.

\bibitem[CQ13]{corwin2013crossover}
Ivan Corwin and Jeremy Quastel.
\newblock Crossover distributions at the edge of the rarefaction fan.
\newblock {\em Ann. Probab.}, 41(3A):1243--1314, 2013.

\bibitem[CS18]{cs1}
Ivan Corwin and Hao Shen.
\newblock Open {ASEP} in the weakly asymmetric regime.
\newblock {\em Commun. Pure Appl. Math.}, 71(10):2065--2128, 2018.

\bibitem[DDY26]{ddy26}
Sayan Das, Evgeni Dimitrov, and Zongrui Yang.
\newblock Pinning in non-critical half-space geometric last passage percolation.
\newblock {\em arXiv preprint arXiv:2603.24616}, 2026.

\bibitem[DFF{\etalchar{+}}21]{dff}
Evgeni Dimitrov, Xiang Fang, Lukas Fesser, Christian Serio, Carson Teitler, Angela Wang, and Weitao Zhu.
\newblock Tightness of {B}ernoulli {G}ibbsian line ensembles.
\newblock {\em Electron. J. Probab.}, 26:1--93, 2021.

\bibitem[DG23]{dg21}
Sayan Das and Promit Ghosal.
\newblock Law of iterated logarithms and fractal properties of the {KPZ} equation.
\newblock {\em Ann. Probab.}, 51(3):930--986, 2023.

\bibitem[DIM77]{durrigle}
Richard~T Durrett, Donald~L Iglehart, and Douglas~R Miller.
\newblock Weak convergence to {B}rownian meander and {B}rownian excursion.
\newblock {\em Ann. Probab.}, 5(1):117--129, 1977.

\bibitem[Dim22]{dimh}
Evgeni Dimitrov.
\newblock {Characterization of $H$-Brownian Gibbsian line ensembles}.
\newblock {\em Probab. Math. Phys.}, 3(3):627--673, 2022.

\bibitem[DM21]{dm21}
Evgeni Dimitrov and Konstantin Matetski.
\newblock Characterization of {B}rownian {G}ibbsian line ensembles.
\newblock {\em Ann. Probab.}, 49(5):2477--2529, 2021.

\bibitem[DNKLDT20]{de}
Jacopo De~Nardis, Alexandre Krajenbrink, Pierre Le~Doussal, and Thimoth{\'e}e Thiery.
\newblock Delta-{B}ose gas on a half-line and the {K}ardar--{P}arisi--{Z}hang equation: boundary bound states and unbinding transitions.
\newblock {\em J. Stat. Mech.}, 2020(4):043207, 2020.

\bibitem[DNV23]{dnv19}
Duncan Dauvergne, Mihai Nica, and B{\'a}lint Vir{\'a}g.
\newblock Uniform convergence to the {A}iry line ensemble.
\newblock {\em Ann. Inst. H. Poincar\'{e} Probab. Statist.}, 59(4):2220--2256, 2023.

\bibitem[DOV22]{dov18}
Duncan Dauvergne, Janosch Ortmann, and B{\'a}lint Vir{\'a}g.
\newblock The directed landscape.
\newblock {\em Acta. Math.}, 229(2):201--285, 2022.

\bibitem[DS25]{ds24}
Sayan Das and Christian Serio.
\newblock Convergence to stationary measures for the half-space log-gamma polymer.
\newblock {\em J. Funct. Anal.}, 289(4):110982, 2025.

\bibitem[DSV22]{3by2}
Duncan Dauvergne, Sourav Sarkar, and B{\'a}lint Vir{\'a}g.
\newblock Three-halves variation of geodesics in the directed landscape.
\newblock {\em Ann. Probab.}, 50(5):1947--1985, 2022.

\bibitem[DSY26]{dsy26}
E.~Dimitrov, C.~Serio, and Z.~Yang.
\newblock The pinned half-space {A}iry line ensemble.
\newblock 2026.
\newblock arXiv preprint arXiv:2601.04546.

\bibitem[Dud04]{dudley}
RM~Dudley.
\newblock {\em Real Analysis and Probability}.
\newblock Cambridge University Press, 2004.

\bibitem[DV21]{dv18}
Duncan Dauvergne and B{\'a}lint Vir{\'a}g.
\newblock Bulk properties of the {A}iry line ensemble.
\newblock {\em Ann. Probab.}, 49(4):1738--1777, 2021.

\bibitem[DW21]{xd1}
Evgeni Dimitrov and Xuan Wu.
\newblock Tightness of $({H},{H}^{RW})$ {G}ibbsian line ensembles.
\newblock {\em arXiv:2108.07484}, 2021.

\bibitem[DY25]{halfairy}
Evgeni Dimitrov and Zongrui Yang.
\newblock Half-space {Airy} line ensembles.
\newblock {\em arXiv preprint arXiv:2505.01798}, 2025.

\bibitem[DZ24a]{dz23}
Sayan Das and Weitao Zhu.
\newblock The half-space log-gamma polymer in the bound phase.
\newblock {\em Commun. Math. Phys.}, 405(8):184, 2024.

\bibitem[DZ24b]{dz22}
Sayan Das and Weitao Zhu.
\newblock Localization of the continuum directed random polymer.
\newblock {\em Probab. Theory Relat. Fields}, pages 1--77, 2024.

\bibitem[DZ24c]{dz22b}
Sayan Das and Weitao Zhu.
\newblock Short-and long-time path tightness of the continuum directed random polymer.
\newblock {\em Ann. Inst. H. Poincar\'{e} Probab. Statist.}, 60(1):343--372, 2024.

\bibitem[DZ25]{dzhou}
Evgeni Dimitrov and Zhengye Zhou.
\newblock Curve separation in supercritical half-space last passage percolation.
\newblock {\em arXiv preprint arXiv:2510.07508}, 2025.

\bibitem[FO24]{ale3}
Patrik Ferrari and Alessandra Occelli.
\newblock Time-time covariance for last passage percolation in half-space.
\newblock {\em Ann. Appl. Probab.}, 34(1A):627--674, 2024.

\bibitem[GH19]{hai}
M{\'a}t{\'e} Gerencs{\'e}r and Martin Hairer.
\newblock Singular {SPDEs} in domains with boundaries.
\newblock {\em Probab. Theory Relat. Fields}, 173:697--758, 2019.

\bibitem[GH22]{gh22}
Shirshendu Ganguly and Milind Hegde.
\newblock Sharp upper tail estimates and limit shapes for the {KPZ} equation via the tangent method.
\newblock {\em arXiv:2208.08922}, 2022.

\bibitem[GH23a]{chaos2}
Shirshendu Ganguly and Alan Hammond.
\newblock {The geometry of near ground states in Gaussian polymer models}.
\newblock {\em Electron. J. Probab.}, 28:1--80, 2023.

\bibitem[GH23b]{gh21}
Shirshendu Ganguly and Milind Hegde.
\newblock Local and global comparisons of the {A}iry difference profile to {B}rownian local time.
\newblock {\em Ann. Inst. H. Poincar\'e Probab. Stat.}, 59(3):1342--1374, 2023.

\bibitem[GH24]{chaos1}
Shirshendu Ganguly and Alan Hammond.
\newblock Stability and chaos in dynamical last passage percolation.
\newblock {\em Commun. Amer. Math. Soc.}, 4(09):387--479, 2024.

\bibitem[Gin24]{vic}
Victor Ginsburg.
\newblock Pinning, diffusive fluctuations, and {Gaussian} limits for half-space directed polymer models.
\newblock {\em Electron. J. Probab.}, 29:1--34, 2024.

\bibitem[GLD12]{gld}
Thomas Gueudr{\'e} and Pierre Le~Doussal.
\newblock Directed polymer near a hard wall and {KPZ} equation in the half-space.
\newblock {\em EPL}, 100(2):26006, 2012.

\bibitem[GT24]{rantao}
Yu~Gu and Ran Tao.
\newblock {Fluctuation exponents of the half-space KPZ at stationarity}.
\newblock {\em arXiv preprint arXiv:2410.01653}, 2024.

\bibitem[Ham19]{ham2}
Alan Hammond.
\newblock Modulus of continuity of polymer weight profiles in {B}rownian last passage percolation.
\newblock {\em Ann. Probab.}, 47(6):3911--3962, 2019.

\bibitem[Ham20a]{ham4}
Alan Hammond.
\newblock Exponents governing the rarity of disjoint polymers in {B}rownian last passage percolation.
\newblock {\em Proc. London Math. Soc.}, 120(3):370--433, 2020.

\bibitem[Ham20b]{ham3}
Alan Hammond.
\newblock A patchwork quilt sewn from {B}rownian fabric: regularity of polymer weight profiles in {B}rownian last passage percolation.
\newblock {\em Forum Math., Pi}, 8, 2020.

\bibitem[Ham22]{ham1}
Alan Hammond.
\newblock {B}rownian regularity for the {A}iry line ensemble, and multi-polymer watermelons in {B}rownian last passage percolation.
\newblock {\em Mem. Amer. Math. Soc.}, 277(1363), 2022.

\bibitem[He24]{he2}
Jimmy He.
\newblock Boundary current fluctuations for the half-space asep and six-vertex model.
\newblock {\em Proc. London Math. Soc.}, 128(2):e12585, 2024.

\bibitem[He25]{he1}
Jimmy He.
\newblock Shift invariance of half space integrable models.
\newblock {\em Probab. Theory Relat. Fields}, pages 1--71, 2025.

\bibitem[HH85]{huse}
David~A Huse and Christopher~L Henley.
\newblock Pinning and roughening of domain walls in {I}sing systems due to random impurities.
\newblock {\em Phys. Rev. Lett.}, 54(25):2708, 1985.

\bibitem[Him24]{himwich}
Zoe Himwich.
\newblock {Stationary measure of the open KPZ equation through the Enaud--Derrida representation}.
\newblock {\em arXiv preprint arXiv:2404.13444}, 2024.

\bibitem[HKS25]{hegde}
Milind Hegde, Yujin~H Kim, and Christian Serio.
\newblock Scaling limit and tail bounds for a random walk model of {SOS} level lines.
\newblock {\em arXiv preprint arXiv:2502.10384}, 2025.

\bibitem[Igl74]{iglehart}
Donald~L Iglehart.
\newblock Functional central limit theorems for random walks conditioned to stay positive.
\newblock {\em Ann. Probab.}, 2(4):608--619, 1974.

\bibitem[IMS22]{ims22}
Takashi Imamura, Matteo Mucciconi, and Tomohiro Sasamoto.
\newblock Solvable models in the {KPZ} class: approach through periodic and free boundary schur measures.
\newblock {\em arXiv preprint arXiv:2204.08420}, 2022.

\bibitem[IS88]{imb}
John~Z Imbrie and Thomas Spencer.
\newblock Diffusion of directed polymers in a random environment.
\newblock {\em J. Stat. Phys.}, 52(3):609--626, 1988.

\bibitem[IT18]{ito}
Yasufumi Ito and Kazumasa~A Takeuchi.
\newblock When fast and slow interfaces grow together: connection to the half-space problem of the {K}ardar--{P}arisi--{Z}hang class.
\newblock {\em Phys. Rev. E}, 97(4):040103, 2018.

\bibitem[Kar85]{kar2}
Mehran Kardar.
\newblock Depinning by quenched randomness.
\newblock {\em Phys. Rev. Lett.}, 55(21):2235, 1985.

\bibitem[KLD18]{kr}
Alexandre Krajenbrink and Pierre Le~Doussal.
\newblock Large fluctuations of the {KPZ} equation in a half-space.
\newblock {\em SciPost Physics}, 5(4):032, 2018.

\bibitem[NZ17]{nz17}
Vu-Lan Nguyen and Nikos Zygouras.
\newblock Variants of geometric {RSK}, geometric {PNG}, and the multipoint distribution of the log-gamma polymer.
\newblock {\em Int. Math. Res. Not.}, 2017(15):4732--4795, 2017.

\bibitem[OSZ14]{osz14}
Neil O’Connell, Timo Sepp{\"a}l{\"a}inen, and Nikos Zygouras.
\newblock Geometric {RSK} correspondence, {W}hittaker functions and symmetrized random polymers.
\newblock {\em Invent. Math.}, 197(2):361--416, 2014.

\bibitem[Par19]{par}
Shalin Parekh.
\newblock The {KPZ} limit of {ASEP} with boundary.
\newblock {\em Commun. Math. Phys.}, 365(2):569--649, 2019.

\bibitem[Par22]{par2}
Shalin Parekh.
\newblock Positive random walks and an identity for half-space {SPDE}s.
\newblock {\em Electron. J. Probab.}, 27:1--47, 2022.

\bibitem[PS02]{prahofer}
Michael Pr{\"a}hofer and Herbert Spohn.
\newblock {Scale invariance of the PNG droplet and the Airy process}.
\newblock {\em J. Stat. Phys.}, 108:1071--1106, 2002.

\bibitem[PSW82]{phy1}
Rahul Pandit, M~Schick, and Michael Wortis.
\newblock Systematics of multilayer adsorption phenomena on attractive substrates.
\newblock {\em Phys. Rev. B}, 26(9):5112, 1982.

\bibitem[RV25]{rv21}
Mustazee Rahman and B{\'a}lint Vir{\'a}g.
\newblock Infinite geodesics, competition interfaces and the second class particle in the scaling limit.
\newblock {\em Ann. Inst. H. Poincar\'e Probab. Statist.}, 61(2):1075--1126, 2025.

\bibitem[Sep12]{timo}
Timo Sepp{\"a}l{\"a}inen.
\newblock Scaling for a one-dimensional directed polymer with boundary conditions.
\newblock {\em Ann. Probab.}, 40(1):19--73, 2012.

\bibitem[Ser23]{serio}
Christian Serio.
\newblock Tightness of discrete {G}ibbsian line ensembles.
\newblock {\em Stoch. Process. Their Appl.}, 159:225--285, 2023.

\bibitem[SI04]{sis}
Tomohiro Sasamoto and Takashi Imamura.
\newblock Fluctuations of the one-dimensional polynuclear growth model in half-space.
\newblock {\em J. Stat. Phys.}, 115(3):749--803, 2004.

\bibitem[SV21]{sv21}
Sourav Sarkar and B{\'a}lint Vir{\'a}g.
\newblock {B}rownian absolute continuity of the {KPZ} fixed point with arbitrary initial condition.
\newblock {\em Ann. Probab.}, 49(4):1718--1737, 2021.

\bibitem[Wu20]{wu}
Xuan Wu.
\newblock Intermediate disorder regime for half-space directed polymers.
\newblock {\em J. Stat. Phys.}, 181(6):2372--2403, 2020.

\bibitem[Wu23a]{wuconv}
Xuan Wu.
\newblock Convergence of the {KPZ} line ensemble.
\newblock {\em Int. Math. Res. Not.}, 2023(22):18901--18957, 2023.

\bibitem[Wu23b]{wu23}
Xuan Wu.
\newblock The {KPZ} equation and the directed landscape.
\newblock {\em arXiv:2301.00547}, 2023.

\bibitem[Wu23c]{wu2}
Xuan Wu.
\newblock Tightness of discrete {G}ibbsian line ensembles with exponential interaction {H}amiltonians.
\newblock In {\em Ann. Inst. H. Poincar\'e Probab. Stat.}, volume~59, pages 2106--2150. Institut Henri Poincar{\'e}, 2023.

\bibitem[Zha24]{zhang}
Xincheng Zhang.
\newblock {TASEP in half-space}.
\newblock {\em arXiv preprint arXiv:2409.09974}, 2024.

\bibitem[Zho26]{zhou}
Zhengye Zhou.
\newblock Uniform convergence of pfaffian point process to the airy line ensemble.
\newblock {\em Electronic Journal of Probability}, 31:1--27, 2026.

\end{thebibliography}

\end{document}